 \documentclass[11pt]{amsart}
\usepackage{amsmath,amssymb}
\usepackage{graphicx}
\usepackage{amsfonts,amscd,latexsym,epsfig,oldgerm,
psfrag}
\usepackage[all]{xy}

\newcommand{\Fr}{{\rm Fr}}
\newcommand{\cl}{{cl}}
\newcommand{\er}{{\Diamond}}
\newcommand{\E}{{\mathbb E}}

\newcommand{\Exp}{{\rm Exp}}
\newcommand{\orr}{{\mathfrak o}}
\newcommand{\Br}{{\rm Br}}

\newcommand{\s}{{\mathfrak s}}

\newcommand{\f}{{\mathfrak f}}
\newcommand{\lm}{\Lambda^{\rm max}\,}

\newcommand{\TY}{{\Ti Y}}
\newcommand{\Tx}{{\Ti x}}
\newcommand{\Ham}{{\rm Ham}}
\newcommand{\ul}{\underline}
\newcommand{\uLa}{{\ul{\La}}}
\newcommand{\Obj}{{\rm Obj}}
\newcommand{\proj}{{\rm proj}}

\newcommand{\Stab}{{\rm Stab}}
\newcommand{\Ti}{\widetilde}
\newcommand{\Tnu}{{\widetilde{\nu}}}
\newcommand{\TC}{{\widetilde{C}}}

\newcommand{\Tmu}{{\widetilde{\mu}}}
\newcommand{\TGa}{{\widetilde{\Ga}}}
\newcommand{\pr}{{\rm pr}}
\newcommand{\bB}{{\bf B}}
\newcommand{\bb}{{\bf b}}
\newcommand{\bC}{{\bf C}}
\newcommand{\bG}{{\bf G}}
\newcommand{\bM}{{\bf M}}
\newcommand{\bV}{{\bf V}}
\newcommand{\bX}{{\bf X}}
\newcommand{\bZ}{{\bf Z}}

   \newcommand{\ubE}{{\und \bE}}
  \newcommand{\ubB}{{\und \bB}}
 
  \newcommand{\upr}{{\und {\rm pr}}}
  \newcommand{\us}{{\und s}}

 \newcommand{\un}{{\underline{n}}}
 \newcommand{\ux}{{\underline{x}}}
  \newcommand{\upi}{{\underline{\pi}}}
    \newcommand{\urho}{{\underline{\rho}}}

  \newcommand{\uA}{{\underline{A}}}
  
  \newcommand{\uU}{{\underline{U}}}
  
    \newcommand{\ubK}{{\underline{\bK}}}
   \newcommand{\uV}{{\underline{V}}}
  \newcommand{\uW}{{\underline{W}}}
   \newcommand{\uphi}{{\underline{\phi}}}
      \newcommand{\upsi}{{\underline{\psi}}}
\newcommand{\uKk}{{\underline{\Kk}}}

\newcommand{\lub}{{\rm l.u.b.}}
\newcommand{\rT}{{\rm T}}
\newcommand{\rd}{{\rm d}}

\newcommand{\oo}{{\mathfrak o}}
\newcommand{\Mor}{{\rm Mor}}

\newcommand{\Tphi}{{\Tilde \phi}}

\newcommand{\ind}{{\rm ind\,}}

\newcommand{\und}{\underline}
\renewcommand{\Hat}{\widehat}

\newcommand{\nodes}{{\it  nodes}}

\newcommand{\PSL}{{\rm PSL}}
\newcommand{\pbar}{{\ov {\p}_J}}
\newcommand{\Cc}{{\mathcal C}}
\newcommand{\Dd}{{\mathcal D}}
\newcommand{\Kk}{{\mathcal K}}
\newcommand{\Ff}{{\mathcal F}}

\newcommand{\forget}{{\it forget}}

\newcommand{\im}{{\rm im\,}}
\newcommand{\less}{{\smallsetminus}}

\newcommand{\supp}{{\rm supp\,}}
\newcommand{\TU}{{\Tilde U}}
\newcommand{\p}{{\partial}}
\newcommand{\al}{{\alpha}}

\newcommand{\be}{{\beta}}

\newcommand{\om}{{\omega}}
\newcommand{\eps}{{\varepsilon}}
\newcommand{\de}{{\delta}}
\newcommand{\De}{{\Delta}}
\newcommand{\ga}{{\gamma}}
\newcommand{\Ga}{{\Gamma}}
\newcommand{\io}{{\iota}}

\newcommand{\la}{{\lambda}}
\newcommand{\La}{{\Lambda}}
\newcommand{\si}{{\sigma}}
\newcommand{\Si}{{\Sigma}}

\newcommand{\Aa}{{\mathcal A}}

\newcommand{\Uu}{{\mathcal U}}
\newcommand{\Bb}{{\mathcal B}}
\newcommand{\Ww}{{\mathcal W}}

\newcommand{\Mm}{{\mathcal M}}
\newcommand{\Ss}{{\mathcal S}}
\newcommand{\Tt}{{\mathcal T}}
\newcommand{\oMm}{{\overline {\Mm}}}
\newcommand{\ov}{\overline}

\newcommand{\id}{{\rm id}}

\renewcommand{\Tilde}{\widetilde}

\newcommand{\TV}{{\Tilde V}}

\newcommand{\TJ}{{\Tilde J}}
\newcommand{\Tsi}{{\Tilde \si}}

\newcommand{\Ee}{{\mathcal E}}
\newcommand{\Ii}{{\mathcal I}}

\newcommand{\Qq}{{\mathcal Q}}
\newcommand{\Vv}{{\mathcal V}}

\newcommand{\N}{{\mathbb N}}

\newcommand{\Q}{{\mathbb Q}}
\newcommand{\R}{{\mathbb R}}
\newcommand{\C}{{\mathbb C}}

\newcommand{\Z}{{\mathbb Z}}
\newcommand{\CP}{{\mathbb CP}}
\newcommand{\Hom}{{\rm Hom}}

\newcommand{\Map}{{\rm Map}}
\newcommand{\Nn}{{\mathcal N}}
\newcommand{\Pp}{{\mathcal P}}
\newcommand{\Hh}{{\mathcal H}}
\newcommand{\codim}{{\rm codim\,}}
\newcommand{\ev}{{\rm ev}}
\newcommand{\SSS}{{\smallskip}}

\newcommand{\bE}{{\bf E}}
\newcommand{\bK}{{\bf K}}
\newcommand{\bz}{{\bf z}}
\newcommand{\bn}{{\bf n}}
\newcommand{\ba}{{\bf a}}
\newcommand{\bP}{{\bf P}}
\newcommand{\bw}{{\bf w}}

\newtheorem{theorem}{Theorem}[subsection]
\newtheorem{thm}[theorem]{Theorem}

\newtheorem{cor}[theorem]{Corollary}
\newtheorem{lemma}[theorem]{Lemma}
\newtheorem{proposition}[theorem]{Proposition}
\newtheorem{prop}[theorem]{Proposition}

\newtheorem{defn}[theorem]{Definition}
\newtheorem{example}[theorem]{Example}

\newtheorem{rmk}[theorem]{Remark}

\numberwithin{figure}{subsection}
\numberwithin{equation}{subsection}

\newcommand{\MS}{{\medskip}}

\newcommand{\NI}{{\noindent}}

   \newcounter{qcounter}
\newenvironment{enumilist}
   { \begin{list} {(\roman{qcounter})\;}{\usecounter{qcounter}
     \setlength{\itemsep}{.5ex} \setlength{\leftmargin}{4.2ex} } }
   { \end{list} }

\newenvironment{itemlist}
   { \begin{list} {$\bullet$}
         {  \setlength{\itemsep}{.5ex} \setlength{\leftmargin}{2.5ex} } }
   { \end{list} }

\newcommand\quotient[2]{
        \mathchoice
            {
                \text{\raise1ex\hbox{$#1$}\Big/\lower1ex\hbox{$#2$}}%
            }
            {
                #1\,/\,#2
            }
            {
                #1\,/\,#2
            }
            {
                #1\,/\,#2
            }
    }
\newcommand\quot[2]{
                \text{\raise1ex\hbox{$#1$}/\lower1ex\hbox{$\scriptstyle#2$}}
  }

\newcommand\qu[2]{
                \text{\raise.8ex\hbox{$\scriptstyle#1\!$}/\lower.8ex\hbox{$\!\scriptstyle#2$}}
  }

\newcommand\ql[2]{
                \text{\lower.6ex\hbox{$\scriptstyle#1\!$}$\backslash$ \raise.6ex\hbox{$\!\scriptstyle#2$}}
  }

\newcommand\qq[2]{
                \text{\raise.8ex\hbox{$#1\!$}/\lower.8ex\hbox{$#2$}}
}

 \title{
 Notes on Kuranishi atlases}
 \author{Dusa McDuff}
 \address{Department of Mathematics,
 Barnard College, Columbia University}
\email{dusa@math.columbia.edu}
\thanks{partially supported by NSF grant DMS 1308669}
\date{May 17, 2015, minor revisions September 2016}

\begin{document}

\maketitle

\tableofcontents
\section{Introduction}\label{s:intro}

These notes aim to explain a joint project with Katrin Wehrheim \cite{MW1,MW2,MWiso,MWgw} that  uses finite dimensional reductions to  construct
a virtual fundamental class (VFC) for 
Gromov--Witten moduli spaces $X$ 
of closed genus zero curves.   Here $X$ is  a compactified moduli space of $J$-holomorphic stable maps, and so can locally 
 be described as the zero set of a Fredholm operator on the space of sections of a bundle over a nodal Riemann surface,
modulo the action of a Lie group that acts with finite stabilizers. 
Thus in the best case scenario $X$ would be a compact finite dimensional orbifold, which
carries a natural orientation 
and so has a fundamental class.  
However, in practice, $X$ usually has a more complicated structure since the operator is not transverse to zero.
Intuitively the fundamental class is therefore the zero set of a suitable perturbation of the Fredholm operator: all the difficulty in
constructing it lies in finding a suitable framework in which to build this perturbation.

Many of the possible approaches to this question are discussed  in \cite{MW2}.  
For example, in the polyfold approach Hofer--Wysocki--Zehnder develop a radically new analytic framework in order to build a suitable ambient space that contains all the relevant objects.  It is then  relatively easy to perturb the section to get a transverse zero set.  In contrast, we use  finite dimensional reduction which  constructs a local model (or chart) for small open subsets $F\subset X$ via traditional analysis, and then builds a suitable ambient space from these local charts using 
some relatively nontrivial point set topology.\footnote
{
There seems to be a law of conservation of difficulty:  in each approach the intrinsic difficulties in the problem
appear  in different guise, and  are solved either using more analysis as in \cite{HWZ1,HWZ2,HWZ3,HWZ}, or more topology as here,
or more sheaf theory as in \cite{Pard}.}
Our method is based on work by Fukaya--Ono~\cite{FO} and Fukaya--Oh--Ohta--Ono~\cite{FOOO};
see also \cite{FOOO12}.  
However we have  reformulated their ideas in order both to build the virtual neighbourhood of $X$ and to
clarify the formal structures underlying the 
construction.  We  make explicit  all important choices (of tamings, shrinkings and reductions), 
thus creating tools with which to give an explicit proof 
 that the virtual class $[X]^{vir}_\Kk$ is independent of these choices.

We start from the idea that there is a good local model for a sufficiently small open subset  $F$ of 
$X$,  that we call a {\bf basic chart} $\bK$ with {\bf footprint}
$F\subset X$. Since typically there is no direct map from one basic chart to another we relate them via {\bf  transition data}
given by \lq\lq transition (or sum) charts" and coordinate changes. 
 A {\bf Kuranishi atlas} $\Kk$  is made from a finite covering family of basic charts together with suitable transition data.

The papers~\cite{MW1,MW2,MWiso,MWgw}  prove the following theorems.
\MS

\NI {\bf Theorem A.}\,\,{\it  Let $(M^{2n},\om,J)$ be a $2n$-dimensional symplectic manifold with tame almost complex structure $J$, let $\oMm_{g,k}(A,J)$ be the  compact space of 
 nodal $J$-holomorphic genus $g$ stable maps in class $A$ with $k$ marked points 
 modulo reparametrization, and let $d= 2n(1-g) + 2c_1(A) + 2k + 6(g-1)$.
Then $X: = \oMm_{g,k}(A,J)$ has an oriented,  $d$-dimensional, 
weak SS Kuranishi atlas $\Kk$ that is well defined modulo oriented concordance (i.e. cobordism over $ [0,1]\times X$).}

\MS

\NI {\bf Theorem B.}\,\,{\it
Let $\Kk$ be an oriented,  $d$-dimensional, 
weak, SS Kuranishi atlas on a compact metrizable space $X$. Then $\Kk$ determines a cobordism class of oriented, compact weighted branched topological manifolds, and an element 
$[X]^{vir}_\Kk$ in the \v{C}ech homology group $\check{H}_d(X;\Q)$.
Both depend only on the oriented cobordism class of $\Kk$.}
\MS

\begin{itemlist}\item
 If the curves in $X$ have trivial isotropy and have smooth (i.e. non nodal) domains, we  construct the invariant 
as an oriented cobordism class of compact smooth manifolds, and then take an appropriate inverse limit to get the \v{C}ech homology class.\footnote
{
We could get an integral  class if we used Steenrod homology as developed in Milnor~\cite{Mil}.  However, as noted in \cite[Remark~8.2.4]{MW2},
integral \v{C}ech homology does not even satisfy the basic axioms of a homology theory. See \cite{IP} for an illuminating discussion of these different homology theories.}
 If there is nontrivial isotropy  we construct 
it as an oriented cobordism class of weighted branched (smooth) manifolds $(Z_\Kk, \La_\Kk)$, 
that is the realization of an explicit   weighted nonsingular branched (wnb) groupoid; see Definition~3.4.5 ff.
This abstract scheme applies equally well whether or not the  isotropy acts effectively.  
It also applies in the nodal case, though in this situation one needs a gluing theorem in order to
build the charts.
\item  Cobordism classes of oriented weighted branched manifolds contain more topological information than the fundamental class. For example, Pontriagin numbers are cobordism invariants; see \cite[Remark~4.7]{Morb}.
\item One aim  of our project 
is to construct the virtual moduli cycle using  traditional tools as far as  possible, and in particular to prove Theorem A using the approach to gluing in \cite{JHOL}.  This  provides continuity of the gluing map  as the gluing parameter $\ba$ converges to zero, but gives no control over derivatives with respect to 
$\ba$ near $\ba=0$.
With this approach,
 the charts are only weakly stratified smooth (abbreviated SS), i.e. they are topological spaces 
 that are unions of even dimensional, smooth strata.
 As we explain in \S\ref{ss:SS}, this
introduces various complications into the arguments, and specially into the construction of perturbation sections for 
 Kuranishi atlases of dimension $>1$.\footnote
 {
In the Gromov--Witten case all lower strata have codimension $\ge 2$, which means that 
in most situations one can avoid these complications by cutting down dimensions by intersecting with appropriate cycles; see \S\ref{ss:var}.}  
On the positive side it means that there is no need to change the usual smooth structure of  Deligne-Mumford space 
or of the moduli spaces $X$ of $J$-holomorphic curves by choosing a gluing profile, which is the  approach both of Fukaya et al and Hofer--Wysocki--Zehnder. 
This part of the project is not yet complete.  Hence in these notes we will either restrict to
the case $d=0$ or will
assume the existence of a gluing theorem
that provides at least $\Cc^1$ control.  
\item  The construction in Theorem~A of the atlas $\Kk$ on  $\oMm_{g,k}(A,J)$ 
extends in a natural way to a construction of a cobordism atlas on the moduli space $\bigcup_{t\in [0,1]}\oMm_{g,k}(A,J_t)$ for paths $(J_t)_{t\in [0,1]}$  of $\om$-tame almost complex structures.  Hence the weighted branched manifolds that represent 
 $[\Mm_{g,k}(A,J_\al)]^{vir}_\Kk$  for $\al = 0,1$ are oriented cobordant.  
 It follows that all  invariants constructed using the class 
 $[\Mm_{g,k}(A,J)]^{vir}_\Kk$ are independent of the choice of $\om$-tame $J$; see Remark~\ref{rmk:Jindep}.  
\end{itemlist}

We begin by developing the abstract theory of Kuranishi atlases, that is on proving theorem~B for smooth  atlases.
The first two sections of these notes  give precise  statements of the main definitions and 
results from \cite{MW1,MW2,MWiso}, and sketches of the 
most important proofs.  For simplicity we first discuss
 the smooth case with trivial isotropy and then  the case  of
nontrivial isotropy, where we allow arbitrary, even noneffective, actions of finite groups. 
We end Section~\ref{s:iso} with  some notes on the stratified smooth (SS) case.

The rest of these notes are more informal, explaining how the theory can be used in practice. 
Section~\ref{s:GW}  outlines the construction of atlases for genus zero Gromov--Witten moduli spaces,
explaining the setup  but omitting most analytic details
needed for a full proof of Theorem~A.
Some of these can be found in \cite{MW2}, though gluing will be treated in \cite{MWgw}.  See also Castellano~\cite{Castell1,Castell2} that completes 
the construction of a $\Cc^1$-gluing theorem in the genus zero case, and also establishes validity of the Kontsevich--Manin  axioms in this case. 
We restrict to  genus zero  here since in this case the relevant Deligne--Mumford space $\oMm_{0,k}$ 
can be understood simply in terms of cross ratios,
which makes the Cauchy--Riemann equation much easier to write down explicitly. The
 argument should  adapt without difficulty  to the higher genus case.  For example, one could no doubt
 use the approach 
in  Pardon~\cite{Pard}, where atlases that are
quite similar  to ours are constructed in all genera.  In \cite{MWgw} we will explain a construction that works for all genera, 
but that relies on polyfold theory.
 In \S\ref{ss:var} we formulate the notion of a GW atlas, aiming to get a better handle 
 on the uniqueness properties of the atlases that we construct.

Section~\ref{sec:ex} discusses some examples. We begin with the case of  atlases with trivial obstruction spaces, 
whose underlying space $X$ is therefore an orbifold, explaining some results in
\cite{Morb}.
We show that in this case the theory is equivalent to the standard way of thinking of an orbifold as 
 the realization of an ep groupoid; see  Moerdijk~\cite{Moe}. 
   \MS
 
\NI {\bf Proposition C.}  {\it  Each compact orbifold 
has a Kuranishi atlas with trivial obstruction spaces. 
Moreover, there is a bijective correspondence between commensurability classes of 
such Kuranishi atlases and Morita equivalence classes of ep groupoids.}
\MS

The proof constructs from each groupoid representative $\bG$ of $X$ an atlas 
that maps to an explicitly describable subcategory (i.e. submonoid) of $\bG$ with realization $X$.
Thus, this atlas is a simpler model for the groupoid that captures all  essential information.
 
We then discuss some examples from  Gromov--Witten theory. 
In particular, we show in \S\ref{ss:nontriv} that if the space $X$ of equivalence classes of stable maps is a 
compact orbifold with obstruction orbibundle $E$ then the method in \S\ref{s:GW}  builds a GW
 atlas $\Kk$ such that $[X]^{vir}_\Kk$ is simply the Euler class of $E$.  
Secondly,  in \S\ref{ss:S1}
we  use Kuranishi atlases to prove a result claimed in \cite{Mcq} about the
vanishing of certain
 two point GW invariants  of the product manifold $S^2\times M$.  This result is a crucial step in establishing
 the properties of the Seidel representation of $\pi_1(\Ham(M))$ in the quantum homology of $M$, 
 where $\Ham(M)$ is the Hamiltonian group of $(M,\om)$; see  Tseng--Wang~\cite{TW} for a different approach to this question
  using Kuranishi structures.

Finally, in Section~\ref{s:order} we discuss some modifications of the basic definitions that are useful when considering products.
The point here is that the product of two Kuranishi atlases is not an atlas in the sense of our original definition. On
the other hand, there are many geometric situations (such as in Floer theory) in which one wants to use induction 
to build
 a family of atlases on 
spaces  where the boundary of one space $X$ is the product of two or more spaces 
 that, via an inductive process, are already provided with atlases.  We explain how to do this in our current framework, 
by generalizing the order structure underlying our notion of atlas. Some details of the adjustments  needed to deal with 
 Hamiltonian Floer theory may be found in Remark~\ref{rmk:HamFT}.
 
\subsection{Outline of the main ideas.}
As explained above,
a {\bf Kuranishi atlas} $\Kk$  is made from a finite covering family of basic charts, that are related to each other
 via  transition charts and coordinate changes.   
Our first aim is to unite all these charts into a {\bf category} $\bB_\Kk$, akin to the \'etale proper  groupoids\footnote
{
The definition of a topological category is given at the beginning of \S\ref{ss:K2}, while
ep groupoids are defined at the beginning of \S\ref{ss:orb} with examples  in Example~\ref{ex:foot}. Besides fitting in well with the notion of an orbifold as an ep groupoid, we will see that this categorical language provides a succinct way to describe the appropriate equivalences  between the different charts in the atls $\Kk$. }
 often used to model orbifolds.  If we ignore questions of smoothness, 
the space of objects $\Obj_{\bB_\Kk}$ of such a topological category  is the disjoint union $\bigsqcup_I U_I$ 
of smooth manifolds of different dimensions. There are at most a finite number of morphisms between any two points.  Therefore the space $|\Kk|$ obtained by quotienting  $\Obj_{\bB_\Kk}$  by the equivalence relation generated by the morphisms  looks something like an orbifold.  In fact, in good cases this space, called the {\bf virtual neighbourhood of $X$},  is a finite union of (non disjoint) orbifolds; see  Remark~\ref{rmk:Kk}. 
This ambient space  $|\Kk|$  supports a \lq\lq bundle"  $|\pr|: |\bE_\Kk|\to |\Kk|$ with canonical section $|\s|:  |\Kk|\to  |\bE_\Kk|$. The latter is the finite dimensional remnant of the original Fredholm operator, and its zero set can be canonically identified with a copy $\io_\Kk(X)$ of $X$.   Hence the idea is that the virtual moduli cycle $[X]^{vir}_\Kk$  should be represented by the zero set of a perturbed  (multi)section $\s+\nu$ that is chosen to be transverse to zero.

\begin{example}\rm   
If $X$ were a manifold, we could take each basic chart to be an open subset $F_i\subset X$, while the \lq\lq transition chart" relating $F_i$ to $F_j$ would simply be the intersection $F_{ij}: = F_i\cap F_j$.
In this case the atlas $\Kk$ would consist of a finite open covering $(F_i)_{i=1,\dots,N}$ of $X$ together with  the collection of nonempty intersections $\bigl(F_I: = \cap_{i\in I}F_i\bigr)$ 
related to each other by the obvious inclusions.
%
As we will see, one advantage of the categorical framework is that it gives a succinct way of expressing the compatibility conditions between all spaces and maps of interest.
\hfill$\er$
\end{example}

The needed abstract structure 
is easiest to understand if we assume that there are no nodal curves and that all isotropy groups are trivial.
Therefore we begin in  \S\ref{s:noiso} by considering smooth atlases with trivial isotropy.  We consider nontrivial  isotropy in \S\ref{s:iso},  briefly discussing the 
modifications needed for nodal case in 
\S\ref{ss:SS}.  
\MS

We now outline the  main steps in the construction of $[X]^{vir}_\Kk$.
\begin{itemlist}\item
The first difficulty in realizing this idea is that in practice one cannot actually construct atlases; instead one constructs a weak atlas, which is like an atlas except that one has less control of the domains of the charts and coordinate changes.
But a weak atlas does not even define a  category, let alone one 
 whose realization $|\bB_\Kk|= \Obj_{\Bb_\Kk}/\!\!\sim \; =: |\Kk|$ has good topological properties. For example,  we would like $|\Kk|$ to be Hausdorff and (in order to make local constructions possible) for the projection $\pi_\Kk: U_I\to |\Kk| $ to be a homeomorphism to its image.
 
In \S\ref{ss:K2} we formulate the {\bf taming conditions} for a weak atlas.  Our main results are:
\begin{itemize}\item[-]Proposition~\ref{prop:Khomeo}, which shows that the realization of a tame atlas 
 has these good  topological properties, and 
\item[-] Proposition~\ref{prop:proper1}, which shows that every  weak smooth atlas can be tamed.
\end{itemize}
We prove these results in \S\ref{ss:tatlas}  for  {\bf filtered topological atlases}, 
in order that they also apply  to the case of smooth atlases with nontrivial isotropy.  
Here,  filtration is a generalization of the notion of additivity that is already built into the definition of an atlas.
As explained in Remark~\ref{rmk:add1}, some version of this condition is 
 crucial here.
(See \S\ref{s:order} where this notion of additivity is weakened to a notion that is compatible with products.)  
 \item 
The taming procedure gives us two categories $\bB_{\Kk}$ and $\bE_{\Kk}$ with a projection functor $\pr: \bE_{\Kk}\to \bB_{\Kk}$
(the \lq\lq obstruction bundle")  and section functor $\s: \bB_{\Kk}\to \bE_{\Kk}$ (defined by the Cauchy--Riemann operator).  However the category has too many morphisms (i.e. the 
 chart domains 
 overlap too much) for us to be able to construct a perturbation functor $\nu: \bB_{\Kk}\to \bE_{\Kk}$ such that $\s+\nu$
 is transverse to zero (written $\s+\nu\pitchfork 0$).  
We therefore  pass to a full subcategory 
$\bB_{\Kk}|_{\Vv}$ of $\bB_{\Kk}$ with objects $\Vv: = \bigsqcup V_I$ that does support suitable functors $\nu: 
\bB_{\Kk}|_{\Vv}\to \bE_{\Kk}|_{\Vv}$; see 
Definition~\ref{def:vicin}.  This subcategory $\bB_{\Kk}|_{\Vv}$ is called a {\bf reduction} of $\Kk$. Its realization   $ |\bB_{\Kk}|_{\Vv}|$ injects
continuously  onto
the subspace $\pi_\Kk(\Vv)\subset |\Kk|$.
Constructing it is similar to passing from the covering of a triangulated space by the stars of its vertices to the covering by the stars of its first barycentric subdivision.   It is the analog of a \lq\lq good coordinate system"  (now called a \lq\lq dimensionally graded system" in \cite{TF}) in the theory of Kuranishi structures.
 
\item We next define the notion of a reduced perturbation  $\nu: \bB_{\Kk}|_\Vv\to \bE_{\Kk}|_\Vv$ of $\s$ (see  Definition~\ref{def:sect}),
and show that, if $\nu$ is precompact in a suitable sense, the realization  $\big|(\s|_\Vv+\nu)^{-1}(0)\big|$ of the
 zero set is compact.  
 The intricate construction of $\nu$ is one of the most difficult parts of the general theory.  To explain the ideas, we give a fairly detailed description in
 Proposition~\ref{prop:ext},  though still do not do quite enough 
 for a complete proof.   
 In the trivial isotropy case the zero set  is a closed submanifold of $|\Kk|$ lying in the precompact \lq\lq neighbourhood"\footnote
{
In fact, $\io_\Kk(X)$ does {\it not} have a compact neighbourhood in $|\Kk|$; 
as explained in Remark~\ref{rmk:piVv} we should think of  $\pi_\Kk(\Vv)$ 
as the closest we can come to  a compact neighbourhood of $\io_\Kk(X)$.}   
$\pi_\Kk(\Vv) \subset |\Kk|$ of $\io_\Kk(X)$.  
The final step is to
construct the fundamental class $[X]^{vir}_\Kk$
 from this zero set.  We define this class to  lie in rational \v Cech homology rather than the more familiar singular  theory,  because the former  theory has the needed continuity properties under inverse limits.\footnote
 {
 One could obtain an integral fundamental class in the Steenrod homology developed in \cite{Mil}.  However, rational \v Cech homology is simply the homology theory dual to   rational  \v Cech cohomology, and so is easier to understand; see \cite[Remark~8.2.4]{MW2}. Also, 
 one must use rational coefficients if there is nontrivial isotropy.}
 
\item As we will see in \S\ref{s:iso} the above ideas adapt readily to the case of nontrivial isotropy via the notion of the {\bf intermediate category}, which is a filtered topological category and so can be tamed by the results in \S\ref{ss:tatlas}.
Although the needed perturbation $\nu$ is multivalued when considered as a section of $|\bE_\Kk|\to |\Kk|$, it is the realization of a {\it single valued} map  $\nu: \Vv\to \bE_\Kk|_\Vv$, which satisfies some compatibility conditions between charts but is {\it not} a functor.\footnote
{
However $\nu$ is a functor from the \lq\lq pruned" category  $\bB_\Kk|_{\Vv}^{\less \Ga}$ to $\bE_\Kk|_\Vv$, where  $\bB_\Kk|_{\Vv}^{\less \Ga}$ is a (non-full!) subcategory of $\bB_\Kk|_{\Vv}$ obtained by discarding appropriate morphisms.} 
This allows us to give
 a very explicit description of the zero set
  $(\s|_\Vv+  \nu)^{-1}(0)$, which forms an \'etale (but nonproper)  weighted nonsingular groupoid in a  natural and functorial way; see Proposition~\ref{prop:zero}.
\item
Of course, to obtain a fundamental class one also needs to discuss  orientations, and in order to prove uniqueness
of this class one also needs to set up an adequate  cobordism theory.   Cobordisms are discussed briefly in  \S\ref{ss:K2} (see Definition~\ref{def:CKS} ff.) and orientations in \S\ref{ss:iso2} (also see Definition~\ref{def:or}).
\end{itemlist}

\begin{rmk}\rm  (i)  
Note that although the cobordism relation is all one needs when proving the uniqueness of $[X]^{vir}_\Kk$ since this is just a homology class, 
it does not seem to be the \lq\lq correct" relation, in the sense that rather different moduli problems might well give rise to 
cobordant atlases.      
The construction in \S\ref{ss:GW} for an atlas on a fixed GW moduli spaces $X$ builds an atlas whose  {\bf commensurability class}
 (see Definition~\ref{def:commen}) 
is independent of all choices. 
However, the construction  involves the use of some geometric procedures (formalized in Definition~\ref{def:GWA}
as the notion of a GW atlas) that have no abstract description.   Therefore commensurability 
is probably not the optimal relation either, though (as shown by Proposition C) it is optimal if all obstruction spaces vanish.
It may well be that Joyce's notion of a Kuranishi space~\cite{Joyce} best captures  the Fredholm index condition on $X$, at least in the smooth context; see also Yang~\cite{Yang}.
The aim of our work is not to tackle such an abstract problem, but to develop a complete and explicit theory that can be used in practice to construct and calculate GW invariants. 
\MS

\NI (ii) Pardon's very interesting approach to the construction of the GW virtual fundamental 
class in~\cite{Pard} uses atlases that have many of the
 features of the theory presented here.   In particular, his notion of implicit atlas (developed on the basis of \cite{MW}) includes  transition charts and coordinate changes  
that are essentially the same as ours.  However he avoids making choices by considering 
{\it all} charts, and he avoids the taming problems
we encounter,  firstly by considering {\it all} solutions to  the given 
equation with specified domain, and secondly by using a different more algebraic way 
to define the VFC (via a version of sheaf theory and some homological algebra) 
that does not involve considering the quotient space $|\Kk|$. Further, instead of using a reduction in order to thin out the footprint covering  $(F_I)_{I\in \Ii_\Kk}$  sufficiently to construct a perturbation section $\nu$, he relates the charts corresponding to different indices $I\subset J$  via a construction he calls \lq\lq deformation to the normal cone".
$\hfill\er$  
 \end{rmk}

\begin{rmk}\label{rmk:FOOOa} \rm  (i) As already said, our approach grew out of the work of Fukaya--Ono in \cite{FO} and 
Fukaya--Oh--Ohta--Ono in \cite{FOOO}.  In particular we use essentially the same Fredholm analysis that is implicit in these papers 
(though we make it much more explicit).  Similarly, our construction of charts is closely related to theirs, 
but more explicit and more global.  The main technical differences are the following:
\begin{itemlist}\item[-] We use
 \lq\lq Fredholm stabilization" (see Remark~\ref{rmk:stab1} in \S\ref{ss:GW}~(VII)) 
in order to avoid some genericity issues with the obstruction bundles. 
\item[-] We make a more explicit choice of the labeling of the added marked points, giving more control over the isotropy groups and permitting
the
construction of transition charts $\bK_I$ with  footprint equal to the whole intersection $F_I$.
\item[-] The development in \cite{FO,FOOO,TF} goes straight from a Kuranishi structure to a \lq\lq good coordinate system"  or DGS (in our language, accomplishing taming and reduction in one step), thus apparently bypassing some of the topological questions
involving in constructing the category $\bB_\Kk$ and virtual neighbourhood $|\Kk|$. 
\end{itemlist}
Our approach does involve developing more basic topology.  One advantage  
 is that we can  exhibit the virtual fundamental class $[X]^{vir}_\Kk$ as an element  
 of the homology of $X$ as stated in Theorem~A, rather than simply as a cobordism class of 
 \lq\lq regularized moduli spaces" ${\oMm}\,\!^\nu$ with a well defined 
 image under an appropriate extension of the natural  map
$X: = \oMm_{0,k}(M,A,J) \stackrel{{\rm stab}\times {\rm eval}}\longrightarrow \oMm_{0,k}\times M^k$.     This more precise definition of $[X]^{vir}_\Kk$
 informs the very different approach of Ionel--Parker~\cite{IoP}. See Remark~\ref{rmk:FOOO1}
  for further discussion of the relation between our work and that of Fukaya et al. Note finally that the
 new version of the Kuranishi structure approach explained in Tehrani--Fukaya~\cite{TF} would also permit the definition of $[X]^{vir}_\Kk$ as an element in the homology of $X$ via their notion of \lq\lq thickening".
  \MS
  
  \NI (ii)  Our ideas were first explained in the 2012 preprint~\cite{MW} which constructed the fundamental cycle for
smooth atlases with trivial isotropy.  This paper has now been rewritten and separated into the  two papers  \cite{MW1,MW2}, the first of which deals the topological questions in more generality than \cite{MW}, and the second of which constructs the fundamental cycle in the trivial isotropy case  (including a discussion of orientations).  A third paper \cite{MWiso}  extends these results to the case of smooth atlases with nontrivial isotropy.  The lectures \cite{McL} give an overview of the whole construction.
\hfill$\er$
\end{rmk}

\MS

\section{Kuranishi atlases with trivial isotropy}\label{s:noiso}

Throughout these notes $X$ is assumed to be a compact and metrizable space.
Further in this chapter we assume (usually without explicit mention) that the isotropy is trivial. 
The proof of Theorem~B in this case is completed at the end of \S\ref{ss:red}.
 For the general case see \S\ref{s:iso}. 

\subsection{Smooth Kuranishi charts, coordinate changes and atlases}\label{ss:K1}

This section gives basic definitions. 

\begin{defn}\label{def:chart}
Let $F\subset X$ be a nonempty open subset.
A {\bf (smooth) Kuranishi chart} for $X$ with {\bf footprint} $F$ (and trivial isotropy)  is a tuple $\bK = (U,E,s,\psi)$ consisting of
\begin{itemize}
\item
the {\bf domain} $U$, which is 
a smooth $k$-dimensional manifold;\footnote
{
We assume throughout that manifolds are second countable and without boundary, unless explicit mention is made to the contrary.}
\item
the {\bf obstruction space} $E$, which is a finite dimensional real vector space;
\item
the {\bf section} $\s:U\to U\times E, x\mapsto (x,s(x))$ which is given by a
smooth map $s: U\to E$;
 \item
the {\bf footprint map} $\psi : s^{-1}(0) \to X$, which is a homeomorphism to the footprint ${\psi(s^{-1}(0))=F}$, which is an open subset of $X$.
\end{itemize}
The {\bf dimension} of $\bK$ is $\dim \bK: = \dim U-\dim E$.
\end{defn}

\begin{figure}[htbp] 
   \centering
   \includegraphics[width=3in]{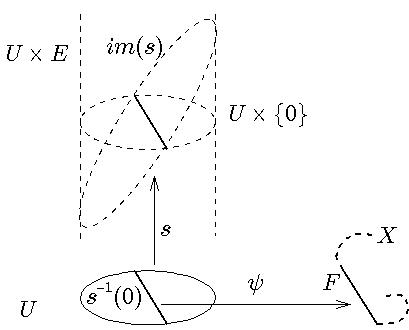}
   \caption{A Kuranishi chart}
   \label{fig:1}
\end{figure}

\begin{defn}\label{def:map} A {\bf map}  $\Hat\Phi : \bK \to \bK'$  between Kuranishi charts is a pair $(\phi,\Hat\phi)$ 
consisting of an embedding $\phi :U \to U'$ and a linear injection $\Hat \phi :E \to E'$ such that
\begin{enumerate}
\item the embedding restricts to $\phi|_{s^{-1}(0)}=\psi'^{-1} \circ\psi : s^{-1}(0) \to s'^{-1}(0)$, the transition map induced from the footprints in $X$;
\item
the embedding intertwines the sections, $s' \circ \phi  = \Hat\phi \circ s$, on the entire domain $U$.
\end{enumerate}
That is, the following diagrams commute:
\begin{equation}
 \begin{array} {ccc}
{U\times E}& \stackrel{\phi\times \Hat\phi} \longrightarrow &
{U'\times E'} 
\phantom{\int_Quark}  \\
\phantom{sp} \uparrow {\s}&&\uparrow {\s'} \phantom{spac}\\
\phantom{s}{U} & \stackrel{\phi} \longrightarrow &{U'} \phantom{spacei}
\end{array}
\qquad
 \begin{array} {ccc}
{s^{-1}(0)} & \stackrel{\phi} \longrightarrow &{s'^{-1}(0)} \phantom{\int_Quark} \\
\phantom{spa} \downarrow{\psi}&&\downarrow{\psi'} \phantom{space} \\
\phantom{s}{X} & \stackrel{{\rm id}} \longrightarrow &{X}. \phantom{spaceiiii}
\end{array}
\end{equation}
\end{defn}

The dimension of the obstruction space $E$ typically varies as the footprint $F\subset X$ changes.
Indeed, the maps $\phi, \Hat\phi$ need not be surjective.  However, as we will see in Definition \ref{def:change}, the maps allowed as coordinate changes are carefully controlled in the normal direction.

\begin{defn} \label{def:restr}
Let $\bK$ be a Kuranishi chart and $F'\subset F$
an open subset of the footprint.
A {\bf restriction of $\bK$ to $\mathbf{\emph F\,'}$} is a Kuranishi chart of the form
$$
\bK' = \bK|_{U'} := \bigl(\, U' \,,\, E'=E \,,\, s'=s|_{U'} \,,\, \psi'=\psi|_{s'^{-1}(0)}\, \bigr)
$$
given by a choice of open subset $U'\subset U$ of the domain such that $U'\cap s^{-1}(0)=\psi^{-1}(F')$.
In particular, $\bK'$ has footprint $\psi'(s'^{-1}(0))=F'$.
\end{defn}

By \cite[Lemma~5.1.6]{MW2}, we may restrict to any open subset of the footprint.
If moreover $F'\sqsubset F$ is precompact, then $U'$ can be chosen to be precompact in $U$, written $U'\sqsubset U$.

The next step is to construct a coordinate change $\Hat\Phi_{IJ}: \bK_I\to \bK_J$ between two charts with 
nested footprints $F_I\supset F_J$.   For simplicity
 we will formulate the definition in the situation 
that is relevant to Kuranishi atlases. That is, we suppose that a finite set of 
Kuranishi charts $(\bK_i)_{i\in \{1,\dots, N\}}$ is given such that for each  $I\subset \{1,\dots, N\}$
with $F_I: = \bigcap_{i\in I} F_i \ne \emptyset$ we have another Kuranishi chart 
$\bK_I$ (called a {\bf transition (or sum) chart}) with 
\begin{eqnarray}
\label{eq:sum0} 
\mbox{ obstruction space } E_I = {\textstyle \prod_{i\in I }}E_i, & \mbox{ and } & 
\mbox{ footprint } F_I: = {\textstyle \bigcap_{i\in I}} F_i. 
\end{eqnarray}

\begin{rmk}\label{rmk:sum}\rm
Since we assume in an atlas that $$
\dim U_I - \dim E_I =:\dim \bK_I = \dim \bK_i = \dim U_i - \dim E_i,\quad \forall i\in I,
$$ 
in general  the domain of the transition chart $U_I$ has dimension strictly larger than $\dim U_i$ for $i\in I$.  
Further, $U_I$ usually cannot be built in some topological way  from the $U_i$ (e.g. by taking products).  
Indeed in the Gromov--Witten situation $U_I$ consists (very roughly speaking) of 
pairs $((e_i)_{i\in I}, u)$ where $u$ is  a solution to an equation of the form $\pbar u = \sum_{i\in I} \la (e_i)$, 
and so cannot be made directly from the $U_i$, which for each $i$ consists of pairs $(e_i,u)$ of solutions to 
the individual equation $\pbar u = \la(e_i)$.
  Note also that we choose the obstruction spaces $E_i$ to cover the cokernel of the linearization of $\pbar$ at the points in $U_i$.  Thus each domain $U_i$ is a manifold 
that is cut out transversally by the equation.  Since the function $s_I:U_I\to E_I$  is the finite dimensional reduction of $\pbar$, 
for each $I\subset K$ the derivative  $\rd_v  s_K$ at a point $v\in \im \phi_{IK}$ has kernel contained in $\rT_v(\im \phi_{IK})$ and
cokernel that is covered by $\Hat\phi_{IK}(E_I)$.  This explains the index condition in Definition~\ref{def:change} 
and Remark~\ref{rmk:tbc} below. See  \S\ref{ss:GW}(VI) for more details.
$\hfill\er$   \end{rmk}

When $I\subset J$ we  write $\Hat\phi: = \Hat\phi_{IJ}: E_I\to E_J$ for  the natural inclusion, omitting it
where no confusion is possible.\footnote
{Note that 
the assumption $E_I = \prod_{i\in I }E_i$ means that the
 family is {\bf additive} in the sense of \cite[Definition~6.1.5]{MW2}.   Therefore all the atlases that we now consider
  are additive, and 
 for simplicity we no longer mention this condition explicitly. We discuss a  weakened version in \S\ref{s:order}. }

\begin{figure}[htbp] 
   \centering
   \includegraphics[width=4in]{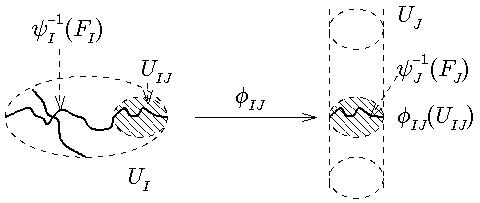}
   \caption{A coordinate change in which $\dim U_{J} =\dim U_I+ 1$.
   Both $U_{IJ}$ and its image $\phi_{IJ}(U_{IJ})$ are shaded.}
   \label{fig:2}
\end{figure}

\begin{defn}\label{def:change}
For  $I\subset J$, let
 $\bK_I$ and $\bK_J$ be Kuranishi charts as above, with domains $U_I, U_J$ and footprints $F_I\supset F_J$.
A {\bf coordinate change} from $\bK_I$ to $\bK_J$ with {\bf domain} $U_{IJ}$ is a map
$\Hat\Phi: \bK_I|_{U_{IJ}}\to \bK_J$, which satisfies the {\bf index condition} in (i),(ii) below, and whose domain 
is an open subset $U_{IJ}\subset U_I$ such that
$\psi_I(s_I^{-1}(0)\cap U_{IJ}) = F_J$.
\begin{enumerate}
\item
The embedding $\phi:U_{IJ}\to U_J$ underlying the map $\Hat\Phi$ identifies the kernels,
$$
\rd_u\phi \bigl(\ker\rd_u s_I \bigr) =  \ker\rd_{\phi(u)} s_J    \qquad \forall u\in U_{IJ};
$$
\item
the linear embedding $\Hat\phi:E_I\to E_J$ given by the map $\Hat\Phi$ identifies the cokernels,
$$
\forall u\in U_{IJ} : \qquad
E_I = \im\rd_u s_I \oplus C_{u,I}  \quad \Longrightarrow \quad E_J = \im \rd_{\phi(u)} s_J \oplus \Hat\phi(C_{u,I}).
$$
\end{enumerate}
\end{defn}

\begin{rmk}\label{rmk:tbc}\rm
We show in \cite[Lemma~5.2.5]{MW2} 
that the index condition is equivalent to the {\bf tangent bundle condition}, which requires isomorphisms for all $v=\phi(u)\in\phi(U_{IJ})$, 
\begin{equation}\label{tbc}
\rd_v s_J : \;\quotient{\rT_v U_J}{\rd_u\phi (\rT_u U_I)} \;\stackrel{\cong}\longrightarrow \; \quotient{E_J}{\Hat\phi(E_I)},
\end{equation}
or equivalently at all (suppressed) base points as above
\begin{equation}\label{inftame}
E_J=\im\rd s_J + \im\Hat\phi_{IJ} \qquad\text{and}\qquad
\im\rd s_J \cap \im\Hat\phi_{IJ} = \Hat\phi_{IJ}(\im\rd s_I).
\end{equation}
Moreover, the index condition implies that $\phi(U_{IJ})$ is an open subset of $s_J^{-1}(\Hat\phi(E_I))$,
and that the charts $\bK_I, \bK_J$ have the same dimension.
$\hfill\er$  
\end{rmk}

\begin{defn}\label{def:Kfamily}
Let $X$ be a compact metrizable space.
\begin{itemlist}
\item
A {\bf covering family of basic charts} for $X$ is a finite collection $(\bK_i)_{i=1,\ldots,N}$ of Kuranishi charts for $X$ whose footprints cover $X=\bigcup_{i=1}^N F_i$.
\item
{\bf Transition data} for a covering family $(\bK_i)_{i=1,\ldots,N}$ is a collection of Kuranishi charts $(\bK_J)_{J\in\Ii_\Kk,|J|\ge 2}$ and coordinate changes $(\Hat\Phi_{I J})_{I,J\in\Ii_\Kk, I\subsetneq J}$ as follows:
\begin{enumerate}
\item
$\Ii_\Kk$ denotes the set of 
nonempty
subsets $I\subset\{1,\ldots,N\}$ for which the intersection of footprints is nonempty,
$$
\Ii_\Kk: = \big\{ \emptyset\ne I\subset \{1,\ldots,N\} \; : \: F_I:= \; {\textstyle \bigcap_{i\in I}} F_i  \;\neq \; \emptyset \bigr\}\;;
$$
\item
$\bK_J$ is a Kuranishi chart for $X$ with footprint $F_J=\bigcap_{i\in J}F_i$ for each $J\in\Ii_\Kk$ with $|J|\ge 2$, and for one element sets $J=\{i\}$ we denote $\bK_{\{i\}}:=\bK_i$;
\item
$\Hat\Phi_{I J}$ is a coordinate change $\bK_{I} \to \bK_{J}$ for every $I,J\in\Ii_\Kk$ with $I\subsetneq J$.
\end{enumerate}
\end{itemlist}
 \end{defn}

According to Definition~\ref{def:change} the domain $U_{IJ}$  of $\Hat\Phi_{I J}$ is part of the
transition data for a covering family.
 Further, this data 
 automatically satisfies a cocycle condition on the zero sets since, due to the footprint maps to $X$, we have for $I\subset J \subset K$:
$$
\phi_{J K}\circ \phi_{I J}
= \psi_K^{-1}\circ\psi_J\circ\psi_J^{-1}\circ\psi_I
= \psi_K^{-1}\circ\psi_I
= \phi_{I K}
\qquad \text{on}\; s_I^{-1}(0)\cap U_{IK} .
$$
Further, the composite maps $\phi_{J K}\circ \phi_{I J}, \Hat\phi_{J K}\circ \Hat\phi_{I J} = \Hat\phi_{IK}$ automatically satisfy the
intertwining relations in Definition~\ref{def:map}.  Hence one can always define a composite coordinate change $\Hat\Phi_{JK}\circ \Hat\Phi_{IJ}$ from $\bK_I$ to $\bK_K$ with domain $U_{IJ}\cap \phi_{IJ}^{-1}(U_{JK})$. 
 (For details, see \cite[Lemma~5.2.7]{MW2} and \cite[Lemma~2.2.5]{MW1}.)  But 
 in general this domain  may have little relation to the domain $U_{IK}$ of $\phi_{IK}$, apart from the fact that these two sets have the same intersection
with the zero set $s_I^{-1}(0)$.
Since there is no natural ambient topological space into which the entire domains of the Kuranishi charts map, the cocycle condition on the complement of the zero sets has to be added as an axiom. There are three natural notions of cocycle condition with varying requirements on the domains of the coordinate changes.

\begin{defn}  \label{def:cocycle}
Let $\Kk=(\bK_I,\Hat\Phi_{I J})_{I,J\in\Ii_\Kk, I\subsetneq J}$ be a tuple of basic charts and transition data. Then for any $I,J,K\in\Ii_K$ with  $I\subsetneq J \subsetneq K$ we define the composed coordinate change $\Hat\Phi_{J K}\circ \Hat\Phi_{I J} : \bK_{I}  \to \bK_{K}$ as above with domain $\phi_{IJ}^{-1}(U_{JK})\subset U_{IJ}$.
We say that the triple of coordinate changes
$\Hat\Phi_{I J}, \Hat\Phi_{J K}, \Hat\Phi_{I K}$ satisfies the
\begin{itemlist}
\item {\bf weak cocycle condition}
if $\Hat\Phi_{J K}\circ \Hat\Phi_{I J} \approx \Hat\Phi_{I K}$, i.e.\ the coordinate changes are equal on the overlap; in particular if
\begin{equation*}
\qquad
\phi_{J K}\circ \phi_{I J} = \phi_{I K}
\qquad \text{on}\;\;
\phi_{IJ}^{-1}(U_{JK}) \cap U_{IK} ;
\end{equation*}
\item {\bf cocycle condition}
if $\Hat\Phi_{J K}\circ \Hat\Phi_{I J} \subset \Hat\Phi_{I K}$, i.e.\  $\Hat\Phi_{I K}$ extends the composed coordinate change; in particular if
\begin{equation}\label{eq:cocycle}
\qquad
\phi_{J K}\circ \phi_{I J} = \phi_{I K}
\qquad \text{on}\;\;
\phi_{IJ}^{-1}(U_{JK}) \subset U_{IK} ;
\end{equation}
\item {\bf strong cocycle condition}
if $\Hat\Phi_{J K}\circ \Hat\Phi_{I J} = \Hat\Phi_{I K}$ are equal as coordinate changes; in particular if
\begin{equation}\label{strong cocycle}
\qquad
\phi_{J K}\circ \phi_{I J} = \phi_{I K}
\qquad \text{on}\; \;
\phi_{IJ}^{-1}(U_{JK}) = U_{IK} .
\end{equation}
\end{itemlist}
 \end{defn}

 The following diagram of sets and maps  between them might be useful in decoding the cocycle conditions.
 \[
  \xymatrix 
  @R= 2pc
  {\phi_{IJ}^{-1}(U_{JK})\cap   U_{IJ}\ar@{^{(}->}[d]\ar@{->}[r]^{\quad\qquad\phi_{IJ}\;\;} & \;\;  U_{JK} \ar@{->}[r]^{\;\;\;\phi_{JK}}  & U_K\\
 U_I  \ar@{<-^{)}}[r] & 
\quad U_{IK}\ar@{->}[ur]_{\phi_{IK}}.
 }
 \]
The relevant distinction between  these versions of the cocycle condition is that the weak  condition can be achieved in practice by constructions of finite dimensional reductions for holomorphic curve moduli spaces, whereas the strong 
condition is needed for our construction of a virtual moduli cycle  from perturbations of the sections in the Kuranishi charts.
The cocycle condition is an intermediate notion which is too strong to be constructed in practice and too weak to induce a VMC, but it does allow us to formulate Kuranishi atlases categorically. This in turn gives rise, via a topological realization of a category, to a virtual neighbourhood of $X$ into which all Kuranishi domains map.

\begin{defn}\label{def:Ku}
A {\bf weak  Kuranishi atlas of dimension $\mathbf d$} on a compact metrizable space
$X$ is a tuple
$$
\Kk=\bigl(\bK_I,\Hat\Phi_{I J}\bigr)_{I, J\in\Ii_\Kk, I\subsetneq J}
$$
consisting of a covering family of basic charts $(\bK_i)_{i=1,\ldots,N}$ of dimension $d$
and transition data $(\bK_J)_{|J|\ge 2}$, $(\Hat\Phi_{I J})_{I\subsetneq J}$ for $(\bK_i)$ as in Definition~\ref{def:Kfamily}, that satisfy the {\it  weak cocycle condition} $\Hat\Phi_{J K}\circ \Hat\Phi_{I J} \approx\Hat\Phi_{I K}$ for every triple
$I,J,K\in\Ii_K$ with $I\subsetneq J \subsetneq K$. A weak  Kuranishi atlas  $\Kk$  is called a {\bf  Kuranishi atlas} if it satisfies the cocycle condition of \eqref{eq:cocycle}.
\end{defn}

\begin{rmk}\label{rmk:smoothatl}\rm  Very similar definitions apply if the isotropy groups are nontrivial, or if $X$ is stratified (for example, it consists of nodal $J$-holomorphic curves).  In the former case we must modify the coordinate changes (see  Definition~\ref{def:change2}), while in the latter case the domains of the charts  are stratified smooth (SS) spaces, which means that we must develop an adequate theory of SS maps.  In \S\ref{ss:tatlas} we introduce a  notion of {\bf topological atlas} that will
provide a common context for the topological constructions.  Therefore, for clarity we will sometimes call the atlases 
of Definition~\ref{def:Ku} {\bf smooth} and {\bf with trivial isotropy}.
$\hfill\er$  
\end{rmk}

\subsection{The Kuranishi category and virtual neighbourhood $|\Kk|$}\label{ss:K2}
After defining the Kuranishi category $\bB_\Kk$ of a Kuranishi atlas $\Kk$ and the associated realization $|\Kk|$, we 
state the main results about the topological space $|\Kk|$, giving all the relevant definitions. Most proofs are deferred to 
\S\ref{ss:tatlas}  where they are carried out in the broader context of topological atlases.
\MS

It is useful to think of the domains and obstruction spaces of a Kuranishi atlas as forming the following categories.
Recall that a {\bf topological category} $\bC$ is a small category (i.e. the
 collections  $\Obj_\bC, \Mor_{\bC}$ of objects and  morphisms
are sets) in which $\Obj_\bC$ and $ \Mor_{\bC}$ are provided with topologies in such a way that 
all structural maps such as source and target maps $s,t:\Mor_\Kk\to \Obj_\Kk$, as well as composition and inverse are 
continuous.  If $\Kk$ is a smooth atlas as defined in Definition~\ref{def:Ku}, then the spaces $\Obj_\bC$ and $ \Mor_{\bC}$ 
are  disjoint unions of smooth finite dimensional manifolds of varying dimensions, and the structural maps are smooth embeddings, but  later we also consider topological atlases. As we will see, the language of categories is a good way to describe how a space (such as $|\Kk|$) is built from simpler pieces.

\begin{defn}\label{def:catKu}
Given a Kuranishi atlas $\Kk$ we define its {\bf domain category} $\bB_\Kk$ to consist of
the space of objects\footnote{
When forming categories such as $\bB_\Kk$, we take always the space of objects 
to be the disjoint union of the domains 
$U_I$, even if we happen to have defined the sets $U_I$ 
as subsets of some larger space such as $\R^2$ 
or a space of maps as in the Gromov--Witten case.
Similarly, the morphism space is a disjoint union of the $U_{IJ}$ even though $U_{IJ}\subset U_I$ for all $J\supset I$.}
$$
\Obj_{\bB_\Kk}:= \bigsqcup_{I\in \Ii_\Kk} U_I \ = \ \bigl\{ (I,x) \,\big|\, I\in\Ii_\Kk, x\in U_I \bigr\}
$$
and the space of morphisms
$$
\Mor_{\bB_\Kk}:= \bigsqcup_{I,J\in \Ii_\Kk, I\subset J} U_{IJ} \ = \ \bigl\{ (I,J,x) \,\big|\, I,J\in\Ii_\Kk, I\subset J, x\in U_{IJ} \bigr\}.
$$
Here we denote $U_{II}:= U_I$ for $I=J$, and for $I\subsetneq J$ use
the domain $U_{IJ}\subset U_I$ of the restriction $\bK_I|_{U_{IJ}}$ to $F_J$
that is part of the coordinate change $\Hat\Phi_{IJ} : \bK_I|_{U_{IJ}}\to \bK_J$.

Source and target of these morphisms are given by
$$
(I,J,x)\in\Mor_{\bB_\Kk}\bigl((I,x),(J,\phi_{IJ}(x))\bigr),
$$
where $\phi_{IJ}: U_{IJ}\to U_J$ is the  embedding given by $\Hat\Phi_{I J}$, and we denote $\phi_{II}:={\rm id}_{U_I}$.
Composition\footnote
{
Note that this is written in the categorical ordering.}
 is defined by
$$
\bigl(I,J,x\bigr)\circ  \bigl(J,K,y\bigr) 
:= \bigl(I,K,x\bigr)
$$
for any $I\subset J \subset K$ and $x\in U_{IJ}, y\in  U_{JK}$ such that $\phi_{IJ}(x)=y$.

The {\bf obstruction category} $\bE_\Kk$ is defined in complete analogy to $\bB_\Kk$ to consist of
the spaces of objects $\Obj_{\bE_\Kk}:=\bigsqcup_{I\in\Ii_\Kk} U_I\times E_I$ and morphisms
$$
\Mor_{\bE_\Kk}: = \bigl\{ (I,J,x,e) \,\big|\, I,J\in\Ii_\Kk, I\subset J,  x\in U_{IJ}, e\in E_I \bigr\}.
$$
\end{defn}

We may also express the further parts of a Kuranishi atlas in categorical terms:

\begin{itemlist}
\item
The obstruction category $\bE_\Kk$ is a bundle over $\bB_\Kk$ in the sense that there is a functor
$\pr_\Kk:\bE_\Kk\to\bB_\Kk$ that is given on objects and morphisms by projection $(I,x,e)\mapsto (I,x)$ and $(I,J,x,e)\mapsto(I,J,x)$ with locally trivial fiber $E_I$.
\item
The sections $s_I$ induce a smooth
section of this bundle, i.e.\ a functor $\s_\Kk:\bB_\Kk\to \bE_\Kk$ which acts smoothly
on the spaces of objects and morphisms, and whose composite with the projection
$\pr_\Kk: \bE_\Kk \to \bB_\Kk$ is the identity. More precisely, it is given by $(I,x)\mapsto (I,x,s_I(x))$ on objects and by $(I,J,x)\mapsto (I,J,x,s_I(x))$ on morphisms.
\item
The zero sets of the sections $\bigsqcup_{I\in\Ii_\Kk} \{I\}\times s_I^{-1}(0)\subset\Obj_{\bB_\Kk}$ form a very special strictly full subcategory $\s_\Kk^{-1}(0)$ of $\bB_\Kk$. Namely, $\bB_\Kk$ splits into the subcategory $\s_\Kk^{-1}(0)$ and its complement (given by the full subcategory with objects  $\{ (I,x) \,|\, s_I(x)\ne 0 \}$) in the sense that there are no morphisms of $\bB_\Kk$ between the two underlying sets of objects.
(This holds by \eqref{eq:zeroIJ} below.)
\item
The footprint maps $\psi_I$ give rise to a surjective functor $\psi_\Kk: \s_\Kk^{-1}(0) \to \bX$ 
to the category $\bX$ with object space $X$ and trivial morphism space, i.e. consisting only of identity maps.
It is given by $(I,x)\mapsto \psi_I(x)$ on objects and by $(I,J,x)\mapsto {\rm id}_{\psi_I(x)}$ on morphisms.
\end{itemlist}

We denote the  {\bf  topological realization of the category} $\bB_\Kk$ by $|\bB_\Kk|$, often abbreviated to $|\Kk|$.  This is the space formed as the quotient of 
$\Obj_{\bB_\Kk}= \bigsqcup_{I} U_I$  by the equivalence relation generated by the morphisms, and is given the quotient topology. 
Thus, for example, 
the realization of the category $\bX$ is the space $X$ itself.
Since $\s_\Kk$ is a functor, the equivalence relation on $\Obj_{\bB_\Kk}$ preserves the zero sets.
More precisely, the fact that the morphisms in $\bB_\Kk$ intertwine the zero sets $s_I^{-1}(0)$ and the footprint maps implies that
\begin{equation}\label{eq:zeroIJ}
 (I,x)\sim (J,y), s_I(x) = 0\;\;\Longrightarrow\;\; s_J(y) = 0, \psi_I(x) = \psi_J(y).
\end{equation}
Hence  $\s_\Kk^{-1}(0)$ can be considered as a full subcategory of $\bB_\Kk$, and this inclusion induces
a natural continuous bijection
from
the realization $|\s_\Kk^{-1}(0)|$ of the subcategory $\s_\Kk^{-1}(0)$ (with its quotient topology)
to
the zero set $|\s_\Kk|^{-1}(0)\subset |\Kk|$ of the function $|\s_\Kk|: |\Kk|\to |\bE_\Kk|$ (with the subspace topology).
As in \cite[Lemma~2.4.2]{MW1}, one can prove directly from the definitions that the inverse is continuous.

\begin{lemma}\label{le:ioK} The inverse of the footprint maps $\psi_I^{-1}: F_I\to U_I$ 
fit together to give an injective map
\begin{equation}\label{eq:ioK}
\io_\Kk: X\to |\s_\Kk|^{-1}(0)\subset |\Kk|
\end{equation}
that  is a homeomorphism to its image $|\s_\Kk|^{-1}(0)$.  
\end{lemma}

Thus the zero set $|\s_\Kk|^{-1}(0)\subset |\Kk|$ has the expected topology.
However, as is shown by Example~\ref{ex:Khomeo} below, 
the topology on $|\Kk|$  itself can be very wild; it is not in general Hausdorff and the natural maps $\pi_\Kk: U_I\to |\Kk|$ 
need not be injective, let alone homeomorphisms to their images.  Further, even though the isotropy is trivial,
 the fibers of the projection $|\pr|:|\bE_\Kk|\to |\Kk|$ need not be vector spaces. 
In order to remedy these problems we introduce the notions of {\bf tameness, shrinking,  metrizability, and cobordism} for weak atlases and prove the following result.

\begin{thm}[see  Theorem~6.3.9 in  \cite{MW2}]\label{thm:K}
Let $\Kk$ be a  weak Kuranishi atlas (with trivial isotropy) on a compact metrizable space $X$.
Then an appropriate shrinking of $\Kk$ provides a  metrizable tame Kuranishi atlas $\Kk'$ with domains 
$(U'_I\subset U_I)_{I\in\Ii_{\Kk'}}$ such that the realizations $|\Kk'|$ and $|\bE_{\Kk'}|$ are Hausdorff in the quotient topology.
In addition, for each $I\in \Ii_{\Kk'} = \Ii_\Kk$ the projection maps $\pi_{\Kk'}: U_I'\to |\Kk'|$ and
$\pi_{\Kk'}:U'_I\times E_I\to |\bE_{\Kk'}|$ are homeomorphisms onto their images and fit into a commutative diagram
\[
\xymatrix
{
U_I'\times E_I   \ar@{->}[d] \ar@{->}[r]^{\pi_{\Kk'}}    &   |\bE_{\Kk'}| \ar@{->}[d]^{|\pr_{\Kk'}|}   \\
U_I' \ar@{->}[r]^{\pi_{\Kk'}}  &|\Kk'|.
}
\]
where the horizontal maps intertwine the vector space structure on $E_I$ with a vector space structure on the fibers of $|\pr_{\Kk'}|$.

Moreover, any two such shrinkings are cobordant by a metrizable tame Kuranishi cobordism whose realization also has the above Hausdorff, homeomorphism, and linearity properties.
\end{thm}

We give most of the details of the proof in the next section \S\ref{ss:tatlas}.  In fact, we will  prove a more general version 
 that will be relevant when we come to consider  smooth atlases with nontrivial isotropy.  Thus, 
we will introduce a notion of {\bf topological atlas} and discuss what   taming means in that context.
In the rest of this section, 
we explain the new notions in the  smooth context, stating more precise versions 
of the above theorem for this case.

\begin{defn}\label{def:tame}
A weak Kuranishi atlas is said to be {\bf tame} if 
for all $I,J,K\in\Ii_\Kk$ we have\footnote
{ In \eqref{eq:tame2} below we write $s_J^{-1}\bigl(\Hat\phi_{IJ}(E_I)\bigr)$ instead of $s_J^{-1}(E_I)$ for clarity; but we usually use the shorter notation, identifying  the subspace $\{(e_j)_{j\in J} \ | \ e_j=0 \ \forall j\in J\less I\}\subset E_J$ 
with $E_I$ when $I\subset J$. 
}
\begin{align}\label{eq:tame1}
U_{IJ}\cap U_{IK}&\;=\; U_{I (J\cup K)}\qquad\qquad\;\;\;\;\,\qquad\forall I\subset J,K ;\\
\label{eq:tame2}
\phi_{IJ}(U_{IK}) &\;=\; U_{JK}\cap s_J^{-1}\bigl(\Hat\phi_{IJ}(E_I)\bigr) \qquad\forall I\subset J\subset K.
\end{align}
Here we allow equalities, using the notation $U_{II}:=U_I$ and $\phi_{II}:={\rm Id}_{U_I}$.
Further, to allow for the possibility that $J\cup K\notin\Ii_\Kk$, we define
$U_{IL}:=\emptyset$ for $L\subset \{1,\ldots,N\}$ with $L\notin \Ii_\Kk$.
Therefore \eqref{eq:tame1} includes the condition
$$
U_{IJ}\cap U_{IK}\ne \emptyset
\quad \Longrightarrow \quad F_J\cap F_K \ne \emptyset  \qquad \bigl( \quad \Longleftrightarrow\quad
J\cup K\in \Ii_\Kk \quad\bigr).
$$
\end{defn}

The notion of tameness generalizes the identities $F_J\cap F_K=F_{J\cup K}$ and $\psi_J^{-1}(F_{K}) = U_{JK}\cap s_J^{-1}(0_J)$ between the footprints and zero sets, which we can include into \eqref{eq:tame1} and \eqref{eq:tame2} as the case $I = \emptyset$, by using the notation 
\begin{equation}\label{eq:empty}
U_{\emptyset J}: = F_J,\qquad \phi_{\emptyset J}:=\psi_J^{-1}.
\end{equation}
Indeed, the first tameness condition \eqref{eq:tame1} extends the identity for intersections of footprints -- which is equivalent to $\psi_I^{-1}(F_J)\cap \psi_I^{-1}(F_K) = \psi_I^{-1}(F_{J\cup K})$ for all $I\subset J,K$ -- to the domains of the transition maps in $U_I$. 
In particular, with $J\subset K$ it implies nesting of the domains of the transition maps,
\begin{equation}\label{eq:tame4}
U_{IK}\subset U_{IJ} \qquad\forall I\subset J \subset K.
\end{equation}
(This in turn generalizes the $I=\emptyset$ case $F_K\subset F_J$ for $J \subset K$.)
The second tameness condition \eqref{eq:tame2} extends the relation between footprints and zero sets -- equivalent to $\phi_{IJ}(\psi_I^{-1}(F_K)) = U_{JK}\cap s_J^{-1}(0)$ for all $I\subset J$ --
to a relation between domains of transition maps and preimages  of corresponding subbundles by 
the appropriate section $s_\bullet$\ .
In particular, with $J=K$ it controls the image of the transition maps,
generalizing the $I=\emptyset$ case $\psi_J^{-1}(F_J) =  s_J^{-1}(0)$ to
\begin{equation}\label{eq:tame3}
\im\phi_{IJ}:= \phi_{IJ}(U_{IJ}) =  s_J^{-1}(E_I) \qquad\forall I\subset J.
\end{equation}
It follows that that  the image of each transition map $\phi_{IJ}$ is a closed subset of the Kuranishi domain $U_J$.
Further,  if $I\subset J\subset K$ we may  combine conditions \eqref{eq:tame2} and \eqref{eq:tame3} to obtain
the identity $\phi_{IJ}(U_{IK}) = U_{JK}\cap s_J^{-1}(E_I) = U_{JK}\cap (\im \phi_{IJ})$,  
 which pulls back via 
$\phi_{IJ}^{-1}$ to $U_{IK} = U_{IJ} \cap \phi_{IJ}^{-1}(U_{JK})$.  This proves the first part of the following proposition.

\begin{prop}\label{prop:tame0} Let $\Kk$ be a tame weak Kuranishi atlas (with trivial isotropy). Then the following holds.
\begin{enumerate} 
\item $\Kk$
 satisfies the strong cocycle condition; in particular it is a Kuranishi atlas.
\item Both
 $|\Kk|$ and  $|\bE_\Kk|$ are Hausdorff, and for each $I\in\Ii_\Kk$ the quotient maps $\pi_{\Kk}|_{U_I}:U_I\to |\Kk|$ and $\pi_{\Kk}|_{U_I\times E_I}:U_I\times E_I\to |\bE_\Kk|$ are homeomorphisms onto their image.
\item There is a unique linear structure on the fibers of  $|\pr_{\Kk}|: |\bE_\Kk| \to |\Kk|$ such that for every $I\in\Ii_\Kk$ the embedding $\pi_{\Kk} : U_I\times E_I \to  |\bE_{\Kk}|$ is linear on the fibers.
\end{enumerate}
\end{prop}
\begin{proof}  See Proposition~\ref{prop:tame1}.
\end{proof}

Thus the quotient topology on the realization $|\Kk|$ of a tame atlas is reasonably  well behaved.  Nevertheless it is almost never metrizable: indeed if there is a coordinate change $\Hat\Phi_{IJ}: \bK_I\to \bK_J$ with $\dim U_I <\dim U_J$  and such that the subset $F_J$ is not closed in $F_I$ then  
 for each $x_I\in \psi_I^{-1}\bigl(F_I\cap (\ov F_J\less F_J)\bigr)\subset U_I$  
the point  $\pi_\Kk(x_I)\in |\Kk|$ does {\it not} have a countable neighbourhood base in the quotient topology; cf. Example~\ref{ex:Khomeo} below.
 
 \begin{defn}[Definition~6.1.14 in \cite{MW2}]\label{def:metric}  
A Kuranishi atlas $\Kk$ is called {\bf metrizable} if there is a bounded metric $d$ on the set $|\Kk|$ such that for each $I\in \Ii_\Kk$ the pullback metric $d_I:=(\pi_\Kk|_{U_I})^*d$ on $U_I$ induces the given topology on  $U_I$.
In this situation we call $d$ an {\bf admissible metric} on $|\Kk|$. 
A {\bf metric Kuranishi atlas} is a pair $(\Kk,d)$ consisting of a metrizable Kuranishi atlas and a choice of  admissible metric $d$.
\end{defn}

\begin{rmk}\rm
We will use this metric  on $|\Kk|$ when constructing the perturbation section $\nu$ in order to 
 control  its domain and size and hence ensure that the perturbed zero set  has compact realization.
 \hfill$\er$
\end{rmk}

Before stating the existence result, it is convenient to introduce the further notion of a {\bf shrinking.}
We write $V'\sqsubset V$ to denote that $V'$ is precompact in $V$, i.e. the closure (written $\ov{V'}$ or $cl_V(V')$) of $V'$ in $V$ is compact.

\begin{defn}\label{def:shr0}
Let $(F_i)_{i=1,\ldots,N}$ be an open cover of a compact space $X$.  We say that $(F_i')_{i=1,\ldots,N}$ is a {\bf shrinking} of $(F_i)$ if $F_i'\sqsubset F_i$ are
precompact open subsets, 
which cover $X= \bigcup_{i=1,\ldots,N} F'_i$, and are such that for all subsets $I\subset \{1,\ldots,N\}$ we have
\begin{equation} \label{same FI}
F_I: = {\textstyle\bigcap_{i\in I}} F_i \;\ne\; \emptyset
\qquad\Longrightarrow\qquad
F'_I: = {\textstyle\bigcap_{i\in I}} F'_i \;\ne\; \emptyset .
\end{equation}
\end{defn}

\begin{defn}\label{def:shr}
Let $\Kk=\bigl(\bK_I,\Hat\Phi_{I J})_{I, J\in\Ii_\Kk, I\subsetneq J}$ be a weak Kuranishi atlas.   We say that a weak Kuranishi atlas $\Kk'=(\bK_I',\Hat\Phi_{I J}')_{I, J\in\Ii_{\Kk'}, I\subsetneq J}$ is a {\bf shrinking} of $\Kk$, and write
$\Kk'\sqsubset \Kk$,~if
\begin{enumerate}
\item  the footprint cover $(F_i')_{i=1,\ldots,N'}$ is a 
shrinking of the cover $(F_i)_{i=1,\ldots,N}$,
in particular the numbers $N=N'$ of basic charts agree, and so do the index sets $\Ii_{\Kk'} = \Ii_\Kk$;
\item
for each $I\in\Ii_\Kk$ the chart $\bK'_I$ is the restriction of $\bK_I$ to a precompact domain $U_I'\subset U_I$
as in Definition \ref{def:restr};
\item
for each $I,J\in\Ii_\Kk$ with $I\subsetneq J$ the coordinate change $\Hat\Phi_{IJ}'$ is the restriction of $\Hat\Phi_{IJ}$  to the open subset $U'_{IJ}: =  \phi_{IJ}^{-1}(U'_J)\cap U'_I$
(see  \cite[Lemma~5.2.6]{MW2}).
\end{enumerate}
\end{defn}

 In order to construct metric tame Kuranishi atlases, we will find it
useful to consider  tame shrinkings $\Kk_{sh}$ of a
weak Kuranishi atlas $\Kk$ that are obtained as shrinkings of an intermediate tame shrinking $\Kk'$ of $\Kk$.
For short we will call such $\Kk_{sh}$ a 
{\bf preshrunk tame shrinking} of $\Kk$ and write $\Kk_{sh}\sqsubset \Kk'\sqsubset \Kk$.

\begin{prop}  \label{prop:proper}
Every weak Kuranishi atlas $\Kk$ has a shrinking $\Kk_{sh}$ that is a tame Kuranishi atlas -- for short called a {\bf tame shrinking}.  
Moreover if $\Kk_{sh}$ is preshrunk with $\Kk_{sh}\sqsubset \Kk'\sqsubset \Kk$ where $\Kk'$ is also tame, then the following holds:
\begin{itemize}\item
the induced map $|\Kk_{sh}|\to |\Kk'|$ is a continuous injection;
\item the atlas $\Kk_{sh}$ is metrizable;
\item we may choose the metric $d$ so that the metric topology on $(|\Kk_{sh}|,d)$ equals its 
topology as a subspace of $|\Kk'|$ with the quotient topology.
\end{itemize}
\end{prop}

 Propositions~\ref{prop:tame0} and \ref{prop:proper} contain all the
results in the first part of Theorem~\ref{thm:K}, and are proved in \S\ref{ss:tatlas}. 
The next important  concept used in Theorem~\ref{thm:K} is that of cobordism.
We develop an appropriate theory of {\bf topological cobordism Kuranishi atlases} in \cite[\S4]{MW1}
over spaces $Y$ with collared boundary; for the smooth theory see \cite[\S6.2]{MW2}.
To give its flavor, we now quote a few key definitions for smooth atlases, where for  simplicity we 
 restrict to the case of {\bf concordances}, i.e. cobordisms over $ [0,1]\times X$.

 \begin{defn} \label{def:CKS}
\begin{itemlist}
\item
Let $\bK^\al=(U^\al,E^\al,s^\al,\psi^\al)$ be a Kuranishi chart on $X$, and let $A\subset[0,1]$ be a relatively open interval. Then we define the {\bf product chart} for $ [0,1]\times  X$ with footprint $A \times  F_I^\al$ as
$$
A \times  \bK^\al :=\bigl(A \times  U^\al, E^\al, s^\al\circ{\rm pr}_{U^\al}  , \id_{A} \times   \psi^\al \bigr) .
$$
\item
A {\bf Kuranishi chart with collared boundary} on $[0,1]\times X$ is a tuple $\bK = (U,E,s ,\psi)$ as in Definition~\ref{def:chart}, with the following variations:
\begin{enumerate}
\item
The footprint $F =\psi (s ^{-1}(0))\subset [0,1]\times X$ intersects the boundary $\{0,1\}\times X$.
\item
The domain $U $ is a smooth manifold whose boundary splits into two parts $\partial U  = \partial^0 U  \sqcup \partial^1 U $ such that $\partial^\al U $ is nonempty iff $F $ intersects $\{\al\}\times X$.
\item
If $\partial^\al U \neq\emptyset$ then there is a relatively open neighbourhood $A^\al\subset [0,1]$ of $\al$ and an embedding $\iota^\al:
A^\al \times \partial^\al U  \hookrightarrow  U $ onto a neighbourhood of $\partial^\al U \subset U $ such that
$$
\bigl( \{\al\}\times  \partial^\al U \,,\, E \,,\, s \circ \iota^\al  \,,\, \psi \circ \iota^\al  \, \bigr)
\; = \; A^\al\times \partial^\al\bK  
$$
is the product of $A^\al$ with a Kuranishi chart $\partial^\al\bK $ for $X$ 
(called the {\bf restriction of $\bK$ to the boundary})
with footprint $F^\al\subset X$ such that $(A^\al\times X) \cap F  = A^\al\times F$.
\end{enumerate}
\end{itemlist}
\end{defn}

\begin{defn} \label{def:Ccc}
\begin{itemlist}
\item
Let $\bK _I,\bK _J$ be Kuranishi charts on  $[0,1]\times X$ 
such that only $\bK _I$ or both
$\bK _I,\bK _J$ have collared boundary.
Then a {\bf coordinate change with collared boundary} $\Hat\Phi_{IJ} :\bK _I\to\bK _J$ is a tuple $\Hat\Phi_{IJ}  = (U_{IJ} ,\phi_{IJ} ,\Hat\phi_{IJ} )$ of domain and embeddings as in Definition~\ref{def:change}, with the following boundary variations and collar form requirement:
\begin{enumerate}
\item
The domain is a relatively open subset $U _{IJ}\subset U _I$ with boundary components
$\partial^\al U _{IJ}:= U _{IJ} \cap \partial^\al U _I$;
\item
If  
$F_J\cap \{\al\}\times X \ne \emptyset$ for  
$\al = 0$ or $1$ (so that $F_I\cap  \{\al\}\times X \ne \emptyset$), 
there
is a relatively open neighbourhood $B^\al\subset [0,1]$
of $\al$
such that
\begin{align*}
(\iota_I^\al)^{-1}(U _{IJ})
\cap \bigl(B^\al\times \partial^\al U _I \bigr)
&\;=\; B^\al \times \partial^\al U _{IJ} , \\
(\iota_J^\al)^{-1}(\im\phi _{IJ})
\cap \bigl( B^\al \times \partial^\al U _J \bigr)
&\;=\;
 \phi _{IJ}( B^\al \times \partial^\al U _{IJ}),
\end{align*}
and
$$
\bigl( \,
 B^\al \times \partial^\al U _{IJ} \,,\, (\iota_J^\al)^{-1}\circ \phi _{IJ} \circ \iota_I^\al  \,,\,  \Hat\phi _{IJ} \, \bigr)
\;=\; {\rm id}_{B^\al} \times \partial^\al\Hat\Phi_{IJ} ,
$$
where $\partial^\al\Hat\Phi_{IJ} : \partial^\al\bK _I \to \partial^\al\bK _J$ is a coordinate change.
\item
If $\p^\al F_J= \emptyset$ but $\p^\al F_I\ne \emptyset$ for  $\al = 0$ or $1$
there is a neighbourhood $B^\al\subset [0,1]$
of $\al$
such that
$$
U _{IJ}\cap \iota_I^\al \bigl(B^\al \times \partial^\al U _I \bigr) = \emptyset.
$$
\end{enumerate}
\item
For any coordinate change with collared boundary $\Hat\Phi_{IJ} $ on $[0,1]\times X$ 
we call the uniquely determined coordinate changes $\p^\al \Hat\Phi_{IJ}$ for  $X$ 
the {\bf restrictions of $\Hat\Phi_{IJ} $ to the boundary} for $\al=0,1$.
\end{itemlist}
\end{defn}

\begin{defn}\label{def:CKS2}
A {\bf (weak) Kuranishi cobordism} on $[0,1]\times X$ 
is a tuple
$$
\Kk  = \bigl( \bK_{I} , \Hat\Phi_{IJ} \bigr)_{I,J\in \Ii_{\Kk}}
$$
of basic charts and transition data as in Definition~\ref{def:Ku} 
with the following boundary variations and collar form requirements:
\begin{itemlist}
\item
The charts of 
$\Kk$  are either Kuranishi charts with collared boundary or standard Kuranishi charts whose footprints are precompactly contained in  $(0,1)\times X$.
\item
The coordinate changes $\Hat\Phi_{IJ}: \bK_{I} \to \bK_{J}$ 
are either standard coordinate changes on $(0,1)\times X$ 
between pairs of standard charts, or coordinate changes with collared boundary between 
pairs of charts, of which at least the first has collared boundary.
\end{itemlist}
Moreover, we call 
$\Kk$ {\bf \ tame} if it satisfies the tameness conditions of Definition~\ref{def:tame}.
\end{defn}

 The essential feature of our definitions is that the charts are now manifolds with collared boundary,  and 
 that we require 
  compatibility of this collar structure with coordinate changes and all other structures, such as metrics.
In particular,  a {\bf metric  Kuranishi concordance}\footnote
{
Concordance is called deformation equivalence in \cite{TF}.}
 $(\Kk,d)$  from $\Kk^0$ to $\Kk^1$
 is a metric atlas $(\Kk,d)$ over $[0,1]\times X$ that for  $\al = 0,1$ restricts to the atlas $\Kk^\al=:\p^\al \Kk$ on $X$,
and near each boundary has an {\it isometric} identification with the product $ A^\al_\eps\times \Kk^\al$ where $A^0_\eps = [0,\eps), A^1_\eps = (1-\eps,1]$.
Here is the main existence result.

\begin{prop}\label{prop:cobord2}
Let $\Kk$ 
be a  weak Kuranishi 
concordance  
on $[0,1]\times X$, 
and let 
$\Kk^0_{sh}, \Kk^1_{sh}$ be preshrunk tame shrinkings of $\p^0\Kk$ 
and $\p^1\Kk$ with admissible metrics $d^\al$.
Then there is a preshrunk tame shrinking $\Kk_{sh}$ of $\Kk$  
that provides a metric tame Kuranishi 
concordance 
from $\Kk^0_{sh}$ to $\Kk^1_{sh}$.  
Further, we may choose the metric on $\Kk_{sh}$ so that for $\al=0,1$  it restricts to $d^\al$ on $\p^\al \Kk_{sh}$.
\end{prop}

The proof is a fairly routine generalization of that of Proposition~\ref{prop:proper}, although it
turns out to be  surprisingly hard to  interpolate between two given  metrics in this way. (Note that we are {\it not} considering Riemannian metrics.) The necessary details are given in 
\cite[Proposition~4.2.4]{MW1}.  An analogous statement  is also true for Kuranishi cobordisms.

A final basic ingredient  of the proof of Theorem~B is that of {\bf orientation}. 
When the atlas $\Kk$ is smooth  with trivial isotropy we show in \S\ref{ss:orient} below that there is
 a real line bundle
 $|\La| \to |\Kk|$ whose pull back $\pi_\Kk^*(|\La|)|_{U_I}$ restricts on each domain $U_I$ to the orientation bundle
 $ \lm \rT U_I\otimes  (\lm E_I)^*$  of $U_I\times E_I$.  (Here, $\lm V$ denotes the maximal exterior power of
 the vector space $V$.)

 \begin{defn}\label{def:or}  We 
 define an orientation of $\Kk$ to be a nonvanishing section $\si$ of  $|\La| \to |\Kk|$.
 \end{defn}   
 
 The pullback of $\si$ to 
 $\pi_\Kk^*(|\La|)|_{U_I}$  then determines an orientation of $U_I\times E_I$, that is preserved by the coordinate changes 
 because the bundle $|\La|$ is defined over $|\Kk|$, and as we show in Lemma~\ref{le:locorient1} induces  compatible orientations
 on the zero sets of a perturbed transverse section.


We next give a simple example of a tame smooth atlas whose realization is neither metrizable nor locally compact.

\begin{example}
[Failure of metrizability and local compactness]  \label{ex:Khomeo}
\rm
For simplicity we will give an example with noncompact $X = \R$. (A similar example can be constructed with $X = S^1$.)
We construct a Kuranishi atlas $\Kk$ on $X$ with two basic charts, 
$\bK_1 = (U_1=\R, E_1=\{0\}, s=0,\psi_1=\id)$ and
$$
\bK_2 = \bigl(U_2=(0,\infty)\times \R,\ E_2=\R, \ s_2(x,y)= y,\ \psi_2(x,y)= x\bigr),
$$
one transition chart $\bK_{12} = \bK_2|_{U_{12}}$ with domain $U_{12} := U_2$, and the coordinate changes $\Hat\Phi_{i,12}$ induced by the natural embeddings of the domains $U_{1,12} := (0,\infty)\hookrightarrow (0,\infty)\times\{0\}$ and $U_{2,12} := U_2\hookrightarrow U_2$.
Then as a set $|\Kk| = \bigl(U_1\sqcup U_2\sqcup U_{12}\bigr)/\!\!\sim$
can be identified with $\bigl(\R\times\{0\}\bigr) \cup \bigl( (0,\infty)\times\R\bigr) \subset \R^2$.
However, the quotient topology at $(0,0)\in|\Kk|$ is strictly stronger than the subspace topology.
That is, for any open set $O\subset\R^2$  the induced subset $O\cap|\Kk|\subset|\Kk|$ is open, but some open subsets of $|\Kk|$ cannot be represented in this way.
In fact,  
for any $\eps>0$ and continuous function $f:(0,\eps)\to (0,\infty)$,
the set
$$
U_{f,\eps} \, :=\; \bigl\{ [x] \,\big|\, x\in U_1, |x|< \eps \}  \;\cup\; \bigl\{ [(x,y)] \,\big|\, (x,y)\in U_2,  |x|< \eps , |y|<f(x)\} \;\subset\; |\Kk|
$$
is open in the quotient topology.  It is shown in \cite[Example~2.4.5]{MW1} that  
 these sets form an (uncountable)  basis for the 
neighbourhoods of 
$[(0,0)]$ in the quotient topology. 

Notice that this atlas $\Kk$ is tame.  
Therefore taming by itself does not give a quotient with manageable  topology.  On the other hand, the only bad point $|\Kk|$ is $(0,0)$.
Indeed, according to  Proposition~\ref{prop:proper} 
 the realization of  any shrinking $\Kk'$ of $\Kk$ injects into $|\Kk|$ and  is metrizable when given the  corresponding subspace topology.  
For example, we could take $U_1': = (-\infty, 2)\sqsubset U_1$, $U_2': = (1,\infty)\times \R\sqsubset U_2$ and $U_{12}': = U_2'$.
\hfill$\er$ \end{example}

\begin{rmk}\label{rmk:poset}\rm (i)
The categories $\bB_\Kk, \bE_\Kk$ have many similarities with the \'etale\footnote
{
For relevant definitions see \S\ref{ss:orb}.}
categories used to model orbifolds i.e. the spaces of objects and morphisms are smooth manifolds and all structural maps (such as the source map, composition and so on) are  smooth. Moreover,   all sets of morphisms in $\bB_\Kk$ or $\bE_\Kk$ between fixed objects are finite.
However, because there could be coordinate changes  $\phi_{IJ}:U_{IJ}\to U_J$ with $\dim U_{IJ} < \dim U_J$, the target map  $t:\Mor_{\bB_\Kk}\to \Obj_{\bB_\Kk}$ is not in general a local diffeomorphism, although it is locally injective.
 Moreover, one cannot in general complete $\bB_\Kk$ to a groupoid by adding inverses and compositions, while keeping the property that the morphism space is a union of smooth manifolds. 
The problem here is that the inclusion of inverses of the coordinate changes, and their compositions, may yield singular spaces of morphisms. Indeed, coordinate changes $\bK_I\to\bK_K$ and $\bK_J\to\bK_K$ with the same target chart are given by embeddings $\phi_{IK}:U_{IK}\to U_K$ and $\phi_{JK}:U_{JK}\to U_K$, whose images may not intersect transversely (for example, often their intersection is contained only in the zero set $s_K^{-1}(0)$); yet this intersection would be a component of the space of morphisms from $U_{I}$ to $U_{J}$.
We show in Proposition~\ref{prop:groupcomp} below that in
 the special case when all obstruction spaces $E_I$ are trivial, i.e.
$E_I=\{0\}$, 
one can adjoin these compositions and inverses to $\bB_\Kk$, obtaining an  \'etale  proper groupoid whose realization is an orbifold.
\MS

\NI(ii)  When, as here, the isotropy groups are trivial, one can think of the category $\bB_\Kk$ as a topological poset; in other words the relation $\le$ on $\Obj_{\Bb_\Kk}$ defined by setting 
$$
(I,x)\le (J,y) 
\Longleftrightarrow \Mor_{\bB_\Kk}\bigl(I,x), (J,y)\bigr)\ne \emptyset
$$
 is a partial order.  As we will see in Lemma~\ref{le:Ku2}, the taming conditions impose further regularity on this order relation.
Indeed part (a) shows that if two elements  $(J,x), (K,y)$  have  a lower bound (i.e. an element $(I,w)$ that is less than or equal to them both), then they have both a  unique greatest lower bound and a unique least upper bound.
The notion of reduction (see Definition~\ref{def:vicin}) further simplifies this partial order; indeed at this point one could reduce to a family of subsets with partial ordering given by the integer lengths $|I|$; cf. the notion of dimensionally graded system (DGS) in \cite{TF}.
$\hfill\er$
\end{rmk}

\subsection{Tame topological atlases}\label{ss:tatlas}

This section explains the proof of a generalization of Theorem~\ref{thm:K}.
The main results are Proposition~\ref{prop:Khomeo}, that establishes the main
properties of tame atlases and Proposition~\ref{prop:proper1}, that constructs them.
Here we explain the key points of the proof. 
 Further details and
analogous results for cobordisms may be found in \cite{MW1}. 

We begin by defining  topological Kuranishi charts and atlases.

\begin{defn}\label{def:tchart}
A {\bf topological Kuranishi chart} for $X$ with open footprint $F\subset X$ is a tuple $\bK = (U,\E,0,\s,\psi)$ consisting of
\begin{itemlist}
\item
the {\bf domain} $U$, which is a separable, locally compact metric space;
\item
the {\bf obstruction ``bundle''} $\pr :\E \to U$, which is a continuous map between 
separable, locally compact metric spaces; 
\item 
the {\bf zero section} 
$0: U \to \E$, which is a continuous map with $\pr \circ 0 = \id_{U}$;
\item
the {\bf section} $\s: U\to \E$, which is a continuous map with $\pr \circ \s = \id_{U}$;
\item
the {\bf footprint map} $\psi : \s^{-1}(0) \to X$, which is a homeomorphism 
between the {\bf zero set} $\s^{-1}(0):=
\s^{-1}(\im 0)=\{x\in U \,|\, \s(x)=0(x)\}$ and 
the {\bf footprint} ${\psi(\s^{-1}(0))=F}$.
\end{itemlist}
\end{defn}

\begin{rmk}\label{rmk:intermed}\rm   
Because finite dimensional manifolds are separable, locally compact, and metrizable, each
smooth Kuranishi chart with trivial isotropy
induces a topological Kuranishi chart with $\E:=U\times E$, $0(x):=(x,0)$, and $\s(x)=(x,s(x))$.
As we will see in \S\ref{s:iso}, a smooth local chart near a point with nontrivial isotropy
is a tuple $(U,E,\Ga,s,\psi)$ where $U,E$ are as in Definition~\ref{def:chart},  $\Ga$ is a finite group 
that  acts on $U$ and $E$, the function $s:U\to E$ is equivariant, and the footprint map $\psi: s^{-1}(0)\to  X$ induces a homeomorphism from $\qu{s^{-1}(0)}{\Ga}$ onto an open subset of $X$.  
Such a tuple gives rise to a topological chart with domain $\qu{U}{\Ga}$ and bundle
$\pr: \E: = \qu{U\times E}{\Ga}\to \qu{U}{\Ga}$, where $\Ga$ acts diagonally on 
$U\times E$.  Note that the fibers of this bundle at points $x\in \qu{U}{\Ga}$ with nontrivial stabilizer $\Stab_x$ are quotients $\qu{E}{\Stab_x}$, and hence in general do not have a linear structure.
Nevertheless there is a well defined zero section  as well as a section $\s:U\to \E$ induced by $s:U\to E$.
$\hfill\er$
\end{rmk}

Although one can define restrictions as before, one must adapt  the notion of 
coordinate change as follows.

\begin{defn}\label{def:tchange}
Let $\bK_I$ and $\bK_J$ be topological Kuranishi charts such that $F_J\subset F_I$.
A {\bf topological coordinate change} from $\bK_I$ to $\bK_J$ is a map $\Hat\Phi: \bK_I|_{U_{IJ}}\to \bK_J$ defined on a restriction of $\bK_I$ to $F_J$. More precisely:  
\begin{itemlist}
\item
The {\bf domain} of the coordinate change is an open subset $U_{IJ}\subset U_I$ such that
$\s_I^{-1}(0_I)\cap U_{IJ} = \psi_I^{-1}(F_J)$.
\item 
The {\bf map} of the coordinate change is a topological embedding (i.e.\ homeomorphism to its image) 
$\Hat\Phi: \E_I|_{U_{IJ}}:=\pr_I^{-1}(U_{IJ}) \to \E_J$ with the following properties.
\begin{enumerate}
\item
It is a bundle map, i.e.\ we have $\pr_J \circ \Hat\Phi = \phi \circ\pr_I |_{\pr_I^{-1}(U_{IJ})}$ for a topological embedding $\phi: U_{IJ}\to U_J$, and it is linear in the sense that $0_J \circ \phi  = \Hat\Phi \circ 0_I|_{U_{IJ}}$.
\item
It intertwines the sections, i.e. $\s_J \circ \phi  = \Hat\Phi \circ \s_I|_{U_{IJ}}$.
\item
It restricts to the transition map induced from the footprints in $X$, i.e.
$$
\phi|_{\psi_I^{-1}(F_I\cap F_J)}=\psi_J^{-1} \circ\psi_I : U_{IJ}\cap \s_I^{-1}(0_I) \to \s_J^{-1}(0_J).
$$
\end{enumerate}
\end{itemlist}

In particular, the following diagrams commute:
\begin{align*} 
& \qquad\quad\;\;
 \begin{array} {ccc}
{\E_I|_{U_{IJ}}}& \stackrel{\Hat\Phi} \longrightarrow &
{\E_J} 
\phantom{\int_Quark}  \\
\phantom{sp} \downarrow {\pr_I}&&\downarrow {\pr_J} \phantom{spac}\\
\phantom{s}{U_{IJ}} & \stackrel{\phi} \longrightarrow &{U_J} \phantom{spacei}
\end{array} 
\qquad
 \begin{array} {ccc}
{\E_I|_{U_{IJ}}}& \stackrel{\Hat\Phi} \longrightarrow &
{\E_J} 
\phantom{\int_Quark}  \\
\phantom{sp} \uparrow {0_I}&&\uparrow {0_J} \phantom{spac}\\
\phantom{s}{U_{IJ}} & \stackrel{\phi} \longrightarrow &{U_J} \phantom{spacei}
\end{array}
 \notag \\
& \qquad\qquad
 \begin{array} {ccc}
{\E_I|_{U_{IJ}}}& \stackrel{\Hat\Phi} \longrightarrow &
{\E_J} 
\phantom{\int_Quark}  \\
\phantom{sp} \uparrow {\s_I}&&\uparrow {\s_J} \phantom{spac}\\
\phantom{s}{U_{IJ}} & \stackrel{\phi} \longrightarrow &{U_J} \phantom{spacei}
\end{array} 
\qquad
 \begin{array} {ccc}
{U_{IJ}\cap \s_I^{-1}(0_I)} & \stackrel{\phi} \longrightarrow &{\s_J^{-1}(0_J)} \phantom{\int_Quark} \\
\phantom{spa} \downarrow{\psi_I}&&\downarrow{\psi_J} \phantom{space} \\
\phantom{s}{X} & \stackrel{{\rm Id}} \longrightarrow &{X}. \phantom{spaceiiii}
\end{array}
\end{align*}
\end{defn}


\begin{rmk} \rm
The above definition does not provide charts with enough structure 
to be able to formulate any equivalent to the index condition.  Indeed, because
$\pr:\E_I\to U_I$ is not assumed to be a vector bundle, there is as yet no notion to replace the quotient bundle
$U_J\times\bigl( \qu{E_J}{\Hat\phi_{IJ}(E_I)}\bigr)$ that appears in \eqref{tbc}.  We will formulate a weaker but 
adequate replacement for this condition when we introduce the notion of  filtration; see
Definition~\ref{def:Ku3} and Lemma~\ref{le:Ku3} below.
\hfill$\er$
\end{rmk}

It is immediate that smooth coordinate changes between smooth charts satisfy the above conditions.
Moreover, one can define 
restriction and composition of coordinate changes as before, as well as the various versions of the cocycle condition.

\begin{defn}\label{def:tKu}
A {\bf (weak) topological Kuranishi atlas}
on a compact metrizable space
$X$ is a tuple
$$
\Kk=\bigl(\bK_I,\Hat\Phi_{I J}\bigr)_{I, J\in\Ii_\Kk, I\subsetneq J}
$$
of a covering family of basic charts $(\bK_i)_{i=1,\ldots,N}$ 
and transition data $(\bK_J)_{|J|\ge 2}$, $(\Hat\Phi_{I J})_{I\subsetneq J}$ for $(\bK_i)$ as in Definition~\ref{def:Kfamily}, 
that consists of topological Kuranishi charts and topological coordinate changes
satisfying the (weak) cocycle condition.
\end{defn}

Because $\Hat\Phi$ is a bundle map,  a weak atlas  is an atlas if and only if  
$U_{IJ}\cap \phi_{IJ}^{-1}(U_{JK})\subset
U_{IK}$ for all $I\subsetneq J\subsetneq K$.
Further, if $\Kk$ is a topological atlas  we can form the topological categories
 $\bB_\Kk, \bE_\Kk$ as before, and then assemble the maps $\pr_I, 0_I,\s_I$ into functors denoted
 $\pr_\Kk, 0_\Kk, \s_\Kk$.
In particular,
 the functor ${\rm pr}_\Kk:\bE_\Kk\to\bB_\Kk$ induces a continuous map
$$
|{\rm pr}_\Kk|:|\bE_\Kk| \to |\Kk|,
$$
which we call the {\bf obstruction bundle} of $\Kk$, although its fibers generally do not have the structure of a vector space.
Similarly,
 the functors $0_\Kk$ and $\s_\Kk$ induce continuous maps
$$
|0_\Kk| : \; |\Kk| \to |\bE_\Kk| , \qquad\qquad |\s_\Kk|:|\Kk|\to |\bE_\Kk|,
$$
that  are sections in the sense that $|\pr_\Kk|\circ|\s_\Kk| =  |\pr_\Kk|\circ |0_\Kk|= {\rm id}_{|\Kk|}$.
Moreover,  the realization of the full subcategory $\s_\Kk^{-1}(0_\Kk)$ is homeomorphic to the zero section
$ |\s_\Kk|^{-1}(|0_\Kk|)\subset |\Kk|$ and
 the analog of Lemma~\ref{le:ioK} holds.
In other words, the footprint functor $\psi_\Kk: \s_\Kk^{-1}(0_\Kk) \to \bX$ descends to a homeomorphism $|\psi_\Kk| :  |\s_\Kk|^{-1}(|0_\Kk|) \to X$ with inverse  given by
$$
\io_{\Kk}:= |\psi_\Kk|^{-1} : \; X\;\longrightarrow\;  |\s_\Kk|^{-1}(|0_\Kk|) \;\subset\; |\Kk|, \qquad
p \;\mapsto\; [(I,\psi_I^{-1}(p))] ,
$$
where $[(I,\psi_I^{-1}(p))]$ is independent of  
the choice of $I\in\Ii_\Kk$ with $p\in F_I$.
\MS

Notice that the above definitions contain no analog of the additivity condition
that we imposed in the smooth case, namely  that 
for each $I\in \Ii_\Kk$ the obstruction space $E_I$ for $\bK_I$
is the product  $\prod_{i\in I} E_i$. As will become clearer below, this condition is needed in order 
to construct shrinkings that satisfy the second
 tameness equation $\phi_{IJ}(U_{IK}) \;=\; U_{JK}\cap s_J^{-1}\bigl(\Hat\phi_{IJ}(E_I)\bigr)$,
 which in turn is a key ingredient of the proof of Proposition~\ref{prop:tame0}.
The topological analog of additivity is the following filtration condition.

\begin{defn}\label{def:Ku3}  
Let $\Kk$ be a  weak topological Kuranishi atlas.  We say that $\Kk$ is {\bf filtered} if 
it is equipped with a {\bf filtration}, that is a tuple of closed subsets $\E_{IJ}\subset \E_J$ for each $I,J\in \Ii_\Kk$ with $I\subset J$, that satisfy the following conditions:

\begin{enumerate}
\item $\E_{JJ}= \E_J$ and $\E_{\emptyset J} = \im 0_J$ for all $J\in\Ii_\Kk$;
\item 
$\Hat\Phi_{JK}(\E_{IJ}) = \E_{IK}\cap \pr_K^{-1}(\im \phi_{JK})$ for all $I,J,K\in\Ii_\Kk$ with
$I\subset J\subsetneq K$;
\item  $\E_{IJ}\cap \E_{HJ} = \E_{(I\cap H)J}$ for all $I,H,J\in\Ii_\Kk$ with $I, H \subset J$;
\item  $\im \phi_{IJ}$ is an open subset of $\s_J^{-1}(\E_{IJ})$
for all $I,J\in \Ii_\Kk$ with $I\subsetneq J$.
\end{enumerate}
\end{defn}

\begin{rmk}\rm  
Applying condition (ii) above to any triple $I=I \subsetneq J$ for $J\in\Ii_\Kk$ gives
\begin{equation}\label{eq:filt}
\Hat\Phi_{IJ}(\E_I)=\E_{IJ}\cap \pr_J^{-1}(\im\phi_{IJ}).
\end{equation}
In particular, by the compatibility of coordinate changes $\s_J\circ\phi_{IJ}=\Hat\Phi_{IJ}\circ\s_I|_{U_{IJ}}$
we obtain
$$
\s_J\bigl(\phi_{IJ}(U_{IJ})\bigr) = \Hat\Phi_{IJ}(\s_I(U_{IJ})) \subset  \Hat\Phi_{IJ}(\E_I) \subset \E_{IJ}.
$$
In other words, the first three conditions above imply 
the inclusion $\phi_{IJ}(U_{IJ}) \subset \s_J^{-1}(\E_{IJ})$.
Condition (iv) strengthens this 
by saying the image is open.  
The proof of the next lemma shows that this can be viewed as topological version of the index condition.
\hfill$\er$
\end{rmk}

\begin{lemma} \label{le:Ku3}
\begin{itemlist}
\item[\rm (a)] 
Any weak smooth Kuranishi atlas is filtered with filtration $\E_{IJ}: = U_J\times \Hat\phi_{IJ}(E_I)$, using the conventions $E_\emptyset:=\{0\}$ and $\Hat\phi_{JJ}:=\id_{E_J}$. 
\item[\rm (b)]
For any filtration $(\E_{IJ})_{I\subset J}$ on a weak topological Kuranishi atlas $\Kk$ we have for any $H,I,J\in\Ii_\Kk$
\begin{align} \label{eq:CIJ}
 H,I\subset J
\quad\Longrightarrow\quad &
\s_J^{-1}(\E_{IJ}) \cap \s_J^{-1}(\E_{HJ}) = \s_J^{-1}(\E_{(I\cap H)J}), 
\end{align}
in particular
\begin{align}\label{eq:CIJ0}
H\cap I=\emptyset \quad\Longrightarrow\quad & \s_J^{-1}(\E_{IJ}) \cap \s_J^{-1}(\E_{HJ}) = \s_J^{-1}(0_J).
\end{align}
\end{itemlist} 
\end{lemma}

\begin{proof}    
To check (a), first note that $U_J\times\Hat\phi_{IJ}(E_I)\subset U_J\times E_J$ is closed since $U_J\subset U_J$ and $\Hat\phi_{IJ}(E_I)\subset E_J$ are closed.
Property (i) in Definition~\ref{def:Ku3} holds by definition, 
while property (iii) holds because $\Hat\phi_{IJ}(E_I)\cap \Hat\phi_{HJ}(E_H) =
\Hat\phi_{(I\cap H)J}(E_{I\cap H})$.
Since
 $\Hat\Phi_{JK}=\phi_{JK}\times\Hat\phi_{JK}$,
we can use   the weak cocycle condition to derive   property (ii) as follows: 
\begin{align*}
\Hat\Phi_{JK}\bigl(U_J\times \Hat\phi_{IJ}(E_I)\bigr) \;
&= \; \im\phi_{JK} \times \Hat\phi_{JK}\bigl(\Hat\phi_{IJ}(E_I)\bigr) \;=\;  \im\phi_{JK} \times \Hat\phi_{IK}(E_I) \\
&= \; \bigl(U_K\times \Hat\phi_{IK}(E_I)\bigr) \cap \bigl( \im \phi_{JK} \times E_K \bigr) \;=\; \E_{IK}\cap (\pr_K^{-1}(\im \phi_{JK})) .
\end{align*}
Finally, property (iv) follows from the index condition for smooth coordinate changes; see  Remark~\ref{rmk:tbc}.
The displayed statements in (b) all follow from applying $\s_J$ to the defining property (ii), and making use of (i) in case $H\cap I=\emptyset$.
\end{proof}

The definition of tameness given in the smooth case in Definition~\ref{def:tame}
readily extends to  topological atlases; in fact the only change is that
the set $s_J^{-1}(\Hat\phi_{IJ}(E_I))$ in \eqref{eq:tame2} is replaced by
$\s_J^{-1}(\E_{IJ})$.

\begin{example}\label{ex:nonfilt}\rm
The prototypical example of a smooth atlas that 
is not filtered is one that has two basic charts $\bK_1, \bK_2$ with overlapping (but distinct)  
footprints, sections $\s_1,\s_2$ whose zero sets $\s_1^{-1}(0), \s_2^{-1}(0)$ have empty interiors,  
and the {\it same} obstruction space $E_i: = E\ne \{0\}$,  which we understand to mean
 that their transition chart also has obstruction space $E_{12} = E$, so that the linear maps  $\Hat\phi_{i(12)} $ are isomorphisms.   Then the index condition for smooth atlases implies that 
 $\phi_{i(12)}(U_{i(12)})$ is an open submanifold of $U_{12}=\s_{12}^{-1}(E) =  \s_{12}^{-1}\bigl(\im(\Hat\phi_{i(12}))\bigr)$
 that contains $\s_{12}^{-1}(0)$.  
 Further filtration condition (i) implies that
 $\E_i = U_i\times E$ for $i=1,2$, which implies by (ii) that $\E_{i(12)}$ contains  $\im \phi_{i(12)} \times E$. 
But then  (iii) implies that  
$$
\bigl(\im \phi_{1(12)} \times E\bigr)\cap \bigl(\im \phi_{2(12)} \times E\bigr) \subset \E_{1(12)}\cap  \E_{2(12)} = \E_{\emptyset (12)} = \s_{12}^{-1}(0),
$$
which means that $\s_{12}^{-1}(0)$ is the intersection of two open subsets of $U_{12}$ and hence is open.
Therefore  $\phi_{1(12)}^{-1}\bigl(\s_{12}^{-1}(0)\bigr) = \psi_i^{-1}(F_{12})$ is also open in $U_i$ which contradicts the choice of section $\s_i$.

We discuss in \S\ref{s:order} how to deal with situations (for example that of products) where the filtration
 assumption does not hold. 
For example, in the situation above it is possible to
redefine the transition chart  so that its obstruction space is the sum of two copies of $E$; see  method 2 in \S\ref{ss:nontriv} and the proof of Proposition~\ref{prop:hybrid1}.
$\hfill\er$
\end{example}

\begin{defn}\label{def:ttame}
A weak
topological Kuranishi atlas is {\bf tame} if \eqref{eq:tame1} holds, i.e.
$U_{IJ}\cap U_{IK}\;=\; U_{I (J\cup K)}$ for $I\subset J\subset K$, and it is 
equipped with a filtration $(\E_{IJ})_{I\subset J}$
such that for all $I,J,K\in \Ii_\Kk$
\begin{align}
\label{eq:ttame2}
\phi_{IJ}(U_{IK}) &\;=\; U_{JK}\cap 
\s_J^{-1}(\E_{IJ})
 \qquad\forall I\subset J\subset K.
\end{align}
\end{defn}

\begin{rmk}\label{rmk:tamef}\rm
The filtration conditions (i) -- (iv) work together with the two tameness identities \eqref{eq:tame1} and \eqref{eq:ttame2} in a rather subtle way.  
If conditions (i)--(iii) hold as well as the tameness identities, then
we can use 
the identity  
$$ 
\phi_{IK}(U_{IK})\cap \phi_{JK}(U_{JK}) =  \phi_{(I\cap J) K}(U_{(I\cap J)K})
$$
of \eqref{eq:tamer} to simplify the equivalence relation. 
This is the key tool in the proof that tame atlases have well behaved realizations.
On the other hand, as we see in the discussion following \eqref{eq:stepB},
condition (iv) for filtrations is just strong enough to allow the inductive construction 
of a  tame shrinking of  an atlas that also satisfies filtration conditions (i)--(iii).
$\hfill\er$
\end{rmk}

We now turn to the proof of Proposition~\ref{prop:tame0} which lists the good properties of tame atlases.
We first  prove part (ii)  for topological  atlases in Proposition~\ref{prop:Khomeo},  and part (iii) in
Proposition~\ref{prop:tame1}. 
The key fact is that in the tame case  the equivalence relation on $\Obj_{\bB_\Kk}$ simplifies significantly.
Let $\preceq$ denote the partial order on $\Obj_{\bB_\Kk}$ given by
$$
(I,x)\preceq(J,y) \quad \Longleftrightarrow \quad \Mor_{\bB_\Kk}((I,x),(J,y))\neq\emptyset .
$$
That is, we have $(I,x)\preceq (J,y)$ iff $x\in U_{IJ}$ and $y=\phi_{IJ}(x)$.
Since morphisms in $\bB_\Kk$ are closed under composition,
  $(I,x)\sim (J,y)$ iff there are elements $(I_j,x_j)$ such that 
\begin{align}\label{eq:ch}
&(I,x) = (I_0,x_0)\preceq  (I_1,x_1) \succeq (I_2,x_2) \preceq  \ldots \succeq (I_k, x_k)=(J,y).
\end{align}
For any $I,J\in\Ii_\Kk$ and subset $S_I\subset U_I$ we denote the subset of points in $U_J$ that are equivalent to a point in $S_I$ by
\begin{align}\label{eq:epsJ}
&\eps_J(S_I) \,:=\; \pi_\Kk^{-1}(\pi_\Kk(S_I)) \cap U_J \;=\;
\bigl\{y\in U_J  \,\big|\, \exists\, x\in S_I : (I,x)\sim (J,y) \bigr\} \;\subset\; U_J.
\end{align}
The next result is adapted from Lemma~3.2.3 in \cite{MW1}.

\begin{lemma} \label{le:Ku2}
Let $\Kk$ be a tame topological Kuranishi atlas.
\begin{enumerate}
\item [(a)] 
For $(I,x),(J,y)\in\Obj_{\bB_\Kk}$ the following are equivalent.
\begin{enumerate}
\item[(i)] $(I,x)\sim (J,y)$;
\item[(ii)] $(I,x)\preceq (I\cup J,z)\succeq (J,y)$ for some $z\in U_{I\cup J}$  (in particular $I\cup J\in \Ii_\Kk$);
\item[(iii)]  either
$(I,x)\succeq (I\cap J,w) \preceq (J,y)$ for some $w\in U_{I\cap J}$ (in particular $\emptyset\ne I\cap J\in \Ii_\Kk$), or $I\cap J=\emptyset$, $\s_I(x)=0_I(x),  \s_J(y) = 0_J(y)$, and $\psi_I(x)=\psi_J(y)$.
\end{enumerate}
\item[(b)]
$\pi_\Kk:U_I \to |\Kk|$ is injective for each $I\in\Ii_\Kk$, that is $(I,x)\sim (I,y)$ implies $x=y$
In particular, the elements $z$ and $w$ in (a) are automatically unique.
\item[(c)] 
For any $I,J\in\Ii_\Kk$ 
 and $S_I\subset U_I$ we have
\begin{align*}
\eps_J(S_I)  \;:=\; U_J\cap \pi_\Kk^{-1}\bigl(\pi_\Kk(S_I)\bigr)
&\;=\; \phi_{J(I\cup J)}^{-1}\bigl( \phi_{I (I\cup J)}(U_{I(I\cup J)}\cap S_I) \bigr) \\
&\;=\; \phi_{(I\cap J) J}\bigl( U_{(I\cap J)I} \cap \phi_{(I\cap J)I}^{-1}(S_I) \bigr) ,
\end{align*}
%
where in case $I\cap J=\emptyset$ the right hand side is $\eps_J(S_I) = \psi_J^{-1}\bigl( \psi_I(\s_I^{-1}(0_I)\cap S_I) \bigr)$,
consistently with \eqref{eq:empty}.  
\end{enumerate}
\end{lemma}
\begin{proof}  
The key step in the proof of (a)  is to show that  
 the taming conditions imply 
the equivalence of (ii) and (iii). In fact we need slightly more than this:  
the existence of a chain  $(I,x)\preceq (K,z) \succeq (J,y)$ implies (iii), while the existence of a chain 
$(I,x)\succeq (H,w) \preceq (J,y)$ implies (ii).

To prove these claims, suppose first that there is a chain $(I,x)\succeq (H,w) \preceq (J,y)$.
Then $w\in U_{HI}\cap U_{HJ} = U_{H(I\cup J)}$ by \eqref{eq:tame1}.  But then 
$$
x=\phi_{HI}(w) \in \phi_{HI}(U_{H(I\cup J)}) = U_{I(I\cup J)}\cap \s_I^{-1}(\E_{IH})
$$
by \eqref{eq:ttame2}, so that $\phi_{I(I\cup J)} (x)$ is defined.  Moreover,
$$
z: = \phi_{I(I\cup J)} (x) = \phi_{I(I\cup J)}\circ \phi_{HI}(w) = \phi_{H(I\cup J)}(w)\in U_{I\cup J},
$$ 
by the cocycle condition.
A similar argument shows that $z=\phi_{H(I\cup J)}(w) = \phi_{J(I\cup J)} (y)$.  Hence (ii) holds.  
Secondly, if there is a chain $(I,x)\preceq (K,z) \succeq (J,y)$
 then $$
 z \in \phi_{IK}(U_{IK})\cap \phi_{JK}(U_{JK}) 
 =\s_K^{-1}(\E_{IK})\cap \s_K^{-1}(\E_{JK}) = \s_K^{-1}(\E_{(I\cap J)K}) 
  $$
 by 
tameness \eqref{eq:ttame2} and the filtration condition \eqref{eq:CIJ}.
Since $\E_{\emptyset K} = \s_K^{-1}(0_K)$ by definition,
%
%
%
if $\s_K(z)\ne 0_K(z)$ we must have $I\cap J \ne \emptyset$. Further, whenever $I\cap J \ne \emptyset$ we have
\begin{equation}\label{eq:tamer}
z \in \phi_{IK}(U_{IK})\cap \phi_{JK}(U_{JK}) =\s_K^{-1}(\im(\Hat\phi_{(I\cap J)K})) =  \phi_{(I\cap J) K}(U_{(I\cap J)K}),
\end{equation}
which implies the  existence of suitable $w\in U_{(I\cap J)K}$. 
  On the other hand, if 
$I\cap J = \emptyset$ then
$\s_K(z)= 0_K(z)$ and the equivalence of (ii) and (iii) follows because as in \eqref{eq:zeroIJ}
$$
 (I,x)\sim (J,y), \s_I(x) = 0_I(x)\;\;\Longrightarrow\;\; \s_J(y) = 0_J(y), \psi_I(x) = \psi_J(y).
 $$
%

 Using this together with the fact that  
by composing morphisms one can get rid of consecutive occurrences of $\preceq$ and of $\succeq$, one shows that
if $\s_I(x)\ne 0_I(x)$ 
any chain as in \eqref{eq:ch} can be simplified 
 either to  $(I,x)\preceq (K,z) \succeq (J,y)$ or to $(I,x)\succeq (H,w) \preceq (J,y)$.  Applying the results once more, 
 one reduces to the case when $K = I\cup J$ and $H= I\cap J$.  
 When $\s_I(x)= 0_I(x)$, a similar argument shows that (i) is equivalent to (ii), and hence also to (iii).
 This proves (a).  Statement (b) then holds by applying (i) with $I=J$.    
  Finally, the formulas for $\eps_J(S_I)$ follow from the equivalent definitions of $\sim$ in (a).  
 \end{proof}

The above lemma is the basis for the proof of the following version of Proposition~\ref{prop:tame0}~(ii), taken from \cite[Proposition~3.1.13]{MW1}.

\begin{prop}\label{prop:Khomeo}
Suppose that the topological Kuranishi atlas $\Kk$ is tame. Then $|\Kk|$ and  $|\bE_\Kk|$ are Hausdorff, and for each $I\in\Ii_\Kk$ the quotient maps $\pi_{\Kk}|_{U_I}:U_I\to |\Kk|$ and $\pi_{\Kk}|_{U_I\times E_I}:U_I\times E_I\to |\bE_\Kk|$ are homeomorphisms onto their image.
\end{prop}
\begin{proof} We prove the claims about $|\Kk|$.  The arguments for $|\bE_\Kk|$ are similar,
based on the analog of Lemma~\ref{le:Ku2} for the category $\bE_\Kk$.

To see that $|\Kk|$ is Hausdorff, note first
that the equivalence relation on $O: =  \Obj_{\bB_\Kk}=\bigsqcup_{I\in \Ii_\Kk} U_I$ is closed, i.e.  the subset 
$$
R: = \bigl\{\bigl((I,x), (J,y)\bigr) \ | \ (I,x)\sim (J,y)\bigr\}\subset O\times O
$$
is closed.   Since $\Ii_\Kk$ is finite and $O\times O$ is the disjoint union of the metrizable sets $U_I\times U_J$, this will follow if we show that for all pairs $I,J$ and all convergence sequences $x^\nu\to x^\infty$ in $U_I$, $y^\nu\to y^\infty$ in $U_J$ with $(I,x^\nu)\sim (J,y^\nu)$ for all $\nu$, we have $(I,x^\infty)\sim (J,y^\infty)$. 
For that purpose denote $H:=I\cap J$.  If $H=\emptyset$, then all the points $x^\nu, y^\nu$ lie in the appropriate
 zero set and the result is immediate.  Hence suppose that $H\ne\emptyset$.  Then
by Lemma~\ref{le:Ku2}(a) there is a sequence $w^\nu\in U_H$ such that $x^\nu = \phi_{HI}(w^\nu)$ and $y^\nu = \phi_{HJ}(w^\nu)$.
Now it follows from the tameness condition \eqref{eq:tame3} that $x^\infty$ lies in the
relatively closed subset $\phi_{HI}(U_{HI})=\s_I^{-1}(\E_{HI})\subset U_I$, and since $\phi_{HI}$ is a homeomorphism to its image we deduce convergence $w^\nu\to w^\infty\in U_{HI}$ to a preimage of $x^\infty=\phi_{HI}(w^\infty)$.  Then by continuity of the transition map we obtain $\phi_{HJ}(w^\infty) = y^\infty$, so that $(I, x^\infty)\sim (J, y^\infty)$ as claimed.  
Thus $R$ is closed.  Hence $|\Kk|$ is Hausdorff by
 \cite{Bourb} (see  Exercise 19, \S10, Chapter 1), which states
 that if  a Hausdorff  space $O$ can be written as the union $\bigcup_{k\ge 1} O_k$ of precompact open subsets
  $O_k$ with $\ov{O_k}\subset O_{k+1}$ for all $k$,  then
  its  quotient by a closed relation  is Hausdorff.    

To show that $\pi_\Kk|_{U_I}$ is a homeomorphism onto its image, first recall that it is injective by Lemma \ref{le:Ku2}~(c).
It is moreover continuous since $|\Kk|$ is equipped with the quotient topology. Hence it remains to show that $\pi_\Kk|_{U_I}$ is an open map to its image, i.e.\ for a given open subset $S_I\subset U_I$ we must find an open subset $\Ww\subset |\Kk|$ such that $\Ww\cap \pi_{\Kk}(U_I) = \pi_\Kk(S_I)$. 
Equivalently, we set $\Ww:=|\Kk|\less\Qq$ and construct the complement 
$$
\Qq := \bigl(\io_\Kk(X) \cup {\textstyle \bigcup_{H\subset I}} \pi_\Kk(U_H)\bigr) \less \pi_\Kk(S_I)\;\subset\; |\Kk| .
$$
With that the intersection identity follows from $\Qq\cap \pi_{\Kk}(U_I) = \pi_\Kk(U_I)\less \pi_\Kk(S_I)$, so it remains to show that $U_J\cap \pi_\Kk^{-1}(\Qq)$ is closed for each $J$.
In case $J\subset I$ we have $U_J\cap \pi_\Kk^{-1}(\Qq) = U_J\less \eps_J(S_I)$,
which is closed iff $\eps_J(S_I)$ is open.
Indeed, Lemma~\ref{le:Ku2}~(d) gives $\eps_J(S_I)=\phi_{JI}^{-1}(S_I)\subset U_J$, and this is open since $S_I\subset U_I$ and hence $S_I\cap U_{IJ}\subset U_{IJ}$ is open and $\phi_{JI}$ is continuous.

In case $J\not\subset I$ 
to show that $U_J\cap \pi_\Kk^{-1}(\Qq)$ is closed, we 
express it 
as the union of 
$Q^0_J:=U_J\cap \pi_\Kk^{-1}(\io_\Kk(X)\less \pi_\Kk(S_I))$ with the union over $H\subset I$ of 
\begin{align*}
Q_{JH} \,:=\;
U_J\cap \pi_\Kk^{-1} \bigl( \pi_\Kk(U_H)\less \pi_\Kk(S_I)\bigr) 
& = U_J\cap \pi_\Kk^{-1} \bigl( \pi_\Kk(U_H\less \eps_H(S_I))\bigr) \\
& = \eps_J\bigl(U_H\less \phi_{HI}^{-1}(S_I)\bigr) \;=\; \eps_J(C_H) .
\end{align*}
Note here that $C_H:=U_H\less \phi_{HI}^{-1}(S_I) \subset U_H$ is closed since as above 
$\phi_{HI}^{-1}(S_I) \subset U_H$ is open.
We moreover claim that this union can be simplified to
\begin{equation}\label{qqquark}
U_J\cap \pi_\Kk^{-1}(\Qq) \; = \; Q^0_J \; \cup \bigcup_{H\subset I\cap J} \eps_J(C_H) .
\end{equation}
Indeed, in case $H\cap J= \emptyset$ we have
$\eps_J(C_H) \subset  Q^0_J$ since Lemma~\ref{le:Ku2}~(d) gives 
\begin{align*}
\eps_J\bigl(U_H\less \phi_{HI}^{-1}(S_I)\bigr) 
&= \psi_J^{-1}\bigl(\psi_{H} \bigl(\s_H^{-1}(0_H)\less \phi_{HI}^{-1}(S_I) \bigr)\bigr) \\
&= \psi_J^{-1}\bigl(F_H \less \bigl(F_I \cap \pi_\Kk(S_I) \bigr)\bigr)
\subset U_J\cap \pi_\Kk^{-1}(\io_\Kk(X)\less \pi_\Kk(S_I)) = Q^0_J.
\end{align*}
For $H\not\subset J$ with $H\cap J\neq \emptyset$ 
we have $\eps_J(C_H)\subset \eps_J(C_{H\cap J})$ since
Lemma~\ref{le:Ku2}~(d) gives 
$$
\eps_J\bigl(U_H\less \phi_{HI}^{-1}(S_I)\bigr) 
= \phi_{(H\cap J)J}\bigl(
U_{(H\cap J)J} \cap
\phi_{(H\cap J)H}^{-1} (U_H\less \phi_{HI}^{-1}(S_I))\bigr) 
\subset
U_J \cap \pi_\Kk^{-1}( C_{H\cap J} ),
$$
where $\phi_{(H\cap J)H}^{-1}\bigl( U_H\less \phi_{HI}^{-1}(S_I) \bigr) \subset  U_{H\cap J}\less \phi_{(H\cap J)I}^{-1}(S_I) = C_{H\cap J}$ by the cocycle condition.
This confirms \eqref{qqquark}.

It remains to show that $Q^0_J$ and $\eps_J(C_H)$ for $H\subset I\cap J$ are closed.
For the latter, Lemma~\ref{le:Ku2}~(d) gives 
$\eps_J(U_H\less \phi_{HI}^{-1}(S_I))=\phi_{HJ}(U_{HJ}\cap C_H)$,
which is closed in $U_J$ since closedness of $C_H\subset U_H$ implies relative closedness of $U_{HJ}\cap C_H \subset U_{HJ}$, and $\phi_{HJ}: U_{HJ} \to U_J$ is a homeomorphism onto a closed subset of $U_J$ by \eqref{eq:tame3}.
Finally,
$$
Q^0_J = 
U_J\cap \pi_\Kk^{-1}(\io_\Kk(X)\less \pi_\Kk(S_I))  \;= \; 
s_J^{-1}(0_J) \less \psi_J^{-1}\bigl( F_J \cap \psi_I(\s_I^{-1}(0_I)\cap S_I)\bigr)
$$
is closed in $s_J^{-1}(0_J)$ and hence in $U_J$, since both $F_J$ and $\psi_I(\s_I^{-1}(0_I)\cap S_I)$ are open in $X$, and $\psi_J:s_J^{-1}(0_J)\to X$ is continuous.
Thus $U_J\cap \pi_\Kk^{-1}(\Qq)$ is closed, as claimed, which finishes the proof of the homeomorphism property of $\pi_\Kk|_{U_I}$.
\end{proof}

\begin{cor}\label{cor:Khomeo}  If $\Kk$ is tame, then for every $I\in \Ii_\Kk$ we have
$$
\ov{\pi_\Kk(U_I)}\;\; \subset \;\;\io(X)\cup{\textstyle  \bigcup_{H\subset I} }\pi_\Kk(U_H).$$
\end{cor}
\begin{proof}  It suffices to check that $\Qq': = \io(X)\cup \bigcup_{H\subset I} \pi_\Kk(U_H)$ is closed in $|\Kk|$.
But $\Qq'$ is the set $\Qq$ of the above proof in the case when $S=\emptyset$.
\end{proof}

\begin{prop}\label{prop:tame1}
 Proposition~\ref{prop:tame0} holds.
 \end{prop}
 \begin{proof}
 We proved part~(i) just before stating  Proposition~\ref{prop:tame0}, and  part (ii) follows from
Proposition~\ref{prop:Khomeo}
because by Lemma~\ref{le:Ku3} every tame smooth  atlas may also be considered as a tame topological  atlas.
  Thus it remains to prove (iii), which asserts that in the smooth case the fibers of
$|\pr_{\Kk}|: |\bE_\Kk| \to |\Kk|$ have a  unique linear structure such that for every $I\in\Ii_\Kk$ the embedding $\pi_{\Kk} : U_I\times E_I \to  |\bE_{\Kk}|$ is linear on the fibers.

To this end, for fixed $p\in |\Kk|$
denote the union of index sets for which $p\in \pi_\Kk(U_I)$ by
$$
I_p:= \bigcup_{I\in\Ii_\Kk, p\in \pi_\Kk(U_I)}  I
\qquad \subset \{1,\ldots,N\} .
$$
To see that $I_p\in\Ii_\Kk$ we repeatedly use the observation that Lemma~\ref{le:Ku2}~(a) implies
$$
p \in \pi_\Kk(U_I)\cap\pi_\Kk(U_J) \quad\Rightarrow\quad (I,\pi_\Kk^{-1}(p)\cap U_I)\sim (J,\pi_\Kk^{-1}(p)\cap U_J) \quad\Rightarrow\quad I\cup J \in \Ii_\Kk .
$$
Moreover, $x_p:=\pi_\Kk^{-1}(p)\cap U_{I_p}$ is unique by Lemma~\ref{le:Ku2}~(c).
Next, any element in the fiber $[I,x,e]\in |\pr_{\Kk}|^{-1}(p)$ is represented  by some vector over $(I,x)\in \pi_\Kk^{-1}(p)$, so we have $I\subset I_p$ and $\phi_{I I_p}(x)=x_p$, and hence
$(I,x, e)\sim (I_p,x_p,\Hat\phi_{I I_p}(e))$. Thus $\pi_\Kk:\{x_p\}\times E_{I_p}\to |\pr_\Kk|^{-1}(p)$ is surjective, and by Lemma~\ref{le:Ku2}~(c) also injective.
Thus the requirement of linearity for this bijection induces a unique linear structure on the fiber $|\pr_\Kk|^{-1}(p)$.
To see that this is compatible with the injections $\pi_\Kk:\{x\}\times E_{I}\to |\pr_\Kk|^{-1}(p)$ for $(I,x)\sim(I_p,x_p)$ note again that $I\subset I_p$ since $I_p$ was defined to be maximal, and hence by Lemma~\ref{le:Ku2}~(b)~(ii) the embedding factors as $\pi_\Kk|_{\{x\}\times E_{I}} = \pi_\Kk|_{\{x_p\}\times E_{I_p}} \circ \Hat\phi_{I I_p}$, where $\Hat\phi_{I I_p}$ is linear by definition of coordinate changes. Thus $\pi_\Kk|_{\{x\}\times E_{I}}$ is linear as well.
\end{proof}

\begin{rmk}\label{rmk:Kk}\rm (i) The above construction gives a rather nice picture of the virtual neighbourhood  $|\Kk|$
for a tame  topological 
atlas.  By Corollary~\ref{cor:Khomeo}, it is a union of sets $\pi_\Kk(U_I)$, each of which 
has  frontier
$\ov{\pi_\Kk(U_I)}\less \pi_\Kk(U_I)$ contained  in the union of the zero set $\io(X)$ with the sets $\bigcup_{H\subset I} \pi_\Kk(U_H)$. 
A pairwise  intersection $\pi_\Kk(U_I)\cap \pi_\Kk(U_J)$ is nonempty only if 
the corresponding footprint intersection $F_I\cap F_J = F_{I\cup J}$ is nonempty, in which case Lemma~\ref{le:Ku2}~(a)
implies that 
$\pi_\Kk(U_I)\cap \pi_\Kk(U_J)\subset \pi_\Kk(U_{I\cup J}) $.  If also $I\cap J\ne \emptyset$, then   $\pi_\Kk(U_I)\cap \pi_\Kk(U_J)$  may be identified with  the subset 
$\pi_\Kk\bigl(\s_{I\cup J}^{-1}(\E_{(I\cap J)(I\cup J)})\bigr)$ of $\pi_\Kk(U_{I\cup J}) $. However, if $I\cap J = \emptyset$ then these two sets intersect only  along the zero set $\io_\Kk(X)$,
where $\io_\Kk$ is as in \eqref{eq:ioK}.
In the smooth case, each set  
$\pi_\Kk(U_I)$ is a homeomorphic image of a manifold, and it is easy to see from the index condition that
if $I\cap J\ne \emptyset$ the intersection of
 $\pi_\Kk(U_I)$ with $\pi_\Kk(U_J)$  is  transverse when considered inside  the \lq\lq submanifold"
$\pi_\Kk(U_{I\cup J})$. 
  For example, the images of the
  domains of two basic charts $\pi_\Kk(U_1)$ and $\pi_\Kk(U_2)$ will in general intersect nontransversally in their common footprint $\io_\Kk(F_{12})$, while the two transition domains $\pi_\Kk(U_{12})$ and $\pi_\Kk(U_{23})$
intersect transversally in the submanifold $\pi_\Kk(U_{2})\cap \pi_\Kk(U_{123})$ of $\pi_\Kk(U_{123})$.
\MS

\NI (ii)
The effect of the taming condition is to reduce the equivalence relation to a two step process:
$(I,x)\sim (J,y)$ iff we can write 
$(I,x)\preceq (I\cup J,z) \succeq (J,y)$, or equivalently (if $s_I(x)\ne 0$)
$(I,x)\succeq (I\cap J,w) \preceq (J,y)$.    The reduction process described in \S\ref{ss:red} below
will simplify the equivalence relation even further to a single step.
In fact, this process discards all the elements in $U_I\less V_I$, for suitable choice of open sets $V_I\subset U_I$, 
so that when $x\in V_I, y\in V_J$ we have $(I,x)\sim (J,y)$  only if 
$(I,x)\preceq  (J,y)$ or $(I,x) \succeq (J,y)$.\MS

\NI (iii)  See  \cite[Example 2.4.3]{MW1}  for a (non tame) atlas for which the map $\pi_\Kk$ is not injective on $U_I$.\footnote
{
In this example the atlas is not additive.  However   it can easily be made additive by replacing the charts $(U_i)_{1\le i\le 3}$ described there  by charts
$(U_{i4})_{1\le i\le 3}$ for a suitable chart $\bK_4$ with $E_4 = \R$ and where $E_i = \{0\}$. }
It is also easy to construct non tame examples where the argument in Proposition~\ref{prop:tame1} fails, so that some 
fibers of $|\pr|: |\bE_\Kk|\to  |\Kk|$ are not  vector spaces.
$\hfill\er$  \end{rmk}

To complete the proof of the first part of Theorem~\ref{thm:K} it remains
to  explain how to construct a metrizable tame shrinking of  a weak atlas.  
To accomplish this,
we again  work in the context of topological atlases in order to be able to apply the argument 
when the isotropy is nontrivial.  Note that the definition of 
a shrinking given in
Definition~\ref{def:shr} makes sense  in this context without change. Moreover, since smooth atlases are automatically filtered
by Lemma~\ref{le:Ku3}, we will suppose that the initial weak atlas $\Kk$  is filtered.

We define the filtration $(\E_{IJ}')_{I\subset J}$ on a shrinking $\Kk'$  by restriction 
$\E_{IJ}': = \E_{IJ}\cap \E_J'$ in the obvious way.  Note that
any shrinking of a  weak Kuranishi atlas preserves the weak cocycle condition (since the latter only requires equality on overlaps). 
Moreover, a shrinking  is determined by the choice of the domains $U'_I\sqsubset U_I$ of the transition charts (since condition (iii) then specifies the domains of the coordinate changes), and so can be considered as the restriction of $\Kk$ to 
the subset
$\bigsqcup_{I\in\Ii_\Kk} U_I'\subset\Obj_{\bB_\Kk}$.
However, for a shrinking to satisfy a stronger form of the cocycle condition (such as tameness) the  domains $U'_{IJ}:= \phi_{IJ}^{-1}(U'_J)\cap U'_I$ of the coordinate changes must satisfy appropriate compatibility conditions, so that the domains $U_I'$ can no longer be chosen independently of each other.
Since the relevant conditions are expressed in terms of the $U_{IJ}'$, we next show  that the construction of a 
tame shrinking  can be achieved by iterative choice of these sets $U_{IJ}'$.

The following result
is proved in \cite[\S3.3]{MW1}.   

\begin{prop}  \label{prop:proper1}
Every filtered weak topological Kuranishi atlas $\Kk$ has a tame shrinking $\Kk'$.
\end{prop}
\begin{proof}[Sketch of proof]
Since $X$ is compact and metrizable and the footprint open cover $(F_i)$ is finite, it has a shrinking $(F_i')$ in the sense of Definition~\ref{def:shr0}.
In particular we can ensure that $F_I'\ne \emptyset$ whenever $F_I\ne \emptyset$ by choosing $\de>0$ so that every nonempty $F_I$ contains some ball $B_\de(x_I)$ and then choosing the $F_i'$ to contain $B_{\de/2}(x_I)$ for each $I\ni i$ (i.e.\ $F_I\subset F_i$).
Then we obtain $F'_I\neq\emptyset$ for all $I\in \Ii_\Kk$ since $B_{\de/2}(x_I)\subset 
\bigcap_{i\in I} F_i' =F_I'$.

In another preliminary step, we find precompact open subsets $U_I^{(0)}\sqsubset U_I$ and open sets $U_{IJ}^{(0)}\subset U_{IJ}\cap U_I^{(0)}$ for all $I,J\in\Ii_\Kk$ such that 
\begin{equation}\label{eq:U(0)}
U_I^{(0)}\cap \s_I^{-1}(0) = \psi_I^{-1}(F_I'),\qquad U_{IJ}^{(0)}\cap \s_I^{-1}(0) = \psi_I^{-1}(F'_I\cap F_J').
\end{equation}
Here we choose any suitable  $U_I^{(0)}$ (which is possible by \cite[Lemma~2.1.4]{MW1}), and 
 then define the  
$U_{IJ}^{(0)}$ by restriction:
\begin{equation}\label{eq:UIJ(0)}
U_{IJ}^{(0)}:= U_I^{(0)}\cap \phi_{IJ}^{-1}( U_J^{(0)})\subset U_{IJ}.
\end{equation}
We  then construct  
the required shrinking 
$\Kk'$ by choosing possibly smaller domains  $U_I'\subset U_I^{(0)}$ and $U_{IJ}'\subset U_{IJ}^{(0)}$ with  the same footprints $F_I'$.
We  also arrange $U_{IJ}' = U_I'\cap \phi_{IJ}^{-1}(U_J')$, so that $\Kk'$ is a shrinking of the original $\Kk$.  
Therefore we just need to make sure that $\Kk'$  satisfies the tameness conditions~\eqref{eq:tame1}  and~\eqref{eq:tame2}.

We construct the domains $U_I', U_{IJ}'$ by a finite iteration, starting with $U_I^{(0)}, U_{IJ}^{(0)}$.
Here we streamline the notation by setting $U_{I}^{(k)}:=U_{II}^{(k)}$ and extend the notation to all pairs of nonempty subsets $I\subset J\subset\{1,\ldots,N\}$ by setting $U_{IJ}^{(k)}=\emptyset$ if $J\notin\Ii_\Kk$. (Note that $J\in\Ii_\Kk$ and $I\subset J$ implies $I\in\Ii_\Kk$.)
Then in the $k$-th step we  construct open subsets $U_{IJ}^{(k)}\subset U_{IJ}^{(k-1)}$ for all $I\subset J\subset\{1,\ldots,N\}$
so that the following holds.
\begin{enumerate}
\item
The zero set conditions $U_{IJ}^{(k)}\cap \s_I^{-1}(0) = \psi_I^{-1}(F_J')$ hold for all $I\subset J$.
\item
The first tameness condition \eqref{eq:tame1} holds for all $I\subset J,K$ with $|I|\le k$, that is
$$
U_{IJ}^{(k)}\cap U_{IK}^{(k)}= U_{I (J\cup K)}^{(k)} .
$$
In particular, we have $U_{IK}^{(k)} \subset U_{IJ}^{(k)}$ for $I\subset J \subset K$ with $|I|\le k$.
\item
The second tameness condition \eqref{eq:tame2} holds for all $I\subset J\subset K$ with $|I|\le k$, that is
$$
\phi_{IJ}(U_{IK}^{(k)}) = U_{JK}^{(k)}\cap \s_J^{-1}(\E_{IJ}) .
$$
In particular we have $\phi_{IJ}(U_{IJ}^{(k)}) = U_{J}^{(k)}\cap \s_J^{-1}(\E_{IJ})$ for all $I\subset J$ with $|I|\le k$.
\end{enumerate}
In other words, we need the tameness conditions to hold up to level $k$.

The above choice of the domains $U_{IJ}^{(0)}$ completes the $0$-th step since conditions (ii) (iii) are vacuous. Now suppose that the $(k-1)$-th step is complete for some $k\geq 1$.
We then define $U_{IJ}^{(k)}:=U_{IJ}^{(k-1)}$ for all $I\subset J$ with $|I|\leq k-1$. For $|I|=k$ we also set $U_{II}^{(k)}:=U_{II}^{(k-1)}$.
This ensures that (i) and (ii) continue to hold for $|I|<k$. In order to preserve (iii) for triples $H\subset I\subset J$ with $|H|<k$ we then require that the intersection $U_{IJ}^{(k)}\cap \s_I^{-1}(\E_{HI})= U_{IJ}^{(k-1)}\cap \s_I^{-1}(\E_{HI})$ is fixed.
In case
$H=\emptyset$, this is condition (i), and since $U_{IJ}^{(k)}\subset U_{IJ}^{(k-1)}$ it can generally be phrased as inclusion (i$'$) below.
With that it remains to construct the open sets $U_{IJ}^{(k)}\subset U_{IJ}^{(k-1)}$ as follows.
\begin{itemize}
\item[(i$'$)]
For all $H\subsetneq I\subset J$ with $|H|<k$ and $|I|\geq k$ we have $U_{IJ}^{(k-1)}\cap \s_I^{-1}(\E_{HI}) \subset U_{IJ}^{(k)}$. Here we include $H=\emptyset$, in which case the condition says that
$U_{IJ}^{(k-1)}\cap \s_I^{-1}(0)\subset U_{IJ}^{(k)}$ (which implies $U_{IJ}^{(k)}\cap \s_I^{-1}(0) = \psi_I^{-1}(F_J')$, as explained above).
\item[(ii$'$)]
For all $I\subset J,K$ with $|I|= k$ we have
$U_{IJ}^{(k)}\cap U_{IK}^{(k)}= U_{I (J\cup K)}^{(k)}$.
\item[(iii$'$)]
For all $I\subsetneq J\subset K$ with $|I|=k$ we have
$\phi_{IJ}(U_{IK}^{(k)}) = U_{JK}^{(k)}\cap \s_J^{-1}(\E_{IJ})$.
\end{itemize}
The construction is then completed in two steps.  

\MS\NI
{\bf Step A}\; constructs $U_{IK}^{(k)}$ for $|I|=k$ and $I\subsetneq K$ satisfying (i$'$),(ii$'$) and

\begin{itemize}
\item[(iii$''$)]
$U_{IK}^{(k)} \subset \phi_{IJ}^{-1}(U_{JK}^{(k-1)})$
for all $I\subsetneq J\subset K$ .
\end{itemize}
{\bf Step B} constructs $U_{JK}^{(k)}$ for $|J|>k$ and $J\subset K$ satisfying (i$'$) and (iii$'$).
\MS

Step A assumes the existence of suitable sets $U_{IJ}^{(k-1)}$ for all $I\subset J$ and uses the following nontrivial result to show that the required sets 
$U_{IJ}^{(k)}$ with $|I|=k$ exist.

\begin{lemma}[Lemma~3.3.6 in \cite{MW1}] \label{le:set}
Let $U$ be a locally compact  metric space, $U'\sqsubset U$ a precompact open set, and $Z\subset U'$ a relatively closed subset. Suppose we are given a finite collection of relatively open subsets $Z_i\subset Z$ for $i=1,\ldots,N$ and open subsets $W_K\subset U'$ with
$$
W_K\cap Z= Z_K: = {\textstyle\bigcap_{i\in K}} Z_i
$$
 for all index sets $K\subset\{1,\ldots,N\}$.
Then there exist open subsets $U_K\subset W_K$ with $U_K\cap Z=Z_K$ and
$U_J\cap U_K = U_{J\cup K}$ for all $J,K\subset\{1,\ldots,N\}$. 
\end{lemma}
 
  Step B then completes the inductive step for $|I|=k$, modifying the sets $U_{JK}^{(k-1)}$ by removing the extra parts that contradict  (iii$'$). In other words we define
\begin{equation}\label{eq:stepB}\textstyle
U_{JK}^{(k)}\,:=\; U_{JK}^{(k-1)}\less \bigcup_{I\subset J, |I|= k}
 \bigl( \s_J^{-1}(\E_{IJ})\less \phi_{IJ}(U^{(k)}_{IJ}) \bigr) .
\end{equation}
Note that $U_{JK}^{(k)}$ is open since  the sets $ \s_J^{-1}(\E_{IJ})\less \phi_{IJ}(U^{(k)}_{IJ})$ that we remove are closed,
which holds because
$\phi_{IJ}$ is a homeomorphism to its image which is open in $\s_J^{-1}(\E_{IJ})$ by  condition (iv) of Definition~\ref{def:Ku3}.

 We illustrate how Step A works by explaining the construction of the level $1$ sets $U^{(1)}_{IJ}$ with $|I| = 1$.   
 Fix $I = \{i_0\}$, and 
 choose an arbitrary order for the set $$
 \{i \ \big| \ i\ne i_0, \{i_0,i\}\in \Ii\} \cong \{1,\dots, N'\}.
 $$  Then the index sets $K\in \Ii$ containing $\{i_0\}$ as a proper subset are in one-to-one correspondence with nonempty index sets $K'\subset \{1,\dots,N'\}$ via $K = \{i_0\}\cup K'$. 
Apply
%
%
%
%
%
Lemma~\ref{le:set} to the metric space $U:=U_{i_0}$, its precompact open subset $U':=U^{(0)}_{i_0}$, the relatively closed subset $Z := U^{(0)}_{i_0} \cap \s_{i_0}^{-1}(0_{i_0}) = \psi_{i_0}^{-1}(F_{i_0}')$ of $U'$, the relatively open subsets $Z_i := \psi_{i_0}^{-1}(F_{i_0}'\cap F'_i) \subset Z$ for all $i\in \{1,\dots, N'\}$, and the open subsets
$$
\textstyle
W_{K'}: = \bigcap_{\{i_0\}\subset J\subset K} \bigl( \phi_{i_0J}^{-1} (U^{(0)}_{JK}) \cap U^{(0)}_{i_0J} \bigr),
$$
where $K = \{i_0\}\cup K'$ as above, and we write $U_{i_0J}: = U_{\{i_0\}J}$. 
With these choices, the assumptions of  Lemma~\ref{le:set} are satisfied when $K,N$ is replaced by $K', N'$. 
 In particular, \eqref{eq:U(0)} implies that
$$ \textstyle
W_{K'}\cap Z =  U^{(0)}_{i_0 K}\cap  \s_{i_0}^{-1}(0_{i_0}) = \psi_{i_0}^{-1}(F_{i_0}'\cap F'_K)
=  \psi_{i_0}^{-1}(F_{i_0}'\cap  \bigcap_{i\in K'} F'_i) 
= \bigcap_{i\in K'} Z_i =: Z_{K'}.
$$
Thus the lemma provides open subsets 
$U^{(1)}_{i_0 K}:= U_{K'} \subset W_{K'}= U^{(0)}_{i_0 K}$
which satisfy the conditions (i$'$), (ii$'$)  and (iii$''$) as follows:

\begin{itemlist}
\item[(i$'$)] 
$U_{i_0 J}^{(0)}\cap  \s_{i_0}^{-1}(0_{i_0}) \subset U_{i_0 J}^{(1)}$ for all $J\supset \{i_0\}$ 
translates to 
 $U_{i_0 J}^{(0)}\cap  Z  \subset U_{J'}$, which holds because $U_{i_0 J}^{(0)}=W_{J'}$
and $W_{J'} \cap  Z  = Z_{J'} = U_{J'}\cap Z \subset U_{J'}$;

\item[(ii$'$)] 
follows from the second property in Lemma~\ref{le:set};

\item[(iii$''$)]  $U_{i_0 K}^{(1)} \subset \phi_{i_0 J}^{-1}(U_{JK}^{(0)})$ for $\{i_0\}\subsetneq J\subset K$ translates to $U_{K'} \subset \phi_{i_0 J}^{-1}(U_{JK}^{(0)})$, which holds because
\eqref{eq:UIJ(0)} and the definition of $W_{K'}$ as an intersection
imply
 $$
 \phi_{i_0 J}^{-1}(U_{JK}^{(0)}) 
 = \phi_{i_0 J}^{-1}\bigl( U_J^{(0)}\cap \phi_{JK}^{-1}( U_K^{(0)}) \bigr)
 =  \phi_{i_0 J}^{-1}( U_J^{(0)}) \cap \phi_{i_0 K}^{-1}( U_K^{(0)}) 
 \supset W_{K'} , 
 $$
 while $U_{K'}\subset W_{K'}$ by construction.  
 \end{itemlist}

If $k=2$ the argument is similar.  Given $I$ wth $|I|=2$ we apply Lemma~\ref{le:set}  with  $U = U_I$ and $U' = U_I^{(1)}$, but now defining $Z, Z_i$ in terms of the level $1$  sets $U_{HK}^{(1)}$ where $|H|=1$.  Note that the images of these sets in $U_I$ lie in $\s_I^{-1}(\E_{HI})$ for some $H\subsetneq I$.
Thus we define 
$$
Z := \textstyle
\bigcup_{H\subsetneq I} \bigl( U_{II}^{(1)} \cap  \s_I^{-1}(\E_{HI}) \bigr)  
\;=\; \bigcup_{H\subsetneq I} \phi_{HI}\bigl(U_{HI}^{(1)}\bigr)
\;\subset\; U'
$$
and then put $Z_i := U^{(1)}_{I (I\cup\{i\})} \cap Z$ for  $i\in\{1,\ldots,N\}\less I$.
Finally, in order to achieve property (iii$'$) we again define $W_{K'}$ to be an intersection, namely, if $K:= I\cup K'$ we set
$$
\textstyle
W_{K'}: = \bigcap_{I\subset J\subset K} \bigl( \phi_{IJ}^{-1} (U^{(1)}_{JK}) \cap U^{(1)}_{IJ} \bigr).
$$
  For further details, see \cite[\S3.3]{MW1}.\end{proof}

 \begin{rmk}\label{rmk:add1}\rm 
 As we  explained in Remark~\ref{rmk:tamef}~(i), the existence of the filtration on the initial weak atlas $\Kk$
 is a crucial ingredient of this inductive proof.  As a concrete example of this, 
 consider a  smooth weak atlas  
that contains just three charts $\bK_1, \bK_2$ and $\bK_{12}$ each with obstruction space $E$ so that $\Hat\phi_{i(12)} = \id$ for $i=1,2$.
 Then when $k=1$ we must construct sets $U_{i (12)}^{(1)}$ for $i=1,2$ that both satisfy
 $\phi_{i(12)}\bigl(U_{i(12)}^{(1)}\bigr) = U_{12}^{(1)}\cap \s_{12}^{-1}(E) =U_{12}^{(1)} $.  Hence the choices of the two level one sets 
$U_{1 (12)}^{(1)}$ and $U_{2 (12)}^{(1)}$  are not independent.
 In an additive situation\footnote
 {
 i.e. one where $E_I = \prod_{i\in I} E_i$}, one can only have $E_1 = E_{12} = E$ if $E_2 = \{0\}$. In this case we still need 
 $\phi_{1(12)}\bigl(U_{1(12)}^{(1)}\bigr) = U_{12}^{(1)}$.  However, the condition for $i=2$ is  $\phi_{2(12)}\bigl(U_{2(12)}^{(1)}\bigr) =
  \s_2^{-1}(0)$, which has been arranged at level $0$.  
$\hfill\er$   \end{rmk}  

Finally we establish the metrizability of preshrunk shrinkings.

\begin{lemma}\label{le:metriz} 
For every tame shrinking $\Kk'$ of a tame topological  atlas $\Kk$ 
\begin{itemize} \item[(i)] the map $|\Kk'| \to |\Kk|$ of quotient spaces is injective and continuous, and
\item[(ii)]  $\Kk'$  is metrizable, with a metric that induces the subspace topology.
\end{itemize}
\end{lemma}
\begin{proof}  
We write $U_I, U_{IJ}$ for the domains of the charts and coordinate changes of $\Kk$ and $U_I', U_{IJ}'$  for those of $\Kk'$, so that $U_I'\subset U_I, U_{IJ}'\subset U_{IJ}$ for all $I,J\in \Ii_\Kk = \Ii_{\Kk'}$.
Suppose that $\pi_\Kk(I,x) = \pi_\Kk(J,y)$ where $x\in U_I', y\in U_J'$.  Then we must show that $\pi_{\Kk'}(I,x) = \pi_{\Kk'}(J,y)$.  Since $\Kk$ is tame, 
Lemma~\ref{le:Ku2}~(a) implies that
if $I\cap J\ne \emptyset$
 there is $w\in U_{I\cap J}$ such that
$\phi_{(I\cap J)I} (w)$ is defined and equal to $x$.  Hence $x\in \s_I^{-1}(\E_{I\cap J})\cap U_I' =\phi_{(I\cap J)I} (U'_{(I\cap J)I})$ by the tameness equation \eqref{eq:tame3}  for $\Kk'$.
Therefore $w\in U'_{(I\cap J)I}$. 
Similarly, because $\phi_{(I\cap J)J} (w)$ is defined and equal to $y$, we have $w\in U'_{(I\cap J)J}$.
Then by definition of $\pi_{\Kk'}$ we deduce
 $\pi_{\Kk'}(I,x) = \pi_{\Kk'}(I\cap J,w) =\pi_{\Kk'}(J,y)$.
 On the other hand, if $I\cap J\ne \emptyset$ then 
 Lemma~\ref{le:Ku2}~(a) implies that $\psi_I(x) = \psi_J(y)$ so that 
  $\pi_{\Kk'}(I,x) = \pi_{\Kk'}(J,y)$ by the injectivity of $\io_\Kk: X\to |\Kk|$ (see   Lemma~\ref{le:ioK}).
 
To prove (ii), notice that 
because  $U_I'$ is precompact in $U_I$ for each $I$,  the closure $\ov{\io(|\Kk'|)}$ of $\io(|\Kk'|)$ in $|\Kk|$ is compact.  
It follows that  $\ov{\io(|\Kk'|)}$ (with the subspace topology) is regular  (i.e.\ points and closed sets have disjoint neighbourhoods) 
and second countable, hence metrizable by Urysohn's metrization theorem. (For  details of these arguments see  \cite[Proposition~3.3.8]{MW1}.)
Therefore one may obtain the desired metric on 
 $|\Kk'|$ by restriction from $\ov{\io(|\Kk'|)}$. 
 \end{proof}

\begin{cor}\label{cor:proper}  Proposition~\ref{prop:proper} holds.
\end{cor}
\begin{proof}  By Proposition~\ref{prop:proper1} each weak atlas has a tame shrinking $\Kk'$
with domains $U_I'$. 
 Now apply Proposition~\ref{prop:proper1} again to obtain a further tame shrinking
$\Kk_{sh}$ of $\Kk'$. with domains $U_I\sqsubset U_I'$.  The metric on $\Kk_{sh}$ provided  by Lemma~\ref{le:metriz} 
pulls back to a metric on $U_I$  that is compatible with the  topology on $U_I$
because the  inclusion $U_I\stackrel{\pi_{\Kk_{sh}} }\longrightarrow   |\Kk_{sh}|\subset |\Kk'|$
also factors as $U_I\to U_I' \stackrel{\pi_{\Kk'} }\longrightarrow |\Kk'|$ and hence is a homomorphism to its image when this is 
  topologized as a subspace of  $|\Kk'|$.
Hence $\Kk_{sh}$ is metrizable in the sense of
Definition~\ref{def:metric}.  
The other assertions of Proposition~\ref{prop:proper}   hold by Lemma~\ref{le:metriz}).
\end{proof}

The above corollary, together with the results in \S\ref{ss:tatlas}, establish the key parts of the first set of statements in Theorem~\ref{thm:K}.
The rest of this theorem concerns cobordisms. The proofs are fairly straightforward generalizations of those given above; for details see \cite{MW1}.

Finally, in order to work in  \S\ref{s:order}
with  atlases that are not filtered, we introduce the  notion of a good atlas as follows.

\begin{defn}\label{def:good}  A topological atlas $\Kk = \bigl(\bK_I, \Hat\Phi_{IJ}\bigr)_{I\subset J, I,J\in \Ii_\Kk}$ is said to be 
{\bf good}  if
\begin{itemize} \item[(i)] the realization $|\Kk|$ of $\Kk$ is Hausdorff in the quotient topology;
\item[(ii)] for all $I\in \Ii_\Kk$ the map $\pi_\Kk: U_I\to |\Kk|$ is a homeomorphism to its image;
\item[(iii)] $\Kk$ is metrizable, i.e. there is a metric $d$ on $|\Kk|$ 
whose pullback by $\pi_\Kk$ induces the given topology on $U_I$ for each $I\in \Ii_\Kk$;
\item[(iv)] for all $I\subset J$ with $I,J\in \Ii_\Kk$ the image $\phi_{IJ}(U_{IJ})$ is a closed subset of $U_J$.
\end{itemize}
\end{defn}

\begin{lemma}\label{le:good} 
 Every preshrunk tame topological atlas is good.
\end{lemma}
\begin{proof}  Proposition~\ref{prop:Khomeo}  shows
that every tame topological atlas satisfies  conditions (i) and (ii) above, while
 Lemma~\ref{le:metriz} shows that (iii) holds.  Further  
 (iv) holds by the tameness condition \eqref{eq:ttame2} with $J=K$.   
\end{proof}

\begin{rmk}\rm  
 Once we have constructed a well behaved virtual neighbourhood $|\Kk|$,
condition~(iv) in Definition~\ref{def:Ku3} is all that we need of the filtration/tameness conditions in order to construct the VFC.
It will be used in the proof that the zero set of  a perturbed section is  compact; cf
Proposition~\ref{prop:zeroS0} below.  $\hfill\er$ 
\end{rmk}

\subsection{Reductions and the construction of perturbation sections}\label{ss:red}

In this section,  
we explain how to 
construct the  virtual moduli cycle $[X]^{vir}_\Kk$ for  an oriented smooth good atlas (with trivial isotropy).
 The discussion has three main parts.
 \begin{itemlist}\item   We first develop the notion of a {\bf reduction $\Vv$ of a good topological atlas}.
 This gives us a subcategory $\bB_\Kk|_\Vv$  of $\bB_\Kk$ whose footprints  cover $X$ and 
 with the property that two of its objects $(I,x), (J,y)$ are equivalent if and only if either $(I,x)\preceq (J,y)$ or $(I,x)\succeq (J,y)$.  Comparing with Lemma~\ref{le:Ku2}, we see
 that the equivalence relation $\sim_\Kk$ becomes significantly simpler when restricted to $\bB_\Kk|_\Vv$.
 \item  Secondly, we define 
the notion of 
a {\bf perturbation section} $\nu:\bB_\Kk|_\Vv\to \bE_\Kk|_\Vv$ of a reduction of a smooth atlas (with trivial isotropy), and 
in Proposition~\ref{prop:zeroS0}
give conditions under which 
 the local zero sets of $\s|_\Vv + \nu$ 
 fit together to form a closed oriented manifold $|\bZ^\nu|$.  
\item   We then  explain how to construct suitable perturbation sections $\nu$.  Although
the conditions formulated in Definition~\ref{def:a-e} are rather intricate, they give a great deal of control
over  $\nu$, and in particular over the zero set of $\s|_\Vv + \nu$.  This will allow us both to
construct cobordisms between admissible perturbations as in Proposition~\ref{prop:ext2} and to adapt the construction to the case of nontrivial isotropy as explained at the end of \S\ref{ss:iso2}.
 \item Finally we outline the proof of Theorem~B in the smooth case with trivial isotropy, showing how to define
 $[X]^{vir}_\Kk$ from the zero sets $|\bZ^\nu|$.
 \end{itemlist}
 
We begin by explaining why it is necessary to \lq\lq reduce" the atlas $\Kk$. 
The cover of $X$ by the footprints $(F_I)_{I\in \Ii_\Kk}$ of all the Kuranishi charts
(both the basic charts and those that are part of the transitional data) is closed under intersection. This makes it easy to express compatibility of the charts, since the overlap of footprints of any two charts $\bK_I$ and $\bK_J$ is the footprint of a third  chart $\bK_{I\cup J}$.
However, this yields so many compatibility conditions that a construction of compatible perturbations in the Kuranishi charts may not be possible. For example, a choice of perturbation (with values in $E_I$) in the chart $\bK_I$ also fixes the perturbation in each chart $\bK_J$ over $\phi_{J (I\cup J)}^{-1}\bigl( \im \phi_{I (I\cup J)}\bigr) \subset U_J$, whenever $I\cup J\subset\Ii_\Kk$.  
Since we do not assume transversality of the coordinate changes, this subset of $U_J$ need not be a submanifold,
\footnote
{
As explained in Remark~\ref{rmk:Kk}, it will be a submanifold if $I\cap J\ne \emptyset$, but usually not otherwise.
}
and hence the perturbation may not extend smoothly to a map from $U_J$ to $E_J$.  Moreover, for such an extension to exist at all, the 
pushforward of the perturbation to $U_{I\cup J}$ 
would have to take values in the  intersection  $$
\Hat\phi_{I(I\cup J)}(E_I) \cap 
\Hat\phi_{J(I\cup J)}(E_J) \;\;\subset \;\; \Hat\phi_{(I\cap J)(I\cup J)} (E_{I\cap J}),
$$ a very restrictive condition.  In fact if $I\cap J = \emptyset$, this would mean that the perturbation would have to vanish over $F_{I\cup J}$.
We will avoid these difficulties, and also make a first step towards compactness, by reducing the domains of the Kuranishi charts to precompact subsets $V_I\sqsubset U_I$ such that all compatibility conditions between $\bK_I|_{V_I}$ and $\bK_J|_{V_J}$ are given by direct coordinate changes $\Hat\Phi_{IJ}$ or $\Hat\Phi_{JI}$.
As explained more fully in \cite{MW1} the reduction process is analogous to replacing the star cover of a simplicial set by the star cover of its first barycentric subdivision; also see  Figure~\ref{fig:3}.

\begin{figure}
   \centering
   \includegraphics[width=4in]{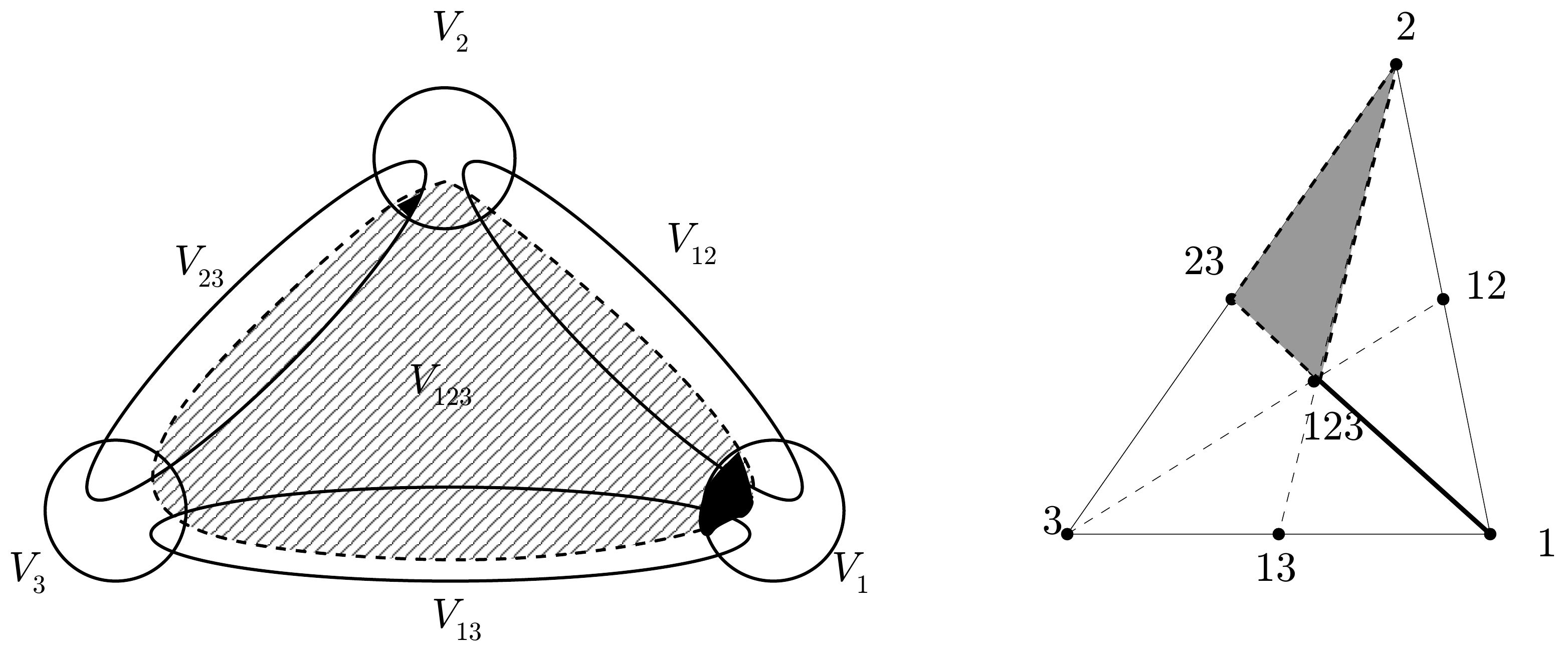}
     \caption{
The right diagram shows the
first barycentric subdivision of the triangle with vertices $1,2,3$.
It has three new vertices labelled $ij$ at the barycenters of the three edges and one vertex labelled $123$ at the barycenter of the triangle.
The left is a schematic picture of the cover $(V_I)$ by the stars  of the   vertices of this barycentric subdivision. 
The black sets are examples of multiple intersections of the new cover, which correspond to the simplices in the barycentric subdivision. E.g.\ $V_2\cap V_{23}\cap V_{123}$ corresponds to the triangle with vertices $2, 23, 123$, whereas $V_1\cap V_{123}$ corresponds to the edge between $1$ and $123$. 
The cover $(V_I)$  has the same intersection properties as the reduction of the original cover $U_1,U_2,U_3$, where $U_i$ is the
 star  of the  vertex $v_i$.}  
\label{fig:3}
\end{figure}

In the following we work with good atlases (with properties as spelled out  in Definition~\ref{def:good}).  
Since all tame atlases are good,  for the present purposes we could equally well work with tame atlases.

\begin{defn}\label{def:vicin}  Let $\Kk$ be a good topological atlas. 
A {\bf reduction} of 
$\Kk$ is an open subset $\Vv=\bigsqcup_{I\in \Ii_\Kk} V_I \subset \Obj_{\bB_\Kk}$ i.e.\ a tuple of (possibly empty) open subsets $V_I\subset U_I$, satisfying the following conditions:
\begin{enumerate}
\item
$V_I\sqsubset U_I $ for all $I\in\Ii_\Kk$, and if $V_I\ne \emptyset$ then $V_I\cap \s_I^{-1}(0)\ne \emptyset$;
\item
if $\pi_\Kk(\ov{V_I})\cap \pi_\Kk(\ov{V_J})\ne \emptyset$ then
$I\subset J$ or $J\subset I$;
\item
the zero set $\iota_\Kk(X)=|\s_\Kk|^{-1}(0)$ is contained in 
$
\pi_\Kk(\Vv) \;=\; {\textstyle{\bigcup}_{I\in \Ii_\Kk}  }\;\pi_\Kk(V_I).
$
\end{enumerate}
Given a reduction $\Vv$, we define the {\bf reduced domain category} $\bB_\Kk|_\Vv$ and the {\bf reduced obstruction category} $\bE_\Kk|_\Vv$ to be the full subcategories of $\bB_\Kk$ and $\bE_\Kk$ with objects $\bigsqcup_{I\in \Ii_\Kk} V_I$ resp.\ $\bigsqcup_{I\in \Ii_\Kk} V_I\times E_I$, and denote by $\s|_\Vv:\bB_\Kk|_\Vv\to \bE_\Kk|_\Vv$ the section given by restriction of 
$\s_\Kk$. 
\end{defn}

There is a similar notion of good cobordism.
It follows easily from condition (ii) above that the realization $|\bB_\Kk|_\Vv|$  of the subcategory $\bB_\Kk|_\Vv$  (i.e. its object space modulo the equivalence relation generated by its morphisms) injects into $ |\bB_\Kk|=:|\Kk|$.
There is a related notion of cobordism reduction (see  \cite[\S5.1]{MW1}), which is just as one would imagine, keeping in mind that all sets have product form near the boundary.

Here is the main existence result.  It is proved in \cite[\S5]{MW1} by first constructing a reduction of the footprint cover (a process well understood in algebraic topology) as in Lemma~\ref{le:cov0} below, and then extending this suitably.

\begin{prop}\label{prop:cov2} 
The following statements hold.
\begin{itemize}
\item[(a)]
Every good topological Kuranishi atlas $\Kk$ has a reduction $\Vv$.
\item[(b)] 
Every good topological Kuranishi cobordism $\Kk^{[0,1]}$ has a cobordism reduction $\Vv^{[0,1]}$.
\item [(c)]  
Let $\Vv^0,\Vv^1$ be reductions of a good topological Kuranishi atlas $\Kk$. Then there exists a cobordism reduction $\Vv$ of $[0,1]\times \Kk$ such that $\p^\al\Vv = \Vv^\al$ for $\al = 0,1$. 
\end{itemize}
\end{prop}
\begin{proof}[Sketch  proof of (a)] 
Define 
$$
\Cc(I):=\{J\in \Ii_\Kk \,|\, I\subset J  \;\text{or}\; J\subset I \},
$$
We show below in Lemma~\ref{le:cov0} that  the footprint cover $(F_I)_{I\in \Ii_\Kk}$ has a reduction 
$(Z_I)_{I\in \Ii_\Kk}$, i.e. there is an open cover of $X$ such that $Z_I\sqsubset F_I$ and $\ov{Z_I}\cap \ov{Z_J}\ne \emptyset\Rightarrow J\in \Cc(I)$.
  For each $I$, choose an open subset $W_I\sqsubset U_I$ such that 
  $$
W_I\cap \s_I^{-1}(0_I) = \psi_I^{-1}(Z_I), \quad \ov{W_I}\cap \s_I^{-1}(0_I) = \psi_I^{-1}(\ov{Z_I}),
$$
and for each $J\notin \Cc(I)$ define
$$
Y_{IJ} \,:= \ov{W_I}\cap \pi_\Kk^{-1}(\pi_\Kk(\ov{W_J}))
\; = \; 
U_I
\cap \pi_\Kk^{-1}\bigl(\pi_\Kk(\ov{W_I})\cap \pi_\Kk(\ov{W_J})\bigr).
$$
Then $Y_{IJ}$  is compact by the  homeomorphism property of $\pi_\Kk$. Further, 
If $J\notin\Cc(I)$, then using $\psi_I^{-1}(\ov{Z_I}) =\s_I^{-1}(0_I)\cap \ov{W_I}$ 
and the notation $\eps_I$ from \eqref{eq:epsJ}
 we have
 \begin{align*}
\psi_I^{-1}(\ov{Z_I}) \cap Y_{IJ}&\;\subset
\;\psi_I^{-1}(\ov{Z_I})\cap \s_I^{-1}(0_I)\cap 
\eps_I(\ov{W_J}) \\
&\;=\; \psi_I^{-1}(\ov{Z_I}) \cap \eps_I\bigl(\s_J^{-1}(0_J) \cap \ov{W_J} \bigr)\\
&\;= \; \psi_I^{-1}(\ov{Z_I})\cap \eps_I\bigl(\psi_J^{-1}(\ov{Z_J})\bigr)
\;= \; \psi_I^{-1}(\ov{Z_I}) \cap \psi_I^{-1}(\ov{Z_J})\;=\; \emptyset,
\end{align*}
where the first equality holds because $\s$ is compatible with the coordinate changes.
Because the $Y_{IJ}$  are compact,  we can find closed neighbourhoods 
$\Nn(Y_{IJ})\subset U_I$ of $Y_{IJ}$
for each $J\notin\Cc(I)$ such that  
$$
\Nn(Y_{IJ})\cap \psi_I^{-1}(\ov{Z_I})= \emptyset.
$$
It is now straightforward to check  that the sets 
\begin{equation}\label{eq:QIVI}
V_I := W_I \;\less {\textstyle\bigcup_{J\notin \Cc(I)}} \Nn(Y_{IJ}).
\end{equation}
form the required reduction. The other parts of the proposition have similar proofs.
\end{proof}

\begin{example}\label{ex:Vv} \rm  (i) A reduction of the atlas $\Kk$ in Example~\ref{ex:Khomeo} has three sets $V_1, V_2, V_{12}$ that cover
the zero set and have the property that $\pi_\Kk(\ov V_1)\cap \pi_\Kk(\ov V_2) = \emptyset$.  For instance, we can take $V_1 = (-\infty, 2)\sqsubset U_1, V_2 = (3,\infty)\times \R\sqsubset U_2$ and $V_{12} = (1,3)\times \R\subset U_{12} (= U_2)$.
\MS

\NI (ii)  It is clear from the above proof that we can construct the reduction $(V_I)_{I\in \Ii_\Kk}$ in such a way that
there is a coordinate change between $V_I$ and $V_J$ only if the corresponding footprints have nonempty intersection  $Z_I\cap Z_J$.  (After choosing the $Z_I$, simply replace $\Cc(I)$ in the above proof by the set $\Cc_Z(I) = \{J: Z_I\cap Z_J\ne \emptyset\}$.)     For example,  if $X = \{p_1,p_2,p_3\}$ and $\Kk$ has two basic charts $U_1, U_2$ with footprints  
$F_1= \{p_1,p_3\}, F_2= \{p_2,p_3\}$, there is one transition chart $\bK_{12}$ with footprint $\{p_3\}$, and we may take
$Z_1 = \{p_1\}, Z_2 = \{p_2\}, Z_{12} = \{p_3\}$.  We can then choose $\Vv$ so that the sets $\bigl(\pi_\Kk(V_I)\bigr)_{I\in \Ii_\Kk}$ are disjoint, so that the only morphisms in the category $\bB_\Kk|_\Vv$ are identity morphisms.
 In other words $V_I\cap \phi^{-1}_{IJ}(V_J) = \emptyset$ for all $I\subset J$, so that the compatibility condition 
between  $\nu_I, \nu_J$ in Definition~\ref{def:sect} is automatically satisfied. \hfill$\er$
\end{example}

\begin{rmk}\label{rmk:piVv} \rm  (i) The  key property of a reduction $\Vv$ is that
 if any intersection $\pi_\Kk(\ov V_I)\cap \pi_\Kk(\ov V_J)$ is nonempty then  $I\subset J$ or $J\subset I$ so that 
 the induced equivalence relation on $\bigsqcup V_I$ has just one step. 
  Contrast this with the equivalence relation for  a tame atlas which, as explained in Remark~\ref{rmk:Kk}~(ii), 
  is given by a two step process.
Another consequence in the smooth case is that the intersection $\pi_\Kk(V_I)\cap \pi_\Kk(V_J)$ is  the image of a submanifold;
see  Remark~\ref{rmk:Kk}~(i).   \MS

\NI (ii)  Another important property is that the inclusion of categories $\bB_\Kk|_\Vv\to \bB_\Kk$ induces an injection on their realizations,
 in other words the equivalence relation on $\bigsqcup_I V_I = \Obj_{\bB_\Kk|_\Vv}$ induced by the
 morphisms in $\bB_\Kk|_\Vv$ coincides with that induced by the morphisms in $\bB_\Kk$. 
(To see that this is not true for arbitrary full
 subcategories of $\bB_\Kk$ consider the full subcategory with objects $\bigcup_i U_i$
which only has identity morphisms.)
It follows that  the quotient topology on $|\Vv|: = \qu{\bigcup V_I}{\sim_\Vv}$ induced by the
morphisms in $\bB_\Kk|_\Vv$ coincides with  that  induced by the morphisms in $\bB_\Kk$, so that $|\Vv|$ 
can be identified with its image $\pi_\Kk(\Vv)$ in $|\Kk|$ with the quotient topology.
Further, because $\Vv$ is precompact in  $ \Obj_{\bB_\Kk}$ one can show that the relative topology 
on $\pi_\Kk(\Vv)$ as a subset of $|\Kk|$ equals its metric topology with respect to any admissible metric on $|\Kk|$; see \cite[Lemma~5.2.5]{MW1}.

\MS

\NI (iii)
   Subsets of $|\Kk|$ of the form $\pi_\Kk(\Vv)$ contain the zero set $\io_\Kk(X)$ and are the analog in the virtual neighbourhood $|\Kk|$  of  \lq\lq precompact neighbourhoods of the zero set".  Since $|\Kk|$ is not locally compact even in the metric topology
(see  Example~\ref{ex:Khomeo}), there are no compact neighbourhoods of the zero set.  
On the other hand, because $V_I\sqsubset U_I$, the subset $\pi_\Kk(\Vv)$ is precompact in $|\Kk|$, and is \lq\lq open" to the extent that it is the image by $\pi_\Kk$ of the open set $\bigsqcup_I V_I$. 
(But, of course, it is almost never actually  open in $|\Kk|$.)  We can interpret Figure~\ref{fig:3} as a 
schematic picture of 
the subsets $\pi_\Kk(V_I)$ in $|\Kk|$, though it is not very accurate since the dimensions of the $V_I$ change. 
$\hfill\er$  \end{rmk}

We now introduce the notion of a perturbation from \cite[\S7.2]{MW2}.  For the rest of this subsection we assume that our atlas is smooth as well as good in the sense of Definition~\ref{def:good}.

\begin{defn}\label{def:sect}   Let $\Kk$ be a good smooth  atlas with trivial isotropy. 
A {\bf reduced perturbation} of $\Kk$ is a smooth functor $\nu:\bB_\Kk|_\Vv\to\bE_\Kk|_\Vv$ between the reduced domain and obstruction categories of 
some reduction $\Vv$ of $\Kk$, 
such that $\pr_\Kk\circ\nu$ is the identity functor. 
That is, $\nu=(\nu_I)_{I\in\Ii_\Kk}$ is given by a family of smooth maps $\nu_I: V_I\to E_I$ such that for each $I\subsetneq J$ we have a commuting diagram
\begin{equation}\label{eq:comp}
\xymatrix{
 V_I\cap \phi_{IJ}^{-1}(V_J)   \ar@{->}[d]_{\phi_{IJ}} \ar@{->}[r]^{\qquad\nu_I}    &  E_I \ar@{->}[d]^{\Hat\phi_{IJ}}   \\
V_J \ar@{->}[r]^{\nu_J}  & E_J.
}
\end{equation}
We say that a reduced perturbation $\nu$ is  {\bf admissible}  if 
\begin{equation}\label{eq:admiss}
\rd_y \nu_J(\rT_y V_J) \subset\im\Hat\phi_{IJ} \qquad \forall \; I\subsetneq J, \;y\in V_J\cap\phi_{IJ}(V_I\cap U_{IJ}) .
\end{equation}
\end{defn}

Each reduced perturbation  $\nu:\bB_\Kk|_\Vv\to\bE_\Kk|_\Vv$ induces a continuous map $|\nu|: |\Vv|\to |\bE_\Kk|_\Vv|$ such that $|\pr_\Kk|\circ |\nu| = \id$, where $|\pr_\Kk|$ is as in Theorem~\ref{thm:K}.
Each such map has the further property that $|\nu|\big|_{\pi_\Kk(V_I)}$ takes values in $\pi_\Kk(U_I\times E_I)$.  
Note that the zero section $0_\Kk$, given by $U_I\to 0 \in E_I$, restricts to an admissible perturbation $0_\Vv:\bB_\Kk|_\Vv\to\bE_\Kk|_\Vv$ in the sense of the above definition. 
Similarly, the canonical section $\s:= \s_\Kk$ of the Kuranishi atlas restricts to a perturbation $\s|_\Vv: \bB_\Kk|_\Vv\to\bE_\Kk|_\Vv$ of any reduction.
However, the canonical section is generally not admissible. In fact, it follows from the index condition
that  for all $y\in V_J\cap \phi_{IJ}(V_I\cap U_{IJ})$ the map 
$$
{\rm pr}_{E_I}^\perp\circ \rd_y s_J \,: \;\;  \quotient{\rT_y U_J} {\rT_y (\phi_{IJ}(U_{IJ}))} \; \longrightarrow \; \quotient{E_J}{\Hat\phi_{IJ}(E_I)} 
$$
is an isomorphism, while for an admissible section it is identically zero.
Thus for any reduction $\Vv$ and admissible perturbation $\nu$, the sum 
\begin{equation}\label{eq:nuindex}
\s|_\Vv+\nu:=(s_I|_{V_I}+\nu_I)_{I\in\Ii_\Kk} \,:\; \bB_\Kk|_\Vv \;\to\; \bE_\Kk|_\Vv
\end{equation}
is a reduced section that satisfies the index condition. \MS

Here are some further definitions. 
\begin{itemlist}
\item We say that two reductions $\Cc, \Vv$ are {\bf nested} and write  $\Cc\sqsubset \Vv$ if
$C_I\sqsubset V_I$ for all $I\in \Ii_\Kk$.   One can show that any two such pairs 
$\Cc_0\sqsubset \Vv_0, \Cc_1\sqsubset \Vv_1$ are cobordant via a nested cobordism  $\Cc^{01}\sqsubset \Vv^{01}$.

\item  A perturbation $\nu:\bB_\Kk|_\Vv\to\bE_\Kk|_\Vv$ is called 
{\bf precompact} if there is a nested  reduction $\Cc\sqsubset \Vv$   such that
\begin{equation}\label{eq:C}
\pi_\Kk\bigl((\s|_\Vv + \nu)^{-1}(0)\bigr)
\;\subset\; \pi_\Kk(\Cc).
\end{equation}
\item
It is called {\bf transverse} if  for all $z\in V_I\cap (s_I|_{V_I}+\nu_I)^{-1}(0)$ 
  the map $\rd_z(s_I+\nu_I): \rT U_I\to E_I$ is surjective.
\end{itemlist}
It is not hard to see if $\nu$ is admissible, transversality of the local sections 
$\nu_I$  is preserved under coordinate changes.  
More precisely,  if $z\in V_I\cap U_{IJ}$ and $w\in V_J$ are such that $\phi_{IJ}(z) = w$,
then $z$ is a transverse zero of $s_I|_{V_I}+\nu_I$ if and only if 
$w$ is a transverse zero of $s_J|_{V_J}+\nu_J$.    
(This holds because the sections $s_I|_{V_I}+\nu_I$ satisfy the index condition; see 
 \eqref{eq:nuindex} above.)
Moreover,  if $\Kk$ is oriented then the local orientation of the zero set at $z\in U_I$ defined as in
 Lemma~\ref{le:locorient1} is taken by $\rd_z\phi_{IJ}$ to the orientation at $w$.

Here is the main result about the zero sets, from \cite[Proposition~7.2.8]{MW2}.
In it, we show that the local zero sets 
 $Z_I =  (s_I|_{V_I} + \nu_I)^{-1}(0)$  can be assembled into 
a  nonsingular  ep groupoid\footnote{
For the definition of ep groupoid see the beginning of \S\ref{ss:orb}; nonsingular means 
that there is at most one morphism between any two points.}
 $\bZ^\nu$ whose realization
is a closed oriented manifold $|\bZ^\nu|$ that injects into $|\Kk|$.

\begin{prop} \label{prop:zeroS0}
Let $\Kk$ be a good smooth $d$-dimensional Kuranishi atlas with trivial isotropy and  reduction $ \Vv\sqsubset \Kk$, and 
suppose that $\nu: \bB_\Kk|_\Vv \to \bE_\Kk|_\Vv$ is a precompact, admissible,  transverse perturbation.
Then $|\bZ^\nu| = |(\s|_\Vv+ \nu)^{-1}(0)|$ is a smooth closed $d$-dimensional manifold, whose 
 quotient topology agrees with the subspace topology on $|(\s|_\Vv+ \nu)^{-1}(0)|\subset|\Kk|$.

Moreover, if $\Kk$ is oriented, so is $|\bZ^\nu|$.
\end{prop}
\begin{proof}   Since $\nu$ is transverse, each local zero set $Z_I: = (s_I|_{V_I} + \nu_I)^{-1}(0)$ is a smooth manifold. 
Hence one may construct a  groupoid $\bZ^\nu$ with objects $\sqcup_I Z_I$ and morphisms 
generated  by  the restriction of the coordinate changes $\phi_{IJ}$ and their inverses. 
Since these are diffeomorphisms and there are a finite number of coordinate patches, 
the realization $|\bZ^\nu|$ (with the quotient topology)  is second countable.  The space $|\bZ^\nu|$ is also Hausdorff because its inclusion into the Hausdorff space
$|\Kk|$  is continuous.  Therefore it is a  smooth $d$-dimensional  manifold.

To check the compactness of $|\bZ^\nu|$,  it suffices to check sequential compactness.
To this end, consider a sequence $(p_k)_{k\in\N}\subset |\bZ^\nu|$. In the following we will index all subsequences by $k\in\N$ as well.
By finiteness of $\Ii_\Kk$ there is a subsequence of $(p_k)$ that has lifts in $(s_I|_{V_I}+\nu_I)^{-1}(0)$ for a fixed $I\in\Ii_\Kk$. 
In fact, by \eqref{eq:C}, and using the language of equation~\eqref{eq:epsJ}, the subsequence lies in
$$
V_I \cap \pi_\Kk^{-1}\bigl(\pi_\Kk(\Cc)\bigr)   \;=\;  V_I \cap {\textstyle \bigcup_{J\in\Ii_\Kk}} \eps_I(C_J)
\;\subset\; U_I .
$$
Here $\eps_I(C_J)=\emptyset$ unless $I\subset J$ or $J\subset I$, due to the intersection property (ii) of Definition~\ref{def:vicin} and the inclusion $C_J\subset V_J$.
So we can choose another subsequence and lifts $(x_k)_{k\in\N}\subset V_I$ with $\pi_\Kk(x_k)=p_k$ such that either 
$$
(x_k)_{k\in\N}\subset V_I \cap \phi_{IJ}^{-1}(C_J)
\qquad\text{or}\qquad 
(x_k)_{k\in\N}\subset V_I \cap \phi_{JI}(C_J\cap U_{JI})
$$ 
for some fixed  $J$ satisfying either $I\subset J$ or  $J\subset I$.
In the first case, compatibility of the perturbations \eqref{eq:comp} implies that there are other lifts
$\phi_{IJ}(x_k)\in (s_J|_{V_J}+\nu_J)^{-1}(0)\cap C_J$, which by precompactness $\ov{C_J}\sqsubset V_J$ have a convergent subsequence $\phi_{IJ}(x_k)\to y_\infty \in (s_J|_{V_J}+\nu_J)^{-1}(0)$.
Thus we have found a limit point in the perturbed zero set 
$p_k = \pi_\Kk(\phi_{IJ}(x_k)) \to \pi_\Kk(y_\infty) \in |\bZ^\nu|$,
as required for sequential compactness.

On the other hand, if $J\subset I$ we use the relative closedness of $\phi_{JI}(U_{JI})\subset U_I$ (which holds by condition (iii) for good atlases)
and the precompactness $V_I\sqsubset U_I$ to find a convergent subsequence 
$x_k\to x_\infty \in \ov{V_I} \cap \phi_{JI}(U_{JI})$.
Since $\phi_{JI}$ is a homeomorphism to its image, this implies convergence of the preimages
$y_k:= \phi_{JI}^{-1}(x_k) \to \phi_{JI}^{-1}(x_\infty) =: y_\infty \in U_{JI}$.
By construction and compatibility of the perturbations \eqref{eq:comp}, this subsequence $(y_k)$ of lifts of $\pi_\Kk(y_k)=p_k$ moreover lies in $(s_J|_{V_J}+\nu_J)^{-1}(0)\cap C_J$.
Now precompactness of $C_J\sqsubset V_J$ implies $y_\infty\in V_J$, and continuity of the section implies $y_\infty\in (s_J|_{V_J}+\nu_J)^{-1}(0)$. Thus we again have a limit point 
$p_k = \pi_\Kk(y_k) \to \pi_\Kk(y_\infty) \in |\bZ^\nu|$.
This proves that the perturbed zero set $|\bZ^\nu|$ is sequentially compact, and hence compact. 
Therefore it is a closed manifold, as claimed. 

So far we have considered $|\bZ^\nu|$ with the quotient topology.
By construction there is a continuous injection $|\bZ^\nu|\to |\Kk|$.  Because $|\Kk|$ is Hausdorff and 
$|\bZ^\nu|$ is compact, this map is a homomorphism to its image, when this is given  the subspace topology.
This proves the last claim in the first paragraph.  Finally note that 
 if $\Kk$ is oriented, we may orient each $Z_I$ as in 
Lemma~\ref{le:locorient1}. The fact that the orientation is defined by a bundle over $|\Kk|$ implies that its restriction to the $Z_I$ is preserved by coordinate changes.  Hence $\bZ$ is oriented if $\Kk$ is.
\end{proof}


There are similar result for cobordisms. In particular, if $\Kk$ is an oriented cobordism (or concordance) from $\Kk^0$ to $ \Kk^1$, 
then for appropriate perturbations  $\nu$ the zero set   $|\bZ^\nu|$ provides an oriented cobordism (or concordance) between 
the zero set of $\nu|_{\Kk^0}$ and that of  $\nu|_{\Kk^1}$.
\MS

\NI {\bf Construction of the perturbation $\nu$:}
The main remaining problem is to construct suitable perturbations  $\nu$.  
 Even though the constructions are rather intricate, the statements of the main results in Propositions \ref{prop:ext} and \ref{prop:ext2} below are very precise. Moreover, our language gives us great control over all aspects of the construction, 
 so that it can be 
 easily adapted  to other situations.  
 
 In order  that the construction apply to the case with isotropy, we will begin by
 working with a good topological atlas.  By definition, we may therefore
choose an {\bf admissible metric} $d$ on $|\Kk|$  as 
in Definition~\ref{def:good}, i.e. a metric whose pullback  $d_I$ to each domain $U_I$ is compatible with its topology.  
We denote
 the $\de$-neighbourhoods of subsets $Q\subset |\Kk|$ resp.\ $A\subset U_I$ for $\de>0$ by
\begin{align*}
B_\de(Q) &\,:=\; \bigl\{w\in |\Kk|\ | \ \exists q\in Q : d(w,q)<\de \bigr\}, \\
B^I_\de(A) &\,:=\; \bigl\{x\in U_I\ | \ \exists a\in A : d_I(x,a)<\de \bigr\}.
\end{align*}
Note that $\phi_{IJ}( B^I_\de(A)) = B^J_\de\bigl(\phi_{IJ}(A)\bigr)$ because
all coordinate changes are isometries.  Similarly $U_I\cap \pi_\Kk^{-1}(B_\de(Q)) = B^I_\de\bigl(U_I\cap \pi_\Kk^{-1}(Q)\bigr)$.
Further, we define a decreasing sequence of nested reductions  $\Vv^k: = (V_I^k)_{I\in \Ii_\Kk}\sqsupset \Vv^{k+1}: = (V_I^{k+1})_{I\in \Ii_\Kk} $, where 
\begin{equation}\label{eq:VIk}
V_I^k \,:=\; B^I_{2^{-k}\de}(V_I) 
\;\sqsubset\; U_I \qquad \text{for} \; k \geq 0,
\end{equation}
and $\de>0$ is chosen so that the following disjointness condition holds when $I\not\subset J, J\not\subset I$:
\begin{equation}\label{desep}
B_\de\bigl(\pi_\Kk(V^k_I)\bigr) \cap B_\de\bigl(\pi_\Kk(V^k_J)\bigr) \subset  
B_{\de + 2^{-k}\de}\bigl(\pi_\Kk(V_I)\bigr) \cap B_{\de + 2^{-k}\de}\bigl(\pi_\Kk(V_J)\bigr) = \emptyset .
\end{equation}
This implies that when $I\subsetneq J$,
\begin{align}\label{eq:N}\notag
V^k_I \cap \pi_\Kk^{-1}(\pi_\Kk(V^k_J))&\;=\;V^k_I  \cap \phi_{IJ}^{-1}(V^k_J)  ,\\
V^k_J \cap \pi_\Kk^{-1}(\pi_\Kk(V^k_I))&\;=\;V^k_J \cap \phi_{IJ}(V^k_I \cap U_{IJ})  \;=:\, N^k_{JI} 
\end{align}
for the sets on which we will require compatibility of the perturbations $\nu_I$ and $\nu_J$.
Similarly, we have precompact inclusions for any $k'>k\geq 0$
\begin{equation} \label{preinc}
N^{k'}_{JI} \;=\;V^{k'}_J \cap  \phi_{IJ}(V^{k'}_I\cap U_{IJ}) \;\sqsubset\; V^k_J \cap \phi_{IJ}(V^k_I\cap U_{IJ})  \;=\;N^k_{JI} .
\end{equation}
We abbreviate
$$
N^k_J \, := \;{\textstyle \bigcup_{J\supsetneq I}} N^k_{JI} \;\subset\; V^k_J ,
$$  
and call the union $N^{|J|}_J$ the {\bf core} of $V^{|J|}_J$, since it is the part of this set on which we will prescribe $\nu_J$ by  compatibility with the $\nu_I$ for $I\subsetneq J$.  (When $|J|=k+1$ we also consider various {\bf enlargements of the core}
of the form $N^{k+\la}_J\supset N^{|J|}_J$ where $\la\in (0,1)$.) 
\begin{figure}[htbp] 
   \centering
   \includegraphics[width=3in]{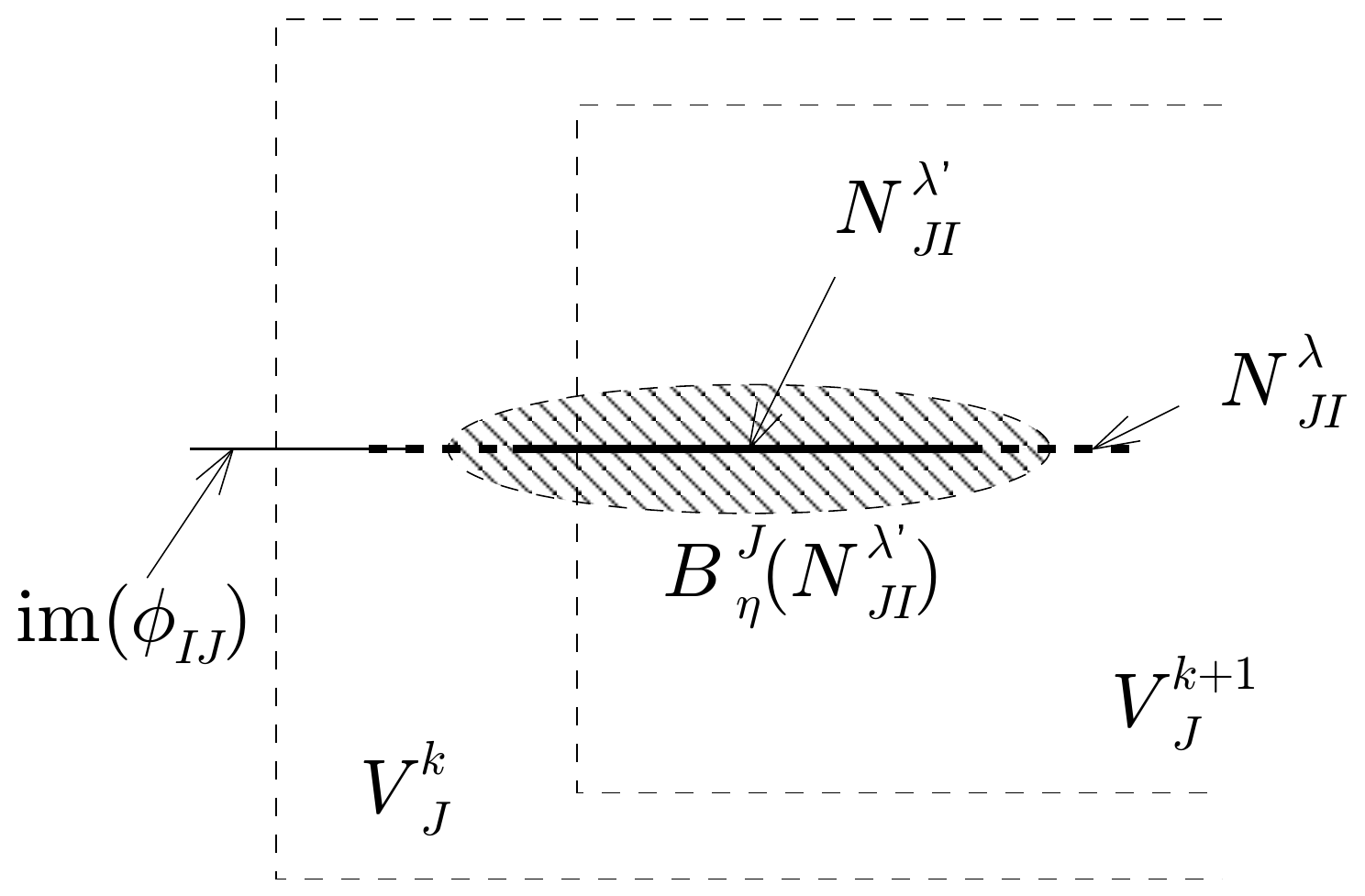} 
 \caption{
This figure illustrates the nested sets $V^{k+1}_J \sqsubset V^k_k$ and 
$ N^{\la'}_{JI}\sqsubset N^{\la}_{JI}\subset \im(\phi_{IJ})\cap V^k_J$ for $k+1>\la'=k+\frac 34 > \la =k+\frac 12>k$, the shaded neighbourhood $B^J_{\eta}(N^{\la'}_{JI})$ for $\eta = 2^{-\la}\eta_0$, and the inclusion given by \eqref{eq:useful}.
}
   \label{fig:4}
\end{figure}

In the iterative argument we work with quarter integers between $0$ and 
$M: = \max_{I\in\Ii_\Kk} |I|  ,
$
and need to introduce another constant $\eta_0>0$ that controls the intersection with $\im\phi_{IJ}=\phi_{IJ}(U_{IJ})$ for all $I\subsetneq J$ as in Figure~\ref{fig:4},
\begin{equation}\label{eq:useful} 
\im \phi_{IJ} \;\cap\;  B^J_{2^{-k-\frac 12}\eta_0} \bigl( N_{JI}^{k+\frac 34} \bigr) \;\subset\; N^{k+\frac 12}_{JI} 
\qquad \forall \; k\in  \{0,1,\ldots,M\}.
\end{equation}
Since $\phi_{IJ}$ is an isometric embedding, this inclusion holds whenever 
$2^{-k-\frac 12}\eta_0 + 2^{-k-\frac 34}\de \leq 2^{-k-\frac 12}\de$ 
for all $k$. To minimize the number of choices in the construction of perturbations, we may thus simply fix $\eta_0$ in terms of $\de$ by
\begin{equation}\label{eq:eta0}
\eta_0 \,:=\; (1 -  2^{-\frac 14} ) \de.
\end{equation} 
Then we also have $2^{-k}\eta_0 + 2^{-k-\frac 12}\de <  2^{-k}\de$,
which provides the inclusions
\begin{equation}\label{eq:fantastic}
B^I_{\eta_k}\Bigl(\;\ov{V_I^{k+\frac 12}}\;\Bigr) \;\sqsubset\; V_I^k \qquad\text{for} \;\; k\geq 0, \; \eta_k:=2^{-k}\eta_0 .
\end{equation}
We then define 
 constants $\de_\Vv>0$ and $\si(\de,\Vv,\Cc)>0$ that depend only on the indicated data as follows.

\begin{defn}  \label{def:admiss0}
Given a reduction $\Vv$ of a good topological  atlas  $(\Kk,d)$, we define $\de_\Vv>0$ to be the maximal constant 
such that any $\de< \de_\Vv$ satisfies 
\begin{align}
\label{eq:de1}
B_{2\de}(V_I)\sqsubset U_I\qquad &\forall I\in\Ii_\Kk , \\
\label{eq:dedisj}
B_{2\de}(\pi_\Kk(\ov{V_I}))\cap B_{2\de}(\pi_\Kk(\ov{V_J})) \neq \emptyset &\qquad \Longrightarrow \qquad I\subset J \;\text{or} \; J\subset I .
\end{align}
Further, given a nested reduction $\Cc\sqsubset\Vv$ of  $(\Kk,d)$ and $0<\de<\de_\Vv$, we define
$$
\eta_0: = (1 -  2^{-\frac 14} ) \de, \qquad 
\eta_{|J|-\frac 12} :=  2^{-|J|+\frac 12} \eta_0
$$
 and
\begin{equation}\label{eq:side}
\si(\de,\Vv,\Cc) \,:=\; \min_{J\in\Ii_\Kk}
\inf \Bigl\{
\; \bigl\| s_J(x) \bigr\| \;\Big| \;
x\in \ov{V^{|J|}_J} \;\less\; \Bigl( \Ti C_J \cup {\textstyle \bigcup_{I\subsetneq J}} B^J_{\eta_{|J|-\frac 12}}\bigl(N^{|J|-\frac14}_{JI}\bigr) \Bigr) \Bigr\},  
\end{equation}
where
$$
\Ti C_J \,:=\;   {\textstyle \bigcup_{K\supset J}} \, \phi_{JK}^{-1}(C_K)   \;\subset\; U_J  .
$$
\end{defn}

\begin{lemma}
 $\si(\de,\Vv,\Cc)>0$.
 \end{lemma}
 \begin{proof}[Sketch of proof.]  Because $\Cc\sqsubset \Vv\sqsubset \Vv^1$ are nested reductions, $\pi_\Kk(\Cc)$ contains the zero set $\io_\Kk(X)$, and also
 $$
 \pi_\Kk(V_J^1)\cap  \pi_\Kk(\Cc) \subset \bigcup_{I\subsetneq J\mbox{ or } J\subset I} \pi_\Kk(C_I).
 $$
Hence the definition of $\Ti C_J$ implies that
$$
 \ov{V^{|J|}_J}\cap s_J^{-1}(0)\; \subset\; \ov{V^{|J|}_J}\cap \pi_\Kk^{-1}(\pi_\Kk(\Cc))\; \subset \;
 {\textstyle \Ti C_J \;  \cup\;   \bigcup_{I\subsetneq J} \, \phi_{IJ}(C_I). }
 $$
 But when $I\subsetneq J$ we have $C_I\subset V_I\subset V_I^k$ for all $k$, so that $ \phi_{IJ}(C_I)$ is a subset of the core $N_J^k$ for all $k$.
 Hence  $\si(\de,\Vv,\Cc)$ is the minimum of 
 $\bigl\| s_J(x) \bigr\|$ as $x\in  \ov{V^{|J|}_J}$ ranges over the complement of a neighbourhood of the zero set of 
 $s_I$, and so is strictly positive. 
\end{proof}

Now, let us suppose that $\Kk$ is smooth, good (e.g. metric and tame) and, as always in this section,  with trivial isotropy.
Here is a slightly shortened version of  \cite[Definition~7.3.5]{MW2}.
 
 \begin{defn}  \label{def:a-e}
Given a nested reduction $\Cc\sqsubset\Vv$ of a smooth good Kuranishi atlas 
$(\Kk,d)$ and constants $0<\de<\de_\Vv$ and $0<\si\le\si(\de,\Vv,\Cc)$, we say that a 
 perturbation $\nu$ of $s_\Kk|_\Vv$ is {\bf $(\Vv,\Cc,\de,\si)$-adapted} if the perturbations $\nu_I:V_I\to E_I$ extend to perturbations over ${V^{|I|}_I}$ (also denoted $\nu_I$) so that the following conditions hold for every $k=1,\ldots, M$ where $\eta_k$ is as in \eqref{eq:fantastic}.
\begin{itemize}
\item[a)]
The perturbations are compatible on $\bigcup_{|I|\leq k} {V^k_I}$, that is
$$
\qquad
\nu_I \circ \phi_{HI} |_{{V^k_H}\cap \phi_{HI}^{-1}({V^k_I})} \;=\; \Hat\phi_{HI} \circ \nu_H |_{{V^k_H}\cap \phi_{HI}^{-1}({V^k_I})} 
\qquad \text{for all} \; H\subsetneq I , |I|\leq k .
$$
\item[b)]
The perturbed sections are transverse, that is $(s_I|_{{V^k_I}} + \nu_I) \pitchfork 0$ for each $|I|\leq k$.
\item[c)]
The perturbations are {\it strongly admissible} with radius $\eta_k$, that is for all $H\subsetneq I$ and $|I|\le k$ we have
$$
\qquad
\nu_I( B^I_{\eta_k}(N^{k}_{IH})\bigr) \;\subset\; \Hat\phi_{HI}(E_H) 
\qquad
\text{with}\;\;
N^k_{IH} = V^k_I \cap \phi_{HI}(V^k_H\cap U_{HI}) .
$$
\item[d)]  
The perturbed zero sets are contained in $\pi_\Kk^{-1}\bigl(\pi_\Kk(\Cc)\bigr)$; more precisely
 $s_I + \nu_I \neq 0$ on ${V^k_I} \less  \pi_\Kk^{-1}\bigl(\pi_\Kk(\Cc)\bigr)$.
\item[e)]
The perturbations are small, that is $\sup_{x\in {V^k_I}} \| \nu_I (x) \| < \si$
for $|I|\leq k$. 
\end{itemize}
\end{defn}

The above conditions are more than needed to ensure that
every $(\Vv,\Cc,\de,\si)$-adapted  perturbation $\nu$ of $\s_\Kk|_\Vv$ is an admissible, transverse perturbation with
 $$
 \pi_\Kk( (\s+\nu)^{-1}(0))\subset\pi_\Kk(\Cc).
 $$ 
Note that the role of the  strong admissibility condition above is to allow us   to deduce d) from e).
 In fact, the definition of $\si$ in \eqref{eq:side} and condition e) imply that the zero set of $s_I|_{V_I^{|I|}} + \nu_I$ 
 must either lie in $\TC_I$ and hence project to $\pi_\Kk(\Cc)$ or lie in some neighbourhood
$\bigcup_{H\subsetneq I} B^I_{\eta}\bigl(N^k_{IH}\bigr)$ of the
 enlarged core.  But the strong admissibility condition c) implies that near the core $N_{IH}$
 the components in $I\less H$ of $s_{I} + \nu_I$ equal those of $s_I$ and hence only vanish on the core itself.
   It follows that any zero of $s_I+\nu_I$ that lies near  the  core 
must in fact lie in the core 
 and hence project to 
     $\pi_\Kk(\Cc)$ by the inductive nature of the construction. 
 
We now outline the argument that such perturbations $\nu$ exist. For full details see
  \cite[Proposition~7.3.7]{MW2}.
 
  \begin{proposition}\label{prop:ext}
Let $(\Kk,d)$ be a smooth good  Kuranishi atlas with nested reduction $\Cc \sqsubset \Vv$.
Then for any $0<\de<\de_\Vv$ and $0<\si\le\si(\de,\Vv,\Cc)$ there exists a $(\Vv,\Cc,\de,\si)$-adapted perturbation $\nu$ of 
$\s_\Kk|_{\Vv}$.  
\end{proposition}
\begin{proof}[Sketch of proof]
  The construction is by an inductive process that 
constructs the required perturbations $\nu_I$ on sets larger than $V_I$.  Namely, this proposition constructs functions
$\nu_I: V^{|I|}_I\to E_I$ by an iteration over $k=0,\ldots,M= \max_{I\in\Ii_\Kk} |I|$, where in step $k$ we will define $\nu_I : V^k_I \to E_I$ for all $|I| = k$ that, together with the $\nu_I|_{V^k_I}$ for $|I|<k$ obtained by restriction from earlier steps, satisfy conditions a)-e) of Definition~\ref{def:a-e}.  Restriction to $V_I\subset V^{|I|}_I$ then yields a $(\Vv,\Cc,\de,\si)$-adapted perturbation $\nu$ of $s_\Kk|_\Vv$.   A key point in the construction is that because the different sets $V^{|I|}_I, |I|=k,$ are disjoint (by \eqref{desep}),
at the $k$th step 
the needed functions $\nu_I$ can be constructed independently of each other.

 Assume inductively that suitable $\nu_I :V^{|I|}_I\to E_I$
have been found for $|I|\le k$, and consider the construction of $\nu_J$ for some $J$ with $|J|=k+1$. We construct $\nu_J$ as a sum $\Tnu_J + \nu_{\pitchfork}$ where 
\begin{itemize}\item[-] $\Tnu_J|_{N^{k+1}_J}=\mu_J|_{N^{k+1}_J}$ where
 $\mu_J: N^{k}_J\to E_J$ is defined on 
the enlarged core $N_J^k$ by the compatibility conditions, and 
\item[-]   $\nu_{\pitchfork}$ is a final perturbation chosen so as to achieve transversality.
\end{itemize}
We construct the extension $\Tnu_J$ by extending each component $\mu_J^j, j\in J,$ of $\mu_J$ separately.  
In turn, we construct each $\mu_J^j$ by an elaborate  iterative process over the increasing family of sets $W_\ell,1\le \ell\le k,$ defined in
equation~(7.3.25) of \cite{MW2}.  Here $W_\ell$ is a carefully chosen neighbourhood of
the part $\bigcup_{|H|\le \ell} N_{JH}^{k+\frac 12}$ of the enlarged core $ N_{J}^{k+\frac 12}$ defined by sets $H$ 
with $|H|\le \ell$. In particular, when $\ell = k$ we define
\begin{equation}\label{eq:WJ}
W_k: = B^J_{\eta_{k+\frac 12}}(N^{k+\frac 12}_J) = : W_J\;\subset \ov{W_J} \subset V_I^k,
\end{equation}
where the last inclusion follows from ~\eqref{eq:fantastic}.
 Thus, omitting $J$ from the notation, we need to construct for $1\le \ell \le k$ functions $\Tmu^j_\ell:W_\ell\to E_j$
that extend  $\mu_J^j$ and satisfy certain vanishing and size conditions that will guarantee (a-e).   
Again simplifying by omitting $j$ from the notation, it turns out to suffice  
to construct the extension $\Tmu_\ell$ on a certain set $B'_\ell$, that is a union of disjoint
sets $B'_L$, one for each $L\subset J$ with $ |L|=\ell$. For each $L$, we
 localize the latter extension problem, reducing it to the construction of extensions $\Tmu_z$ near each point $z$ 
 in the set $B_L'$  of \cite[Equation~(7.3.27)]{MW2}, that we then sum up using a partition of unity.
 In most cases we can choose 
$\Tmu_z $ either to be zero or to be given by the compatibility conditions.  In fact,  the only case in which this extension is nontrivial is when $z$ is in the enlarged core, more precisely
the case $z\in \ov{N^{k+\frac 12}_{JL}}$.   In that case we  define $\Ti\mu_z$ by extending $\mu_J^j|_{B^J_{r_z}(z)\cap N^k_{JL}}$ to be \lq\lq constant in the normal directions" as described in the third bullet point of the construction.

When these extensions have all been constructed for $k\le\ell$ and $j\in J$, we define
\begin{equation}\label{eq:Tmu}
\Ti\nu_J:= \beta \cdot \bigl( {\textstyle\sum_{j\in J}} \, \Ti\mu^j_k\bigr),
\end{equation}
where $\beta:U_J \to [0,1]$ is a smooth cutoff function with $\beta|_{N^{k+\frac 12}_J}\equiv 1$ and $\supp\beta\subset 
W_J$,
 so that $\Ti\nu_J$ extends trivially to $U_J\less W_J$.
Here are some important points.
\begin{itemlist}
\item[(A)] By \eqref{eq:fantastic} the constants $\eta_k$ are chosen so that  $W_J\cap N^{k}_J $  is compactly contained in  $N^{k + \frac 14}_J $.  Further, \eqref{eq:eta0} implies that $B^J_{\eta_{k+\frac 12}}(N^{k+\frac 34}_{JI}\cap N^k_{JI}) = N^{k+\frac 12}_{JI}$.
Together with  tameness this gives
$$
s_J^{-1}(E_I)\cap  B^J_{\eta_{k+\frac 12}}(N^{k+\frac 34}_{JI})\subset N^{k+\frac 12}_{JI}
$$
 which means that at each 
point of the closure 
$cl\bigl(B^J_{\eta_{k+\frac 12}}(N^{k+\frac 12}_{JI})\bigr)\less N^k_{JI}$
at least one component of $(s_J^j)_{j\in J\less I} $ is nonzero.  
This  gives control over  zero sets as in (C), (D) below, for all sufficiently small perturbations $\Tnu_J$ that are  admissible as  in (B).  
\item[(B)] The following  strong admissibility condition holds: if $I\subsetneq J$ and $j\in J\less I$ then
 $\Tmu_J^j=0$ on $B^J_{\eta_{k+\frac 12}}(N_{JI}^{k+\frac 12})\subset W_J$  and on $N_{JI}^{k}$.
\item[(C)] The section $s_J + \Ti\nu_J$  is transverse to $0$ on 
$B: = B^J_{\eta_{k+\frac 12}}(N^{k+\frac 34}_J)\subset W_J$.
(This holds because by (A),(B) all zeros in $B$ must lie in the subset $N_J^{k+\frac 12}$ of $B$; but here the perturbation is pulled back from $\nu_I, I\subsetneq J,$ so that zeros are transverse by the inductive hypothesis.)
\item[(D)]  For any perturbation $\Tnu_J$ with support in $W_J$, $\|\Tnu_J\|< \si$, and satisfying the admissibility condition (B), we have 
  $$V^{k+1}_J\cap (s_J + \Ti\nu_J)^{-1}(0)\subset N^{k + \frac 14}_J \cup(\TC_J\less \ov {W_J}).$$
\end{itemlist}
The set $B$ in (C) above compactly contains the neighbourhood $B': = B^J_{\eta_{k+1}}(N_{J}^{k+1})$ of the core 
$N: = N_J^{|J|}$ on which compatibility requires $\nu_J|_N = \mu_J|_N = \Tnu_J|_N$.

At this stage conditions a), c), d), e) hold, so that we only need work to achieve transversality b) while keeping 
$\nu_{\pitchfork}$  so small that e) and hence d)  remain true.    The difficulty here is to keep control of the zero set since it is not a compact subset of $V^{k+1}_J$ though of course it is compact in $\ov{V^{k+1}_J}$.
We first choose a relatively open 
neighbourhood $W$ of $\ov{V_J^{k+1}}\less B$ in $\ov{V_J^{k+1}}$ so that
$$
(s_J + \Ti\nu_J)^{-1}(0)\cap W\subset O: =  \ov{V_J^{k+1}}\cap \TC_J,\quad \ov{B'} \cap W = \emptyset.
$$
This  is possible by the zero set condition  (D) and the fact $B'\sqsubset B\subset W_J$.  
By (C),  transversality holds in $B$ and hence outside the compact subset $\ov{V_J^{k+1}}\less B\subset W$.
Hence there is an open 
precompact subset $P\sqsubset W$  (the complement of a slight shrinking of $B$)
so that transversality holds on $W\less P$.    Finally we choose  
$\nu_{\pitchfork}$ to be a very small smooth function with support in $P$ and values in $E_J$ so that 
the $s_J + \Ti\nu_J + \nu_{\pitchfork}$ is  transverse on all of $W$ and hence on all $V^{k+1}_J$,
 and is such that
$(s_J + \Ti\nu_J + \nu_{\pitchfork})^{-1}(0) \subset O\subset \TC_J$.
This completes the inductive step, and hence the construction of $\nu$.
For more details see \cite{MW2}.
\end{proof}

 To show that different choices lead to cobordant zero sets we also need  a relative version of this construction. 
The relevant constant $\si_{\rm rel}(\de,\Vv,\Cc)$ depends on the given data, and in particular on the constants 
$\si(\de,\Vv^\al,\Cc^\al), \al=0,1$ that occur in Proposition~\ref{prop:ext}.  (See \cite[Definition~7.3.8]{MW2} for a precise formula.)

\begin{proposition}\label{prop:ext2}
Let $(\Kk,d)$ be a metric tame Kuranishi cobordism with nested cobordism reduction $\Cc\sqsubset\Vv$, 
let $0<\de<\min\{\eps,\de_\Vv\}$, where $\eps$ is the collar width of $(\Kk,d)$
and the reductions $\Cc,\Vv$. 
Then we have $\si_{\rm rel}(\de,\Vv,\Cc)>0$ and the following holds.
\begin{enumerate}
\item
Given any 
$0<\si\le \si_{\rm rel}(\de,\Vv,\Cc)$, there exists an admissible, precompact, transverse cobordism perturbation $\nu$ of $\s_\Kk|_\Vv$ with $\pi_\Kk\bigl((\s_\Kk|_\Vv+\nu)^{-1}(0)\bigr)\subset \pi_\Kk(\Cc)$, whose restrictions $\nu|_{\p^\al\Vv}$  for $\al=0,1$ are $(\p^\al\Vv,\p^\al\Cc,\de,\si)$-adapted 
perturbations of $\s_{\p^\al\Kk}|_{\p^\al\Vv}$.
\item
Given any perturbations $\nu^\al$ of $\s_{\p^\al\Kk}|_{\p^\al\Vv}$ for $\al=0,1$ that are $(\p^\al\Vv,\p^\al\Cc,\de,\si)$-adapted
with $\si\le \si_{\rm rel}(\de,\Vv,\Cc)$, the perturbation $\nu$ of $\s_\Kk|_\Vv$ in (i) can be constructed to have boundary values $\nu|_{\p^\al\Vv}=\nu^\al$ for $\al=0,1$.
\end{enumerate}
\end{proposition}

This is proved by making minor modifications in the construction given above.
If $\Kk$ is  a metric tame Kuranishi atlas  with nested reduction $\Cc\sqsubset \Vv$,  then
 a perturbation $\nu$ of $\s_\Kk|_\Vv$ is said to be {\bf strongly $(\Vv,\Cc)$-adapted} if 
 it is $(\Vv,\Cc,\de,\si)$-adapted for some $\de<\de_\Vv$ and $\si < \si_{\rm rel}(\de,[0,1]\times\Vv,[0,1]\times\Cc)$ where the 
 constant $\si_{\rm rel}(\de,\Vv,\Cc)$ is defined with respect to the product cobordism $[0,1]\times \Kk$ and reductions 
 $[0,1]\times \Cc\sqsubset  [0,1]\times \Vv$. 
 Note that the constant $\si_{\rm rel}(\de,[0,1]\times\Vv,[0,1]\times\Cc)$ in this definition depends only on $\Kk,\Vv,\Cc$ and the  choice of $\de$. 
 Using
 Proposition~\ref{prop:ext2} one can then show the following.\footnote
 {
 This corollary is not explicitly stated in \cite{MW2} but can be extracted from the proof given there
 that the VFC is independent of choices; see \cite[Remark~7.3.9]{MW2}.}
 
\begin{cor}\label{cor:ext2} Given any two strongly $(\Vv,\Cc)$-adapted perturbations $\nu^0, \nu^1$ of $\Kk$, there is 
an admissible, precompact, transverse cobordism perturbation $\nu$ of $\s_{[0,1]\times \Kk}|_{[0,1]\times \Vv}$ that restricts to
$\nu^\al$ on $\Kk \equiv\{\al\}\times  \Kk$ for $\al = 0,1$.
\end{cor}

\MS

\NI {\bf Sketch proof of Theorem~B.}   By Theorem~\ref{thm:K}, 
each weak Kuranishi atlas has a tame shrinking $\Kk$ that is unique up to concordance, and by
Proposition~\ref{prop:cov2} $\Kk$ has a reduction $\Vv$ that is also unique up to concordance.  
 Proposition~\ref{prop:ext} shows that there is a corresponding admissible, precompact, transverse perturbation $\nu$, such that the corresponding zero set 
is an oriented manifold $|\bZ^\nu|$ by Proposition~\ref{prop:zeroS0}.   Moreover, there is a natural inclusion $\io_Z: |\bZ^\nu|\to |\Vv|\subset |\Kk|$, and using
Corollary~\ref{cor:ext2} one can show that if $\nu$ is $(\Vv,\Cc)$-strongly adapted for some $\Cc\sqsubset \Vv$  the oriented bordism class of this inclusion is independent of choices.
By construction $|\Vv|$ lies in a $\de$-neighbourhood of the zero set $|\s|^{-1}(0) = \io_\Kk(X)$ of $|\Kk|$.  Hence, 
we may obtain a well defined element in the   rational
\v{C}ech homology group $\check{H}_d(X;\Q)$  by taking   an appropriate inverse limit as $\de\to 0$.  
This is the basic idea of the proof.  For (many) more details see
 \cite[\S8.2]{MW2}.

\section{Kuranishi atlases with nontrivial isotropy}\label{s:iso}

 The main change in this case
is that  the domains of the charts are no longer smooth manifolds, 
but rather group quotients $(U,\Ga)$ where $\Ga$ is a finite group acting on $U$.  
We will begin by assuming that $U$ is smooth, considering more general domains in \S\ref{ss:SS}.
Our definitions are chosen so that the  quotient of a Kuranishi chart 
$\bK$ with isotropy group $\Ga$ is
 a topological Kuranishi chart $\ubK$ (with trivial isotropy) that we call the {\bf intermediate} chart. 
 Similarly the quotient of a weak Kuranishi atlas by its finite isotropy groups is a weak topological Kuranishi atlas.\footnote
 {
 Note that we do not consider topological atlases with nontrivial isotropy.}  
  Hence we can apply the results of \S\ref{ss:tatlas} to tame the intermediate 
atlas, and then lift this to a taming of $\Kk$.

The other key new idea is that the coordinate changes $\bK_I\to \bK_J$ should no longer be given by inclusions $\phi_{IJ}$ of an open subset $U_{IJ}\subset U_I$ into $U_J$.
These inclusions exist on the intermediate level as $\uphi_{IJ}:\uU_{IJ}\to \uU_J$.  However, it is the inverse $$
\uphi_{IJ}^{-1}: 
\uphi_{IJ}(\uU_{IJ})\to \uU_{IJ}
$$
 that lifts to the charts themselves:
 there is a $\Ga_J$-invariant submanifold $\TU_{IJ}\subset s_J^{-1}(E_I)$ on which the kernel of the natural projection $\Ga_J\to \Ga_I$ acts freely with quotient homeomorphic to a $\Ga_I$-invariant  subset $U_{IJ}$ of $U_I$.  
This gives  a {\bf covering map} $\rho_{IJ}: \TU_{IJ}\to U_{IJ}\subset U_I$ that descends on the intermediate level to the inverse $\uphi_{IJ}^{-1}$ of $\uphi_{IJ}$.
In the Gromov--Witten setting, these maps $\rho_{IJ}$ occur very naturally as maps that forget certain sets of added marked points. (See the end of Lecture 2 in \cite{McL}, and \S\ref{s:GW} below.)  
Another important point is that the  coordinate changes preserve stabilizer  subgroups.  Hence the construction  equips
each point of $X$ with  a stabilizer  subgroup,  well defined up to isomorphism.
(In practice, one of course starts with a space $X$ whose points do have well defined 
stabilizer  subgroups, see for instance Example~\ref{ex:zorb} below.)

Most of the proofs are routine generalizations of those in \S\ref{s:noiso}; the only real difficulty is to 
make appropriate  definitions.  This section therefore consists mostly of notation and definitions.
The main reference  is \cite{MWiso}.    

\subsection{Kuranishi atlases}

\begin{defn}\label{def:gq}
A  {\bf group quotient} is a pair $(U,\Ga)$ consisting of a smooth manifold $U$ and a finite group $\Ga$ together with a smooth action $\Ga\times U \to U$.
We will denote the quotient space by
$$
\uU: = \qq{U}{\Ga},
$$
giving it the quotient topology, and write $\pi: U\to \uU$ for the associated projection.
Moreover, we denote the {\bf stabilizer} of each $x\in U$ by
$$
\Ga^x := \{\ga\in \Ga \,|\, \ga x = x \} \subset \Ga .
$$
\end{defn}

Both the basic and transition charts of Kuranishi atlases will be group quotients, related by coordinate changes that are composites of the following kinds of maps.

\begin{defn}\label{def:inject}
Let $(U,\Ga), (U',\Ga')$  be group quotients.
A {\bf group  embedding} $$
(\phi, \phi^\Ga): (U,\Ga)\to (U',
\Ga')
$$
is a smooth embedding $\phi: U\to U'$ with the following properties:
\begin{itemize}\item[-]  it is  equivariant with respect to an 
injective group homomorphism
$\phi^\Ga:\Ga\to \Ga'$;
\item[-] for each $x\in U$,
$\phi^\Ga$ induces an isomorphism between the stabilizer subgroup $\Ga^x\subset \Ga$ and the
stabilizer $\Ga^{\phi(x)}\subset \Ga'$ ;
\item[-]
 it induces an injection $\uphi:\uU\to \uU'$ on the quotient  spaces.  
\end{itemize}
\end{defn}

Note that if $\phi^\Ga:\Ga\to \Ga':=\Ga$  is surjective (and hence an isomorphism)  then
the condition on the stabilizers is automatically satisfied.   However,  in general, this condition is needed to avoid the situation where $\Ga'$ contains elements that act trivially on $U'$ and hence do not affect the quotient $\uU'$ and yet are not in the image of $\Ga$.

In a Kuranishi atlas we often consider embeddings $(\phi, \phi^\Ga): (U,\Ga)\to (U',\Ga)$ where $\dim U <\dim U'$ and $\phi^\Ga:\Ga\to \Ga':=\Ga$ is the identity map.   On the other hand,
group quotients of the same dimension are usually related
either by restriction or by coverings as follows.

\begin{defn}\label{def:grprestr}
Let $(U,\Ga)$ be a group quotient and $\uV\subset \uU$ an open subset.  Then the {\bf restriction of $(U,\Ga)$ to $\uV$} is the group quotient $(\pi^{-1}(\uV), \Ga)$.
\end{defn}

Note that the inclusion $\pi^{-1}(\uV)\to U$ induces an equidimensional group embedding $(\pi^{-1}(\uV), \Ga)\to (U,\Ga)$ that covers the inclusion $\uV\to\uU$.
The third kind of map that occurs in a coordinate change is a group covering.  This notion
is less routine; notice in particular the requirement in (ii) that $\ker \rho^\Ga$ act freely.
Further, 
the two domains $\TU, U$ will necessarily have the same dimension since they are related by a regular covering $\rho$.

\begin{defn}\label{def:cover}
Let $(U,\Ga)$  be a group quotient. A {\bf group covering} of $(U,\Ga)$ is a tuple 
$(\TU,\TGa,\rho,\rho^\Ga)$
 consisting of
\begin{enumilist}
\item[{\rm (i)}] a surjective group homomorphism $\rho^\Ga: \TGa\to \Ga$,
\item[{\rm (ii)}] a group quotient $(\TU,\TGa)$ where $\ker \rho^\Ga$ 
 acts freely,
\item[{\rm (iii)}]
a regular covering $\rho: \TU \to U$ that is the quotient map $\TU\to\qu{\TU}{\ker \rho^\Ga}$ composed with a diffeomorphism $\qu{\TU}{\ker \rho^\Ga}\cong U$ that is  equivariant with respect to the induced $\Ga = \im \rho^\Ga$ action on both spaces.
\end{enumilist}
Thus $\rho:\TU\to U$ is a principal $\ker \rho^\Ga$-bundle that is $\rho^\Ga$-equivariant.
We denote by $\urho: \ul{\TU}\to \uU$ the induced map on quotients.
\end{defn}

Here is an elementary but important lemma (\cite[Lemma~2.1.5]{MWiso}). 

 \begin{lemma}\label{le:vep}
Let $(U,\Ga)$ be a group quotient.
\begin{enumilist}
\item[{\rm (i)}]
The projection $\pi: U\to \uU$ is open and proper (i.e. the inverse image of a compact set is compact).
\item[{\rm (ii)}]
Every point $x\in U$ has a neighbourhood $U_x$ that is invariant under $\Ga^x$ and is such that inclusion $U_x\hookrightarrow U$ induces a homeomorphism from $\qu{U_x}{\Ga^x}$ to $\pi(U_x)$.
In particular, the inclusion $(U_x,\Ga^x)\to (U,\Ga)$ is a group embedding.
\item[{\rm (iii)}]
If $(\TU,\TGa,\rho,
\rho^\Ga)$ is a group covering of $(U,\Ga)$, then 
\begin{itemize}\item[-] for all $\Tx\in \TU$,  $\rho^\Ga$ induces an isomorphism $\Ga^{\Tx}\to \Ga^{\rho(\Tx)}$;
\item[-]  $\urho:   \ul{\TU}\to \uU$ is a homeomorphism.
\end{itemize}
\end{enumilist}
\end{lemma}
\begin{proof}[Sketch of proof] Parts (i) and  (ii) are well known.
To prove (iii) notice first that  because $\ker \rho^\Ga$ acts freely on $\TU$ we have
$\Ga^{\Tx}\cap \ker \rho^\Ga = \{\id\}$.  
 Hence $ \rho^\Ga$ induces an injection from $\Ga^{\Tx}$ onto a subgroup of $\Ga^{x}$, where $x: = \rho(\Tx)$, and
it  suffices to show that $\Ga^{\Tx}, \Ga^x$
have the same order.
But this holds because
 the subgroup $( \rho^\Ga)^{-1}(\Ga^x)$ has order equal to the product $|\Ga^x|\ |\ker \rho^\Ga|$
and acts transitively on the set $\rho^{-1}(x)$, which has  $ |\ker \rho^\Ga|$ elements.   
The last claim then follows from (ii). \end{proof}  

\begin{defn}\label{def:chart2}
A smooth Kuranishi chart for $X$ is a tuple $\bK = (U,  E, \Ga,s,\psi)$ made up of  
 \begin{itemlist}
 \item
the {\bf domain} $U$ which is a smooth finite dimensional manifold;
\item  a finite dimensional vector space $E$ called the {\bf obstruction space};
\item 
a finite {\bf isotropy group} $\Ga$ with a smooth action on $U$ and a linear action on $E$;
\item a smooth $\Ga$-equivariant function $s:U\to E$,
called the {\bf section};
 \item
a continuous map $\psi : s^{-1}(0) \to X$ that induces a homeomorphism
$$
\und{\psi}:\und{s^{-1}(0)}: = \qq{s^{-1}(0)}{\Ga} \;\to\; F
$$
with open image $F\subset X$, called the {\bf footprint} of the chart.
\end{itemlist}
The {\bf dimension} of $\bK$ is $\dim(\bK): =\dim U-\dim E$, and
we will say that the chart is 
\begin{itemlist}
\item {\bf minimal} if there is a point $x\in  s^{-1}(0)$ at which $\psi$ is injective, i.e.\ $x= \psi^{-1}(\psi(x))$, or equivalently $\Ga x = x$.
\end{itemlist}
 \end{defn}


In order to extend topological constructions from \S\ref{s:noiso} to the case of nontrivial isotropy, we will also consider the following notion of intermediate  Kuranishi charts which have trivial isotropy but less smooth structure.

\begin{defn}\label{def:quotlev}
We associate to each Kuranishi chart $\bK = (U, E,\Ga,s, \psi)$ the {\bf intermediate chart} $\und{\bK}: = (\und U, \und{U\times E}, \und {\s}, \und{\psi})$, where $\und{U\times E}$ is the quotient by the diagonal
action of $\Ga$ and $\und {\s}$ is the section of the bundle $\upr:\und{U\times E}\to \und{U}$ induced by $\s=\id_U\times s :U \to U\times E$.
 \end{defn}
 
 Thus, as we explained in Remark~\ref{rmk:intermed}, $\und{\bK}$  is a topological chart in the sense of Definition~\ref{def:tchart}.
%
We write $\pi: U\to \uU: = \qu{U}{\Ga}$ for the projection from the Kuranishi domain to the intermediate domain, and will distinguish everything connected to the intermediate charts by underlines. Moreover if a chart $\bK_I= 
(U_I, E_I,\Ga_I,s_I, \psi_I)$ has the label $I$, the corresponding projection is 
denoted $\pi_I: U_I\to \uU_I$.

We will find that many results (in particular the taming constructions) 
from \S\ref{s:noiso}  carry over to nontrivial isotropy via the intermediate charts, 
since precompact subsets of $\und U$ lift to 
precompact subsets of $U$ by Lemma~\ref{le:vep}~(i).
An important exception is the construction of perturbations which must be done on the smooth
 spaces $U$ rather than on their quotients $\und U$.

\begin{defn} \label{def:restr2}
Let $\bK = (U,E, \Ga, s,\psi)$ be a Kuranishi chart and $F'\subset F$
an open subset of its footprint.
A {\bf restriction of $\bK$ to $\mathbf{\emph F\,'}$} is a Kuranishi chart of the form
$$
\bK' = \bK|_{\uU'} := \bigl(U', E, \Ga, s'=s|_{U'} \,,\, \psi'=\psi|_{s'^{-1}(0)}\, \bigr), \qquad U'=\pi^{-1}(\uU')
$$
given by  a choice of open subset $\uU'\subset \uU$ such that $\uU'\cap \upsi^{-1}(F) = \upsi^{-1}(F')$.
We call $\uU'$ the {\bf intermediate domain} of the restriction and $U'$ its {\bf  domain}.
\end{defn}

Note that the restriction $\bK'$ in the above definition has footprint $\psi'(s'^{-1}(0))=F'$, and its domain group quotient
$(U',\Ga)$ is the restriction of $(U,\Ga)$ to $\uU'$ in the sense of Definition~\ref{def:grprestr}.

Moreover, because the restriction of a chart is determined by a subset of the intermediate domain $\und{U}$, all results about restrictions  are easy to generalize to the case of nontrivial isotropy.  
In particular the following result holds, where we use the notation  $\sqsubset$ to denote a  
precompact inclusion and ${\rm cl}_V(V')$ to denote the closure of a subset $V'\subset V$ in the relative topology of $V$.

\begin{lemma}\label{le:restr0}
Let $\bK$ be a Kuranishi chart. Then for any open subset $F'\subset F$ there  is a restriction $\bK'$ to $F'$ with  domain $U'$  such that ${\rm cl}_U(U')\cap s^{-1}(0) = \psi^{-1}({\rm cl}_X(F'))$.
Moreover, if $F'\sqsubset F$ is precompact, then $U'\sqsubset U$ can be chosen precompact. 
\end{lemma}

Most definitions in \S\ref{s:noiso} extend with only
minor changes to the case of nontrivial isotropy. However, the notion of coordinate change 
 needs to be generalized significantly to include a covering map. 
 We will again formulate the definition in the situation 
that is relevant to Kuranishi atlases. That is, we suppose that a finite set of 
Kuranishi charts $(\bK_i)_{i\in \{1,\dots, N\}}$ is given such that for each  $I\subset \{1,\dots, N\}$
with $F_I: = \bigcap_{i\in I} F_i \ne \emptyset$ we have another Kuranishi chart 
$\bK_I$ with\footnote
{
Generalizations of these conditions will be considered in \S\ref{s:order}.}
\begin{itemize}
\item[-] 
group $\Ga_I = \prod_{i\in I} \Ga_i$,
\item[-]  
obstruction space $E_I = \prod_{i\in I }E_i$ on which $\Ga_I$ acts with the product action,
 \item[-]  
footprint $F_I: = \bigcap_{i\in I} F_i$. 
\end{itemize}
Then for $I\subset J$ note that the natural inclusion $\Hat\phi_{IJ}: E_I\to E_J$  is equivariant with respect to the inclusion $\Ga_I \hookrightarrow \Ga_I\times \{\id\} \subset \Ga_J$, and we have a natural splitting $\Ga_J\cong\Ga_I\times \Ga_{J\less I}$, so that the complement $\Ga_{J\less I}$ of the inclusion
acts trivially on the image $\Hat\phi_{IJ}(E_I)\subset E_J$.

\begin{defn} \label{def:change2}  Given  $I\subset J\subset \{1,\dots, N\}$ let
 $\bK_I$ and $\bK_J$ be Kuranishi charts as above, so that $F_I\supset F_J$.
A {\bf coordinate change} $\Hat\Phi_{IJ}$
 from $\bK_I$ to $\bK_J$  consists of
\begin{itemize}
\item 
a choice of domain $\uU_{IJ}\subset\uU_I$ such that  $\bK_I|_{\uU_{IJ}}$ is a restriction  of $\bK_I$ to $F_{J}$,
\item the splitting $\Ga_J\cong\Ga_I\times \Ga_{J\less I}$ as above, and the induced inclusion $\Ga_I\hookrightarrow \Ga_J$
and projection $\rho_{IJ}^\Ga:\Ga_J\to \Ga_I$,
\item a $\Ga_J$-invariant submanifold 
$ \TU_{IJ}\subset U_J$ on which $\Ga_{J\less I}$ acts freely, and the induced $\Ga_J$-equivariant inclusion $\Tphi_{IJ}:\TU_{IJ}\to U_J$;
\item
a group covering $(\TU_{IJ},\Ga_J,\rho, \rho^\Ga)$ 
of $(U_{IJ}, \Ga_I)$, 
where $U_{IJ}: = \pi_I^{-1}(\uU_{IJ})\subset U_I$,
\item
the linear $\Ga_I$-equivariant injection $\Hat\phi_{IJ}:E_I\to E_J$ as above,
 \end{itemize}
such that the $\Ga_J$-equivariant inclusion $\Tphi:\TU_{IJ} \hookrightarrow U_J$ intertwines the sections and footprint maps,
\begin{equation}\label{eq:change2}
s_{J}\circ\Tphi_{IJ} = \Hat\phi_{IJ}\circ s_I\circ \rho \ \mbox{ on }  \TU_{IJ},\qquad
\psi_{J}\circ\Tphi_{IJ} = \psi_I\circ \rho_{IJ} \ \mbox{ on } \rho^{-1}(s_I^{-1}(0)).
\end{equation}
Moreover, we denote $s_{IJ}:=s_I\circ\rho_{IJ}:\Ti U_{IJ}\to E_I$ and require the {\bf index condition}:
\begin{enumerate}
\item
the embedding $\Tphi_{IJ}:\TU_{IJ}\hookrightarrow  U_J$ identifies the kernels,
$$
\rd_u\Tphi_{IJ} \bigl(\ker\rd_u s_{IJ} \bigr) =  \ker\rd_{\Tphi_{IJ}(u)} s_J    \qquad \forall u\in \TU_{IJ};
$$
\item
the linear embedding $\Hat\phi_{IJ}:E_{I}\to E_J$ identifies the cokernels,
$$
\forall u\in \TU_{IJ} : \qquad
E_{I} = \im\rd_u s_{IJ} \oplus C_{u,I}  \quad \Longrightarrow \quad E_J = \im \rd_{\Tphi_{IJ}(u)} s_J \oplus \Hat\phi(C_{u,I}).
$$
\end{enumerate}
The subset   $\uU_{IJ}\subset \uU_I$ is called the {\bf domain} of the coordinate change, while $\TU_{IJ}\subset U_J$ is its {\bf lifted domain}.
\end{defn}

\begin{rmk}\label{rmk:change} \rm
\begin{enumilist}
\item If the isotropy and covering $\rho_{IJ}=:\phi_{IJ}^{-1}$ are both trivial,
this definition  
agrees with that in \S\ref{ss:K1}
 with $\TU_{IJ}=\phi_{IJ}(U_{IJ})$.
Further, the index condition together with the condition that $\TU_{IJ}$ is a submanifold of $U_J$ implies that $\TU_{IJ}$ is 
an open subset of $s_J^{-1}(E_I)$.

\item
The following diagram of group embeddings and group coverings
is associated to each coordinate change:
\begin{equation}\label{eq:phiIJ}
\begin{array}{clcc}
&(\TU_{IJ},\Ga_J) 
& \stackrel {(\Tphi_{IJ}, \id)}{\longrightarrow } &(U_J,\Ga_J)
\vspace{.1in} \\
&\;\; \; \;\downarrow  (\rho_{IJ}, \rho_{IJ}^\Ga) &&\vspace{.1in} \\
(U_I,\Ga_I)\ \longleftarrow &(U_{IJ},\Ga_I)  &&
\end{array}
\end{equation}

\item Since $\urho_{IJ}: \und{\TU}_{IJ} \to \uU_{IJ}$ is a homeomorphism by Lemma~\ref{le:vep}~(iii), each coordinate change $(\phi_{IJ},\Hat\phi_{IJ},\rho_{IJ}): \bK_{I}|_{\uU_{IJ}}\to \bK_J$
induces an injective map 
$$
\und{\phi_{IJ}}:=\und{\Tphi_{IJ}}\circ  \urho_{IJ}^{-1} : \uU_{IJ}\to \uU_J
$$ 
on the domains of the intermediate charts.
In fact, there is an induced orbifold coordinate change $\und{\Hat\Phi_{IJ}} : \ubK_{I}|_{\uU_{IJ}}\to \ubK_J$ on the level of the intermediate charts, given by the bundle map $\und{\Hat\Phi_{IJ}} : \und{U_{IJ}\times E_I} \to \und{U_{J}\times E_J}$ which is induced by the multivalued map $(\Tphi_{IJ}\circ\rho_{IJ}^{-1})\times\Hat\phi_{IJ}$ and hence covers 
$\und{\Tphi_{IJ}} \circ  \urho_{IJ}^{-1}=:\uphi_{IJ}$.
This is a topological coordinate change as in Definition~\ref{def:tchange} that respects the orbibundle structure of the projection $\und{\pr}_I:  \und{U_{I}\times E_I} \to \uU_I$.
%

\item
Conversely, suppose 
we are given an orbifold coordinate change $\und{\Hat\Phi_{IJ}}: \ubK_{I} \to \ubK_J$ with domain ${\uU_{IJ}}$.
Then any coordinate change from $\bK_I$ to $\bK_J$ that induces $\und{\Hat\Phi_{IJ}}$ is determined by the $\Ga_J$-invariant set $\TU_{IJ}:=\pi_J^{-1}(\uphi(\uU_{IJ}))$ and a choice of $\Ga_I$-equivariant diffeomorphism between
$\qu{\TU_{IJ}}{\Ga_{J\less I}}$ and $U_{IJ} := \pi_I^{-1}(\uU_{IJ})$.
When constructing coordinate changes in the Gromov--Witten setting, we will see that there is a natural choice of this diffeomorphism since the covering maps $\rho_{IJ}$ are given by forgetting certain added marked points.

\item Because $\TU_{IJ}$ is defined to be a subset of $U_J$ it is sometimes convenient 
to think of an element $x\in 
\TU_{IJ}$ as an element in $U_J$, omitting the notation for the inclusion 
map $\Tphi_{IJ}:\TU_{IJ}\to U_J$.$\hfill\er$ 
\end{enumilist}
  \end{rmk}

The next step is to consider restrictions and composites of coordinate changes. Restrictions behave as before.
 Thus,
for $I\subset J$,
given any restrictions $\bK_I': = \bK_I|_{\uU_I'}$ and $\bK_J': = \bK_J|_{\uU_J'}$ 
whose footprints $F_I'\cap F_J'$ have nonempty intersection, and any coordinate change $\bK_I|_{\uU_{IJ}}\to \bK_J$, there is an induced {\bf restricted coordinate change} $\bK_I'|_{\uU_{IJ}'}\to \bK_J'$ for any subset $\uU_{IJ}'\subset \uU_{IJ}$ satisfying the conditions
\begin{equation}\label{eq:coordres}
 \uU_{IJ}'\subset \uU_I'\cap \uphi^{-1}(\uU_J'),
\qquad  \uU_{IJ}'\cap \us_I^{-1}(0) = \upsi_I^{-1}(F_I'\cap F_J').
\end{equation}
However, coordinate changes now do not directly compose due to the coverings involved. The induced coordinate changes on the intermediate charts still compose directly, but the analog of \cite[Lemma~2.2.5]{MW1} is the following.
The proof is routine.

\begin{lemma}\label{le:compos2} 
Let $I\subset J\subset K$ 
(so that automatically $F_I\cap F_K= F_J$)
and suppose that $\Hat\Phi_{IJ}: \bK_I\to \bK_J$ and  $\Hat\Phi_{JK}: \bK_J\to \bK_K$ are coordinate changes with domains ${\uU_{IJ}}$ and $\uU_{JK}$ respectively.
Then the following holds.
\begin{enumerate}
\item
The domain
$\uU_{IJK}:= \uU_{IJ}\cap \uphi_{IJ}^{-1}(\uU_{JK})\subset \uU_I$
defines a restriction $\bK_I|_{\uU_{IJK}}$ to $F_K$. 
\item
The composite $\rho_{IJK}: = \rho_{IJ}\circ \rho_{JK}:\TU_{IJK} \to U_{IJK}: =\pi_I^{-1}(\uU_{IJK})$ is defined on $\TU_{IJK}: = \pi_K^{-1}\bigl((\uphi_{JK}\circ\uphi_{IJ})(\uU_{IJK})\bigr)$ via the natural identification of $\rho_{JK}(\TU_{IJK})\subset U_J$ with a subset of $\TU_{IJ}$.
Together with the natural projection $\rho_{IK}^\Ga: \Ga_K\to \Ga_I$ with kernel $\Ga_{K\less I}$
(which factors $\rho_{IK}^\Ga= \rho_{IJ}^\Ga\circ\rho_{JK}^\Ga$),
this forms a group covering $(\TU_{IJK},\Ga_K,\rho_{IJK},\rho_{IK}^\Ga)$ of $(U_{IJK}, \Ga_I)$.
%
\item
The inclusion $\Tphi_{IJK}:\TU_{IJK}\hookrightarrow U_K$ together with 
the natural inclusion
$\Hat\phi_{IK}:E_I\to E_K$ (which factors $\Hat\phi_{IK}= \Hat\phi_{JK}\circ \Hat\phi_{IJ}$)
and $\rho_{IJK}$ satisfies  \eqref{eq:change2} and the index condition with respect to the indices $I,K$.
%
\end{enumerate}
Hence this defines a  {\bf composite coordinate change} $\Hat\Phi_{IJK}= (\Tphi_{IJK}, \Hat\phi_{IJK},\rho_{IJK})$ from
$\bK_I$ to  $\bK_K$ with domain $\uU_{IJK}$.
\end{lemma}

\begin{rmk}\label{rmk:coord}\rm  
The induced 
orbifold
coordinate change $\und{\Hat\Phi}_{IJK}= (\uphi_{IJK},\und{\Hat\phi}_{IJK})$
between the intermediate charts $\ubK_I$ and $\ubK_K$ is the composite
$\und{\Hat\Phi}_{JK}\circ \und{\Hat\Phi}_{IJ}$ as considered in \S\ref{ss:K1}.
(For more detail, see \cite[Lemma~2.2.5]{MW1}.)
$\hfill\er$  \end{rmk}

With these notions of Kuranishi charts and coordinate changes in place, we can now directly extend the definition of Kuranishi atlas from \S\ref{ss:K1}. The notions of a covering family $(\bK_i)_{i=1,\ldots,N} $ of basic charts for $X$ and 
of transition data $$
(\bK_J)_{J\in\Ii_\Kk,|J|\ge 2}, (\Hat\Phi_{I J})_{I,J\in\Ii_\Kk, I\subsetneq J}
$$ are as before.
The cocycle conditions can now mostly be expressed in terms of the intermediate charts.

\begin{defn}  \label{def:cocycle2}
Let $\bK_\al$ for $\al = I,J,K$ be Kuranishi charts with $I\subset J\subset K$ and let $\Hat\Phi_{\al\be}:\bK_\al|_{\uU_{\al\be}}\to \bK_\be$ for $(\al,\be) \in \{(I,J), (J,K), (I,K)\}$ be coordinate changes. We say that this triple
$\Hat\Phi_{I J}, \Hat\Phi_{J K}, \Hat\Phi_{I K}$ satisfies the
\begin{itemlist} \item {\bf weak cocycle condition} if $\Hat\Phi_{J K}\circ \Hat\Phi_{I J} \approx \Hat\Phi_{I K}$ are equal on the overlap in the sense
\begin{align}\label{eq:wcocycle}
\rho_{IK} &= \rho_{IJ}\circ \rho_{JK} \quad \mbox{ on } \; \TU_{IK}\cap \rho_{JK}^{-1}(\TU_{IJ}\cap U_{JK});
\end{align}
\item {\bf cocycle condition}
if $\Hat\Phi_{J K}\circ \Hat\Phi_{I J} \subset \Hat\Phi_{I K}$, i.e.\  $\Hat\Phi_{I K}$ extends the composed coordinate change in the sense that \eqref{eq:wcocycle}  holds and
\begin{eqnarray}\label{eq:cocycle2}
& \uU_{IJ}\cap \uphi_{IJ}^{-1}(\uU_{JK})\subset  \uU_{IK};
\end{eqnarray}
\item {\bf strong cocycle condition}
if $\Hat\Phi_{J K}\circ \Hat\Phi_{I J} = \Hat\Phi_{I K}$ are equal as coordinate changes, that is if \eqref{eq:wcocycle} 
  holds and
\begin{eqnarray}\label{strong cocycle2}
& \uU_{IJ}\cap \uphi_{IJ}^{-1}(\uU_{JK})=  \uU_{IK}.
\end{eqnarray}
\end{itemlist}
 \end{defn}

As in  \cite[Lemma~2.2.13]{MWiso}, it is straightforward to check that 
the cocycle condition \eqref{eq:cocycle2} implies that 
\begin{equation}\label{eq:cocyclea}
\rho_{IK} \; =\; \rho_{IJ}\circ \rho_{JK} \quad \mbox{ on } \;  \rho_{JK}^{-1}(\TU_{IJ}\cap U_{JK})\subset \TU_{IK}.
\end{equation}
Similarly, condition \eqref{eq:wcocycle} implies
\begin{align}\label{eq:uwcocycle}
\uphi_{IK} &= \uphi_{JK} \circ \uphi_{IJ}\quad  \mbox{ on } \; \uU_{IK}\cap \bigl(\uU_{IJ}\cap \uphi_{IJ}^{-1}(\uU_{JK})\bigr).
\end{align}

Therefore we are led to the following definition.  It is identical to the original Definition~\ref{def:Ku} of an atlas, except that now we allow more general charts and coordinate changes.

\begin{defn}\label{def:Ku2}
A {\bf (weak) smooth Kuranishi atlas of dimension $\mathbf d$} on a compact metrizable space $X$ is a tuple
$$
\Kk=\bigl(\bK_I,\Hat\Phi_{I J}\bigr)_{I, J\in\Ii_\Kk, I\subsetneq J}
$$
consisting of a covering family of basic charts $(\bK_i)_{i=1,\ldots,N}$ of dimension $d$
and transition data $\bigl((\bK_J)_{|J|\ge 2}, (\Hat\Phi_{I J})_{I\subsetneq J}\bigr)$ for $(\bK_i)_{i=1,\ldots,N}$ that satisfies the (weak) cocycle condition for every triple $I,J,K\in\Ii_K$ with $I\subsetneq J \subsetneq K$.

The corresponding {\bf intermediate atlas} has charts $(\ubK_I)_{I\in \Ii_\Kk}$ and coordinate changes
$(\und{\Hat\Phi}_{IJ})_{I\subset J, I,J\in\Ii_\Kk}$.
 \end{defn}
 
 \begin{lemma}\label{le:interm}
The intermediate atlas $\uKk$ corresponding to a (weak) smooth Kuranishi atlas $\Kk$ is a (weak) filtered topological atlas with $\E_{IJ} = \und{U_J\times \Hat\phi_{IJ}(E_I)}$ for all $I\subset J$.
 \end{lemma}
 \begin{proof} This holds by  the proof of Lemma~\ref{le:Ku3} (a).
 \end{proof}

\begin{example}\label{ex:foot} \rm  (i)  We show in Proposition~\ref{prop:orb} that every compact smooth orbifold has an atlas.
As an example, consider the ``football"  $Y = S^2$ with two orbifold points, one at the north pole of order $2$ and one at the south pole of order $3$.
Take charts $(U_1, \Z_2)$, $(U_2, \Z_3)$ about north/south pole with  
$\uU_1\cap \uU_2=\uU_{\{12\}}=\uA$ an annulus
around the equator.
Let $A_i= \pi_i^{-1}(\uA)$ where $\pi_i:U_i\to \uU_i$ is the projection.\footnote
{
For simplicity, we here identify 
each $\uU_I$ with its image $F_I\subset Y$.
} 
Then the restriction of $(U_1, \Z_2)$ over $\uA$ is $(A_1, \Z_2)$,
whereas the restriction of $(U_2, \Z_3)$ over $\uA$ is $(A_2, \Z_3)$.
There is no direct relation between these restrictions because the coverings $A_1\to \uA$ and $A_2\to \uA$ are incompatible.  However, they do have a common free covering, namely the pullback defined by the diagram
\[
\xymatrix{
U_{12}   \ar@{->}[d] \ar@{->}[r]   & U_1 \ar@{->}[d]^{\pi_1}   \\
U_2 \ar@{->}[r]^{\pi_2}  & Y
}
\]
i.e. $U_{12} := \{(x,y)\in U_1\times U_2 \,|\, \pi_1(x) = \pi_2(y)\}
$
with group 
$\Ga_{12}: = \Ga_1\times \Ga_2$.  This defines  an atlas with two basic charts and one transition chart.
For a noneffective example with trivial obstruction spaces, see \cite[Example~3.6]{Morb}.

Note that this construction cannot be imitated for general orbifolds since the above fiber product $U_{12}$ is not smooth 
if there are any points in the overlap $\pi_1(U_1)\cap \pi_2(U_2)$ with nontrivial stabilizers.  However, it turns out that there is a very simple substitute construction for the domain of the transition chart.  Namely,  if the charts $(U_i,\Ga_i)$ inject into a groupoid representative for $Y$ then we can take $U_{12}$ to be the morphisms in this groupoid from $U_1$ to $U_2$; see  the proof of Proposition~\ref{prop:orb}.\MS

\NI (ii)  For another example that illustrates the effect of a noneffective group action, consider an atlas over $X = S^1$ with two basic charts, where
\begin{itemize}\item $\bK_1$ has $U_1=S^1, E_1 = \{0\}$, the group $ \Ga_1 = \Z_2$ acting trivially,
and   footprint map $\psi_1= \id: U_1\to S^1$, while 
\item $\bK_2$ has $U_2= S^1\times \R^2$,  $E_2=\R^2$, $s_2(x,e) = e$,
and $\Ga_2 = \Z_2$ where $\ga\ne \id $ acts on $U_2$ and $E_2$ by $(x,e)\mapsto (x,-e)$ and $  e\mapsto -e$, so that $\psi_2: s_2^{-1}(0) = S^1\times \{0\}\to S^1$ 
is  induced by the projection.
\end{itemize}
The transition chart $\bK_{12} $ has $E_{12} = \Hat\phi_{2,12}(E_2) =\R^2$ and 
$\Ga_{12} = \Z_2\times \Z_2$. We must have $U_{1,12} = U_1$ and if we take $U_{2,12} = U_2$ (there is some choice here), then
in order for $\Ga_i$ to act freely on $\TU_{j,12}$ for $i\ne j$,  we must take
 $U_{12} = \TU_{2,12} = \sqcup_{\be = \pm} (S^1\times \R^2)_\be$ to be two copies of
 $S^1\times \R^2$ that are interchanged by the action of either group $\Ga_i$. 
Thus 
 the group $\Ga_{12} = \Z_2\times \Z_2$ acts as follows:
\begin{align*} 
\mbox{ on } U_{12}  \mbox{ by } &\
(\ga_1,\ga_2)(x,e)_\be =\bigl( (x,\ga_2e)\bigr)_{\ga_1\ga_2(\be)} \; \mbox{ where } \; \ga_i(\be) = (-1)^{\ga_i}\be,\\
\mbox{ on } E_{12}  \mbox{ by }  &\ (\ga_1,\ga_2)(e) = \ga_2 e, 
\end{align*}
and the coordinate changes are given by
\begin{align*} 
&\  \TU_{1,12} = \{(x,0)_\be \ | \ x\in S^1\}, \quad \rho_{1,12}: (x,0)_\be \mapsto x,\\
 &\  \TU_{2,12} = U_{12},\quad \rho_{2,12}: (x,e)_\be\mapsto  (x,e).
\end{align*}
Notice that although the action of $\Ga_1$ is totally noneffective, the group $\Ga_{12}$ acts effectively on $U_{12}$.  Moreover the stabilizer
$\Ga_{12}^{\Tx}$ of the point $\Tx= (x,0)_\be \in \TU_{1,12}$ is the diagonal subgroup $(\ga,\ga)\in \Ga_{12}$.
Thus although $\rho^\Ga_{1,12}: \Ga_{12}\to \Ga_1$ does induce an isomorphism $\Ga^{\Tx}_{12} \to \Ga^x_1$ for all $\Tx\in \rho_{1,12}^{-1}(x)$ (cf.  Lemma~\ref{le:vep}~(iii)),
 it does not preserve the property of acting noneffectively.    See Figure~\ref{fig:7}  in Example~\ref{ex:foot2} for an illustration of the fundamental cycle of this atlas. \hfill$\er$
 \end{example}

\begin{rmk}\label{rmk:tamea}\rm   
 The definition of atlas requires that  each transition chart $\bK_I$ has  obstruction spaces and groups that
are products over $i\in I$. We will generalize the requirements on $E_I$ in \S\ref{s:order}.
Assuming for now that the obstruction spaces are products, we can generalize the requirements on the groups $\Ga_I$ as follows.
We assume given charts $(\bK_I)_{I,J\in \Ii_\Kk}$ whose domains are group quotients $(U_I,\Ga_I)$ 
with coordinate changes given by group coverings $(\rho_{IJ}, \rho^\Ga_{IJ}): \TU_{IJ}\to U_{IJ}$ 
and linear embeddings $\Hat\phi_{IJ}:E_I\to E_J$ as before, where 
$\rho_{IJ}^\Ga: \Ga_J\to \Ga_I$ is a family of split surjective homomorphisms
with splittings $\io^\Ga_{IJ}:\Ga_I\to \Ga_J$
such that
\begin{itemize}
\item  $\rho_{IJ}\circ \rho_{JK} = \rho_{IK}$, 
 $\rho^\Ga_{IJ}\circ \rho^\Ga_{JK} = \rho^\Ga_{IK}$  and $\io^\Ga_{JK}\circ \io^\Ga_{IJ}=\io^\Ga_{IK}$ whenever $I\subset J\subset K$;
 \item for all $I\subset J$ 
\item the  linear inclusion $\Hat\phi_{IJ}:E_I\to E_J$  is $\io^\Ga_{IJ}$-equivariant and 
 $\ker(\rho^\Ga_{IJ})$ acts trivially on its image;
\end{itemize}
We will call such  atlases {\bf tameable}.  They are  very natural when one considers products; cf. \S\ref{ss:hybrid}.
It is not hard to check that all the constructions in this section,  go through as before with this more general kind of atlas.
In particular, in the construction of $\Hat\bZ^\nu$ in Proposition~\ref{prop:zero} the group $\Ga_{I\less F}$  that appears in 
\eqref{eq:morZ0} should be replaced by $\ker \rho^\Ga_{FI}$.  
$\hfill\er$  \end{rmk}

\begin{rmk}\label{rmk:FOOO1}  \rm 
We now show how to obtain a Kuranishi structure in the sense of \cite{FOOO12} from a Kuranishi atlas in the case 
when the isotropy groups act effectively on the domains $U$, the only case they consider.
Recall that to define a Kuranishi structure one needs to specify a family of Kuranishi  charts $(\bK_p)_{p\in X}$  with 
footprints $F_p\ni p$, together with coordinate changes\footnote
{
For consistency with our conventions we write $\uphi_{qp}$ rather than $\uphi_{pq}$ here. 
}
 $(\uphi_{qp}: \ubK_q\to \ubK_p)_{q\in F_p}$ 
that are defined on the level of the intermediate atlas and satisfy the weak cocycle condition stated below.  Even though there are uncountably many charts $\bK_p$, Fukaya et al. construct them from a finite covering family (called a {\it primary family} in \cite{TF}) in much the same way that we  now describe.

\begin{itemlist} \item First choose a precompact \lq\lq shrinking" $\{ G_i \sqsubset F_i\}_{i=1,\dots,N}$ of  the footprints,
and for $p\in X$, define $I_p: =  \{i \,|\, p\in \ov G_i\}$.
\item For  $p\in X$ define $\bK_p$ by choosing a restriction of the  chart $\bK_{I_p}$ to 
$$
F_p:={\textstyle \bigl(\bigcap_{i\in I_p} F_i\bigr) \less \bigl(\bigcup_{j\notin I_p} \ov G_j\bigr).}
$$ 
More precisely, since $p\in F_p$, we choose a lift $\oo_p$ of $p$ to $U_{I_p}$, take $\Ga_p$  to be the stabilizer subgroup of $\oo_p$ (so that the chart is minimal as required by \cite{FOOO})
and take the domain $U_p\subset U_{I_p}$ to be a $\Ga_p$-invariant neighbourhood of $\oo_p$ such that the induced map $\qu{U_p}{\Ga_p}\to \qu{U_{I_p}}{\Ga_{I_p}}$ is injective.  (Suitable $(U_p,\Ga_p)$ exist by Lemma~\ref{le:vep}~(ii).)

\item 
For $q\in F_p$ we 
have $I_q\subset I_p$ since by construction $F_p\cap G_i=\emptyset$ for $i\not\in I_p$. 
So we obtain a coordinate change $\Hat\Phi_{qp}: \bK_q\to \bK_p$ from a suitable restriction of $\Hat\Phi_{I_q I_p}$ to a $\Ga^{x_q}_q$-invariant neighbourhood $U_{qp}\subset U_q$ of $x_q$.  
 More precisely, we choose $U_{qp}\subset U_q$ small enough so that the projection $\rho_{I_q I_p}: U_p\cap \TU_{I_q I_p}\to U_{I_q I_p}$ has a continuous section over $U_{qp}$.
We denote its image by $\TU_{qp}$ and thus obtain an embedding $\phi_{qp}:= \rho_{I_q I_p}^{-1} : U_{qp}\to \TU_{qp}\subset U_p \cap \TU_{I_q I_p}$. 
Since the projection $\rho_{I_q I_p}$ induces an isomorphism on stabilizer subgroups by 
 Lemma~\ref{le:vep}~(iii), 
  this is equivariant with respect to a suitable injective homomorphism $h_{qp}:\Ga_q\to \Ga_p$ and induces an injection $\uphi_{qp}: \uU_{qp}: = \qu{U_{qp}}{\Ga_q} \to \uU_{p}: = \qu{U_{p}}{\Ga_p}$.
Moreover, because the induced map $\uU_{p}= \qu{U_{p}}{\Ga_p}\to \uU_{I_p} = \qu{U_{I_p}}{\Ga_{I_p}}$ is a homeomorphism to its image by construction
and similarly for $q$,
we can identify the underlying map  $\uphi_{qp}$ of the coordinate change with a suitable restriction of $\uphi_{I_qI_p}$.
The coordinate change $\Hat\Phi_{qp} = (U_{qp},\uphi_{qp})$ is then given by the domain $U_{qp}$ and the restriction of 
 $\uphi_{I_qI_p}$ to $\uU_{qp}\subset \uU_q$.

The weak cocycle condition for $\Kk$ implies the  required compatibility condition, namely
 for all triples 
$p,q,r\in X$ with  $ q\in F_p$ and $r\in  \psi_q(U_{qp}\cap s_q^{-1}(0))\subset F_q\cap F_p$, the equality
$\uphi_{qp}\circ \uphi_{rq} = \uphi_{rp}$  holds  on the common domain
$\uphi_{rq}^{-1}(\uU_{qp})\cap \uU_{rp}$ of the maps in this equation.

\end{itemlist}

This process of passing to small charts loses information about isotropy, and
seems a little inefficient, in that one  needs to rebuild larger charts in order to get a \lq\lq good coordinate system".
  In \cite{FOOO12} this is done only on the level of the intermediate atlas, so that they need to assume
  the group actions are effective.\footnote
  {
  Although this is almost always true in the Gromov Witten setting and can always be arranged by adding marked points,
  there are a few situations where it does not hold; cf. Example~\ref{ex:zorb}.}
   Further, although it is not clear how relevant the extra information  contained in a Kuranishi atlas is to 
 the question of how to define Gromov--Witten invariants
 for closed curves in smooth manifolds,
 it might prove useful in other situations, for example  in the case of orbifold Gromov--Witten invariants, or in the 
 recent work of Fukaya et al \cite{FOOO15} where the authors consider a process
  that rebuilds a Kuranishi structure  from a coordinate system. 
One way of formulating the approach in \cite{FOOO12} is that they do 
    taming and reduction 
simultaneously.  It might be possible to do that in our context.  However,  it is actually very useful to have 
the intermediate object of a Kuranishi atlas since this captures 
all the needed information about the coordinate changes in the simplest possible form.  
$\hfill\er$  \end{rmk}

\begin{rmk}\rm
Although it seems that many choices are needed in order to construct a
Kuranishi atlas, this is somewhat deceptive.   For example, in the Gromov--Witten case  considered 
in Section~\ref{s:GW} below, 
the choices involved in the construction of a family of basic charts $(\bK_i)_{i=1,\dots, N}$ essentially induce the transition data as well.
Namely, for each $I\subset \{1,\dots,N\}$ such that $F_I: = \bigcap_{i\in I} F_i$ is nonempty, we will construct a ``transition chart" $\bK_I$ with group $\Ga_I: = \prod_{i\in I} \Ga_i$ and obstruction space\footnote
{
We will use the  stabilization process developed in \cite{MWgw} (see also \S4.1~(VI) below) that allows us to do this for {\it any} set of $E_i$;  there is no need for a transversality requirement such as Sum Condition II$'$ in \cite[Section~4.3]{MW2}.
}	
$E_I: = \prod_{i\in I}  E_i$. 
Moreover, each $E_i$ is a product of the form $E_i = \prod_{\ga\in \Ga_i}  (E_{0i})_\ga$ of copies of a vector space $E_{0i}$ that are permuted by the action of $\Ga_i$, and $\Ga_I$ acts on $E_I$ by the obvious product action.  

More precisely, each basic chart $\bK_i$  is constructed by adding a certain tuple $\bw_i$ of marked points\footnote
{
One difference between our construction and that in \cite{FOOO12,TF} is that here the added marked points $\bw$ are labelled, while in 
 \cite{FOOO12,TF} they are unordered.} 
 to the domains of the stable maps $[\Si_f,\bz,f]$, given by the preimages of a fixed hypersurface of $M$.
When seen on spaces of  equivalence classes of maps,  the action of $\Ga_i$ is easy to understand,\footnote
{
This point is explained in detail in \S\ref{ss:GW}~(IX), where we describe the action both on parametrized maps as in \eqref{eq:perm},\eqref{eq:permn} and on equivalence classes of maps as in the discussion after \eqref{eq:coordf3}. }
since it simply permutes this  set of marked points $\bw_i$.
Similarly, elements of the domains $U_J$ of the transition charts consist of 
certain maps $f:\Si\to M$ with the given marked points $\bz$ together with $|J|$ sets of added tuples of marked points $(\bw_j)_{j\in J}$,
each  taken by $f$ to certain hypersurfaces in $M$.
Each factor $\Ga_j$ of the group $\Ga_J$ acts by permuting the elements of the $j$-th tuple of points,  leaving the others alone.
Moreover, the covering map $\TU_{IJ}\to U_I$ simply forgets the extra tuples $(\bw_j)_{j\in J\less I}.$
Thus it is immediate from the construction that the group $\Ga_{J\less I}$ acts freely on the subset $\TU_{IJ}$ of $U_J$, and that the covering map is equivariant in the appropriate sense.
Further, when $I\subset J\subset K$ the compatibility condition
$\rho_{IK} = \rho_{IJ}\circ \rho_{JK}$ holds whenever both sides are defined.  
Therefore, just as in the case with trivial isotropy, once given the basic and transition charts, the only new choice needed to construct an atlas is that of the domains $\uU_{IJ}$  of the coordinate changes which are required to intersect the zero set $\us_I^{-1}(0)$ in 
 $\upsi_I^{-1}(F_J)$.
Since there is no obvious way  to
organize these choices 
so as to satisfy the cocycle condition,  concrete constructions will usually only satisfy a weak cocycle condition.
$\hfill\er$  \end{rmk}

\subsection{Categories and  tamings}\label{ss:iso2}

Just as in \S\ref{ss:K2}, 
we will associate to each
Kuranishi atlas $\Kk$ 
two topological categories $\bB_\Kk,\bE_\Kk$ together with functors
$$
\pr_\Kk:\bE_\Kk\to \bB_\Kk,\quad s_\Kk: \bB_\Kk\to \bE_\Kk,\quad
\psi_\Kk: s_\Kk^{-1}(0) \to \bX,
$$
where $\bX$ is the category with objects $X$ and only identity morphisms. 
Recall here that the morphism spaces will only be closed under composition (and thus generate an equivalence relation that defines the realization $|\Kk|$ as ambient space for $X$) if the cocycle condition holds. Thus for the following we assume that $\Kk$ is a Kuranishi atlas.  
Then, as before, the domain category $\bB_\Kk$ has objects
$$
\Obj_{\bB_\Kk}:= \bigsqcup_{I\in \Ii_\Kk} U_I \ = \ \bigl\{ (I,x) \,\big|\, I\in\Ii_\Kk, x\in U_I \bigr\},
$$
where we  identify $x\in U_I$ with $(I,x)\in \Obj_{\bB_\Kk}$.
The morphisms in $\bB_\Kk$ are composites of morphisms of the following two types.

\begin{itemize}
\item[(a)] 
For each $I\in\Ii_\Kk$ the action of $\Ga_I$ gives rise to morphisms between points in $U_I$. These form a space $U_I\times \Ga_I$ with source and target maps
\begin{align*}  
 s\times t: \quad U_I\times \Ga_I  \; & \longrightarrow \quad\;\; U_I\times U_I\qquad \subset\; \Obj_{\bB_\Kk}\times \Obj_{\bB_\Kk}\\ 
 ( x, \ga ) \;\;\, & \longmapsto \;\bigl((I,\ga^{-1}x), (I, x)\bigr),
\end{align*}
and inverses $(x,\ga)^{-1}=(\ga^{-1} x , \ga^{-1})$.
\item[(b)] \label{b morph}
For each $I\subsetneq J$ the coordinate change $\Hat\Phi_{IJ}$ gives rise to non-invertible morphisms from points in $U_I$ to points in $U_J$ given by the space $\TU_{IJ}$ with source and target maps
\begin{align*}  
 s\times t: \quad \TU_{IJ}  \; & \longrightarrow \qquad\quad\;\; U_I\times U_J\qquad\quad \subset\; 
 \Obj_{\bB_\Kk}\times \Obj_{\bB_\Kk}\\ 
 \ y \quad & \longmapsto \;\bigl((I,\rho_{IJ}(y)), (J,\Tphi_{IJ}(y))\bigr) .
\end{align*}
\end{itemize}

In order to determine the general morphisms in $\bB_\Kk$ we will unify types (a) and (b) by allowing $I=J$, in which case we interpret $U_{II}: = U_I,\ \TU_{IJ}: = U_J,\ \Tphi_{II} =  \rho_{II}: = \id$. 
Thus the morphisms of type (b) are described by their targets $\Tphi_{IJ}(y)\in U_J$ and the covering map $\rho_{IJ}$.
In comparison, recall that in \S\ref{ss:K1}, we have no morphisms of type (a) 
and the morphisms of type (b) are described by their source $x\in U_{IJ}\subset U_I$ and the embedding $\phi_{IJ}:U_{IJ}\to U_J$.
When the isotropy groups are all trivial, it makes no difference whether we use 
source or target since $\phi_{IJ}=\rho_{IJ}^{-1}$ and $\Tphi_{IJ} = \id_{U_J}$. The corresponding isomorphism of categories is given in Lemma~\ref{le:und} below.
For nontrivial isotropy, however, the only way to obtain a continuous description 
of the morphism spaces is to parametrize them essentially by  their targets as follows.

\begin{lemma}\label{le:morph}   
Let $\Kk$ be a Kuranishi atlas. 
Then the space of morphisms in $\bB_\Kk$  is the disjoint union
$$
\Mor_{\bB_\Kk} = \bigsqcup_{I\subset J} \TU_{IJ}\times \Ga_I
= \bigl\{ (I,J,y,\gamma) \,\big|\, I\subset J, y\in \TU_{IJ}, \gamma\in \Ga_I \bigr\},
$$
with source and target maps given by 
\begin{align}\label{eq:morph}  
 s\times t: \quad  \TU_{IJ}\times \Ga_I\; & \; \longrightarrow \qquad\quad U_I\times U_J\qquad\qquad  \subset \; \Obj_{\bB_\Kk}\times \Obj_{\bB_\Kk},\\ \notag
 \bigl(I,J,y,\ga\bigr) & \; \longmapsto\;  \bigl((I,\ga^{-1} \rho_{IJ}(y), \ (J,\Tphi_{IJ}(y))\bigr),
\end{align}
and composition given by 
\begin{equation}\label{eq:compos1}
   \Tphi_{IJ}(x) = \de^{-1}\rho_{JK}(y)\;\Longrightarrow \; \bigl(I,J,x,\ga\bigr) \circ \bigl(J,K,y,\de\bigr)=  \bigl(I,K, y, \rho^\Ga_{IJ}(\de) \ga\bigr).
\end{equation}
\end{lemma}

The {\bf intermediate Kuranishi category} $\ubB_\Kk$ is defined by the topological atlas $\uKk$ as in \S\ref{ss:tatlas}.
 It has 
$$
\Obj_{\ubB_\Kk} : = \bigsqcup_{I\in \Ii_\Kk} \uU_I,\quad 
\Mor_{\ubB_\Kk} : = \bigsqcup_{I,J\in \Ii_\Kk, I\subset J} \uU_{IJ}
$$
with  source and target maps 
$$
s\times t: \;\uU_{IJ} \to \uU_I\times \uU_J\subset \Obj_{\ubB_\Kk}\times \Obj_{\ubB_\Kk},\quad
(I,\ux)\mapsto \bigl( (I,\ux), (J,\uphi_{IJ}(\ux))\bigr).
$$
The identity maps $\uphi_{II}$ on $\uU_{II}=\uU_I$ are included, 
giving rise to the identity morphisms.

As before, we denote by $|\Kk|$ the {\bf realization} of the category $\bB_\Kk$, i.e.
 the topological space obtained as the quotient of  $\Obj_{\bB_\Kk}$ by the equivalence relation generated by the morphisms in $\bB_\Kk$.
The quotient map $\pi_\Kk: \Obj_{\bB_\Kk} \to |\Kk|, \; (I,x)\mapsto [I,x]$ now factors
through the intermediate category,
$$
\pi_\Kk: \Obj_{\bB_\Kk} \to \Obj_{\ubB_\Kk} \to  |\Kk|.
$$
In particular 
the two categories $\bB_\Kk$ and $\ubB_\Kk$ have the same realization. 
In the latter case, we denote the natural projection by  $\und\pi_{{\Kk}}:  \Obj_{\ubB_\Kk}\to |\Kk|$.
More precisely, we can formulate this as follows.

\begin{lemma}\label{le:und}   
Let $\Kk$ be a Kuranishi atlas. 
Then there is a 
functor $\pi_\Ga: \bB_\Kk\to \ubB_\Kk$ that is given on objects by the
quotient maps 
$\pi_I: U_I \to \uU_I$, 
and on morphisms by
$$
 \TU_{IJ}\times \Ga_I \; \to \;  \uU_{IJ}, \qquad
 \bigl(I,J, y, \ga\bigr) \; \mapsto \; \bigl(I,J,\ul{\rho_{IJ}(y)}\bigr) .
$$
If all isotropy groups $\Ga_I=\{\id\}$ are trivial, then $\pi_\Ga$ is an isomorphism of categories with identical object spaces, and $\ubB_\Kk$ is the category associated to the Kuranishi atlas in \S\ref{ss:K1}.

In general, the realization $|\Kk|$ of $\bB_\Kk$ can be identified as topological space (with the quotient topology) with that of $\ubB_\Kk$ via factoring the quotient map $\pi_\Kk=\ul\pi_\Kk \circ \pi_\Ga$ into the functor $\pi_\Ga$ given on objects  by quotienting by the group actions and the  projection $\upi_\Kk: \ubB_\Kk\to |\Kk|$, that can be considered as a functor to a topological category with only identity morphisms.
Moreover, $\pi_\Ga$ is proper, i.e.\ compact subsets of $\Obj_{\ubB_\Kk}$ have compact preimage in $\Obj_{\bB_\Kk}$.
\end{lemma}

There are similar obstruction space categories $\bE_\Kk$ and $\ubE_\Kk$.
For example,
$$
\Obj_{\bE_\Kk} = \bigsqcup_I U_I\times E_I,\qquad
\Mor_{\bE_\Kk} = \bigsqcup_{I\subset J}  \TU_{IJ}\times \Hat\phi_{IJ}(E_I)\times \Ga_I.
$$
The projections $\pr_I$ and sections $\s_I$ fit together to functors
$$
\pr_\Kk: \bE_\Kk\to \bB_\Kk,\; \s_\Kk:\bB_\Kk\to \bE_\Kk, \quad  \und{\pr}_\Kk: \ubE_\Kk\to \ubB_\Kk,
\;\und{\s}_\Kk:\ubB_\Kk\to \ubE_\Kk.
$$

\begin{prop} \label{prop:realization} 
Let $\Kk$ be a Kuranishi atlas. 
\begin{enumilist}
\item[{\rm (i)}]
The functors ${\rm pr}_\Kk:\bE_\Kk\to\bB_\Kk$  and $\und{\pr}_\Kk: \ubE_\Kk\to \ubB_\Kk$ induce the same continuous map
$$
|{\rm pr}_\Kk|:|\bE_\Kk| \to |\Kk|,
$$
which we call the {\bf obstruction bundle} of $\Kk$, although its fibers generally do not have the structure of a vector space.  However, it has a continuous zero section
$$
|0_\Kk| : \; |\Kk| \to |\bE_\Kk| , \quad [I,x] \mapsto [I,x,0] .
$$
\item[{\rm (ii)}]
The sections $\s_\Kk:\bB_\Kk\to \bE_\Kk$ and $\und{\s}_\Kk:\ubB_\Kk\to \ubE_\Kk$ descend to the same continuous section
$$
|\s_\Kk|:|\Kk|\to |\bE_\Kk| .
$$
Both of these are sections in the sense that
$|\pr_\Kk|\circ|\s_\Kk| =  |\pr_\Kk|\circ |0_\Kk|= {\rm id}_{|\Kk|}$.
\item[\rm (iii)]
There is a natural homeomorphism from the realization of the subcategory $\s_\Kk^{-1}(0)$
to the zero set of $|\s_\Kk|$, with the relative topology induced from $|\Kk|$,
$$
\bigr| \s_\Kk^{-1}(0)\bigr| \;=\; \quotient{\s_\Kk^{-1}(0)}{\sim_{\scriptscriptstyle \s_\Kk^{-1}(0)}}
\;\overset{\cong}{\longrightarrow}\;
|s_\Kk|^{-1}(0) \,:=\; \bigl\{[I,x] \,\big|\, s_I(x)=0  \bigr\}  \;\subset\; |\Kk| .
$$
\end{enumilist}
\end{prop}

The proof is not difficult.\MS

The next task is to establish the following  analog of Theorem~\ref{thm:K}.  Note that we say   that {\bf  $\Kk$ is  tame or metrizable} if its intermediate atlas $\uKk$ has this property.
In particular,  we require that the  metric $d$ on  $|\Kk|$ pulls back to a 
 metric $\und d_I:=(\upi_\Kk|_{\uU_I})^*d$ on $\uU_I=\qu{U_I}{\Ga_I}$ that induces the quotient topology. (When $\Ga_I\ne \id$ the pullback to $U_I$ cannot be a metric.)
Similarly, a {\bf shrinking of $\Kk$} is the pullback by the functor $\pi_\Ga$ in Lemma~\ref{le:und} of a shrinking of $\uKk$.  

\begin{thm}\label{thm:K2}
Let $\Kk$ be a weak Kuranishi atlas on a compact metrizable space $X$.
Then there is a metrizable tame shrinking $\Kk'$ of $\Kk$ with domains $(U'_I\subset U_I)_{I\in\Ii_{\Kk'}}$ such that the realizations $|\Kk'|$ and $|\bE_{\Kk'}|$ are Hausdorff in the quotient topology.
Further, for each $I\in \Ii_{\Kk'} = \Ii_\Kk$ the projection maps $\und \pi_{\Kk'}: \und U_I'\to |\Kk'|$ and $\und \pi_{\Kk'}:\und {U'_I\times E_I}\to |\bE_{\Kk'}|$ are homeomorphisms onto their images. In addition, these projections $\pi_{\Kk'} $ fit into a commutative diagram
$$
\begin{array}{ccc}
U_I'\times E_I & \stackrel{\pi_{\Kk'}}\longrightarrow & |\bE_{\Kk'}|  \quad \\
 \downarrow & & \;\; \downarrow \scriptstyle |\pr_{\Kk'}| \\
U_I' &
\stackrel{\pi_{\Kk'}} \longrightarrow  &|\Kk'| \quad
\end{array}
$$
where the horizontal maps intertwine 
the linear structure on the fibers of $U'_I\times E_I \to U'_I$ with the 
induced orbibundle structure on the fibers of $|\pr_{\Kk'}|$.
More precisely, for each $p\in |\Kk'|$ there is $I_p\in \Ii_{\Kk'}$ such that $p\in \pi_{\Kk'}(U_{I_p})$ and so that
the composite map
$$
\pi_{\Kk'}: \bigl(U_{I_p}'\cap  \pi_{\Kk'}^{-1}(p)\bigr)\times E_{I_p}  
 \;\to\;  |\pr_{\Kk'}|^{-1}(p)\subset |\bE_{\Kk'}|
$$
induces an isomorphism from the quotient $\qu{(U_{I_p}'\cap  \pi_{\Kk'}^{-1}(p))\times E_{I_p}}{\Ga_{I_p}}$ onto the fiber $|\pr_{\Kk'}|^{-1}(p)$.

Moreover, any two such shrinkings are cobordant by a metrizable tame Kuranishi cobordism that also has the above Hausdorff, homeomorphism, and linearity properties.
\end{thm}

\begin{proof}  This is an immediate consequence of Theorem~\ref{thm:K} applied to the intermediate atlas $\uKk$, which is a filtered topological atlas by Lemma~\ref{le:interm}.
The only nontrivial statement concerns the linear structure on the fibers of $|\pr_{\Kk'}|$. But this follows by arguing as in the proof of Proposition~\ref{prop:tame1}.
\end{proof}

\subsection{Orientations}\label{ss:orient}
This section develops  the notion of orientation. 
(See \cite[\S8.1]{MW2} for the  trivial isotropy case, and \cite[\S3.1]{MWiso} for  the general case.)
As in Definition~\ref{def:or}, we define an orientation of $\Kk$ 
to be a nonzero section of its determinant line bundle.   There are two possible definitions of this bundle, first as a bundle $\det(\Kk)$ over $\Kk$ that restricts on  each domain $U_I$ to the local orientation bundle $\lm TU_I\otimes (\lm E_I)^*$ of Definition~\ref{def:or},  and second as
a bundle $\det(\s_\Kk)$ whose restriction to $U_I$ is the standard determinant bundle 
of the (Fredholm) section $s_I$ (see  \cite[\S~A.2]{JHOL}). Although the local description of $\det(\s_\Kk)$ is more complicated because the dimension of $\ker s_I$ varies, its coordinate changes are more natural, being induced directly from
the chart coordinate changes $\Hat\Phi_{IJ}$ with no auxiliary choices, while those for
$\det(\Kk)$ do involve the sections $s_I$ as well as the choice of a normal bundle to $\TU_{IJ}$ in $U_J$.
  Thus we first prove that $\det(\s_\Kk)$ is a bundle,
and then show it is isomorphic to the bundle $\det(\Kk)$.  A similar argument shows that the restriction of 
$\det(\Kk)$ to $\Kk|_\Vv$ is isomorphic to $\det(\s_\Kk|_\Vv + \nu)$.  But if $(s_\Kk|_\Vv + \nu)\pitchfork 0$, an orientation
of $\det(\s_\Kk|_\Vv + \nu)$ induces one on the zero set $\bZ^\nu$; see Lemma~\ref{le:locorient1}.
We do not attempt to give full proofs here, but at least define all the maps involved.

We begin by defining the notion of a bundle over a smooth atlas.
Notice that in distinction to the \lq\lq obstruction bundle" with fibers $E_I$ over 
the domains $U_I$, we now assume that the fiber dimension is constant.

\begin{defn} \label{def:bundle}
A {\bf vector bundle} $\La=\bigl(\La_I,\Ti\phi_{IJ}\bigr)_{I,J\in\Ii_\Kk}$ {\bf over a weak Kuranishi atlas} $\Kk$ is a collection $(\La_I \to U_I)_{I\in \Ii_\Kk}$ of  vector bundles such that 
\begin{itemlist}\item  for each $I\in \Ii_\Kk$ the action of  $\Ga_I$ on $U_I$ lifts to $\La_I$;
\item  there is an action of $\Ga_J$ on the pullback bundle 
$\rho^*_{IJ}(\La_I|_{U_{IJ}})$ 
such that \begin{itemize}\item[-]  the projection $\rho^*_{IJ}(\La_I|_{U_{IJ}})\to\La_I|_{U_{IJ}}$  factors as the 
quotient  by the  action of $\Ga_{J\less I}$ followed by an isomorphism, and  \item[-]
 the inclusion $\Tphi_{IJ}: \TU_{IJ}\to U_J$ lifts to a  $\Ga_J$-equivariant embedding  \newline
$\Tphi_{IJ}^\La\;:\; \rho^*_{IJ}(\La_I|_{U_{IJ}})\to \La_J$.
\end{itemize}
\item  the weak cocycle condition holds 
$$
\Tilde \phi^\La_{IK} =  \Tilde \phi^\La_{JK}\circ \rho_{JK}^*( \Tilde \phi^\La_{IJ})\;\mbox{ on } \;\rho^{-1}_{JK}(\Tphi_{IJ}(\TU_{IJ}))\cap \TU_{IK}, \quad \forall \;I\subset J\subset K.
$$   
\end{itemlist}
A {\bf section} of a bundle $\La$ over $\Kk$ is a collection of smooth $\Ga_I$-equivariant sections $\si=\bigl( \si_I: U_I\to \La_I \bigr)_{I\in\Ii_\Kk}$ that are compatible with the pullbacks $\rho_{IJ}^*$ and bundle maps $\Ti\phi^\La_{IJ}$.
Thus there are commutative diagrams for each $I\subset J$,
\[
\xymatrix{
\La_I|_{U_{IJ}}    &   \ar@{->}[l]_{\rho_{IJ}}
  \rho_{IJ}^*(\La_I|_{U_{IJ}})   \ar@{->}[r]^{\;\;\;\;\;\;\Tilde\phi^\La_{IJ}}    &  \La_J  \\
U_{IJ} \ar@{->}[u]^{\si_I}       &  \ar@{->}[l]_{\rho_{IJ}}
  \TU_{IJ}  \ar@{->}[u]^{\rho_{IJ}^*(\si_I)}   \ar@{->}[r]^{\Tphi_{IJ}}  & U_J  \ar@{->}[u]_{\si_J} .
}
\]
%
An {\bf isomorphism} $\Psi: \La\to \La'$ between vector bundles over $\Kk$ is a collection
$(\Psi_I: \La_I\to \La'_I)_{I\in \Ii_\Kk}$ of $\Ga_I$-equivariant bundle isomorphisms covering the identity on $U_I$, that intertwine the transition maps,
i.e.\ $\Ti\phi^{\La'}_{IJ}\circ\rho^*_{IJ}(\Psi_I) = \Psi_J \circ \Ti \phi^\La_{IJ}|_{\TU_{IJ}}$
for all $I\subsetneq J$.
%
\end{defn}

\begin{rmk}\rm  (i) If $\Kk$ has trivial isotropy, this is just the definition in \cite{MW2}.  Since the smooth structure on $U_I$ is not used,
there is a corresponding notion for a bundle over the intermediate category $\ul{\Kk}$.  Further any bundle $\ul{\La} = 
(\ul{\La}_I, {\ul{\Tphi}}\,\!^\La_{IJ})_{I\in \Ii_\Kk}$ over $\ul{\Kk}$ pulls back to a bundle over $\Kk$ via the projections $\pi_I: U_I\to \uU_I$.
\MS

\NI (ii) There are similar definitions
for  bundles over a cobordism.
As always, the key point is that the bundle has a chosen  product structure (or collar)  near its boundary.  
We will not need this notion in this generality here since we will only consider the case of
a determinant bundle that is locally oriented and hence is 
 pulled back from a line bundle over  the intermediate (cobordism) atlas as described in Remark~\ref{rmk:preorient}.
Hence we can use the definitions and results from \cite{MW2} that apply to cobordism atlases with trivial  isotropy.
$\hfill\er$
\end{rmk}

The next task is to establish that each of the two forms of the determinant bundle gives rise to a line bundle over $\Kk$. 
Recall that  we denote the top exterior power of a real vector space   $V$ by $\La^{\max} V$.

We define the
{\bf determinant line of a linear map} $D:V\to W$  to be 
$$
\det(D):= \lm\ker D \otimes \bigl( \lm \bigl( \qu{W}{\im D} \bigr) \bigr)^*,
$$ 
and use the following conventions from \cite{MW2}; see the discussion after \cite[Definition~8.1.6]{MW2}.
Each isomorphism $F: Y \to Z$ between finite dimensional vector spaces induces an isomorphism
\begin{equation}\label{eq:laphi}
\La_F :\; \lm Y   \;\overset{\cong}{\longrightarrow}\; \lm Z , \qquad
y_1\wedge\ldots \wedge y_k \mapsto F(y_1)\wedge\ldots \wedge F(y_k) .
\end{equation}
For example, if $I\subsetneq J$ and  $\Tx\in \TU_{IJ}$ is such that $\rho_{IJ}(\Tx) = x$, 
then because $$
s_I\circ \rho_{IJ}=:s_{IJ}: \TU_{IJ}\to E_I
$$ 
the derivative $\rd_{\Tx}\rho_{IJ}:  \ker \rd s_{IJ}\to \ker \rd s_I$ induces a isomorphism 
\begin{equation}\label{eq:bunrho}
\La_{\rd_\Tx\rho_{IJ}}\otimes (\La_{\id})^*:\det(\rd_\Tx s_{IJ}) = \rho_{IJ}^*(\det \rd_x s_I) \to \det(\rd_x s_{I}).
\end{equation}
Further, it follows from 
the index condition in Definition~\ref{def:change} that  with $y:= \Tphi_{IJ}(\Tx)$ the map 
\begin{equation}\label{eq:bunIJ}
\La_{IJ}(\Tx): = \La_{\rd_\Tx\Tphi_{IJ}} \otimes 
\bigl(\La_{[\Hat\phi_{IJ}^{-1}]}\bigr)^*
\, :\; \det(\rd_\Tx s_{IJ}) \to \det(\rd_{y} s_J)
\end{equation}
is an isomorphism, induced by the isomorphisms $\rd\Tphi_{IJ}:\ker\rd s_{IJ}\to\ker\rd s_J$ and
$[\Hat\phi_{IJ}] : \qu{E_I}{\im\rd s_I}\to\qu{E_J}{\im\rd s_J}$.

\begin{lemma}\label{le:detsK}  There is a line bundle $\det(\s_\Kk) = \bigl(\det (\rd s_I)\bigr)_{I\in \Ii_\Kk}$  on $\Kk$ 
where 
$$
\det(\rd s_I):=\bigcup_{x\in U_I} \det(\rd_x s_I) \;\to\; U_I \qquad 
\text{for}\; I\in\Ii_\Kk, 
$$
and the isomorphisms $\Tphi_{IJ}^\La: = \La_{IJ}(\Tx)$ in \eqref{eq:bunIJ} for $I\subsetneq J$ and $\Tx\in \TU_{IJ}$.
\end{lemma}
\begin{proof} 
We showed in \cite[Proposition~8.1.8]{MW2} that the assignments $x\mapsto  \det(\rd_x s_I)$ and $
\Tx\mapsto \La_{IJ}(\Tx)$ are smooth. Further, the weak cocycle condition holds by functoriality.
It remains to check the conditions in Definition~\ref{def:bundle} pertaining to the group actions. 
Since $\ga\in \Ga_I$ acts  by diffeomorphisms on $U_I$ and by linear transformations on $E_I$ and since $s_I$ is $\Ga_I$ equivariant,
the  action of $\ga\in \Ga_I$  on $U_I$ lifts to $\La_I$ by
$$
\La_{\rd_x \ga} \otimes 
\bigl(\La_{\ga^{-1}}\bigr)^*
\, :\; \det(\rd_x s_{I}) \to \det(\rd_{\ga x} s_I).
$$
Further, because the image $\Hat\phi_{IJ}(E_I)\subset E_J$ is invariant under the $\Ga_J$ action, there is a similar action of $\Ga_J$ on the pullback  $\rho_{IJ}^*(\det(\rd s_I))$.  Thus the required group actions exist.  The map $\rho^*_{IJ}(\La_I|_{U_{IJ}})\to\La_I|_{U_{IJ}}$ 
 is  the quotient by  $\Ga_{J\less I}$ followed by an isomorphism since this holds for $\rho_{IJ}$.  Moreover, 
 the natural lift of $\Tphi_{IJ}$  to 
$\Tphi_{IJ}^\La : = \La_{IJ}(y)$ is $\Ga_J$ equivariant because the isomorphisms
 $\rd\Tphi_{IJ}$ and $$
 \Tphi_{IJ} \times  [\Hat\phi_{IJ}]: \TU_{IJ}\times \qu{E_I}{\im \rd s_I} \to U_J \times \qu{E_J}{\im \rd s_J}
 $$
  are $\Ga_J$-equivariant.
  \end{proof}

The hard part of the above proof is to show that  the assignments $x\mapsto  \det(\rd_x s_I)$ and $
\Tx\mapsto \La_{IJ}(\Tx)$ are smooth.\footnote
{
The local argument is essentially the same as that in \cite[Thm.~A.2.2]{JHOL}, but we choose a different order for the added basis vectors in order to get consistency of signs.}
  For this one  constructs  local trivializations that 
use the contraction isomorphisms defined  in the next lemma.
These maps will also be used  in the  definition of the second version of the determinant bundle: see Proposition~\ref{prop:det2} below.

\begin{lemma}[\cite{MW2}~Lemma~8.1.7]
\label{lem:get}  Let $V,W$ be finite dimensional vector spaces.
Every linear map $F:V\to W$ together with an isomorphism $\phi:K\to \ker F$ induces an isomorphism
\begin{align}\label{Cfrak}
\mathfrak{C}^{\phi}_F \,:\; \lm V \otimes \bigl(\lm W \bigr)^* 
&\;\overset{\cong}{\longrightarrow}\;  \lm K \otimes \bigl(\lm \bigl( \qu{W}{F(V)}\bigr) \bigr)^*  
\end{align}
given by
\begin{align}
(v_1\wedge\dots v_n)\otimes(w_1\wedge\dots w_m)^* &\;\longmapsto\;
\bigl(\phi^{-1}(v_1)\wedge\dots \phi^{-1}(v_k)\bigr)\otimes \bigl( [w_1]\wedge\dots [w_{m-n+k}] \bigr)^* ,
\notag
\end{align}
where $v_1,\ldots,v_n$ is a basis for $V$ with ${\rm span}(v_1,\ldots,v_k)=\ker F$, and $w_1,\dots, w_m$ is a basis for $W$ whose last $n-k$ vectors are $w_{m-n+i}=F(v_i)$ for $i=k+1,\ldots,n$.

In particular, for every linear map $D:V\to W$ we may pick $\phi$ as the inclusion $K=\ker D\hookrightarrow V$ to obtain an isomorphism
\begin{equation}\label{eq:CD}
\mathfrak{C}_{D} \,:\;  \lm V \otimes \bigl(\lm W \bigr)^* \;\overset{\cong}{\longrightarrow}\;  \det(D) .
\end{equation} 
\end{lemma}

 Recall that the tangent bundle condition \eqref{tbc} implies that $\rd s_J$ restricts 
at all $y= \Tphi_{IJ}(\Tx)$\footnote
{
To keep notation similar  to that in \cite{MWiso} we always denote $y:=\Tphi_{IJ}(\Tx)$ where  $\Tx\in \TU_{IJ}$.
}
  to an isomorphism 
 $\qu{\rT_{y} U_J}{\rd_\Tx\Tphi_{IJ}(\rT_\Tx \TU_{IJ})}\overset{\cong}{\to} \qu{E_J}{\Hat\phi_{IJ}(E_I)}$. 
Therefore, if we choose a smooth normal bundle $$
N_{IJ} ={\textstyle \bigcup_{\Tx\in \TU_{IJ}}} N_{IJ,\Tx}
\subset \Tphi_{IJ}^*( \rT U_J)
$$
 to 
the submanifold 
$\im \Tphi_{IJ}
\subset U_J$,  
the subspaces 
$\rd_y s_J(N_{IJ,\Tx})$ 
form a smooth family of subspaces of $E_J$ that are complements  
to
$\Hat\phi_{IJ}(E_I)$.
Hence letting 
$\pr_{N_{IJ}}(\Tx) : E_J \to \rd_y s_J(N_{IJ,\Tx}) \subset E_J$
be the 
smooth family of 
projections with kernel  
$\Hat\phi_{IJ}(E_I)$,
we obtain 
a smooth family of 
linear maps
$$
F_\Tx \,:= \;
\pr_{N_{IJ}}(\Tx) 
\circ \rd_y s_J \,:\; \rT_y U_J \;\longrightarrow\; 
E_J
\qquad\text{for}\; \Tx\in \TU_{IJ}, \;\; y=\Tphi_{IJ}(\Tx)
$$
with images $\im F_\Tx=\rd_y s_J(N_{IJ,\Tx})$,
and isomorphisms to their kernel
$$
\phi_\Tx  \,:= \;  \rd_\Tx\Tphi_{IJ} \,:\; \rT_\Tx \TU_I \;\overset{\cong}{\longrightarrow}\;  \ker F_\Tx =  \rd_\Tx\Tphi_{IJ}(\rT_\Tx \TU_{IJ}) \;\subset\; \rT_y U_J .
$$
We may arrange that the maps $F_\Tx$ are $\Ga_J$-equivariant by  choosing the bundle $N_{IJ}$ to be $\Ga_J$-invariant, since
 $s_J$ is $\Ga_J$-equivariant, as is the subspace  $\Hat\phi_{IJ}(E_I)\subset E_J$ 
 (this time with trivial action of $\Ga_{J\less I}$).  
Then the isomorphism $\phi_{\Tx}$ is also $\Ga_J$-equivariant. 
By Lemma~\ref{lem:get} the maps  $F_\Tx, \phi_{\Tx}$ induce $\Ga_J$-equivariant isomorphisms 
$$
\mathfrak{C}^{\phi_\Tx}_{F_\Tx} \,:\; \lm \rT_{y} U_J \otimes \bigl(\lm E_J \bigr)^* 
 \;\overset{\cong}{\longrightarrow}\;  \lm \rT_\Tx \TU_{IJ} \otimes  \Bigl(\lm \Bigl(\qq{E_J}
{\im F_\Tx}
 \Bigr) \Bigr)^* .
$$
We may combine this with the dual of the  isomorphism $\lm \bigl(\qu{E_J}
{\rd_y s_J(N_{IJ,\Tx})}
\bigr) \cong \lm E_I$ induced 
via \eqref{eq:laphi} by $\pr^\perp_{
N_{IJ}}(\Tx) \circ\Hat\phi_{IJ} : E_I \to \qu{E_J}
{\rd_y s_J(N_{IJ,\Tx})}
$ to obtain $\Ga_J$-equivariant isomorphisms
\begin{align} \label{CIJ}
\mathfrak{C}_{IJ}(\Tx) \,: \;
 \lm \rT_{y} U_J \otimes \bigl(\lm E_J \bigr)^*  
\;\overset{\cong}{\longrightarrow}\;  \lm \rT_\Tx \TU_{IJ} \otimes \bigl(\lm E_I \bigr)^*  
\end{align}
given by $\mathfrak{C}_{IJ}(\Tx) :=  \bigl( \id_{ \lm \rT_\Tx \TU_{IJ}} \otimes 
\La_{(\pr^\perp_{
N_{IJ}}(\Tx)\circ\Hat\phi_{IJ})^{-1}}^*  \bigr) \circ \mathfrak{C}^{\phi_\Tx}_{F_\Tx}$, where again $\Ga_{J\less I}$ acts trivially on 
$E_I$.


\begin{prop}[\cite{MWiso}~Proposition~3.1.13]\label{prop:det2} Let $\Kk$ be a weak Kuranishi atlas.  Then there is a well defined line bundle $\det(\Kk)$ over $\Kk$ given by the line bundles $$
\La^\Kk_I := \lm \rT U_I\otimes \bigl(\lm E_I\bigr)^* \to U_I,\quad I\in\Ii_\Kk
$$
and the transition maps $\mathfrak{C}_{IJ}^{-1}: \rho_{IJ}^*\bigl(\La^\Kk_I |_{U_{IJ}}\bigr) \to \La^\Kk_J $ 
from \eqref{CIJ} for $I\subsetneq J$. 
In particular, the latter isomorphisms are independent of the choice of 
normal bundle $N_{IJ}$.

Furthermore, the contractions $\mathfrak{C}_{\rd s_I}: \La^\Kk_I \to \det(\rd s_I)$ from \eqref{eq:CD} with $D =\rd s_I$  define an isomorphism $\Psi^{\s_\Kk}:=\bigl(\mathfrak{C}_{\rd s_I}\bigr)_{I\in\Ii_\Kk}$ from $\det(\Kk)$ to $\det(\s_\Kk)$.
\end{prop}
\begin{proof}  
The first step is to show  that the maps $\mathfrak{C}_{\rd_x s_I}$ are isomorphisms for each $I$, and in particular that they vary smoothly with $x\in U_I$.  This holds just as before.  
Next, consider  the following diagram:
\begin{equation}\label{cclaim}
\xymatrix{
\La^\Kk_I  = \lm  \rT_x U_I \otimes \bigl( \lm E_I \bigr)^* 
 \ar@{->}[r]^{ \qquad\quad\qquad\mathfrak{C}_{\rd_x s_I}}  
 &
\det(\rd_x s_I)   \\
\rho_{IJ}^*\bigl(\La^\Kk_I\bigr) =  \lm  \rT_\Tx \TU_{IJ} \otimes \bigl( \lm E_I \bigr)^* 
 \ar@{->}[r]^{ \qquad\qquad\qquad\mathfrak{C}_{\rd_\Tx s_{IJ}}}  
 \ar@{->}[u]^{\La_{\rd_{\Tx}\rho_{IJ}}\otimes \La_{\id}}
 &
\det(\rd_\Tx s_{IJ})  \ar@{->}[d]^{\La_{IJ}(\Tx)} 
 \ar@{->}[u]_{\La_{\rd_{\Tx}\rho_{IJ}}\otimes \La_{\id}}
\\
\La^\Kk_J = \lm \rT_y U_J \otimes \bigl( \lm E_J \bigr)^*
 \ar@{->}[r]^{\qquad\quad\quad \mathfrak{C}_{\rd_y s_J}}  
 \ar@{->}[u]^{\mathfrak{C}_{IJ}
 (\Tx)} 
&
\det(\rd_y s_J) .
}
\end{equation}
 The top square has right hand map given by  \eqref{eq:bunrho}, and  commutes because $s_{IJ} = s_I\circ \rho_{IJ}$.  
Hence the maps $\mathfrak{C}_{\rd_\Tx s_{IJ}}$   vary smoothly with $\Tx\in \TU_{IJ}$, and we can define   
preliminary transition maps 
\begin{equation}\label{tiphi}
(\Ti\phi^\La_{IJ})':= \mathfrak{C}_{\rd s_J}^{-1} \circ \La_{IJ} \circ \mathfrak{C}_{\rd s_{IJ}}
\,:\; \rho_{IJ}^*\bigl(\La^\Kk_I|_{\TU_{IJ}}\bigr)\to \La^\Kk_J \qquad\text{for}\; I\subsetneq J \in \Ii_\Kk
\end{equation}
by the transition maps \eqref{eq:bunIJ} of $\det(\s_\Kk)$ and the isomorphisms $\mathfrak{C}_{\rd s}$.
These define a line bundle $\La^\Kk:=\bigl(\La^\Kk_I, \Ti\phi_{IJ} \bigr)_{I,J\in\Ii_\Kk}$ since the weak cocycle condition follows directly from that for the $\La_{IJ}$. Moreover, this automatically makes the family of bundle isomorphisms $\Psi^\Kk:=\bigl(\mathfrak{C}_{\rd s_I}\bigr)_{I\in\Ii_\Kk}$ an isomorphism from $\La^\Kk$ to $\det(\s_\Kk)$. 
It remains to see that $\La^\Kk=\det(\Kk)$ and $\Psi^\Kk=\Psi^{\s_\Kk}$, i.e.\ we claim equality of transition maps $\Ti\phi_{IJ}=\mathfrak{C}_{IJ}^{-1}$. This also shows that $\mathfrak{C}_{IJ}^{-1}$ and thus $\det(\Kk)$ is independent of the choice of 
normal bundle $N_{IJ}$ in \eqref{CIJ}.
So to finish the proof it suffices to show that the lower square commutes.  But this can be proved by slightly modifying the proof
of the corresponding result in the trivial isotropy case.  In fact, 
  since this  statement  can be proved at each point separately, 
we can treat the map $\rho_{IJ}$ as the identity, identifying $x$ with $\Tx$ and  hence reduce to the diagram 
\cite[(8.1.18)]{MW2}
which is commutative by the proof of \cite[Proposition~8.1.12]{MW2}.
 \end{proof}

This completes the construction of the two isomorphic line bundles $\det(\s_\Kk)$ and $\det(\Kk)$ over $\Kk$.
Similar arguments show that if $\Kk$ is a cobordism atlas as in Definition~\ref{def:CKS} the 
corresponding bundles   
$\det (\Kk)$ and $\det(\s_\Kk)$ are well defined and isomorphic. 

\begin{rmk}\label{rmk:preorient}\rm (i)  If  the bundle 
  $
\La^\Kk_I: =   \La^{\max} \rT U_I\otimes (\La^{\max}E_I)^*\to U_I
  $
  has  a nonvanishing $\Ga_I$-equivariant section $\si_I$ for each $I\in \Ii_\Kk$, then  $\det(\Kk)$ is the 
  pullback of a line bundle on the intermediate atlas $\uKk$.
  Indeed, the existence of  $\si_I$  implies that 
 $\La^\Kk_I$ is the pullback via the projection $\pi_I: U_I\to \uU_I$ of some bundle $\uLa_I$ over $\uU_I$
that is trivialized by $\si_I$.  Further, the  fact that the 
transition maps $\Tphi_{IJ}^\La: \rho^*_{IJ}(\La^\Kk_I) \to \La^\Kk_J$   are $\Ga_J$-equivariant implies that they
descend to well defined maps
 $\Ti\uphi_{IJ}^\La: \uLa_I|_{\uU_{IJ}}\to \uLa_J$.  Note that there is no need to assume here that
the local sections $\si_I$ are compatible. A similar remark applies to $\det(\s_\Kk)$.\MS

\NI (ii)
In the case of Gromov--Witten atlases the obstruction 
space is a direct sum $E_I: = \oplus_{\ga\in \Ga_I} E_I^0$ of copies of a 
fixed space $E_I^0$, and $\Ga_I$ acts by permuting its factors.  If we assume that  $\dim E_I^0$ is even then 
this action preserves an orientation of $E_I$ so that the local section $\si_I$ in (i) exists
exactly if the group $\Ga_I$ acts by orientation preserving maps of $U_I$.  $\hfill\er$
\end{rmk}

Here are some basic definitions .  For more detail about the isomorphisms 
$(\Ti\io^{\La,\al}_I)_{I\in\Ii_{\p^\al\Kk}}$ see \cite[equation~(8.1.13)]{MW2}.  

\begin{defn}\label{def:orient} 
A  weak Kuranishi atlas or Kuranishi cobordism $\Kk$ is {\bf orientable} if 
there exists a nonvanishing section $\si$ of the determinant bundle $\det(\s_\Kk)$.
An {\bf orientation} of $\Kk$ is a choice of nonvanishing section $\si$ of $\det(\s_\uKk)$, which pulls back to a section, also denoted $\si$, of $\det(\s_\Kk)$. 
An {\bf oriented Kuranishi atlas or cobordism} is a pair $(\Kk,\si)$ consisting of a  Kuranishi atlas or cobordism and an orientation $\si$ of $\Kk$.

For an oriented Kuranishi cobordism $(\Kk,\si)$, the {\bf induced orientation of the boundary} $\p^\al\Kk$ for $\al=0$ resp.\ $\al=1$ is the orientation of $\p^\al\Kk$,
$$
\p^\al\si \,:=\; \Bigl( \bigl( (\Ti\io^{\La,\al}_I)^{-1} \circ\si_I \circ \io^\al_I \bigr)\big|_{\partial^\al U_I \times\{\al\} } \Bigr)_{I\in\Ii_{\p^\al\Kk}}
$$
induced by the isomorphism $(\Ti\io^\al_I)_{I\in\Ii_{\p^\al\Kk}}$ between a collar neighbourhood of the boundary in $\Kk$ and the product Kuranishi atlas $
A^\al_\eps\times \p^\al\Kk$, followed by restriction to the boundary $\p^\al \Kk=\p^\al\bigl( A^\al_\eps\times \p^\al \Kk\bigr)$, where we identify $\partial^\al U_I \times\{\al\} \cong \partial^\al U_I$.

With that, we say that two oriented weak Kuranishi atlases $(\Kk^0,\si^0)$ and $(\Kk^1,\si^1)$ are {\bf oriented cobordant} 
(resp. {\bf oriented concordant}) if there exists a weak Kuranishi cobordism (resp. concordance) $\Kk$ from $\Kk^0$ to $\Kk^1$ and a section $\si$ of $\det(s_{\Kk})$ such that  $\partial^\al\si=\si^\al$ for $\al=0,1$.
\end{defn}

\begin{rmk}\rm  (i)  As we saw in Remark~\ref{rmk:preorient} the existence of a nonvanishing local section $s_I$  of
 $\La^\Kk_I$ (or $\det(\s_I)$) implies that this bundle is pulled back from a trivial bundle over the intermediate chart.   
 The orientability condition implies that one can choose compatible local orientations for all these charts.\MS

\NI (ii) As in \cite{MW2}, we define the induced orientation on the boundary $\p^\al \Kk$ of a cobordism so that it is completed to an orientation of the collar by adding the  positive unit vector $1$ along $A^\al_\eps\subset \R$  rather than the more usual outward normal vector. $\hfill\er$
\end{rmk}

We show in \cite[Proposition~8.1.13]{MW2} that if  $(\Kk,\si)$ is  an oriented, 
tame Kuranishi atlas  with trivial isotropy and with
reduction $\Vv$,  
then  the zero set $|\bZ^\nu|$  of any  admissible, precompact, transverse perturbation $\nu$ of $\s_\Kk|_\Vv$
is an oriented closed manifold whose oriented cobordism class is independent of the choice of $\nu$.  
The proof is based on the next lemma (a simplified version of \cite[Lemma~3.1.14]{MWiso}) that describes 
how to use the orientation of $\Kk$ to orient the local zero sets $Z_I$.  One then proves  that these local orientations are compatible
with coordinate changes  by adapting the proof of 
 Proposition~\ref{prop:det2} above.

\begin{lemma}  \label{le:locorient1}  Let $\Kk$ be oriented 
 and $F:U_I\to  E_I$ be a smooth map that is transverse to $0$.  Then the  zero set
$Z_I: = F^{-1}(0)$ inherits the structure of a smooth oriented manifold.
\end{lemma}
\begin{proof} 
By transversality $Z_I$ is a manifold and $\im (\rd_z F)=E_I$ for each $z\in Z_I$. Hence 
$\lm\bigl( \qu{E_I}{\im (\rd_z s_I + \rd_z\nu_I)}\bigr)= \lm\{0\} = \R$, which gives
 a natural isomorphism between the orientation bundle of $Z_I$ and the restriction of the determinant line bundle of $F$:
\begin{align*}
\lm \rT Z_I &\;=\; {\textstyle \bigcup_{z\in Z_I}} \lm \ker (\rd_z F) \\
&\; \cong\;  {\textstyle \bigcup_{z\in Z_I}} \lm \ker (\rd_z F) \otimes \R 
\;=\; \det(F)|_{Z_I} .
\end{align*}

By hypothesis, we are given a nonvanishing section of the bundle 
$\La^\Kk_I: =\lm\bigl( \TU_I\bigr) \otimes \lm\bigl(E_I\bigr)^*$.
The proof of Proposition~\ref{prop:det2} adapts to show that the contraction
$$
\mathfrak{C}_{\rd F} \,:\;
\La^\Kk_I(z) \;\longrightarrow\; \det(F)|_z = \lm \rT_z Z_I  \qquad \text{for} \; z\in Z_I 
$$
is an isomorphism that varies continuously with $z\in Z_I$.  
Hence the given section of $\La^\Kk_I$ pushes forward to a
 nonvanishing section of $\lm \rT Z_I $, i.e. an orientation of $Z_I$.
\end{proof}

When the isotropy groups of $\Kk$ are nontrivial, the above lemma shows how to  orient the local zero set.  However the global zero set  is now a branched manifold.  All the above work applies in this case, but we postpone further discussion until we have introduced the relevant language.

\subsection{Perturbation sections  and construction of the VFC}\label{ss:zero}
With the notion of orientation in hand, we are in a position to complete the construction of the VFC. 
Again, we represent this class via the zero set of a suitable reduced perturbation $\nu = (\nu_I)$.  
After making the relevant definitions, we explain how to construct a cycle from the local zero sets. 

The notion of reduction is essentially 
the same as before. Recall from Lemma~\ref{le:interm} that the intermediate atlas $\uKk$ is a filtered topological atlas, and 
that we define $\Kk$ to be tame if $\uKk$ is.  Similarly, we define a reduction of 
a tame Kuranishi atlas $\Kk$ to be
the pullback of a reduction $\und{\Vv}$ of the intermediate atlas $\uKk$.   Here is a formal definition.

\begin{defn}[see \ Definition~\ref{def:vicin}]
\label{def:vicin2}  
A {\bf reduction} $\Vv=\bigsqcup_{I\in \Ii_\Kk} V_I \subset \Obj_{\bB_\Kk}$ of a tame Kuranishi atlas $\Kk$ is 
 a tuple of (possibly empty)  precompact  open subsets $V_I\sqsubset U_I$, satisfying the following conditions:
\begin{enumerate}
\item $V_I = \pi_I^{-1}(\uV_I)$  for each $I\in \Ii_\Kk$, i.e.\ $V_I$ is pulled back from the intermediate category and so is $\Ga_I$-invariant;
\item
$V_I\sqsubset U_I $ for all $I\in\Ii_\Kk$, and if $V_I\ne \emptyset$ then $V_I\cap s_I^{-1}(0)\ne \emptyset$;
\item
if $\pi_\Kk(\ov{V_I})\cap \pi_\Kk(\ov{V_J})\ne \emptyset$ then
$I\subset J$ or $J\subset I$;
\item
the zero set $\iota_\Kk(X)=|\s_\Kk|^{-1}(0)$ is contained in 
$
\pi_\Kk(\Vv) \;=\; {\textstyle{\bigcup}_{I\in \Ii_\Kk}  }\;\pi_\Kk(V_I).
$
\end{enumerate}
Given a reduction $\Vv$, we define the {\bf reduced domain category} $\bB_\Kk|_\Vv$ and the {\bf reduced obstruction category} $\bE_\Kk|_\Vv$ to be the full subcategories of $\bB_\Kk$ and $\bE_\Kk$ with objects $\bigsqcup_{I\in \Ii_\Kk} V_I$ resp.\ $\bigsqcup_{I\in \Ii_\Kk} V_I\times E_I$, and denote by $\s|_\Vv:\bB_\Kk|_\Vv\to \bE_\Kk|_\Vv$ the section given by restriction of $\s_\Kk$. 
\end{defn}

It is  crucial  in this context   that the quotient map $\Obj \ \bB_\Kk\to \Obj\ \ubB_\Kk$ is proper (see  Lemma~\ref{le:vep}), so that the pullback $V_I$ of a precompact subset $\uV_I\sqsubset \uU_I$ is still precompact in $U_I$.   Because of this, we can
establish the existence and uniqueness of reductions modulo cobordism  by working in the intermediate category, hence proving the analog of Proposition~\ref{prop:cov2}.
 Further, 
the results on {\bf nested (cobordism) reductions} can  be interpreted  at the intermediate level, and hence go through as before.
Here, we say that two reductions $\Cc, \Vv$ are {\bf nested} (written $\Cc\sqsubset \Vv$) if $C_I$ is a precompact subset of $V_I$ for all $I$.
The only real change needed to the discussion in \S\ref{ss:red} above  is that, to achieve transversality  we should work with \lq\lq multisections" rather than sections.\footnote
{
To see that one cannot always achieve transversality   by equivariant sections of an orbibundle, consider the bundle $\pi: (S^1\times \R, \Z_2)  \to (S^1,{\id})$ where the generator of $\Z_2$ acts on $S^1\times \R$ by $(x,e)\mapsto (x,-e)$.  Then the only equivariant section is the constant section, which is not transverse to $0$.}
In our categorical framework, these can be defined very easily.

\begin{defn}\label{def:sect2} 
A {\bf reduced  perturbation} $\nu$ of $\Kk$ is 
a smooth map 
$$
\nu:\Vv = \bigsqcup_{I\in \Ii_K} V_I \;\longrightarrow\; \Obj_{\bE_\Kk|_\Vv}
$$
between the spaces of objects in the reduced domain and obstruction categories of 
some reduction $\Vv$ of $\Kk$, 
such that $\pr_\Kk\circ\nu$ is the identity. 
Further, we require  that $\nu=(\nu_I)_{I\in\Ii_\Kk}$ is given by a family of smooth maps $\nu_I: V_I\to E_I$ 
that are compatible with coordinate changes in the sense that
\begin{equation}\label{eq:compatc}
\nu_J\big|_{\TV_{IJ}}\  =\ 
 \Hat\phi_{IJ}\circ\nu_I\circ \rho_{IJ}\big|_{\TV_{IJ}}
\quad \mbox{ where }\;\; \TV_{IJ}: = \TU_{IJ}\cap V_J\cap \rho_{IJ}^{-1}(V_I).
%
\end{equation}
We say that  $\nu$ is:
\begin{itemlist}\item {\bf admissible } if
$$ 
\rd_y \nu_J(\rT_y V_J) \subset\im\Hat\phi_{IJ} \qquad \forall \; I\subsetneq J, \;y\in \TU_{IJ}\cap
V_J\cap \rho_{IJ}^{-1}(V_I).
$$ 
\item   {\bf precompact} if there is a nested reduction $\Cc\sqsubset \Vv$ such that 
$$\;\bigcup_{I\in \Ii_\Kk} \pi_\Kk\bigl((s_I|_{V_I} + \nu_I)^{-1}(0)\bigr)
\;\subset\; \pi_\Kk(\Cc);
$$
\item {\bf transverse}  (written $\pitchfork 0$) if $s_I|_{V_I} + \nu_I: V_I\to E_I$ is transverse to $0$  for each $I\in \Ii_\Kk$.
\end{itemlist}
\end{defn}

The above compatibility condition implies that when $I\subset J$ the perturbation  $\nu_J$ is determined by $\nu_I$ on the part 
$\TV_{IJ}$ of $V_J$ that lies over $V_I$.
In particular it takes values in $E_I\subset E_J$ and is invariant under the action of $\Ga_{J\less I}$,
and this means that $\nu$ is compatible with morphisms of type (b) on page~\pageref{b morph}.
However, $\nu$ is {\it not} in  general a functor $\bB_\Kk\big|_\Vv \to \bE_\Kk\big|_\Vv$  since it is not required to be equivariant under the group actions.  Hence it induces a multivalued map $|\nu|$ on the realization $|\Vv|$, given by
$$
|\nu_I([I,x])| = \bigcup_{\ga\in \Ga_I} \pi_{\bE_\Kk}([I, \nu(\ga x)])\in |\pr|^{-1}([I,x]).
$$
Thus one can think of $\nu$ as a \lq\lq multisection", even though it is defined by a single valued function on $\Vv$.

\begin{rmk}\rm  The reduced perturbation $\nu: \bB_\Kk|_\Vv\to \bE_\Kk|_\Vv$ 
is defined in \cite{MWiso} to be  a functor between the  {\bf pruned}  domain and obstruction categories $\bB_\Kk|_\Vv^{\less \Ga}$ 
and  $\bE_\Kk|_\Vv^{\less \Ga}$ of \cite[Lemma~3.2.3]{MWiso}.  These categories have the same object spaces as 
 $\bB_\Kk|_\Vv$ and  $\bE_\Kk|_\Vv$ but fewer morphisms.  For example in  $\bB_\Kk|_\Vv^{\less \Ga}$  one only keeps the morphisms of the form
 $(I,J, y): (I,\rho_{IJ}(y))\mapsto (J,y)$  given by the  projections $\rho_{IJ}$.  
\hfill$\er$
\end{rmk}

Because the perturbation section
$\nu$ satisfies essentially the same compatibility conditions as does a perturbation section of an atlas with trivial isotropy, it may be constructed by 
 the same method.   Notice in particular  that, although $\nu$ must be constructed on the space of objects $\bigsqcup_I V_I$ of $\bB_{\Kk}|_\Vv$, the sets such as  $V^{|J|}_J$ and $B_\de^J(N^{|J|}_{JI})$ used in \S\ref{ss:red} to describe its
 inductive construction can all be pulled back from corresponding subsets of the domains of  the intermediate category
 (which after all is where the metric lives). 
It follows that the construction of $\nu$ goes through without essential change.  

However, because of the presence of nontrivial isotropy, the local zero sets $Z_I: = (s_I|_{V_I} + \nu_I)^{-1}(0)$ no longer fit together to form a manifold.  Rather we must use the notions of  weighted nonsingular branched (wnb) groupoid
(defined below)
 and 
weighted branched manifold developed in  \cite{Mbr}.  Since this requires considerable preparation, we will 
first state the main result and show how to use it to construct the VFC.  

\begin{prop}\label{prop:zero}  Let $(\Kk,\si)$ be a tame, oriented, Kuranishi atlas of dimension $d$ with nested  reduction $\Cc\sqsubset \Vv$, and let  $\nu:  \bB_{\Kk|_{\Vv}}\to 
\bE_{\Kk|_{\Vv}}$ be an admissible reduced perturbation that is  transverse and precompact with local zero sets  
$$
Z_I: = (s_I|_{V_I} + \nu_I)^{-1}(0) \subset 
V_I\cap \pi_\Kk^{-1}(\pi_\Kk(\Cc)).
$$
Let $\bZ^\nu$ be the full subcategory of $\bB_{\Kk|_\Vv}$ with object space $\bigcup_I Z_I$.  
Then $\bZ^\nu$ can be completed to 
a  wnb groupoid  $(\Hat{\bZ}^\nu, \La^\nu)$ with the same realization $|\Hat{\bZ}^\nu| = |\bZ^\nu|$.  Moreover, 
the maximal Hausdorff quotient $ |\Hat\bZ^\nu|_\Hh$ is compact and  inclusion
induces a 
 map $|\io_{\nu}|_\Hh:   |\Hat\bZ^\nu|_\Hh \to |\Kk|$ with image in  $\pi_\Kk(\Cc)$.
\end{prop}


\NI{\bf Sketch proof of Theorem B.}  
By definition, the maximal Hausdorff 
quotient $|\Hat \bZ|_\Hh$ (which is defined below) of a wnb groupoid $\Hat\bZ$  is a weighted branched manifold.  
We show in \cite[Proposition~3.25]{Mbr}  that when  $Z: = |\Hat\bZ|_\Hh$ is compact, it carries a fundamental class $[Z]\in H_d(Z;\Q)$, where $d$ is the dimension of $\Hat\bZ$.  (See Examples~\ref{ex:wnb} and~\ref{ex:foot2} below.)
 With this information we can now repeat the arguments sketched at the end of \S2
to establish the existence of the class $[X]^{vir}_\Kk\in \check H_d(X;\Q)$. To prove that this class does not depend on our auxiliary choices,
 we first establish a version of Proposition~\ref{prop:zero} for cobordisms, and then argue as in the case with trivial
isotropy.
\hfill$\Box$
\MS\MS

%
%
%
%
%
%
%
%
%
%
%
%

With this in hand, we now return to the discussion of wnb groupoids $\bG$.
Because these groupoids\footnote{
See the beginning of \S\ref{ss:orb} for the definition of an \'etale groupoid; nonsingular means 
that there is at most one morphism between any two points.}
 are in general not  proper, the realization $|\bG|$ may not be Hausdorff, and we write 
$|\bG|_{\Hh}$ for its {\bf maximal Hausdorff quotient}. Thus  $|\bG|_\Hh$ is a Hausdorff quotient of $|\bG|$ that  satisfies the 
 following universal property: any continuous map from $|\bG|$ to a Hausdorff  space $Y$ factors though 
 the projection
 $|\bG|\to |\bG|_\Hh$.  (The existence of such a quotient for any topological space is proved in \cite[Lemma~3.1]{Mbr}; see also \cite[Lemma~A.2]{MWiso}.) 
We denote the natural maps by
$$
\pi_\bG:\Obj_{\bG}\to |\bG|,\qquad   \pi_{|\bG|}^\Hh: |\bG|\longrightarrow |\bG|_\Hh, \qquad 
\pi^{\Hh}_\bG:= \pi_{|\bG|}^{\Hh}\circ \pi_\bG : \Obj_{\bG}\to |\bG|_\Hh .
 $$
Moreover, for $U\subset \Obj_\bG$ we write $|U| := \pi_\bG(U) \subset |\bG|$ and $|U|_\Hh := \pi_\Hh(U)\subset |\bG|_\Hh$.
The {\bf branch locus} of $\bG$ is defined to be the subset of $|\bG|_\Hh$ consisting of points with more than one preimage in $|\bG|$.

 \begin{defn} 
A   {\bf weighted nonsingular branched groupoid} (or {\bf wnb groupoid} for short) of dimension $d$
is a pair $(\bG,\La)$ consisting of an oriented, nonsingular \'etale  groupoid 
$\bG$ of dimension $d$, together with a rational weighting function $\La:|\bG|_{\Hh}\to \Q^+: = \Q\cap (0,\infty)$ that satisfies the following compatibility conditions.
For each $p\in |\bG|_{\Hh}$ there is an open neighbourhood $N\subset|\bG|_{\Hh}$ of $p$,
a collection  $U_1,\dots,U_\ell$ of disjoint open subsets of $(\pi_{\bG}^{\Hh})^{-1}(N)\subset \Obj_{\bG}$ (called {\bf local branches}), and a set of positive rational weights $m_1,\dots,m_\ell$ such that the following holds: 
\SSS

\NI
{\bf (Covering) } $( \pi_{|\bG|}^{\Hh})^{-1}(N) = |U_1|\cup\dots \cup |U_\ell| \subset |\bG|$;
\SSS

\NI
{\bf (Local Regularity)}  
for each $i=1,\dots,\ell$ the projection 
$\pi_{\bG}^{\Hh}|_{U_i}: U_i\to |\bG|_{\Hh}$ is a homeomorphism onto a relatively closed subset of $N$;
\SSS

\NI
{\bf (Weighting)}  
for all $q\in N$, 
the number 
$\La(q)$ is the sum of the weights of the local
branches whose image contains $q$:
$$
\La(q) = 
\underset{i:q\in |U_i|_{\Hh}}\sum m_i.
$$
\end{defn}

 Further we define a 
{\bf weighted branched manifold} of dimension $d$ to be a pair $(Z, \La_Z)$ consisting of a topological space $Z$ together 
with a function $\La_Z:Z\to \Q^+$ 
and an equivalence class\footnote{
The precise notion of equivalence is given in \cite[Definition~3.12]{Mbr}. In particular it ensures that the induced function $\La_Z: = \La_\bG\circ f^{-1}$ and the dimension of $\Obj_{\bG}$ are the same for equivalent structures $(\bG,\La_\bG, f)$. 
} 
of wnb $d$-dimensional groupoids $(\bG,\La_\bG)$ and homeomorphisms $f:|\bG|_\Hh\to Z$ that induce the function $\La_Z = \La_\bG\circ f^{-1}$.
Analogous definitions of a wnb cobordism groupoid (as always, assumed to have collared boundaries) and of a wnb cobordism 
are spelled out in \cite[Appendix]{MWiso}.

 \begin{example}\label{ex:wnb} \rm 
(i) 
A compact weighted branched manifold
of dimension $0$  consists of
a finite set of points $Z=\{p_1,\ldots,p_k\}$, each with a positive rational weight $m(p_i)\in \Q^+$  and orientation $\mathfrak o(p_i)\in \{\pm\}$.  Its fundamental class pushes forward to the sum   $\sum_i \orr(p_i) m(p_i)\in H_0(pt,\Q) = \Q$ under any map $Z\to \{pt\}$.
Any representing groupoid $\bG$ has as object space $\Obj_{\bG}$  a set with the discrete topology, that is equipped with an orientation function $\orr: \Obj_{\bG}\to \{\pm\}$.  The morphism space $\Mor_{\bG}$ is also a discrete set and, because we assume that $\bG$ is oriented, 
defines an equivalence relation on $\Obj_{\bG}$ 
such that  $x\sim y\Longrightarrow \orr(x) = \orr(y)$.  Moreover, 
because  $|\bG|$ is  Hausdorff, we can identify $ |\bG|= |\bG|_\Hh$ and hence conclude that $\Obj_{\bG}$ has precisely $k$ equivalence classes.
\MS

\NI
(ii)
For the prototypical example of a $1$-dimensional weighted branched cobordism $(|\bG|_\Hh,\La)$, take $\Obj(\bG)=I\sqcup I'$ equal to two copies of the interval $I=I'=[0,1]$ with nonidentity morphisms from $x\in I$ to $x\in I'$ for $x\in [0,\frac 12)$ and their inverses,
 where we suppose that $I$ is oriented in the standard way.
 Then the realization and its Hausdorff quotient are
$$
\begin{array}{cl}
 |\bG| &= \;\qq{I\sqcup I'}{\bigl\{(I,x)\sim (I',x) \; \text{iff}\;  x\in [0,\tfrac 12)\bigr\}},\\
 |\bG|_\Hh& = \;\qq{I\sqcup I'}{\bigl\{(I,x)\sim (I',x) \; \text{iff}\;  x\in [0,\tfrac 12]\bigr\}},
\end{array}
$$
and the branch locus is a single point $\Br(\bG)=
\bigl\{[I,\frac 12]=[I',\frac 12]\bigr\}
\subset |\bG|_\Hh$.
The choice of weights $m, m'>0$ on 
the two local branches
$I$ and $I'$ determines the weighting function $\La: |\bG|_\Hh \to (0,\infty)$ as
$$
\La([I,x])
 = \left\{
 \begin{array}{ll} 
 m+m'  & \mbox{ if }\;\; x\in [0,\frac 12],  \vspace{.1in}\\
m &\mbox{ if }\;\;  x\in (\frac 12,1],
\end{array}\right. 
\qquad
\La([I', x]) 
 = \left\{
 \begin{array}{ll} 
m+m'  & \mbox{ if }\;\; x\in [0,\frac 12],  \vspace{.1in}\\
m'  &\mbox{ if }\;\;  x\in (\frac 12,1].   
\end{array}\right.
$$
For example, if we give each branch $I, I'$ the weight $m=m'=\frac 12$ and choose appropriate collar functors $\io^\al_\bG$, we obtain a weighted branched cobordism $(|\bG|_\Hh, \io_\bG^0, \io_\bG^1, \La)$ with
$|\p^0\bG|_\Hh=\bigl\{[I,0]=[I',0]\bigr\}$,
which is a single point with weight $1$, and $|\p^1\bG|_\Hh=
\bigl\{[I,1],[I',1]\bigr\}$, 
which consists of two points with weight $\tfrac 12$, all with positive orientation.
The fundamental class $[Z]$ of $Z = |\bG|_\Hh$ is then represented in $H_1(Z,\p Z;\Q)$ by the weighted sum of the two branches, thought of as singular $1$-chains.   Note that its image under the boundary map $H_1(Z,\p Z;\Q)\to H_0(\p Z)$
is $[\p^1 Z] - [\p^0 Z]$, where $\p^\al Z: = |\p^\al\bG|$.

Another choice of collar functors for the same weighted groupoid $(\bG,\La)$ 
might give rise to a different partition of the boundary into incoming $\p^0\bG$ and outgoing $\p^1\bG$, for example yielding
a weighted branched cobordism with $|\p^0\bG|_\Hh=\bigl\{[I,0]=[I',0] , [I,1]\bigr\}$, consisting of two points with weights and orientations
 $(1,+)$ and $(\frac 12, -)$ 
and $|\p^1\bG|_\Hh=\bigl\{[I',1]\bigr\}$, consisting of one point with weight $(\tfrac 12,+)$.

See Figure~\ref{fig:7} for an illustration of the realization of a similar example.
\hfill$\er$
\end{example}

\NI{\bf Sketch  proof of Proposition~\ref{prop:zero}.} 
Note first that 
each set $$
\TV_{IJ}: = V_J\cap \rho_{IJ}^{-1}(V_I)\subset \TU_{IJ} = V_J\cap \pi_{\Kk}^{-1}\bigl(\pi_\Kk(V_I)\bigr)
$$ 
is invariant under the action of $\Ga_J$, and 
is an open subset of the closed submanifold  $\TU_{IJ}=s_J^{-1}(E_I)$ of $V_J$, where the last equality holds by
the  tameness condition \eqref{eq:tame2}.
Further
if $F\subset I\subset J$, 
\begin{equation}\label{eq:HIJ}
V_J\cap \rho_{IJ}^{-1}(\TV_{FI}) = \TV_{FJ}\cap \TV_{IJ}= V_J\cap \pi_{\Kk}^{-1}\bigl(\pi_\Kk(V_F)\cap \pi_\Kk(V_I)\bigr)\subset \TU_{FJ}.
\end{equation}
In fact, the second equality holds for any pair of subsets $F,I\subset J$.  However, because $\Vv$ is a reduction, the intersection is empty unless
$F,I$ are nested, i.e. either $F\subset I$ or $I\subset F$.
Finally, the group $\Ga_{I\less F}$ acts freely on $\TU_{FI}$ (by Definition~\ref{def:change2} for a coordinate change)  and hence also on $\TV_{FI}$.

 We begin by establishing the following properties of the local zero sets $Z_I$:
\begin{enumerate}
\item  each  $
Z_{I}: = (s_I|_{V_I}+\nu_I)^{-1}(0)\subset V_I
$
 is a submanifold of dimension $d$; 
\item for all $I\subset J$, the intersection
 $\TV_{IJ}\cap Z_{J}$ is an open subset of $Z_{J}$, and 
each $\al\in \Ga_{J\less I}$ induces a diffeomorphism $Z_{J}\cap \TV_{IJ}\to 
Z_{J}\cap \TV_{IJ}: y\mapsto \al y$;
\item for all $I\subset J$,
 $\rho_{IJ}: \TV_{IJ}\cap Z_{J} \to Z_{I}$ is a submersion onto 
  an open subset of $Z_{I}$;
  \item  if $\Kk$ is oriented and each $Z_I$ is oriented as in Lemma~\ref{le:locorient1}, 
  then the maps in (i) and (ii) preserve orientation.
\end{enumerate}
The proof that these properties hold is straightforward;. In particular, the first claim in (ii) holds by the index condition for the inclusion $\Tphi_{IJ}: \TU_{IJ}\to U_J$, while the second holds by the compatibility condition
\eqref{eq:compatc}. Indeed, if  $x\in \TV_{IJ}$ and $\al\in \Ga_{J\less I}$ then compatibility implies that $\nu_J(\al x) = \nu_J(x) $, while 
equivariance implies that $s_J(\al x) = \al s_J(x) = s_J(x)$ since  $s_J|_{\TV_{IJ}} = \Hat\phi_{IJ}\circ \rho_{IJ}|_{\TV_{IJ}}$ takes values in $E_I$ and
$\Ga_{J\less I}$  acts trivially on $E_I$.  Hence $(s_J + \nu_J)(\al x)=0$ if and only if $(s_J + \nu_J)(x)=0$.

We then define the wnb groupoid $\Hat\bZ: = \Hat\bZ^\nu$ to have objects  $\Obj_{\Hat\bZ} = \bigsqcup_I Z_I$, and  morphisms
that are 
generated by the covering maps $\rho_{IJ}: \TV_{IJ}\to V_I$.  In other words, we take all morphisms of this kind (which are denoted
$(I,J,\id,y)\in \Mor_{\bB_\Kk}$) and then complete to a groupoid by including all composites and inverses. As we sketch below, it turns out that the resulting space of morphisms
with source in $Z_{I}$ and  target in $Z_{J}$
 for $I\subset J$ is given by
\begin{align}\label{eq:morZ0}
&\quad \Mor_{\Hat\bZ}(Z_{I}, Z_{J}) = \bigcup_{\emptyset\ne F\subset I} \bigl(Z_{J}\cap  \TV_{IJ}\cap
\TV_{FJ}\bigr)\times \Ga_{I\less F}, \\ \notag
&\qquad = \bigl\{(I,J,\al, y) \ \big| \  \al \in \Ga_{I\less F}, y\in \TV_{FJ}\cap \TV_{IJ} \cap Z_{J} \mbox{ for some }  F\subset I,  \ F\ne \emptyset \bigr\},
\end{align}
where we interpret  $\Ga_{\emptyset} = {\id}$; in particular  the identity morphisms are given by taking $F=I=J$.
Here we use the same notation as for the category $\bB_\Kk$.  In other words,
source and target of these morphisms are given by
\begin{equation}\label {eq:morZ}
(I,J,\al,y) \in \Mor_{\bG}\bigl( (I,\al^{-1}\rho_{IJ}( y)), \ (J, y)\bigr).
\end{equation}
The category $\Hat\bZ$ also contains the inverses of these morphisms.
Thus we obtain $\Hat\bZ$ from the full subcategory of $\bB_\Kk$ with objects $ \bigsqcup_I Z_I$ by throwing away many of the morphisms, and then adding inverses. 
Note that  a point $(J,x)\in Z_J$ that does not lie in any subset $\TV_{IJ}$ for $I\subsetneq J$ is the target of exactly one morphism, namely the identity morphism, while if it lies in some $\TV_{IJ}$  it is the target of morphisms with source in $Z_I$ as well as morphisms with source in $Z_J$ given by the (free) action of $\Ga_{J\less I}$. 

 It is not too hard to see that this set of morphisms is closed under composition.
For example, one can check that  if $I\subset J\subset K$ and $\rho_{IJ}(y') = \rho_{IK}(y)$ so that $y' = \rho_{JK}(y)$,
 then the  composite\footnote
 {
 as usual written in categorical order} 
on the left of the following identity  is well defined and equals the morphism on the right:
  $$
(I,J,\id,y')^{-1}\circ (I,K,\id,y) = (J,K,\id, y).
$$
On the other hand, 
 if $I\subset K\subset J$ and there are $y'\in \TV_{IJ}\cap V_J,\;  y\in \TV_{IK}\cap V_K$ with
$$
\rho_{IJ}(y') = \rho_{IK}(\rho_{KJ}(y')) = \rho_{IK}(y)\in V_I, 
$$
 then
$\rho_{KJ}(y')$ and $y$ lie in the same $\Ga_{K\less I}$-orbit, so that there is $\be \in \Ga_{K\less I}$ such that
$y = \be\, \rho_{KJ}( y') =  \rho_{KJ}(\be y') \in V_K$, and we have
\begin{align}\label{eq:ACinv}
& (I,J,\id,y')^{-1}\circ (I,K,\id,y) \;  =\; (J,J,\be,\be\,y')\circ (K,J,\id,\be\, y')^{-1}.
\end{align}
This shows that if we  patch the local zero sets together by the projections $\rho_{IJ}$ and then complete to a groupoid by adding inverses
and composites, one necessarily must add the morphisms from $Z_J$ to $Z_J$ given by the actions of the groups $\Ga_{J\less I}$  on the sets $\TV_{IJ}$.
The fact that these actions are free implies that $\Hat\bZ$ is nonsingular. Further the category is \'etale by properties (ii) and (iii) above, and oriented by (iv).  One might informally think that the manifolds $Z_I$ form the local branches.  However, this is not strictly true because the map $Z_J\to  |\Hat\bZ|$ is {\it not} injective but rather quotients out by the
free actions of the subgroups $\Ga_{J\less I}$ on the sets $\TV_{IJ}\cap Z_J$.

Once $\Hat\bZ$ is constructed, the next step is to understand its maximal Hausdorff quotient.   
 By \cite[Lemma~3.2.10]{MWiso}, this is the realization of the groupoid $\Hat\bZ_\Hh$ that is obtained from $\Hat\bZ$ by closing the equivalence relation on $\Obj_{\Hat\bZ}$ generated by $\Mor_{\bZ}$.  Thus $\Hat\bZ_\Hh$ has extra morphisms 
from $Z_J$ to $Z_J$ given by 
 \begin{align*}\label{eq:morZ3}
&\quad \Mor_{\Hat\bZ_\Hh}(Z_{J}, Z_{J}) = \bigcup_{F\subset J} \bigl(Z_{J}\cap
\Fr(\TV_{FJ})\bigr)\times \Ga_{J \less F}, \\ \notag
&\qquad = \bigl\{(J,J,\al, y) \ \big| \ \exists F\subset J \ \al \in \Ga_{J\less F}\ s.t.\  y\in \Fr(\TV_{FJ})\cap Z_{J}\bigr\},
\end{align*}
where ${\rm Fr}(A): = \ov A \less A$ denotes the frontier of the set $A$.
As one can clearly see in the football  example given below, these extra morphisms lie over the branch locus.
Finally, one  checks that the following formula 
gives a well-defined
 weighting function  $\La:|\Hat\bZ_\Hh|\to \Q^+$:
  $$
  \La(p): =  \sum_{ x\in Z_J \cap \pi_{\Hat\bZ_\Hh}^{-1}(p)} \frac {1}{|\Ga_J|}.
  $$
Thus each local zero set $Z_J$ is assigned the weight $\frac {1}{|\Ga_J|}$, and the weight of a point $p\in |\Hat\bZ_\Hh|: = |\Hat\bZ|_\Hh$
   is obtained by multiplying this weight by the number of its inverse images in $Z_I$. 
   This completes the sketch proof. \hfill$\Box$

Here are some examples to illustrate  this construction. 

\begin{example}\label{ex:foot2}\rm  
We begin with the orbifold case, i.e. where the atlas $\Kk$ on $X$ has trivial (i.e. zero dimensional) obstruction bundles so that $\s_\Kk$ and $\nu$ are identically zero and  $\io_\Kk: X \to |\Kk|$ is a homeomorphism.  Recall from Proposition~\ref{prop:zero} that  $\bZ$ is defined as a subcategory of $\bB_\Kk|_\Vv$, while $\Hat \bZ$ is its groupoid completion.  Therefore to avoid confusion below we often designate groupoids with hats.
\smallskip

\NI (i)  First suppose that $X = \qu{M}{\Ga}$ is the quotient of a compact oriented smooth manifold $M$ by the action of a finite group $\Ga$, and that $\Kk$ is the atlas with a
single chart  with domain $M$ and $E = \{0\}$.  Then 
$\bZ$ is the category with objects $M$ and only identity morphisms, because there are no pairs  $I,J\in \Ii_\Kk$ such that $\emptyset\ne I\subsetneq J$.
Therefore
 $\bZ = \Hat\bZ_\Hh$ has realization 
$|\Hat\bZ_\Hh| = M$ and 
  the weighting function $\La:M\to \Q$ is given by $\La(x) = \frac{1}{|\Ga|}$. 
  If the action of $\Ga$ is effective on every open subset of $M$ then
 the pushforward of $\La$ by $\io_{\Hat\bZ_\Hh}: M\to X$, which is defined by
$$
(\io_{\Hat\bZ_\Hh})_*\La(p): = \sum_{x\in  (\io_{\Hat\bZ_\Hh})^{-1}(p)} \La(x)
$$
takes the value $1$ at every smooth point (i.e. point with trivial stabilizer) of the orbifold $\qu{M}{\Ga}$.     On the other hand, if $\Ga$ acts by the identity so that the action is totally noneffective then $\io_{\Hat\bZ_\Hh}: M\stackrel{\cong}\to X$ is the identity map and 
the weighting function $X\to \Q^+$ takes the constant value $\frac 1{|\Ga|}$.  

We claim that this choice of weights gives a  fundamental  class that is consistent with standard conventions.
Recall from \cite{Mbr} that
if the branch locus $\Br(|\Hat\bZ|)$  (i.e. set of points with more than one inverse image in $|\Hat\bZ|$) of a weighted branched manifold $|\Hat\bZ|_\Hh$ is sufficiently nice then there is a triangulation of $|\Hat\bZ|_\Hh$ such that $\Br(|\Hat\bZ|)$ is contained in the codimension $1$-skeleton, and each top dimensional simplex lies in a single branch.
  In this case, the compatibility condition on the weights implies that the  fundamental class is represented by
  the sum of the top dimensional simplices, each provided with the appropriate orientation and weight.
For example, if the action of $\Ga$  on $M$ is effective and  preserves an orientation
then its singular set $M^{sing} = \{x:\Ga^x\ne \id\}$ has codimension at least $2$ 
and  the fundamental class of the weighted branched manifold $(M,\La)$  pushes forward to the \lq\lq usual" fundamental class of $X$, which is represented by the pseudocycle
$X\less X^{sing} $, where $X^{sing}: = \qu{M^{sing}}{\Ga}$.
\MS

\begin{figure}[htbp] 
   \centering
   \includegraphics[width=3in]{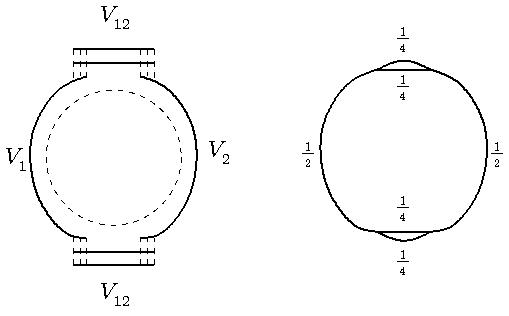} 
   \caption{This illustrates the zero set $\bZ$ in the case of the atlas on $S^1$ constructed in Example~\ref{ex:foot}~(ii).  The left diagram shows $\Obj_{\bZ^\nu} : = \sqcup V_I$, with vertical lines to denote the existence of morphisms, while the right shows the realization and weighting function.}
   \label{fig:7}
\end{figure}
\NI (ii)
If the action is noneffective, then one might wonder about the correct definition for the fundamental class. 
However, in dimension $d=0$ we saw in Example~\ref{ex:wnb}~(i) that a branched manifold  $Z = |\bZ|_\Hh$ is a finite collection of points $\{p_1,\dots, p_k\}$, one for each equivalence class in $\Obj_{\bZ}$, where 
the point $p_i$ corresponding to an equivalence class with stabilizer $\Ga^i$ has weight  $\frac 1{|\Ga^i|}$.  If each point is positively oriented, then our construction counts 
the \lq\lq number of elements" in $|\bZ|_\Hh$ as the sum
$\sum_{i=1}^k \frac 1{|\Ga^i|}$, which is the Euler characteristic of the groupoid; see  Weinstein~\cite{Wein} for  further discussion and references.  Example \ref{ex:zorb} below gives another relevant example.
\MS

\NI (iii)   Again suppose that $X = \qu{M}{\Ga}$ where $\Ga = \Z_2$ acts trivially on $M$, and choose an atlas $\Kk$ with one chart 
$$
\bigl(U = M\times \R^2,\; E=\R^2,\; \Ga: = \Z_2, \; \s: = \pr_E: U\to E ,\; \psi = \pr_M|_{M\times \{0\}}\bigr),
$$
where the nonidentity element $\ga \in \Z_2$ acts by  $(x,e)\mapsto (x,-e)$ on $U$ and by $e\mapsto -e$ on $E$.
The reduction $\Vv$ is given by a choice of any open neighbourhood $V$ of $\s^{-1}(0) = M\times \{0\}\subset U$.
Further, any function $\nu:V\to \R$ of the form $\nu(x,e) = f(x)$ is an admissible transverse perturbation, and 
defines a zero groupoid
$\Hat\bZ^\nu$ whose space of objects $\{(x,e) |  f(x)  +e = 0\}$  can be identified with the graph of $-f:M\to \R^2$. 
As in (i), $\Hat\bZ^\nu$ has only identity morphisms and its
weighting function is the constant $\frac 12$.
\MS

\NI (iv) For a more interesting  example, consider the \lq\lq football" discussed in Example~\ref{ex:foot}, with  reduction $\Vv$  given by two discs
 $V_1\sqsubset U_1, V_2\sqsubset U_2$ with disjoint images in $X$, together with an open annulus $V_{12}\sqsubset U_{12}$.  
 For $j = 1,2$ the sets $\TV_{j,12}\subset V_{12}$ are disjoint open annuli that project  into
$V_j$ by a covering map of degree $3$ for $j=1$ (that quotients out by $\Ga_{(12)\less 1} = \Ga_2 = \Z_3$) and degree $2$ for $j=2$.  
Then $Z_I = V_I$ so that  $\Obj_{\bZ} = V_1\sqcup V_2\sqcup V_{12}$.  For $j = 1,2$ the category $\Hat\bZ$ 
has the following morphisms (besides identities);
\begin{itemize}\item morphisms  $V_j\to V_{12}$ given by the projection $\rho_{j, 12}: \TV_{j,12} \to V_j$, together with their inverses;
\item morphisms $V_{12}\to V_{12}$ given by the action of $\Z_3 = \Ga_{(12)\less 1}$ on $\TV_{1,12}$, resp. of $\Z_2
=\Ga_{(12)\less 2} $ on $\TV_{2,12}$.
\end{itemize}
To obtain $\Hat\bZ_\Hh$ we add the morphisms given by the action of $\Ga_{(12)\less 1}$ on 
the boundary $\Fr(\TV_{1,12}): = V_{12}\less \TV_{1,12}$  and the action of 
$\Ga_{(12)\less 2}$ on $\Fr(\TV_{2,12}): = V_{12}\less \TV_{2,12}$. 
The weighting function $\La$ is given by:
\begin{align*}
\La(p) & \ = \tfrac 12 \mbox { if } p\in \ov{\pi_{\Hat\bZ_\Hh}(V_1)} = \pi_{\Hat\bZ_\Hh}(V_1) \cup \pi_{\Hat\bZ_\Hh}(\cl(\TV_{1,12})) \\
& \ = \tfrac 13 \mbox { if } p\in \ov{\pi_{\Hat\bZ_\Hh}(V_2)}= \pi_{\Hat\bZ_\Hh}(V_2) \cup \pi_{\Hat\bZ_\Hh}(\cl(\TV_{2,12})) \\
& \ = \tfrac 16 \mbox { if } p\in \pi_{\Hat\bZ_\Hh}(V_{12} \less \cl\bigl(\TV_{1,12} \cup \TV_{2,12}))
\end{align*}
Notice that for $j = 1,2$  the weighting function does not change along the boundary of the intersection 
$\pi_{\Hat\bZ_\Hh}(V_j) \cap \pi_{\Hat\bZ_\Hh}(\TV_{j,12})$, i.e. there is no branching there, while it does change along
the internal boundaries $ \pi_{\Hat\bZ_\Hh}(\Fr(\TV_{j,12}))$ in $\pi_{\Hat\bZ_\Hh}(V_2)$.
 Note further  that the pushforward of $\La$ by the map $|\io_{\Hat\bZ_\Hh}|: |\Hat\bZ_\Hh|\to |\bB_\Kk\big|_\Vv|$ takes the value $1$ at all 
 points except the two poles $N,S$:
 $$
  |\io_{\Hat\bZ_\Hh}|_*(\La)(q): = \sum_{p\in  |\io_{\Hat\bZ_\Hh}|^{-1}(q)} \ \La(p) = 1, \quad \forall q\in X\less \{N,S\}. 
  $$

\NI (v)  The above examples all consider the case when the bundle over the orbifold has dimension zero; for further discussion see \S\ref{ss:orb}.  Example~3.2.12 in \cite{MWiso} builds  Kuranishi atlases that model the tangent bundles  of $S^2$ and the football, and then constructs an explicit perturbation that calculates the  Euler class.  Atlases for more general bundles over orbifolds are considered  in \S\ref{ss:nontriv} below. \hfill$\er$
\end{example}

\begin{example}\rm\label{ex:zorb}    Here is an example from Gromov--Witten theory that further illustrates our choice of weighting function in the noneffective case.
Consider the space $X$ of equivalence classes of genus zero stable maps to $S^2$ of degree $2$ with no marked points.  
An element $[\Si,f]\in X$ is  represented either by a degree $2$ map $f:\Si=S^2\to S^2$ (classified  modulo reparametrization by its two (unordered) critical values) 
or by a map $f:\Si=S^2\sqcup S^2\to S^2$  (classified by the image of the nodal point).    Elements of the first kind form a stratum $\Ss_1$
diffeomorphic to $\qu{\{(x,y)\in S^2\times S^2 | x\ne y\}}{\!\sim}$ where $(x,y)\sim (y,x)$, and one 
can show that elements of the second kind are Gromov limits of sequences in $\Ss_1$ in which the two critical value converge 
to points on the diagonal; e.g. see \cite[Lemma~3.7]{Macs}. Thus, as a space we can consider  $X$ to be homeomorphic to $\CP^2 = \qu{S^2\times S^2}{\sim}$.  However, each element $[\Si,f]$ has a nontrivial order two  symmetry $\ga_f:\Si\to \Si$ such that
$f=f\circ \ga_f$, so that in the current context we should consider $X$ as an orbifold equipped with the identity action of the group $\Ga = \Z_2$.  Thus $[X]^{vir} = \frac 12 [X]$.
To justify this factor of $\frac 12$, consider the space 
$\TY$ of equivalence classes of genus zero, degree $2$ stable maps to $S^2$ that are equipped with three marked points $z_0, z_1,z_2$ with the fixed images $0,1,\infty$.  Since the elements of $\TY$ have trivial isotropy, $\TY$ is a manifold so that $[\TY]^{vir}$ is its usual fundamental class $[\TY]$.
As a map of manifolds, the forgetful map $\phi: \TY\to X$ has degree four:  given a generic element $[\Si, f]\in X$ (i.e. one that does not branch at $\{0,1,\infty\}$) there are two choices of position for
each  of the marked points, and hence eight choices overall, and but only four distinct lifts to $\TY$ because of the action 
of the M\"obius group.
(In fact, \cite[Lemma~3.7]{Macs} shows that $\TY\cong \CP^2$ and identifies the forgetful map  with the branched cover
$\CP^2\to \CP^2, [x,y,z]\mapsto [x^2,y^2,z^2]$.)  On the other hand as a map of orbifolds $\phi:(Y, \id)\to (X,\Ga)$ 
we should consider $\phi$ to have degree $8$, and this is what is reflected in the identity
$\phi_*([\TY]^{vir}) = 8 [X]^{vir}$.\hfill$\er$
\end{example}

\subsection{Stratified smooth atlases}\label{ss:SS}

As we will see in \S\ref{s:GW} in order to build a smooth Kuranishi atlas on a GW moduli space such as $X=\oMm_{0,k}(M,J,A)$ 
we need a smooth version of the gluing theorem that builds a family of curves with smooth domain from one with nodal domain. Even if we ask that the structural maps in $\Kk$ are $\Cc^1$-smooth  rather than $\Cc^\infty$-smooth, this is more than is provided by the simplest gluing theorems such as that in \cite{JHOL}.   On the other hand, in order to get a VFC we do not need the domains $U_I$ of the Kuranishi charts to be smooth manifolds: since all we want is a homology class that we define as the zero set of a transverse perturbation $\s+\nu$, it is enough that $U_I$ is stratified, with smooth strata  and lower strata of codimension at least $2$.\footnote
{
Pardon~\cite{Pard} uses a homological way to define $[X]_\Kk^{vir}$ and hence only needs 
the $U_I$ to be topological manifolds.  Thus a gluing theorem such as that  in \cite{JHOL} suffices.
}
  Here we briefly explain some elements of the approach in \cite{MWgw}.  What we describe is enough to define
  $[X]_\Kk^{vir}$ if this is zero dimensional (therefore with one dimensional cobordisms), and hence enough to calculate all numerical GW invariants; see  \S\ref{ss:var} and Castellano~\cite{Castell2}.
  We begin with some basic definitions.
  
  \begin{defn}\label{def:strat}
A pair $(X,\Tt)$ consisting of  a topological space $X$ together with a finite partially ordered set $(\Tt,\le)$  is called a {\bf stratified space with strata} $(X_T)_{T\in \Tt}$ if:
\begin{enumerate}
\item   $X$ is the disjoint union of the strata, i.e.\  $X = \bigcup_{T\in \Tt} X_T$, where $X_T\ne \emptyset$ for all $T\in \Tt$ and  $X_S\cap X_T = \emptyset$ for all $S\neq T$;
\item the closure of each stratum intersects only deeper strata, i.e.\ $cl(X_T) \subset \bigcup_{S\le T} X_S$.
\end{enumerate} 

\noindent
We denote the induced strict order by $S<T$ iff $S\leq T, S\neq T$.
\end{defn}

There are different stratifications that one could put on a Gromov--Witten moduli space.  However, we will use the most simple, namely that with strata labelled by the number of nodes in the domain of the stable maps.
It follows from the Gromov compactness theorem that this stratification satisfies the above conditions.

\begin{defn}\label{def:stratcts}
A {\bf stratified continuous} map $f:(X,\Tt)\to (Y,\Ss)$ between stratified spaces is a continuous map $f:X\to Y$ that induces a map $f_*:\Tt\to \Ss$ which preserves  strict order. More precisely,
\begin{enumerate}
\item $f_*$ preserves  strict order in the sense that $\; T<S \Rightarrow f_*T < f_*S \;$, and
\item $f$ maps strata into strata in the sense that $f(X_T)\subset Y_{f_*T}$ for all $T\in \Tt$.
\end{enumerate} 
A stratified continuous map $f:(X,\Tt)\to (Y,\Ss)$ is called a {\bf stratified homeomorphism}, 
if $f$ is a homeomorphism and  $f_*$ is bijective.  
In this case the spaces $(X,\Tt),(Y,\Ss)$ are called {\bf  stratified homeomorphic}. 
\end{defn}

\begin{example}\label{ex:RC}\rm   Let $M$ be a smooth $k$-dimensional manifold and let $n\in\N_0$.
The {\bf standard SS space} $M\times \C^{\un}$ 
is the topological space $M\times \C^n$ with the following extra structure:
\begin{itemize}
\item
the stratification $M\times \C^\un=\bigcup_{T\in \Tt^{n}} (M\times \C^\un)_T$,
where $\Tt^{n}$ 
is the set of all (possibly empty) subsets $T\subset \{1,\dots,n\}$, partially ordered by the subset relation, and
whose strata are given by
$$
 (M\times \C^\un)_T : = \bigl\{(x ; \ba)\in M\times \C^n\,\big|\, \ba=(a_1,\dots,a_n)\mbox{ with } a_i\ne 0\Leftrightarrow i\in T\bigr\};
$$
\item 
the smooth structure induced on each stratum by the embedding
 $ (M\times \C^\un)_T \hookrightarrow  M\times \C^{|T|}$, $(x;\ba)\mapsto (x; (a_i)_{i\in T})$.
\end{itemize}
Every subset $U\subset M\times \C^{\un}$ inherits a stratification $\bigl(U_T=U\cap (M\times\C^\un)_T \bigr)_{T\in\Tt_U}$ in the sense of Definition~\ref{def:strat}, which is called the {\bf SS stratification} on $U$. Here we denote by $\Tt_U$  the subset of $T\in\Tt$ for which the stratum $U_T$ is nonempty. \hfill$\er$
\end{example}

Thus $\C^{\und 1}$ is the space $\C$ equipped with the two smooth strata 
$\{0\} = (\C^{\und 1})_{\emptyset}$ and $\C\less \{0\} = (\C^{\und 1})_{\{1\}}$, while if $M = \R^k$ we have  
 the {\bf Euclidean SS space} $\R^k\times \C^{\un}$.
 Note that the (real) codimension of the stratum $(M\times \C^\un)_T\subset M\times \C^\un$ is $2(n-|T|)$ and hence is always even.  We think of the components $a_j$ of $\ba\in\C^n$ as {\bf strata variables}, while the components $x_i$ of a local coordinate system near $x\in M$ are called {\bf smooth variables}.  
This is natural in the context of Gromov--Witten invariants, where the strata variables are
 given by complex gluing parameters at nodes. 
It is also convenient notationally since it allows us to distinguish between the two types of variables.  
However, maps between coordinate charts need not be in any sense holomorphic, and 
will not in general  preserve the distinction between these two types of variable.
 
 \begin{defn}\label{def:SS0} 
Let $f:U\to Y\times \C^{\und m}$ be a stratified continuous map defined on an open subset $U \subset M\times \C^{\un}$.
We call $f$ {\bf weakly stratified smooth} (abbreviated {\bf weakly SS}) if it restricts to 
a smooth map $U_T=U\cap  (M\times \C^\un)_T \to  (Y\times \C^{\und m})_{f_*T}$ on each stratum $T\in\Tt^{n}_U$. 
 A {\bf weakly SS diffeomorphism} is an injective, weakly SS map $\phi:U\to \phi(U)$ with open image and a weakly SS inverse. 
\end{defn}

One can  check that
 the composite of two weakly SS maps is weakly SS.
  Hence it is possible to define the notion of a {\bf weakly SS manifold}, namely
 a stratified, second countable, topological manifold whose transition functions are weakly SS diffeomorphisms between open subsets of  Euclidean SS spaces.   
    Such manifolds (if closed and with oriented top stratum) do have a fundamental class, since the singular set has codimension $\ge 2$.  However, as we now explain, 
 this context is not sufficiently rich to support a good theory of transversality.  

 \begin{defn}\label{def:SS1} 
   If  $f:U\to \R^k $ is weakly SS and $f(w)= 0$ for $w\in U_S$, we  say that  {\bf $f$ is transverse to $0$ at $w$} if 
the derivative $\rd_w^S f: \rT_w (U_S)\to \R^k$ is surjective at $w$, where $\rd_w^S$ is the differential of the smooth restriction of $f$ to the stratum $U_S$.
\end{defn}

\begin{example}\label{ex:SS}\rm 
This transversality  condition is not  open.  
    For example, the weakly SS function $$
    f:\R\times \C^{\und{1}} \to \R,\qquad (x;a)\mapsto \left\{ \begin{array}{ll}  
    x + x^{\frac 13} \sin(\frac 1{|a|}) & \mbox{ if } a\ne 0 \\
    x& \mbox{ if } a= 0\end{array}\right.
    $$
    is 
 transverse  to zero at $(0;0)$ in the above sense, since its restriction to the stratum $a=0$ is transverse.  However its zero set is not a manifold near $(0,0)$ since for $a\ne 0$ it sometimes has $1$ and sometimes $3$ solutions.  \hfill$\er$
 \end{example}

The above example shows that we cannot hope to define
 the VFC as the zero set of a weakly  SS function.  
 We  avoid this problem here by restricting consideration to atlases of dimension $d=0$ with cobordisms of dimension $1$.  
 The above transversality condition then implies that  the zero set lies in the top stratum where all functions are smooth.

 We now briefly explain how to modify the theory of atlases to deal with the case when 
$(X,\Tt)$ is stratified. 
 We say that $\bK_I = (U_I, E_I,\Ga_I,s_I, \psi_I)$ is a {\bf weakly SS  Kuranishi chart} on $X$ if  
 the conditions of Definition~\ref{def:chart2} hold  in the category of weakly SS manifolds and 
 weakly SS diffeomorphisms.  Thus
 $U_I$ is an open subset of a weakly SS manifold,  all  maps are weakly SS diffeomorphisms,
 and the footprint map $\psi: s_I^{-1}(0)\to F_I$ is stratified continuous.
Some care is required with the definition of a coordinate change and the index condition.
In the smooth situation, our definitions imply that the domain $\TU_{IJ}$  is a 
submanifold of $U_J$ cut out transversally by the function $pr_{IJ}\circ s_J: U_J\to \qu{E_J}{\Hat\phi_{IJ}(E_I)}$.  
We make a similar requirement here, but stratum by stratum 
as follows.
\begin{itemize}\item [(a)]
$\TU_{IJ}$ is a weakly SS manifold with weak SS action by $\Ga_J$ such that $\Ga_{J\less I}$ acts freely;
\item[(b)]  $\rho_{IJ}:\TU_{IJ}\to U_{IJ}$ is the composite of the quotient map followed 
by a weakly SS diffeomorphism from $\qu{\TU_{IJ}}{\Ga_{J\less I}}$ to the open subset $U_{IJ}$ of $U_I$,
and the inclusion $\Tphi_{IJ}:\TU_{IJ}\to U_J$ is weakly SS; 
\item[(c)]  
For each stratum $(U_J)_S$ the intersection  
$ \Tphi_{IJ}(\TU_{IJ})\cap (U_J)_S$  is  
cut out transversally by the function $\pr_{IJ}\circ s_J: (U_J)_S\to \qu{E_J}{\Hat\phi_{IJ}(E_I)}$. 
\end{itemize}
Finally, we say that $\Kk = (\bK_I, \Hat\Phi_{IJ})_{I\subset J, I,J\in \Ii_\Kk}$  is a {\bf weakly SS (weak) Kuranishi atlas }
 if all the conditions of Definition~\ref{def:Ku2}  hold in the weakly SS category.
Note that $\Kk$ has a  subatlas $\Kk^{top}$ formed by its top strata, and this carries an orientation bundle $\det(\Kk^{top})$ as before; see  \S\ref{ss:orient}.
We will say that $\Kk$ is oriented if $\det(\Kk^{top})$ is provided with a non-vanishing section.

Our aim is to establish the following theorem.

   \begin{prop}\label{prop:VFCstrat} Let $\Kk$ be an oriented,  $0$-dimensional, 
weak, effective,  weakly SS Kuranishi atlas on a compact metrizable stratified space $(X, \Tt)$. Then $\Kk$ 
determines
a class $[X]^{vir}_\Kk \in \check H_0(X;\Q)$
that 
 depends only on the oriented cobordism class of $\Kk$.
 \end{prop} 

As before,
 $[X]^{vir}_\Kk$ may be  represented by a weighted branched manifold of dimension zero that maps into $|\Kk|$, i.e. by a finite union of oriented, weighted points.
 
 \begin{proof} This result may be proved by 
 making minor changes 
 to the arguments in the smooth case 
 since the strata in $U_I^{sing}: = U_I\less U_I^{top}$ have dimension at least $2$ less than $\dim E_I$.   
Since  the intermediate atlas $\uKk$ is a filtered  topological atlas in the sense of Definitions~\ref{def:tKu}
and~\ref{def:Ku3}, the taming and reduction constructions work as before.  Therefore 
we may assume that $\Kk$ is preshrunk and tame, with reduction $\Vv$.  The only 
question is how to construct a suitable perturbation section $\nu$.

For simplicity, we first go through the construction in the case when
 the isotropy groups are trivial.   In this case, condition (b) above says that the subset $\TU_{IJ} =\phi_{IJ}(U_{IJ})$  of $U_J$ is provided with 
 the structure of a weakly  SS manifold  such that the inclusion $\TU_{IJ}\hookrightarrow U_{J}$ is weakly SS.   In this situation, without further condition on the directions normal to $\TU_{IJ}$,  it may happen that 
%
 weakly SS functions on 
 $\TU_{IJ}$ do not have weakly SS extensions to $U_J$.    Hence we will relax the smoothness assumptions for sections given in Definition~\ref{def:sect} and later requirements on
perturbation sections (both for atlases and cobordism atlases) as follows:
  \begin{itemize}\item each $\nu_I: V_I\to E_I$ is continuous overall, and smooth on the top stratum $V^{top}_I$;
  \item 
   $\nu$  is {\bf admissible} if, for all $I\subsetneq J$, the components $(\nu_J^j)_{j\in J\less I}$ vanish in a neighbourhood of
  $\Tphi_{IJ}(\TU_{IJ})\subset U_J$;
  \item $\nu$  is {\bf transverse}  if $s_I+\nu_I\ne 0$ on $(V_I)_S$ for all strata $S$ except $V_I^{top}$, and if
  $s_I+\nu_I\pitchfork 0$ on $V_I^{top}$.
  \end{itemize}
 The notion of precompactness is unchanged.   With these definitions, 
the statement in Proposition~\ref{prop:zeroS0} that the
realization $|\Hat\bZ^\nu|$ of the zero set of a precompact, admissible, transverse perturbation is a compact 
smooth manifold of dimension $d=0$ makes sense, and can be proved as before.  Moreover, because 
$|\Hat\bZ^\nu|\subset |\Kk^{top}|$, it inherits an orientation from  $ |\Kk^{top}|$. Similarly, a cobordism section with boundary values $\nu^0$ and $\nu^1$
defines an oriented smooth one-dimensional cobordism between $|\Hat\bZ^{\nu^0}|$ and $|\Hat\bZ^{\nu^1}|$.

It remains to construct a suitable perturbation, which we do by following the method of  Proposition~\ref{prop:ext}.
Recall from \eqref{eq:VIk} that we have a family $\bigl(\Vv^k = (V^k_I)\bigr)_{k\ge 0}$ of nested reductions with $\Vv\sqsubset \Vv^k
\sqsubset \Vv^{k-1}$ for all $k$.
At the first step we need to choose a suitable function $\nu_i:V_i^1\to E_i$ for each basic chart $\bK_i$, 
which we may do separately for each $i$.  In this case, because
there is no function $\mu_J$ (since there are no $J$ with $|J|< 1$),  we just need to choose the (suitably small ) perturbation $\nu_i: = \nu_{\pitchfork}$ in such a way  that $s_i + \nu_{\pitchfork} \pitchfork 0$.  
Since $ V_i^{sing}: = V_i^0\less (V_i^0)^{top}$ is a union of manifolds of dimension $\le \dim E_i-2$, we can first choose 
a continuous function 
$\nu_{\pitchfork}:   (V_i^0)^{sing}\to E_i$ so that $s_i + \nu_{\pitchfork}$ does not vanish on this set; then extend 
$\nu_{\pitchfork}$ continuously over $(V_i^0)^{top}$.  Because $V_i^1\sqsubset V_i^0$, there is $\eps>0$ such that
$s_i + \nu_{\pitchfork} \ne 0$  in the neighbourhood 
$B^i_{4\eps}( (V_i^1)^{sing})$ of the singular strata.\footnote
{
Recall that  $B^i_\eps(Q)$ denotes an $\eps$-neighbourhood of the subset $Q\subset U^i$ with respect to the 
metric $d^i$ pulled back from $d$ on $|\Kk|$.}  Now smooth $\nu_{\pitchfork}$ outside 
$ B^i_{\eps}( (V_i^1)^{sing})$ so it still has no zeros in  
$B^i_{3\eps}((V_i^1)^{sing})$.  Finally, use standard transversality theory to perturb $\nu_{\pitchfork}$ smoothly 
outside  $ B^i_{\eps}((V_i^1)^{sing})$
so that all zeros of $s_i + \nu_{\pitchfork}$ both lie outside 
$ B^i_{2\eps}( (V_i^1)^{sing})$  (and hence in the region where $s_i + \nu_{\pitchfork}$ is smooth)
  and  are transverse.   

The rest of the construction proceeds in a similar way.  When $|J|>1$ there are two steps: we must first construct the 
extension $\Tnu_J$ of $\mu_J|_{N_J^{k+1}}$, and then the final perturbation $\nu_{\pitchfork}$ to achieve transversality.
The function $\Tnu_J$  must satisfy conditions (A-D) and also be smooth in a neighbourhood of the part of the zero set
$(s_J + \Tnu_J)^{-1}(0)$ on which transversality is achieved, i.e. on the part of the zero set lying in $B$.
Then we define   $\nu_{\pitchfork}$ to have support in an 
open precompact subset $P\sqsubset W$ where $W$ is a neighbourhood of $\ov{V^{k+1}_J}\less B$. Again we need 
the sum $s + \Tnu_J + \nu_{\pitchfork}$ to be smooth in a neighbourhood of its zero set, which we will arrange as above 
by making it smooth 
outside a very small neighbourhood of the singular set $(\ov{V^{k+1}_J})^{sing}$.

Here are some more details of the construction.
By \eqref{eq:Tmu}, the function $\Tnu_J $ has the form $ \be\cdot \sum_{j\in J} \Tmu^j_k $, where $\be$ is a cut off function
and $\Tmu^j_k$ is a certain extension of the $j$th component of $\mu_J$ that is obtained from local extensions $\Tmu_z$ by 
summing with respect to a partition of unity.  
We can assume inductively  that there are constants $0< \eps'<\eps$ such that for all $I\subset J$ the functions $s_I|_{V_I^{|I|}} + \nu_I$ do not vanish in an $\eps$-neighbourhood of
$(V_I^{|I|})^{sing}$ and that they are smooth outside an $\eps'$-neighbourhood.  Since the coordinate changes are isometries, we
 can therefore choose the local extensions $\Tmu_z$ so that they are smooth  for all $z$ outside a 
 $\eps'$-neighbourhood of $(N^{k+\frac 12}_J)^{sing}$.  Since we can choose the bump function $\be$ and partition of unity to be smooth with respect to some system of local coordinates on $V_J^{k}$, we can construct   $\Tnu_J$ with support in $W_J$ so that,
as well as satisfying conditions (A-D), it  is smooth 
 outside a neighbourhood $B^{\eps''}_J\bigl((N^{k+\frac 12}_J)^{sing}\bigr)$.  Here we may choose   $\eps''$  (with   $\eps'<\eps'' <  \eps$)
sufficiently small that the zeros of  $s_J + \Tnu_J$ in  $B$ (which in fact lie in $N^{k+\frac 12}_J$ by (C)) also lie outside this neighbourhood.
This completes the first step.  

For the second step we need to choose a suitable $\nu_{\pitchfork}$ with support in an 
open precompact subset $P\sqsubset W$.  As before, this set $P$ contains all the zeros of $s_J + \Tnu_J$ that do not lie in $N^{k+\frac 12}_J$.
Moreover we can choose $\nu_{\pitchfork}$ to be any sufficiently small function with support in $P$ 
that achieves the transversality and smoothness conditions.  Hence we can construct $\nu_{\pitchfork}$ by the same method that was used
to construct $\nu_i$.   Thus we first choose a very small continuous function $ \nu_{\pitchfork}$ so that $s_J + \Tnu_J + \nu_{\pitchfork}$ does not vanish
near  $P^{sing} = P\cap V_J^{sing}$, then  smooth $\Tnu_J + \nu_{\pitchfork}$ outside a smaller neighbourhood of $P^{sing}$, and finally perturb $\nu_{\pitchfork}$ to achieve transversality everywhere.

This completes the construction of $\nu$  in the case of trivial isotropy.  One adapts
the construction  in \cite[Proposition~7.3.10]{MW2} of cobordism perturbations in essentially the same way.
When the isotropy is nontrivial,  the argument is also essentially the same; see  the discussion after Definition~\ref{def:sect2}.
For more details, see Castellano~\cite[\S5]{Castell2}.
\end{proof}

%
%


\section{Gromov--Witten atlases}\label{s:GW}

We begin by discussing the proof of Theorem A in the genus zero case.
Consider a closed $2n$-dimensional symplectic manifold $(M,\om)$ with $\om$-tame almost complex structure $J$, and let
 $X=\oMm_{0,k}(M,A,J)$,  the space of
equivalence classes of genus zero, $k$-marked stable maps to $M$ in class $A\in  H_2(M;\Z)$. 
Section~\ref{ss:GW} explains  how to construct a $d$-dimensional  weak Kuranishi atlas on $X$, where 
\begin{equation}\label{eq:indA}
d: = \ind(A) = 2n + 2c_1(A) + 2k-6.
\end{equation}
This atlas is either weakly SS or ($\Cc^1\slash\Cc^\infty$)-smooth, depending on the gluing theorem that we use, see (VII) below, and is unique up to concordance as required by Theorem~A. (In fact, it is unique up to the more restrictive relation of commensurability; see  Definition~\ref{def:commen}.)
Though we describe the construction of a chart  in enough detail to be able to formulate the equation in local coordinates, we do not carry out the necessary analysis; for this, see \cite{Castell1}. 
The construction of  a single chart is straightforward though rather lengthy.   It is trickier to write down exactly 
how to combine the slicing conditions in the construction of a transition chart.  For this it is simplest to adopt the \lq\lq coordinate free" approach of \eqref{eq:coordf1}, but to understand this one does need the notion of an \lq\lq admissible labelling" as in Remark~\ref{rmk:lifts}.  

Section~\S\ref{ss:var} 
 explains variants of the basic method.

We concentrate here on the genus zero case  because the relevant spaces, group actions and equations can be written down very explicitly in terms of the coordinates on the Deligne--Mumford space $\oMm_{0,k}$ provided by cross ratios.   No doubt these arguments can be adapted to the higher genus case, though they would have to be based on a thorough understanding of the relevant
Deligne--Mumford space as in Robbin--Salamon~\cite{RSAL}.
However, one can construct smooth Kuranishi atlases for moduli spaces of $J$-holomorphic curves of any genus by using the ambient polyfold  that is constructed by Hofer--Wysocki--Zehnder in~\cite{HWZ1,HWZ2,HWZ3,HWZ,HWZ:DM}.  This polyfold is an infinite dimensional analog of an ep groupoid, that contains $X=\oMm_{g,k}(M,A,J)$ and is modelled on an sc-Banach space.  It carries a strong polyfold bundle whose fibers are certain $(0,1)$ forms, and one can extract a Kuranishi atlas over $X$ from this groupoid plus bundle by much the same 
construction used in  Proposition~\ref{prop:orb} below  to construct strict atlases for orbifolds.  Full details are worked out in~\cite{MWgw}. 
Of course one might wonder what the advantages of Kuranishi atlases are if one  has to use polyfolds to construct them. 
One point is that one might be able to  get more control over the geometric properties of elements in $|\Kk|$ by making a careful choice of the obstruction bundles.  Another is that the construction of the analog of the perturbation section is very complicated 
in the polyfold world; one certainly has more  control  over its  branching using the techniques described in 
Proposition~\ref{prop:zero} above.  Finally, the existence of this connection between the two approaches should imply that the Gromov--Witten invariants constructed via polyfold theory are the same as those constructed using atlases, and also should serve as a bridge between more traditional geometric methods and the new analytic approach via polyfolds.

\subsection{Construction of charts 
in the genus zero GW setting}\label{ss:GW}

We begin by explaining how to build a basic chart $\bK = (U,E,\Ga,s,\psi)$  near a point $[\Si_0,\bz_0,f_0]
\in X$, where   
 $[\Si_0,\bz_0,f_0]$ denotes the equivalence class of the stable map 
$(\Si_0,\bz_0,f_0)$.   If  the domain $(\Si_0,\bz_0)$  is stable,  the construction is relatively easy since in this case
 $[\Si_0,\bz_0,f_0]$ has no automorphisms and we may take $\Ga = \id$.   As we will see below, we can fix the positions of enough nodal and marked points to determine a parametrization of the domain as a union of spheres (below we call this choice is called a {\bf normalization}), thereby obtaining a trivialization of the universal curve $\Cc\to \oMm_{0,k}$ near   $[\Si_0,\bz_0]$.
   For any smooth map $f:\Si\to M$,  we consider ${\rm graph} f$ to be embedded in the product $\Cc\times M$, and, as in \eqref{eq:linE},  choose  a linear map $\la$
from $E$ to the space of sections of a bundle over $\Cc\times M$ whose fibers consist of appropriate $(0,1)$-forms.
Assuming that the image of $\la$ is sufficiently large to cover the cokernel of the linearized  Cauchy--Riemann operator at $f_0$,
we then take
$U$ to consist of tuples $(e,f,\bz)$ where $e\in E$ is a vector and $f\approx f_0:\Si\to M$ is a solution to the perturbed Cauchy--Riemann equation $\pbar f = \la(e)|_{\rm graph f}$: see \eqref{eq:delbar}.   Even in this simple case there is still the nontrivial question of putting an appropriate smooth structure on $U$.  To do this requires a gluing theorem, as we discuss further in Remark~\ref{rmk:glueXX}.

In general, before starting the above construction we must stabilize the domain of $[\Si_0,\bz_0,f_0]$ by adding some extra marked points, and then investigate the action of the automorphism group $\Ga$ in terms of these marked points.    
In the language developed in  \cite[Appendix~D]{JHOL}, we can think of the
 domain  $\Si_0$ as a connected finite union of standard spheres  whose combinatorics can be described in terms of a labelled tree $(T,\Ee,\al)$ with vertices $T$ and symmetric edge  relation $\Ee\subset T\times T$, and where the labelling map $\al:\{1,\dots,k\}\to T, i\mapsto \al_i$ determines which sphere contains the $i$th marked point.
 Thus $\Si_0$ is the union of spheres 
 $(S^2)_{\al\in T}$ joined at nodal pairs 
$$
(S^2)_\al \ni n^0_{\al\be} = n^0_{\be\al}\in (S^2)_\be, \quad  \al\sim_\Ee \be
$$
   with the marked point $z_0^i$ lying on the component $(S^2)_{\al_i}\less \{z_{\al_i\be}: \al_i\sim_\Ee \be\}$.
Further the map
$$f_0: \Si_0= \bigcup_{\al\in T} (S^2)_{\al}\to M
$$ satisfies  $f_0(n^0_{\al\be}) = f_0(n^0_{\be\al})$.
 For short, the nodal points are denoted $\bn_0 : = (n^0_{\al\be})_{\al \sim_\Ee\be}$, so that we may identify
 $(\Si_0,\bz_0,f_0)$ with the tuple $\bigl(\bn_0,\bz_0, f_0\bigr)$.  Two such tuples $\bigl(\bn_0,\bz_0, f_0\bigr), \bigl(\bn_0',\bz_0', f_0'\bigr)$ modelled respectively on the labelled trees $(T,\Ee,\al), (T',\Ee',\al')$ are called equivalent (and define the same element $[\Si_0,\bz_0,f_0]\in \oMm(M,A,J)$) if there is a tree isomorphism $\io:T\to T'$  and a function $T\to \PSL(2,\C), \al\mapsto \phi_\al$
such that
\begin{equation}\label{eq:treeaut}
f_{\io(\al)} \circ \phi_\al = f_\al,  \quad n'_{\io(\al)\io(\be)} = \phi_\al (n_{\al\be}),\quad z_i' = \phi_{\al_i}(z_i),
\end{equation}
for all indices $\al, \be, i$: see \cite[Definition~5.1.4]{JHOL}.

Quite a few choices are involved in constructing the chart; we list the main ones here. 

\MS

\NI  {\bf (I): The added marked points.}
The chart is  determined by the choice of a  {\bf slicing manifold} $Q$,  which is an orientable  codimension $2$ 
(open, possibly disconnected) submanifold of $M$ that is transversal to  $\im f_0$ and meets it in enough points 
$f_0^{-1}(Q) = 
 \{w_0^1,\dots,w_0^L\} = :\bw_0$  to stabilize its domain, i.e. so that there are at least three special points (nodal or marked) on each component.
 We assume the points $w_0^\ell$ are disjoint from $\bz_0\cup\bn_0$.
(If $[\Si_0,\bz_0]$ is already stable there is no need to add these points.  In this case we allow $\bw_0$ to
be  the empty tuple.)   
 Since $(\Si_0, \bw_0,\bz_0)$ is stable, it is described up to biholomorphism by its tuple of  special points.  Thus we may write  
the element $\de_0\in  \oMm_{0,k+L}$ either as $[\bn_0,\bw_0,\bz_0]$ or as $ [\Si_0,\bw_0,\bz_0], $  and will work over a suitable neighbourhood $\De$ of $\de_0$ in  $  \oMm_{0,k+L}$.

 \MS
 
 \NI
 {\bf (II): The group.}
We take $\Ga$ to be the stabilizer subgroup  of 
 $[\Si_0,\bz_0,f_0]$, so that each $\ga\in \Ga$ acts on $\Si_0$ by a biholomorphism ${\ga}:\Si_0\to \Si_0$ of the form detailed in \eqref{eq:treeaut}, 
 permuting  the points in $\bw_0$ (and hence also  $T$  and $\bn_0$)
 while fixing those in $\bz$ and leaving $f_0$ unchanged:  
 $f_0= f_0\circ \ga$.  
 We  therefore consider $\Ga$ to be a subgroup of $S_L$, the symmetric group on $L$ letters, acting via\footnote
 {
 We could take $\Ga$ to be any subgroup of  $S_L$ that contains the isotropy group, but this complicates the description of the action given in \eqref{eq:phiga} below.}
\begin{equation}\label{eq:action}
 \bw_0 \mapsto \ga\cdot \bw_0 : = (w^{\ga(\ell)})_{1\le \ell\le L}.
\end{equation}
Note that $(\ga_1\ga_2)\cdot \bw = \ga_2\cdot(\ga_1\cdot \bw)$.
Further, the induced action $\al\mapsto \ga(\al)$ on the tree $T$ has the property that
$$
w_0^\ell\in (S^2)_\al\;  \Longrightarrow\;  
 (\ga\cdot \bw_0)^\ell  = w^{\ga(\ell)}\in (S^2)_{\ga(\al)}.
 $$
 Correspondingly we define a $\Ga$ action on $\bn$ by
 $(\ga\cdot\bn)_{\al\be}: = n_{\ga(\al)\ga(\be)}$.
We choose small disjoint neighbourhoods 
 $(D_0^\ell)_{\ell=1,\dots,L}\subset \Si_0\less (\bz_0\cup \nodes)$ of the $(w_0^\ell)$, averaging them over   
the $\Ga$-action, so that $\Ga$ acts on them by permutation.
Later we will use these discs to control  the added marked points, specially those  in a transition chart. 
\MS

\NI {\bf (III): The normalization conditions and  universal curve.}
The above description of $\Si_0$ in terms of its nodal points $\bn_0$  is not  
unique since we have not yet explained how to identify each component of $\Si_0$ with a sphere,
a step that is needed in order to write down the equation satisfied by the elements in the domain  of the chart.
To this end,  we  
fix the positions of three of the special points $\bn_0\cup \bw_0 \cup \bz_0$ on each component to be $0,1,\infty$.
Thus we choose an injective function  
\begin{equation}\label{eq:bP}
\bP: T\times \{0,1,\infty\}\to \;\{(\al,\be)\ |\ {\al E\be}\}) \cup \{1,\dots,L\}\cup \{1,\dots,k\}
\end{equation}
 that takes
$\{\al\}\times \{0,1,\infty\}$ to three labels for  points in $(S^2)_\al$.
We denote the set of points with labels in $\im(\bP)$ by  $(\bn_0)_\bP\cup (\bw_0)_{\bP}\cup (\bz_0)_{\bP}$,
and write $p_n, p_w,p_z$ for the number of points of each type.  Thus  $3|T|  = p_n+p_w+p_z$.  
We then  parametrize $\Si_0$ by identifying the collection of points with labels in $\bP$ with the corresponding fixed positions on the standard sphere $S^2$, 
 denoting this parametrization of $\Si_0$ by $\Si_{\bP,0}$.    Note that this normalization $\bP$ does not
uniquely determine the domain $\Si_0$ up to biholomorphism since the positions of the nodal points in
 $\bn_0\less (\bn_0)_{\bP}$ must still be specified.  
If we want, we can reduce this indeterminacy by putting nodal labels into $\im(\bP)$ wherever possible, but we cannot always eliminate it; see  Figure~\ref{fig:5}.  
Thus as a stable curve, the tuple $(\Si_{\bP,0}, (\bw_0)_{\bP}, (\bz_0)_{\bP})$ represents
the element  $\de_{\bP,0}: = [\bn_0, (\bw_0)_{\bP}, (\bz_0)_{\bP}]\in \oMm_{0,p}$, where $p: = p_w+p_z$.
We denote the 
nearby
elements in $\oMm_{0,p}$ by  
$\de_{\bP}: = [\bn, \bw_{\bP}, \bz_{\bP}]$, reserving the name $\de$ to denote stable curves
$[\bn, \bw, \bz]\in \oMm_{0,k+L}$.

   \begin{figure}[htbp] 
   \centering
   \includegraphics[width=2.5in]{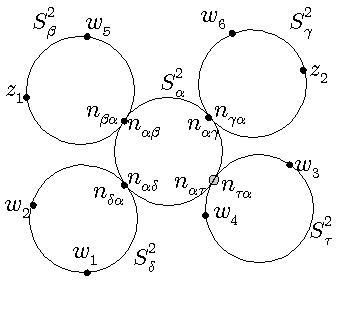} 
   \caption{Here all the nodes and special points are in $\bP$ except for $n_{\al\tau}$.  When the node joining $(S^2)_\al$ to $(S^2)_\de$ is resolved,
   two points of $\bP$ are removed, namely the nodal pair, so that 
   the new glued component contains $4$ points of $\bP$,
 one more  than is needed  stabilize it.  This extra point records the gluing parameter.  
 For example,  if  we  fix the parametrization of the new glued sphere by $n_{\al\be},n_{\al\ga}, w_1$,  the gluing parameter is 
    determined by
   the position of $w_2$, i.e. by the cross ratio $cr(n_{\al\be},n_{\al\ga}, w_1,w_2)$.  On the other hand, if we resolve at $n_{\al\tau}$, then we lose one point of $\bP$, and we can take the cross ratio $cr(n_{\al\be}, n_{\al\ga}, n_{\al\de}, w_3)$ to parametrize the position of the nodal point $n_{\al\tau}$, while
   $cr(n_{\al\be}, n_{\al\ga}, n_{\al\de}, w_4)$ gives the gluing parameter. Similarly, if we resolve at a node with neither point in $\bP$ then, after gluing, the three points in $\bP$ not needed for stability  parametrize the positions of the two nodal points and the gluing parameter.
   }
   \label{fig:5}
\end{figure}

We now discuss the structure of a neighbourhood $\De_\bP$ of $\de_{\bP,0}$ 
 in the Deligne--Mumford space $\oMm_{0,p}$.  We  denote the universal curve  over $\De_\bP$ by $\Cc|_{\De_\bP}$ with fibers $\Si_{\de_\bP}, \de_\bP\in \De_\bP.$  
 A normalized representation of the surface $\Si_{\de_\bP}$ may be obtained from $\Si_{\bP,0}$ by varying the positions of the  nodal points not in $\im(\bP)$ and then gluing. More precisely, if  $\Si_{\bP,0}$ has $K$ nodes, then there are $2K$ nodal points $\bn_0$,
 $2K-p_n$ of which can move,
 and $K$ small complex
 gluing parameters  $\ba: = (a_1,\dots,a_K)$, one at each node, such that all nearby fibers $\Si_{\bP,\ba,\bb}$ may be obtained from
$\Si_{\bP,0}$ by first varying the $2K-p_n$ points in $\bn_0\less \bP$ via  complex parameters denoted $\bb = (b_1,\dots,b_{2K-p_n})$ and then cutting out discs of radius $\sqrt{|a_i|}$ near the $i$th pair of nodal points,  gluing the boundaries of these discs with the twist $\arg(a_i)$.\footnote
{These choices allow one to identify the gluing parameter with a cross-ratio: see \cite[equation~(10.9.3)]{JHOL}.}
  We suppose $|a_i|, |b_j|<\eps$, where
 $\eps>0$ is chosen so that the union  $\Nn^{2\eps}_{\nodes}$ of the  $2\eps$-discs around the nodes of $\Si_{\bP,0}$ does not intersect the discs 
$D_0^\ell$ or the marked  points $\bz_0$.    
Thus, for some small neighborhood $B^{6K-2p_n}$ of $0$  in $\C^{3K-p_n}$, we have a fiberwise embedding
\begin{eqnarray}\label{eq:gluedom}
 \io_\bP &:& 
\bigl(\Si_{\bP,0}\less \Nn^{2\eps}_\nodes\bigr)\times B^{6K-2p_n} \to \Cc|_{\De_{\bP}}, \mbox{ where }  \\ \notag
\io_{\bP,\ba,\bb} &:& 
\bigl(\Si_{\bP,0}\less \Nn^{2\eps}_\nodes 
\bigr)\times \{\ba,\bb\}\; \mapsto\; \Si_{\bP,\ba,\bb}\less \Nn^{2\eps}_\nodes ,
\end{eqnarray}
Thus  $\io_{\bP,\ba,\bb}$  takes the $p=p_w+p_z$ marked points in 
$\Si_{\bP,0}\less \Nn^{2\eps}_\nodes$ to the marked points $\bw_\bP, \bz_\bP$ in 
the fiber $\Si_{\bP,\ba,\bb}$, and the discs  $\bigcup_\ell D_0^\ell\subset \Si_{\bP,0}\less \Nn^\eps_\nodes$ to corresponding discs in 
$\Si_{\bP,\ba,\bb}$.  For each $\ba, \bb$ the injection $\io_{\bP,\ba,\bb}$ is defined on the subset of $\Si_{\bP,0}$ that is not cut out 
by the gluing, i.e. on $\bigcup_\al\bigl((S^2)_\al\less \bigcup_\be D_{n_{\al\be}}(|a_{\al\be}| + |b_{\al\be}|)\bigr)$, where $a_{\al\be}, b_{\al\be}\in \C$ are the relevant small parameters $\ba,\bb$ at the nodal point $n_{\al\be}$.

\begin{rmk}\label{rmk:norm}\rm  \begin{itemlist}\item[(i)]
These coordinates $(\ba,\bb)$ for the neighbourhood $\De_\bP\subset \oMm_{0,p}$ are given by
 the positions  of the free nodes (parametrized by $\bb$) and the gluing parameters $\ba$, and are the most convenient ones in which to write down the equation; cf. (VI). 
In order to understand the group action  it is helpful to note that one can read off the parameters $\ba,\bb$ from the (extended) cross ratios\footnote
{
In this extension we allow at most pairs of points to coincide, so that $cr$ may equal $0,1,\infty$.
The presence of such special values signals the existence of a node, and the resulting  combinatorics gives the tree.}
 of the points 
 $\bw_\bP, \bz_\bP$ in the fiber $\Si_{\bP,\ba,\bb}$; see  Figure~\ref{fig:5} and \cite[Appendix~D]{JHOL}.  Hence 
 we can write down the group action in terms of the induced permutation of the special points as in \eqref{eq:phiga} below.
 \item[(ii)]
We will often denote the normalized domain of  the stable curve
$\de_\bP: = [\bn,  \bw_\bP, \bz_\bP] = [\Si_{\de_\bP}, \bw_\bP, \bz_\bP] $
 as  $\Si_{\bP,\de_\bP}$ instead of $  \Si_{\bP,\ba,\bb}$.  Thus
 $$
 \Si_{\bP,\de_\bP}: =  \Si_{\bP,\ba,\bb}.
 $$
\item[(iii)]   As a check on dimensions, note that
 $\dim_\C(\oMm_{0,p}) = p - 3$, while there are 
 $3K-p_n$ parameters $\ba,\bb$, and $p + p_n = 3K+3$ by definition of $\bP$, so that the total number of parameters $\ba,\bb$ is $p-3$.
 Note also that the normalization $\bP$ of the central fiber $\Si_0$ labels enough points to 
 normalize the nearby fibers, since  one needs three fewer points in $\bP$
 for  
 each node that is glued. 
  As illustrated in Figure~\ref{fig:5}, some points in $\bP$ may be cut out by a 
  gluing, but the extra elements in $\bP$ can always be interpreted  in terms of  gluing parameters $\ba$
 and the parameters $\bb$ pertaining to the nodes that have been glued.  (This point is discussed more fully in (VIII)[b].) $\hfill\er$ 
 \end{itemlist}
  \end{rmk}
%
%

Now consider the stable curves $\de: = [\Si_\de,\bw,\bz]$ with the full set of marked points.
If we  parametrize  the domain  
as $\Si_{\bP,\de_\bP}$, then the points in $ \bw_{\bP}, \bz_{\bP}$ have fixed positions while
the other marked points (as well as the nodes not in $\bn_\bP$) can move.
 The map $\io_\bP$ in \eqref{eq:gluedom} therefore extends to a parametrization of the universal curve $\Cc|_{\De}$
 away from the nodes:  
\begin{eqnarray}\label{eq:gluedom1}
 \io_\bP :
\bigl(\Si_{\bP,0}\less \Nn^\eps_\nodes\bigr)\times B^{6K-2p_n}\times B^{2(k+L-p)}\; \to \;\Cc|_{\De}: 
 \end{eqnarray}
 where the small parameters  $\om^\ell,\zeta^j \in B^{2(k+L-p)}\subset \C^{k+L-p}$ describe the positions of the points in 
 $\io_{\bP,\ba,\bb}^{-1}(\bw\cup\bz)\less (\bw_\bP\cup \bz_\bP)$, taking the value $0$ at $\bw_0,\bz_0$.
 The map 
$\oMm_{0,k+L}\to  \oMm_{0,p}$ that forgets the points in $(\bw\cup\bz)\less (\bw_\bP\cup \bz_\bP)$ lifts to a forgetful map
 $\forget: \Cc|_\De\to \Cc_{\De_\bP}$ that fits into the following commutative diagram
\begin{equation}\label{eq:gluedom2}
\xymatrix{
\bigl(\Si_{\bP,0}\less \Nn^\eps_\nodes\bigr)\times B^{6K-2p_n}\times B^{2(k+L-p)}  \ar@{->}[d]_{\proj} \ar@{->}[r]^{\qquad\qquad\qquad\qquad\io_\bP}    & \Cc|_{\De} \ar@{->}[d]_{\forget}   \\
\bigl(\Si_{\bP,0}\less \Nn^\eps_\nodes\bigr)\times B^{6K-2p_n}\ar@{->}[r]^{\qquad\qquad \io_\bP}   & \Cc|_{\De_\bP}.
}
  \end{equation}
We will denote the element $\de\in \De$ as $\de: = [\Si_\de, \bw, \bz]=[\bn, \bw,\bz]$, with chosen
 representative denoted either  $(\Si_{\bP,\de},\bw,\bz)$ or $(\Si_{\bP,\ba,\bb},\bw,\bz)$.
Here $\bw, \bz$ are tuples of points in the curve $\Si_{\bP,\de}=\Si_{\bP,\ba,\bb}$; their pullbacks by
$\io_{\bP,\ba,\bb}$ to the fixed fiber $\Si_{\bP,0}$ are given by the complex parameters $\vec \om, \vec \zeta$,
 that we assume to vanish at $\bw_0,\bz_0$ and have length $<\eps$ so that  
 \begin{equation}\label{eq:omell}
 \om^\ell: =\io_{\bP,\ba,\bb}^{-1}(w^\ell) \in D^\ell_0.
 \end{equation}

\MS

\NI{\bf (IV): The group action.}  Since $\Ga$ is the stabilizer of $[\Si_0,\bz_0,f_0]$ and acts on the added marked points $\bw_0$  by permutation, with an associated action on the nodes, this  action  extends to a neighbourhood of   $[\Si_0,\bw_0,\bz_0]$.  Hence we
may assume that $\De$ is invariant under this action $\de\mapsto \ga^*(\de)$ of $\Ga$, 
where
 $\ga^*(\de)= [\ga\cdot\bn,\ga\cdot \bw, \bz]=: [\bn', \bw{\hspace{0.002in}}',\bz]$ as 
 in \eqref{eq:action} ff.
 Correspondingly there is an action
 $[\bn,\bw,\bz, f]\mapsto  [\ga\cdot\bn,\ga\cdot \bw, \bz, f] = [\bn',\bw{\hspace{0.002in}}',\bz,f]$ on the space of stable maps.
To obtain an explicit formula for this action, we normalize the domains via the labelling $\bP$.
We may assume that $f$ is defined on the normalized domain $\Si_{\bP,\de}$.  However,
$\Si_{\bP,\de}\ne \Si_{\bP,\ga^*(\de)}$ since in $\Si_{\bP,\ga^*(\de)}$ the points whose {\it new}
 labels are in $\bP$
 are put in standard position.    Therefore the normalized action may be written as
\begin{equation}\label{eq:actionf}
 \bigl(\Si_{\bP,\de},\bw,\bz, f\bigr)\mapsto \bigl(\Si_{\bP,\ga^*(\de)}, \phi_{\ga,\de}^{-1}(\ga\cdot\bw),
 \phi_{\ga,\de}^{-1}(\bz),f\circ \phi_{\ga,\de}\bigr)
\end{equation}
 where 
 \begin{equation}\label{eq:phiga0}
 \phi_{\ga,\de}:\Si_{\bP,\ga^*(\de)}\to \Si_{\bP,\de}
 \end{equation}
  is defined to be 
 the unique  biholomorphic map  
 that takes the special points $\bn',\bw{\hspace{0.002in}}',\bz' $  in $ \Si_{\bP,\ga^*(\de)}$ with labels in $\im(\bP)$ (that are in standard position)  to the corresponding points in $\Si_{\de}$, i.e.   the map
 $ \phi_{\ga,\de}$ takes
\begin{equation}\label{eq:phiga}
 n_{\al\be}'\mapsto n_{\ga(\al)\ga(\be)},\;\;
 (w') ^\ell \mapsto  w^{\ga(\ell)},\;\; (z')^i \mapsto z^i\; \;\mbox{ if }  (\al,\be), \ \ell, \ i\in \im(\bP). 
\end{equation}
The positions of the other special points in $\bn',\bw{\hspace{0.002in}}',\bz' $ in $ \Si_{\bP,\ga^*(\de)}$ are then determined by 
\eqref{eq:actionf}; in particular, $\bw{\hspace{0.002in}}' =  \phi_{\ga,\de}^{-1}(\ga\cdot\bw)$ and $\bz'=
 \phi_{\ga,\de}^{-1}(\bz)$  as claimed in \eqref{eq:actionf}.

 One can also pull these maps $\phi_{\ga,\de}$ back to partially defined maps $\phi_{\bP,\ga,\de}$ on the fixed surface
 $ \Si_{\bP,0}\less \Nn^{2\eps}_{\nodes} $,  as follows:
%
 \begin{eqnarray}\label{eq:gastab}&& \qquad\qquad  \phi_{\bP,\ga,\de}\ :  =   \\ \notag
 && \Si_{\bP,0}\less \Nn^{2\eps}_{\nodes}  \stackrel{ \io_{\bP}(\cdot,\ba',\bb')}\longrightarrow \Si_{\bP,\ba',\bb'} = \Si_{\bP,\ga^*(\de)}
  \stackrel{\phi_{\ga,\de}}\longrightarrow \Si_{\bP,\ba,\bb} = \Si_{\bP,\de}  \stackrel{ \io_{\bP}(\cdot,\ba,\bb)^{-1}}\longrightarrow \Si_{\bP,0}\less \Nn^\eps_{\nodes}.
 \end{eqnarray}
Then $\phi_{\bP,\ga,\de}$ is almost equal to $\ga:  \Si_{\bP,0}\to  \Si_{\bP,0}$, because 
 the inverse image by $\io_{\bP}(\cdot,\ba,\bb)$ of $w^{\ga(\ell)} \in \Si_{\bP,\de}$ is close to $w_0^{\ga(\ell)}$
 while $\io_{\bP}(\cdot,\ba',\bb')^{-1}\circ  \phi_{\ga,\de}^{-1}(w^{\ga(\ell)})$ is close to $w_0^\ell$.

\begin{figure}[htbp] 
   \centering
   \includegraphics[width=3in]{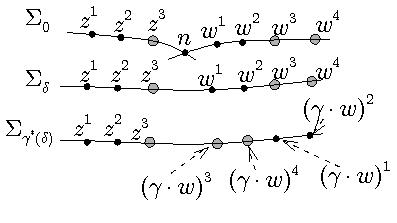} 
   \caption{The points in $\bP$ are the two nodal points plus $z^1,z^2,w^1,w^2$ with solid dots.  These are shown on $\Si_0$ and 
   $\Si_\de$. The group $\Ga = \Z/2\Z$ interchanges $w^1, w^3$ and $w^2, w^4$, so that $\Si_{\ga^*(\de)}$ has
    the same marked points as $\Si_\de$, but with different labels.  Hence it is normalized differently, using the points labelled $w^3, w^4$ on $\Si_\de$ instead of $w^1, w^2$.
  Usually  $cr(z^1,z^2,w^1,w^2)\ne cr(z^1,z^2,(\ga\cdot w)^1,(\ga\cdot w)^2)$, so that  the normalized 
  domains $\Si_{\bP,\de}$ and $\Si_{\bP,\ga^*(\de)}$ 
are obtained from $\Si_0$ by  gluing with  different parameters.  We have drawn the figure so that the map $\phi_{\ga,\bP,\de}:\Si_{\bP,\ga^*(\de)}\to
  \Si_{\bP,\de}$ identifies points vertically, taking $(\ga\cdot w)^1$ to $w^3$ and so on.
 }
   \label{fig:6}
\end{figure}

Because 
the permutation action $\bw\mapsto \ga\cdot \bw$ satisfies
$(\al\ga)\cdot \bw = \ga\cdot(\al\cdot w)$, we have
$(\al\ga)^*(\de) = [(\al\ga)\cdot \bw,\bz] = \ga^*(\al^*(\de))$.
Hence the composite
$ \phi_{\al,\de} \circ  \phi_{\ga,\al^*\de}$ is defined and maps from
$\Si_{\bP,\ga^*(\al^*(\de))}$ through $\Si_{\bP,\al^*(\de)}$ to $\Si_{\bP,\de}$.
It follows easily that
  \begin{equation}\label{eq:phiga2}
\phi_{\al\ga, \de} = \phi_{\al,\de} \circ  \phi_{\ga,\al^*\de}:\Si_{\bP,(\al\ga)^*(\de)}\to \Si_{\bP,\de}.
\end{equation}
Figure~\ref{fig:6} explains this action in a case in which  there is a trivial induced action of $\Ga$ on the set of components of $\Si_0$ and hence on the nodes.

 Using the map $\io_\bP$ in \eqref{eq:gluedom1},
we can push the discs  $(D_0^\ell)_\ell$ in (I) forward to subsets of $\Si_{\bP,\de}$, and then average them over the $\Ga$-action in $\De_\de$
to obtain discs 
\begin{equation}\label{eq:discsde}
(D_\de^\ell)_\ell \subset \Si_{\bP,\de}\;\mbox{ where } \; 
\phi_{\ga,\de}(D_{\ga^*(\de)}^{\ga(\ell)}) = D_\de^\ell \;\;\forall \de,\ell.
\end{equation}

\MS

\NI {\bf (V): The obstruction space.}
Consider the complex vector  bundle $\Hom_J^{0,1}(\Cc|_{\De}\times M)$ whose fiber at $(z, x)$ is $\Hom_J^{0,1}(T_z \Si_{\bP,\de}, T_x M)$ in normalized coordinates. 
%
%
Choose a complex vector space $E_0$ and a (not necessarily injective) complex  linear map $\la: E_0\to \Cc^{\infty}\bigl(\Hom_J^{0,1}(\Cc|_{\De}\times M)\bigr)$ whose image consists of sections that vanish near the nodal points of the fibers.
More precisely, the sections should be supported in the image of the embedding $\io_\bP$ of \eqref{eq:gluedom1}.
Define $E: = \prod_{\ga\in \Ga} E_0$, the product of $|\Ga|$ copies of $E_0$ with elements $\vec e: = (e^\ga)_{\ga\in \Ga}$, on which $\Ga$ acts by permutation so that
$ \bigl(\al\cdot \vec e\bigr)^\ga= e^{\al\ga}$ for $\al\in \Ga$.
Then extend $\la$ equivariantly to a linear map 
\begin{equation}\label{eq:linE}
\la: E\to \Cc^{\infty}\bigl(\Hom_J^{0,1}(\Cc|_{\De}\times M)\bigr), \qquad 
\vec e: = (e^\ga)_{\ga\in \Ga}\mapsto {\textstyle \sum_{\ga\in \Ga}}\; \ga^*(\la(e^\ga)).
\end{equation}
Here we use the fact that the isotropy group $\Ga$ acts fiberwise on $\Cc|_{\De}$ as explained in (IV), taking the fiber  $\Si_{\bP,\de}$ (with relabelled marked points $\bw$)   to the fiber
$\Si_{\bP,\ga^*(\de)}$ by a map that in normalized coordinates is $(\phi_{\ga,\de})^{-1}$; see  \eqref{eq:actionf},  \eqref{eq:phiga}.  The induced action of $\Ga$ on a section $\nu\in   \Cc^{\infty}\bigl(\Hom_J^{0,1}(\Cc|_{\De}\times M)\bigr)$ is by pullback as follows: for $z\in \Si_{\bP,\de}$ we have
\begin{equation}\label{eq:ga*}
\ga^*(\nu)(z,x): =(\phi_{\ga,\de}^{-1})^*(\nu)(z,x) = \nu(\phi_{\ga,\de}^{-1}(z),x) \circ \rd \phi_{\ga,\de} ^{-1}(z) : \rT_z \Si_{\bP,\de}\to \rT_x M.
\end{equation} 
 It follows from \eqref{eq:phiga2} that
$(\ga\al)^*(\nu) = \ga^*(\al^*(\nu))$.

There is quite a bit of choice for the space $E_0$.  For example, we could ask that it is the pullback via \eqref{eq:gluedom2} of a space of sections of $\Hom_J^{0,1}(\Cc|_{\De_\bP}\times M)$.  However, we do need $E$ to 
consist of sections over $\De$ in order for it to support a $\Ga$-action. Later
we will require that $E_0$ is chosen so that  the linearized Cauchy--Riemann operator is surjective; more precisely that  condition (*) in (VI) holds.  

\MS

\NI {\bf (VI): The equation.}
  The elements of the domain $U$ of a basic chart  near the point $[\Si_0, \bz_0, f_0]\in X$
have the form
$(\vec e,\ba,\bb, \vec \om,\vec \zeta, f)$, where:
\begin{itemize}\item[(i)] $\vec e\in E$, for $E$  chosen sufficiently large as specified below;
\item[(ii)]  the (small) parameters $\ba,\bb$ determine the normalized domain $\Si_{\bP,\ba,\bb}$
while the (small) parameters
$ \vec \om, \vec \zeta$ describe the positions  in $\Si_{\bP,0}\less \Nn^\eps_{\nodes}$
of the points  
$\io_{\bP,\ba,\bb}^{-1}\bigl((\bw\less \bw_{\bP})\cup (\bz\less \bz_\bP)\bigr)$ as in the discussion after 
\eqref{eq:gluedom1} and \eqref{eq:omell}; in particular,
\begin{itemize}\item   the tuple  $\ba,\bb, \vec \om,\vec \zeta$ determines a unique
fiber $\de: = [\Si_{\bP,\ba,\bb}, \bw,\bz]$ in $\Cc|_\De$ whose underlying surface
we call  either $\Si_{\bP,\ba,\bb}$ or $\Si_{\bP,\de}$;
\item we suppose $\vec \om$ so small that $w^\ell\in D^\ell_\de$ , where  $ D^\ell_\de$ is as in \eqref{eq:discsde};
\end{itemize}
\item[(iii)]  the map  $f: \Si_{\bP,\ba,\bb}\to M$ represents the class $A\in H_2(M)$ and
  is a solution of the equation
\begin{equation}\label{eq:delbar}
\pbar (f) = \la(\vec e)|_{{\rm graph } f} : = {\textstyle \sum_{\ga\in \Ga}}\; \ga^*\bigl(\la(e^\ga)\bigr)|_{{\rm graph } f}.
\end{equation}
where $\ga^*(\la)$ is defined as in \eqref{eq:ga*},with
$\de$ as in (ii).
\end{itemize}
The solution set of this equation is the zero set of the section
\begin{eqnarray}\label{eq:delbarF}
&& \quad F: E\times  B^{2(L+k-p)}\times W^{1,p}(\Cc|_\De, M) \to L^p \bigl(\Hom^{0,1}_J(\Cc|_\De\times M)\bigr),\\ \notag
&&\qquad\qquad \quad 
(\vec e,\ba,\bb,\vec \om,\vec \zeta, f)\mapsto \pbar f - \la(\vec e)|_{{\rm graph } f} \in L^p \bigl(\Hom^{0,1}_J(\Si_{\bP,\ba,\bb}\times M)\bigr),
\end{eqnarray}
where the domain $W^{1,p}(\Cc|_\De, M)$ is the Sobolev space of $(1,p)$ maps from the fibers of $\Cc|_\De$ to $M$, and 
the range consists of $L^p$ sections of the bundle considered in (V).
If we fix $\ba,\bb$ so that the domain $\Si_{\bP,\ba,\bb}$ of $f$ is fixed,  then the operator $F$ is $\Cc^1$
because  \eqref{eq:ga*} shows that $ \la(\vec e)$  is a sum of terms
\begin{equation}\label{eq:ga1}
z\mapsto \ga^*\bigl(\la(e^\ga)\bigr)|_{(z,f(z))} = \la(e^\ga)(\phi_{\ga,\de}^{-1}(z), f(z)) \circ \rd_z \phi_{\ga,\de}^{-1}
\end{equation}
where  $\phi_{\ga,\de}$ does {\it not} depend explicitly on $f$ but just on the parameters  $\ba,\bb$ (which we have fixed) and on
 $\vec \om,\vec \zeta$.  (We allow the  points  in $(\bw\cup\bz)\less(\bw_\bP\cup\bz_\bP)$  to vary freely until we have solved the equation.)  Hence\footnote
{
See \cite[\S3,4]{MW2} for the analytic details, and also Remark~\ref{rmk:anal}.
}
 the operator has a linearization $\rd F$.
Consider 
   the restriction $F_0$  of $F$
  to a neighbourhood of $\vec 0\times f_0$ in
  the space $E\times  W^{1,p} (\Si_{\bP,0})$ 
 of tuples  with the fixed domain 
  $\Si_{\bP,0}$.  Then $F_0(\vec e,f) = \pbar f - \la(\vec e)|_{{\rm graph } f}$.
 It follows that   
$$
\rd_{(\vec 0,f_0)} F_0 (\xi, \vec e)  = \rd_{f_0} (\pbar)(\xi) - \la(\vec e)|_{{\rm graph }f_0},
$$ where  
\begin{equation}\label{eq:lindbar}
\rd_{f_0}( \pbar):\Dd_0\to \prod_{\al\in T} L^p(\Hom^{0,1}_J((S^2)_{\al}, f_{0,\al}^*(\rT M))
\end{equation}
 has domain\footnote
 {
 Here we assume that the domain $\Si_0$ is connected, i.e. we identify the different components at the nodal points,
so that tangent vectors must satisfy  $\xi_\al(n_{\al\be})= 
\xi_\be(n_{\be\al}) $.  Equivalently, one could set up the equation on the disjoint union of spheres 
$\bigsqcup_{\al} (S^2)_\al$  and require that the evaluation map  $\ev_{node}$ at the nodes  is transverse to the corresponding 
 diagonal $ \bigl\{(x_{\al\be}):  \al E \be \Rightarrow x_{\al\be} = x_{\ba\al} \bigr\}\subset  M^{2K}$, where $K$ is the number of nodes, and hence the number of edges in the tree.
 For variety, we took this second approach in the discussion of condition $(*_c)$ below; see  \eqref{eq:cond*c}.
 }
\begin{equation}\label{eq:Ldelbar}
\Dd_0: = {\textstyle \bigl\{\xi_\al\in \prod_{\al\in T} W^{1,p}((S^2)_\al,f_{0,\al}^*(\rT M)) \ \big| \ \xi_\al(n_{\al\be})= 
\xi_\be(n_{\be\al}) \ \forall \al E \be \bigr\}.
}
\end{equation}  
   Therefore, the requirement on the obstruction space $E: = \prod_{\ga\in \Ga} E_0$ is as follows:
\begin{itemize}\item[$(*)$]  
the elements in the image of
$\la: E\to \C^{\infty}\bigl(\Hom_J^{0,1}(\Cc|_W\times M)\bigr)$ restrict on ${\rm graph }f_0$ to a subspace of 
$\prod_\al L^p(\Hom^{0,1}_J((S^2)_{\al\in T}, f_{0,\al}^*(\rT M))$ that covers the cokernel of 
$\rd_{f_0}(\pbar)$.
\end{itemize}
Since the regularity condition is open, if we allow the nodal parameters $\bb$ and also $\vec \om,\vec \zeta$ 
 to vary (fixing the gluing parameters $\ba=0$) 
we get transversality for each nearby domain in the given stratum of $\oMm_{0,k+L}$.  However, 
as we explain in more detail in Remark~\ref{rmk:glueXX}~(ii),
 in general we need a
 gluing theorem in order to claim that condition (*)  implies
 that the linearization  $\rd F$ is 
 surjective for all sufficiently close tuples $(\vec e,\ba,\bb,\vec \om,\vec \zeta, f)$, and that the space of solutions near the center point 
 $(0, 0,0,0,0, f_0) $ (where $\vec e, \ba,\bb,\vec \om, \vec \zeta$ all vanish) is the product of the space of solutions at $\ba=0$ with 
 a small neighborhood of $0$ in the parameter space $\ba$.
The gluing theorem in \cite{JHOL} suffices for this purpose, but it does not show that the resulting set of  solutions is a smooth manifold: 
although the solution depends smoothly on $\bb$, and on $\ba$ as long as no component goes to zero, 
it does not establish any differentiability with the respect to the gluing parameters $a_i$ as these converge to $0$. Thus 
the solution space has a weakly SS structure as in Definition~\ref{def:SS0}.
Therefore one must either work in a stratified smooth situation or prove a more powerful  gluing theorem.  

We will assume here that we have a more powerful  gluing theorem that gives at least $\Cc^1$ smoothness with respect to (a rescaled) $\ba$.  (See Remark~\ref{rmk:glueXX} below, and \cite{MWgw} for the general case.)
 We then
define the {\bf unsliced domain}
\begin{equation}\label{eq:Hat U} 
\Hat U
\end{equation}
 to be a small $\Cc^1$-open neighbourhood  of $(\vec 0,0,0,0,0,f_0)$ in $F^{-1}(0)$ where $F$
is as in  \eqref{eq:delbarF}.
 Condition (*) on $E$ implies that $\Hat U$ is a manifold of dimension 
 $\dim \Hat U = \dim E + 2L +  \ind(A)$, where $\ind(A) = 2n + 2c_1(A) + 2k -6$ as in \eqref{eq:indA}.

We now show that $\al\in \Ga$ acts on the solutions of \eqref{eq:delbar} by
\begin{equation}\label{eq:GaU} 
\al^*\bigl(\vec e,\ba,\bb,\vec \om,\vec \zeta, f\bigr)_\bP = \bigl(\al\cdot\vec e, {\ba}',\bb',\phi_{\bP,\al,\de}^{-1}(\al\cdot \vec \om), \phi_{\bP, \al,\de}^{-1}(\vec \zeta), 
f\circ \phi_{\al, \de})_\bP.
\end{equation} 
where  we added the subscript $_\bP$ to emphasize that these tuples are $\bP$-normalized,  where $\Si_{\bP,\de}: = \Si_{\bP,\ba,\bb}$,  $\Si_{\bP,\al^*(\de)} = \Si_{\bP,\ba',\bb'}$,  and $\phi_{\al,\de}:
 \Si_{\bP,\ba',\bb'} \to  \Si_{\bP,\ba,\bb}$ is as in 
\eqref{eq:phiga0}, with $\phi_{\bP,\al, \de}:\Si_0\less N^{2\eps}_{\nodes}$ being its normalization defined in \eqref{eq:gastab}, with the obvious induced action on the parameters $\vec \om,\vec \zeta$. 
To simplify the calculation we consider a point 
$v\in \Si_{\bP,\al^*(\de)}$, and write $z: = \phi_{\al,\de}(v)$. 
If $(\vec e,\ba,\bb,\vec \om,\vec \zeta, f)_\bP$ is a solution, then by \eqref{eq:ga1} we 
 have for fixed $\al\in \Ga$ that 
\begin{eqnarray*}
\la  \bigl(\al\cdot\vec e)\big|_{(v, f\circ \phi_{\al,\de}(v))} & =&
{\textstyle \sum_{\ga\in \Ga} }\; \la(e^{\al\ga})\bigl(\phi^{-1}_{\ga, \al^*(\de)}(v), f\circ \phi_{\ga, \de}(v)\bigr) \circ \rd_v\phi_{\ga, \al^*(\de)}^{-1}\\ \vspace{.07in}
&=& 
{\textstyle  \sum_{\ga\in \Ga} }\; \la(e^{\al\ga})\bigl(\phi^{-1}_{\ga,\al^*(\de)}(\phi_{\al,\de}^{-1}(z)), f(z)\bigr) \circ \rd_v\phi_{\ga, \al^*(\de)}^{-1}\\ \vspace{.07in}
&=& 
{\textstyle  \sum_{\al\ga\in \Ga} }\;\la\bigl(e^{\al\ga})(\phi^{-1}_{ \al\ga,\de}(z), f(z)\bigr) \circ \rd_z\phi_{\al\ga,\de}^{-1}\circ \rd _v\phi_{\al,\de}\\ \vspace{.07in}
&=&
\pbar  ( f) (z) \circ \rd _v\phi_{\al,\de} 
\; =\; \pbar(f\circ \phi_{\al,\de})(v).
\end{eqnarray*}
where the third equality uses \eqref{eq:phiga2} twice and the next one uses \eqref{eq:ga1}.
Hence, because $\Ga$ fixes the element $(\vec 0, 0, 0, 0,0,f_0)\in \Hat U$, we may assume that $\Hat U$ is $\Ga$-invariant.
(Replace $\Hat U$ by $\bigcap_{\ga\in \Ga} \ga^*(\Hat U)$.) Notice also that although the $\Ga$-action of \eqref{eq:GaU} looks quite complicated in normalized coordinates, the induced action on the equivalence classes 
$\bigl(\vec e,[\bn, \bw,\bz, f]\bigr)$  of 
the elements in $ \Hat U$ modulo biholomorphic reparametrizations (where we now describe the marked points by their images in the fiber $\Si_\de=[\bn, \bw,\bz]$) may be written in the notation of (II) as
\begin{equation}\label{eq:coordf1}
\ga^*\bigl(\vec e,[\bn, \bw,\bz, f]\bigr)= \bigl(\ga\cdot \vec e, [ \ga\cdot \bn, \ga\cdot \bw, \bz, f]\bigr).
\end{equation}
\MS

\NI {\bf (VII) The basic chart $\bK: = (U,E,\Ga, s,\psi)$:}

To obtain a chart from the solution space $\Hat U$ we impose  {\bf  slicing conditions} on the tuples $(\vec e,
\ba,\bb,\vec \om,\vec \zeta,f)$
in   $\Hat U$.   Because 
$\im f_0|_{D_0^\ell}$ meets $Q$ transversally in a unique point for each $\ell= 1,\dots, L$ and $\Ga = \Stab[\Si_0,\bz,f_0]$,
we may choose the small  $\Ga$-invariant  $\Cc^1$-open neighbourhood $\Hat U$ of (VI) so that it  satisfies the following
condition:
\begin{itemize}\item[(iv)] for
all $(\vec e,\ba,\bb,\vec \om,\vec \zeta,f)\in \Hat U$ and  $1\le \ell\le L$ the image $\im f\circ \io_{\bP,\ba,\bb}|_{D^\ell_0}$ 
meets $Q$ transversally in a single point; moreover, $(f\circ \io_{\bP,\ba,\bb})^{-1}(Q)
\subset \bigcup_{1\le \ell \le L} D^\ell_0$.
\end{itemize}
Now consider the following set $U'$,
\begin{equation*}
U': = \bigl\{(\vec e,\ba,\bb,\vec \om,\vec \zeta,f)_{\bP}\in \Hat U\ \big| \ f\circ  \io_{\bP,\ba,\bb}(\om^\ell)\in Q \ \;  \forall 1\le \ell \le L\bigr\}.
\end{equation*}  
 Note that:
 \begin{itemize}\item [-]  the $\Ga$ action of
 \eqref{eq:GaU}
$$ 
\al^*\bigl((\vec e,\ba,\bb,\vec \om,\vec \zeta, f)_{\bP}\bigr) = \bigl(\al\cdot\vec e, {\ba}',\bb',\phi_{\bP, \al,\de}^{-1}(\al\cdot \vec \om), \phi_{\bP,\al, \de}^{-1}(\vec \zeta), f\circ \phi_{\al, \de})_{\bP}.
$$ 
preserves the slicing conditions because, by \eqref{eq:gastab}, $\phi_{\bP,\al, \de}$ is the pullback of $\phi_{\al, \de}$ to the fixed domain $\Si_{\bP,0}$.  Thus $\phi_{\bP,\al, \de}^{-1}: \Si_{\bP,0}\to  \Si_{\bP,0} $ is close to the group element $\al^{-1}:\Si_{\bP,0}\to  \Si_{\bP,0}$ that 
permutes the components of $\vec w_0$. Hence   $\phi_{\bP, \al,\de}^{-1}\bigl((\al\cdot \vec \om)^\ell\bigr)$ represents a point that is close to $w_0^\ell$.
\item [-] 
 the slicing conditions are transverse (see  \cite[\S4.3]{MW2}) so that
 the dimension of $U'$ is $\dim E +  \ind(A)$ as required. 
 \end{itemize}
  Define
 \begin{equation}\label{eq:spsi} 
s\bigl((\vec e, \ba,\bb,\vec \om,\vec \zeta, f)_{\bP}\bigr) : = \vec e\in E,\quad \psi\bigl((\vec 0,\ba,\bb, \vec \om,\vec \zeta, f)_{\bP}\bigr) = [\Si_{\bP,\ba,\bb}, \bz,f] \in X.
\end{equation} 
This tuple $(U',E,\Ga, s,\psi)$ satisfies all the requirements for a Kuranishi chart, except possibly the footprint condition:
we need $\psi: s^{-1}(0) \to X$ to induce a homeomorphism from  the quotient $\qu{s^{-1}(0)}{\Ga}$ onto an open subset of $X$.
The forgetful map $\psi: s^{-1}(0) \to X$ factors through the quotient $\qu{s^{-1}(0)}{\Ga}$.  
Further, if  $ \psi\bigl((\vec 0,\ba,\bb, \vec \om, \vec \zeta, f)_{\bP}\bigr) = \psi\bigl((\vec 0,\ba',\bb', \vec \om', \vec \zeta', f')_{\bP}\bigr)$ there are biholomorphisms 
 \begin{equation}\label{eq:phi'} 
\phi: \Si_{\bP,\ba',\bb'}\to  
\Si_{\bP,\ba,\bb}, \quad \phi_{\bP}: = \io_{\bP,\ba,\bb}\circ \phi\circ  \io_{\bP,\ba',\bb'}:\Si_{\bP,0}\to \Si_{\bP,0},
\end{equation} 
 such that 
$
f\circ \phi = f', \phi_{\bP}^{-1}(\vec \zeta) = \vec \zeta'$ and, by condition (iv) above, a permutation $\pi:\{1,\dots,L\}\to \{1,\dots,L\}$ such that  
$\phi_{\bP}^{-1}(\om^{\pi(\ell)}) = (\om')^\ell$.  We need to see that $\pi\in \Ga$.  Without further conditions on $U'$ this may not hold.  However, since $\Ga = \Stab([\Si_0, \bz_0, f_0])$,
we can choose $U'$ so that it holds 
at  $(\vec 0,0,0, 0,0, f_0)_{\bP}$ itself and hence also on a sufficiently small neighbourhood of $(\vec 0,0,0,0, 0, f_0)_{\bP}$ by continuity.
  Hence we may put a final condition on the domain $U$.
\begin{itemize}\item[(v)] for all $(\vec e,\ba,\bb,\vec \om,\vec \zeta,f)_{\bP}\in U$ 
and permutations $\pi:\{1,\dots,L\}\to \{1,\dots,L\}$, there is a tuple
$(\vec e,\ba',\bb',\vec \om',\vec \zeta',f')_{\bP}\in U$ and maps $\phi,\phi_{\bP}$ as in \eqref{eq:phi'} such that
$$
f\circ \phi = f', \quad \phi_{\bP}^{-1}(\vec \zeta) = \vec \zeta', \quad \phi_{\bP}^{-1}(\om^{\pi(\ell)}) = (\om')^{\ell}, 1\le\ell\le L,
$$
if and only if $\pi\in \Ga$.
\end{itemize}
With this condition the footprint map is injective.
 It requires somewhat more work to show that its image $F$ is open in $X$.  The proof in the non-nodal situation may be found in 
 \cite[Proposition~4.1.4]{MW2}. 
  In general, this is a consequence of the gluing theorem. 
 
 \begin{defn}  We define the {\bf chart } $\bK: = (U,E,\Ga, s,\psi)$ with $\Ga = \Stab([\Si_0,\bz_0, f_0]$ and $E$ as in (IV) by requiring that $U$, satisfying (v), be constructed as above from a set $\Hat U$ 
 that satisfies (iv), and then defining $s, \psi$ as in \eqref{eq:spsi}.
\end{defn}
This construction 
 depends on the following choices:
\begin{itemize}\item a {\bf center} $\tau: = [\Si_0,\bz_0,f_0]$ used to fix the parametrization;
\item  a  {\bf slicing manifold} $Q$ that is transverse to $\im f_0$,  disjoint from $f_0(\bz_0)$, orientable, and 
chosen so that the $k$ points $\bz_0$ together
 with the $L$ points in $f_0^{-1}(Q)$ stabilize the domain of $f_0$;
 \item  a {\bf normalization}  $\bP$ for $\Si_0$ as in \eqref{eq:bP} which fixes the parametrization of $\Si_{\bP,0}$;
 \item a {\bf disc structure }
$\bigsqcup_{1\le \ell\le L} D^\ell\subset \Si_0\less \Nn^{\eps}_{\nodes}$ consisting of small disjoint neighbourhoods $D^\ell_0\ni w_0^\ell$ of the $L$ points in $f_0^{-1}(Q)$ that are  averaged over $\Ga$ so that the $\Ga$-action permutes them;
 and 
 \item an {\bf obstruction space} $E$ and $\Ga$-invariant map $\la: E\to \Cc^{\infty}\bigl(\Hom^{0,1}_J(\Cc|_\De\times M)\bigr)$ as in (V), where 
 $\De$ is a  small neighbourhood of $[\Si_0, \bw_0, \bz_0]$ in $\oMm_{0,k+L}$.  
  \end{itemize}

The next remark 
explains what part of the data used to define $U$ can be seen in the footprint $F$ and so carries over
 to the coordinates of a different chart.

\begin{rmk}\label{rmk:lifts} \rm The footprint map $\psi: s^{-1}(0)\to F$ has the form
$(0,\ba,\bb,\vec \om,\vec \zeta,f)_{\bP}\mapsto [\Si_{\ba,\bb},\bz,f]$ and, as we saw above, quotients out by the action of $\Ga$.
Conversely, each element $\tau: = [\Si_{\tau},\bz,f] \in F$ has $|\Ga|$ potentially different lifts to $s^{-1}(0)$ that may be described as follows.
Add the set of points $\bw: = f^{-1}(Q)$ to the domain, giving them one of the $|\Ga|$ labellings 
that occur in the elements of $\psi^{-1}(\tau)\subset s^{-1}(0)$.   Note that these labellings are permuted by  the  $\Ga$ action as in \eqref{eq:coordf1}. One can check that
two  labellings give rise to the same element in $s^{-1}(0)$ if, in the notation of \eqref{eq:GaU}, 
we have $f =  f\circ \phi_{\al,\de}$.
For short we will say that these $|\Ga|$ labellings of the elements in $\bw$ are {\bf admissible}.
By condition (v) above, the domains $\de\in \De\subset \oMm_{k+L}$ of the elements in $\psi^{-1}(\tau)$ form a $\Ga$-orbit,
so that we can identify the domain $\Si_{\tau}$ 
 of $\tau = [\Si_{\tau},\bz,f] \in F$
with the quotient of  $\Si_{\bP,\de}$ by the $\Ga$-action.  Therefore, 
 because we chose the discs $(D^\ell_\de)_\ell \subset \Si_{\bP,\de}$ in \eqref{eq:discsde} so that they are permuted by
 the $\Ga$ action,
the  domain $\Si_{\tau}$   contains a set of $|\Ga|$ disjoint discs $\bigcup_\ell D^\ell_\de$
 with a labelling that is well defined modulo the $\Ga$ action.  Thus  the domain $\Si_{\tau}$ of the element $\tau= [\Si_{\tau},\bz,f] $ in $X$ is provided with a well defined set of discs that have a set of $|\Ga|$ admissible labelings.  Since this structure is well defined by the chart $\bK$ for the elements in its footprint $F$, it can (and will) be used when we construct transition charts.
$\hfill\er$  \end{rmk} 

\begin{rmk}\label{rmk:stab1}\rm  In \cite{MW2} we defined charts whose domains 
are sets of normalized maps that satisfy some equation.  This approach meant that we could not define arbitrary transition charts, but had to assume that the obstruction spaces were suitably transverse; see   Sum Condition II$'$ in \cite[Theorem~4.3.1]{MW2}.
However, we later realized that such conditions are unnecessary if, as above,  we define the elements of the domain $U$ to be tuples
that contain the elements $\vec e\in E$ as one component.  This is explained in detail in \cite{MWgw}, where we call equations such as
\eqref{eq:delbar} a {\bf  decoupled Fredholm stabilization}.   (Here \lq\lq decoupled" refers to the fact that all marked points can vary freely since we have not yet imposed slicing conditions.)   The other\footnote
{
In fact, \cite{FT} also adopts this stabilization method to avoid having to choose generic basic charts.}
 main difference with the construction by Fukaya et al.  
is that, via the notion of admissible labellings, we can keep track of the labelling of all the added marked points in a transition chart, which allows us to construct transition charts with large footprints.   Pardon independently uses the same approach as ours, but 
constructs his charts much less explicitly. \hfill$\er$
\end{rmk}

The following remark explains the key analytic issues in a little more detail.

\begin{rmk}\label{rmk:anal} {\bf (Some comments on the analysis)} \rm (i) First note that 
\begin{align} \label{action}
\Theta: \PSL(2,\C) \times W^{k,p}(S^2,M) \to W^{k,p}(S^2,M), \quad
(\gamma,f) \mapsto f\circ\gamma
\end{align}
does not even have directional derivatives at maps $f_0\in W^{k,p}(S^2,M) \less W^{k+1,p}(S^2,M)$ since the differential\footnote{
Here the tangent space to the automorphism group $\rT_{\rm Id}G_\infty \subset\Ga(\rT S^2)$ is the finite dimensional space of holomorphic (and hence smooth) vector fields $X:S^2 \to \rT S^2$ that vanish at $\infty\in S^2$.
}
\begin{align}\label{eq:actiond}
{\rm D}\Theta ({\rm Id},f_0) : \;
\rT_{\rm Id}\PSL(2,\C)
\times W^{k,p}(S^2, f_0^*\rT M) &\;\longrightarrow\; W^{k,p}(S^2, f_0^*\rT M) \\
 {(X,\xi)} \qquad\qquad\qquad\;\;\; &\;\;\mapsto\; \;\;\;\;   \xi + \rd f_0 \circ X \nonumber
\end{align}
is well defined only if $\rd f_0$ is of class $W^{k,p}$.
In fact, even at smooth points $f_0\in{\mathcal C}^\infty(S^2,M)$, this ``differential'' only provides directional derivatives of \eqref{action}, for which the rate of linear approximation depends noncontinuously on the direction. Hence \eqref{action} is not classically differentiable at any point.
Hence if we impose the slicing conditions  in (iv) above {\it before} we solve the equation $F=0$ where $F$ is as in \eqref{eq:delbarF}, the 
expression~\eqref{eq:ga1} for $\la(\vec e)$ will involve a  group element $\phi_f$ that depends explicitly on $f$ and hence will not be differentiable.  
On the other hand the slicing conditions  in (iv) are differentiable.  More precisely,  the implicit function theorem implies that
 the map
$$
 W^{k,p}(S^2,M)\to S^2, \quad f \mapsto f^{-1}(Q) .
$$
is $\Cc^\ell$-differentiable if $k>\ell + 2/p$ such that $W^{k,p}(S^2)\subset \Cc^\ell(S^2)$.  For more detail on these points see \cite[\S3]{MW2}.\MS

\NI (ii)
Similar issues arise when changing coordinates and constructing transition charts, and lead to our choosing obstruction bundles
that are geometrically constructed by pulling back finite rank subspaces
\begin{equation}\label{graphsp}
E\subset \Cc^\infty(\Hom^{0,1}_J(S^2,M))
\end{equation}
of the space of smooth sections of the bundle over $S^2\times M$ of $(j,J)$-antilinear maps $\rT S^2 \to \rT M$.
This point is discussed in detail in \cite[\S4]{MW2}.\MS

\NI (iii)
This choice of obstruction bundle also allows us to apply gluing theorems such as that in \cite{JHOL} via the graph construction. The point here is that
\cite{JHOL} proves a gluing result\footnote
{
\cite{JHOL}  only considers the case when $\Si$  has one node, but since the estimates are local it is easy to generalize to more nodes.}
 for the linearization of the homogeneous operator $f\mapsto \pbar f$ in a neighbourhood of a map 
$f: \Si: = \Si_{\bP,\ba,\bb}\to M$ 
with fixed domain $\Si$ and surjective linearization.  Thus it only applies when we can take  $E = 0$.  However
the form of the inhomogeneous term in \eqref{eq:delbarF} allows us to 
convert this equation  into a homogeneous equation using the graph construction. 
Namely, consider the fibration $\pr: V = E\times \Si \times M \to E$ and provide the fiber $\vec e \times  \Si \times M$ with the almost complex structure $\TJ_{\vec e} : \rT_z \Si \oplus \rT_x M \to \rT_z\Si \oplus \rT_x M$ that restricts  to the standard complex structure on $ \rT_z \Si$ and to $J$ on $ \rT_x M$ but that has an extra component in $\Hom(\rT_z \Si, \rT_x M)$ given by $-2 J \la(\vec e)(z,x)$.
Then one can check that $f: \Si\to M$ satisfies  \eqref{eq:delbarF} for a given $\vec e\in E$  precisely if  ${\rm graph }f: z\mapsto (z,f(z))$ is 
$(\TJ_{\vec e})$-holomorphic. 
Now consider  the  family of almost complex structures $\TJ = (\TJ_{\vec e})$ in the fibers of the bundle  $E\times \Si \times M \to E$,
and look at the corresponding Cauchy-Riemann equation for maps such as ${\rm graph }f$ into these fibers.  
Since  the derivatives 
$\p_{\vec e} \TJ$ of  $\TJ$  with respect to directions in $E$ surject onto $E$, condition (*)  for $f:\Si\to M$ implies that the corresponding linearized operator 
at ${\rm graph }f$ is surjective.  Compare the discussion of $J$-holomorphic spheres in the fiber in  \cite[Ch.~6.7]{JHOL}.
\hfill$\er$
\end{rmk}

\begin{rmk}\label{rmk:glueXX} \rm {\bf(More on gluing)}
(i)  The geometric gluing process for a domain with one nodal point and (small) gluing parameter $\ba\in \C$ involves cutting out discs of radius $\sqrt{|\ba|}$ from each sphere (with respect to the round metric) and identifying their boundaries with a twist.  The resulting sphere has a natural smoothing: see \cite[\S10.2]{JHOL}.  As explained in  \cite[\S10.9]{JHOL}, with this choice the parameter $\ba$ can be interpreted as the cross ratio
of the images in the glued domain of four (appropriately ordered) points in standard position on the original spheres.  (In \cite{JHOL} the nodal points are put at $\infty$ and $0$, and then we consider the images in the glued domain of the points $0,1$ on the first sphere and $1,\infty$ on the second.)  The resulting gluing map (described in more detail in part (ii) below) is weakly stratified smooth in the sense of Definition~\ref{def:SS1}, and varies smoothly with respect to the gluing parameter $\ba$ in the domain $\ba\in \C\less \{0\}$.  However, to get differentiability with respect to $\ba$ at $\ba=0$  one must rescale $\ba$, replacing it by a variable $\Tilde\ba: = 
\phi(|\ba|)\frac{\ba}{|\ba|}$, for an appropriate function $\phi$ called a {\it gluing profile}.  For more details see for example \cite{Castell1}.
\MS

\NI (ii)  The gluing map is a local isomorphism from a neighbourhood of the center point in $U|_{\{\ba = 0\}}\times B(\ba)$ to 
$U$, where $B(\ba)$ is a neighbourhood of $0$ in the space of gluing parameters $\ba$.  We need to verify its smoothness properties when it is written in terms of the product coordinates
$\bigl((\vec e, \bb,\om,\zeta, f)_{\bP},\ba\bigr)$ on its domain and the coordinates $(\vec e,\ba, \bb,\om,\zeta, f)_{\bP}$ on the target $U$.
 We will construct this gluing map
as the composite of the geometric gluing constructed in \cite{JHOL} with a renormalization. In its turn, the geometric gluing map is constructed in \cite[Ch~10]{JHOL} by a 
two step process.  First one considers the pregluing $\oplus_{\ba,\bb}f$ that is  made using cut off functions from the restriction of  $f$ to the domain $\Si_{\bP,\ba,\bb}$ formed as in \eqref{eq:gluedom}, and then one uses a Newton process to adjust the pregluing
to a  solution of equation~\eqref{eq:delbar}  of the form 
$$
gl'_{\vec e,\ba,\bb}(f): = \Exp_{\oplus_{\ba,\bb}f}(\xi),
$$
 where $\xi$ is a vector field along the image of  $\oplus_{\ba,\bb}f$, that lies in some chosen complement  to the kernel of the linearized Cauchy--Riemann operator at
$\oplus_{\ba,\bb}f$.   Finally,  the coordinates $\om': = gl'_{\vec e,\ba,\bb,f}(\om)$ are defined so that the corresponding points $\bw'$ on 
the glued domain $\Si_{\bP,\ba,\bb}$ satisfy the slicing conditions with respect to the map $gl'_{\vec e,\ba,\bb}(f)$.  
Hence the geometric glung map takes the form:
\begin{align*}
\bigl((\vec e, \bb, \vec \om,\vec \zeta, f )_\bP, \ba\bigr) &\ \mapsto \bigl(\vec e, [\Si_{\bP,\ba,\bb}, \bw', \bz':=\oplus_{\ba,\bb} \bz, gl'_{\vec e,\ba,\bb}(f)]\bigr),
\end{align*}
where $\bw'$ depends on all the variables $\vec e,\ba,\bb,f$.
The second step is to normalize the resulting tuple.
  Thus,  in normalized coordinates the gluing map takes the form 
\begin{align}\label{eq:gluemap}
\bigl((\vec e, \bb, \vec \om,\vec \zeta, f )_\bP, \ba\bigr) &\ \mapsto \bigl(\vec e, [\Si_{\bP,\ba,\bb}, \bw', \bz':=\oplus_{\ba,\bb} \bz, gl'_{\vec e,\ba,\bb}(f)]\bigr)\\ \notag
&\ \stackrel{renorm}{ \mapsto} \bigl(\vec e, \ba',\bb',  gl_{\vec e,\ba,\bb,\vec \om, f}(\vec\om), gl_{\vec e,\ba,\bb,\vec \om, f}(\vec\zeta), gl_{\vec e,\ba,\bb}(f)\bigr)_\bP,
\end{align}
where the parameters $\ba',\bb'$ are functions of $\ba$, $\bb$ and the positions of the points $\bw',\bz$ that lie in $\im \bP$. 
\hfill$\er$ \end{rmk}

\NI {\bf (VIII): Change of coordinates:}  Before discussing transition charts, we consider the effect on a single chart  of changing the  normalization, center and slicing conditions.   These formulas will be needed to understand the coordinate changes in the atlas, but are not necessary for the construction of a transition chart.  Hence this subsection can be omitted at first reading.
 \MS
 
 \NI {\bf  [a]  Change of normalization:}   
 
 When defining a chart, the center and slicing manifold are needed to
 set up the framework, i.e. to specify the added marked points $\bw$ and hence the neighborhood $\De$ of the stabilized domain
 $[\Si_0,\bw_0,\bz_0]$ in $\oMm_{0,k+L}$.  
The  normalization is then used
in order  to write down the equation \eqref{eq:delbar} in coordinates so that one can understand its analytic properties.
However, the equation itself makes sense as a section of a bundle over the space 
$\Map^\infty (\Cc|_\De; M)$  of $C^\infty$ maps from the fibers of the universal curve to $M$.  
Therefore the following holds.

  \begin{itemize}\item{\it  
 If we fix $\tau$ and $Q$ and consider two possible normalizations $\bP_1,\bP_2$, then any chart $U_{\bP_1}$ constructed using 
$\bP_1$ is isomorphic to some chart $U_{\bP_2}$ constructed using $\bP_2$. In particular its footprint will not change.}
  \end{itemize}
  
To see this, let $\phi_{\bP_2,\bP_1}: \Si_{\bP_2,0}\to \Si_{\bP_1,0}$ be the unique biholomorphism that 
  takes the points  $\bn_0, \bw_0, \bz_0$ in $\Si_{\bP_2,0}$ with labels in $\im(\bP_1)$  to their standard positions in  
  $\Si_{\bP_1,0}$.  Then for each $\de\in \De$ 
  the fiber $\Si_\de$ has two normalizations $\Si_{\bP_i,\de} = \Si_{\bP_i,\ba_i,\bb_i}, i=1,2$,
where $  \ba_i,\bb_i$ are determined by appropriate cross ratios of the marked points $\bw,\bz$ in $\Si_\de$; cf. Remark~\ref{rmk:norm}~(i).
The change of normalization is given by a biholomorphism $  \phi_{\bP_2,\bP_1, \de}:\Si_{\bP_2,\ba_2,\bb_2}\to \Si_{\bP_1,\ba_1,\bb_1}$
that satisfies the formula
  $$
  \phi_{\bP_2,\bP_1, \de}:= \io_{\bP_1,\ba_1,\bb_1}\circ \phi_{\bP_2,\bP_1}\circ  \io_{\bP_2,\ba_2,\bb_2}^{-1}: \; \Si_{\bP_2,\ba_2,\bb_2}\to \Si_{\bP_1,\ba_1,\bb_1},
  $$
  wherever the RHS is defined, and 
in particular, at the points $\bw$.  Hence 
there is an induced map $U_{\bP_1}\to U_{\bP_2}$ of the form
\begin{eqnarray}\label{eq:changeU} 
U_{\bP_1}\ni  (\vec e,\ba_1,\bb_1, \vec \om,\vec \zeta,f)\;\mapsto \phi^*(\vec e,\ba_1,\bb_1, \vec \om,\vec \zeta,f)
\\ \notag
=\; \bigl(\vec e,\ba_2,\bb_2,\phi^{-1}(\vec \om),\phi^{-1}(\vec \zeta),f\circ  \phi_\de\bigr)\in U_{\bP_2}, 
\end{eqnarray} 
where $\phi: = \phi_{\bP_2,\bP_1}$ and  $ \phi_{\de} =   \phi_{\bP_2,\bP_1, \de}.$  Note that
$\phi_{\bP_2, \bP_1,\de} $ varies as $\de= [\bn, \bw,\bz]$ varies, and that the elements in $\vec e$ are not affected by the action.  
\MS

  \NI {\bf  [b]  Change of center:} 
  
   Now suppose 
  given a chart $U_{\bP_1}$ constructed using  $\tau_1, Q, \bP_1$ and with footprint $F$, and 
  that we change the center from 
  $\tau_1: = [\Si_{01},\bz_{01},f_{01}]$ to $ \tau_2: = [\Si_{02},\bz_{02},f_{02}] \in F$,
  but keep the same slicing manifold and the same normalization (as far as possible).  
    Thus we choose a lift $
  (\vec 0,\ba_{02},\bb_{02},\bw_{02}, \bz_{02}, f_{02})$ of  $[\Si_2,\bz_2,f_2]$ to $U_{\bP_1}$, and
 take a normalization $\bP_2$ of this stable map at  $\de_2: = [\Si_{02},\bw_{02}, \bz_{02}] $
  that includes all the nodal points in $\bP_1$ that have not been glued (suppose there are  $p_n-s$ of these), together with an appropriate subset of the points in $\bw_{\bP_1}, \bz_{\bP_1}$, if necessary assigned to 
different  points $0,1,\infty$ in $S^2$; see  the example in Figure~\ref{fig:5}.  
  If the stratum $X_{S_2}$ containing $\tau_2$ is strictly larger than $X_{S_1}$ (i.e. $\tau_2$ has fewer nodes than $\tau_1$), then we cannot hope to represent the whole footprint $F$ in the coordinates based at $\tau_2$.
  However we claim:
  
  \begin{itemize}\item{\it   there is a $\Ga$-invariant neighbourhood
  $U_{\bP_1}|_{\De_2}$ of $\psi^{-1}(F\cap X_{\ge S_2})$ in $U_{\bP_1}$
  that can be represented in  terms of the normalization $\bP_2$.}
  \end{itemize}

To prove the claim, let us suppose that
 $m$ of the gluing parameters $\ba_{02}$ are nonzero, say $a^{K-m+1}_{02}, \dots, a^{K}_{02}$, so that $\Si_2$ has $K-m$ nodes.
 Then
 the   \lq\lq extra" nodes and marked points in $\bP_1$ (namely those in 
  $(\bn_{\bP_1}\cup \bw_{\bP_1}\cup \bz_{\bP_1}) \less (\bn_{\bP_2}\cup \bw_{\bP_2} \cup  \bz_{\bP_2})$ 
  can 
  now move freely;  cf.   Figure~\ref{fig:5} 
   Consider the parametrization
 $$
 \io_{\bP_1} : 
\bigl(\Si_{(\bP_1,\de_1),0}\less \Nn^\eps_\nodes\bigr)\times B^{6K-2p_n}\times B^{2(k+L-p)}\; \to \;\Cc|_{\De}
$$
of \eqref{eq:gluedom1}  near $\de_{01}: = [\Si_{01},\bw_{01},\bz_{01}]$, 
 and let $\De_2\subset \De_1$ be a neighbourhood of $\de_2$ that contains the domains of the elements in
$\psi^{-1}(F\cap X_{\ge S_2})$.
There is a similar  parametrization
 $$
 \io_{\bP_2} :
\bigl(\Si_{(\bP_2,\de_2),0}\less \Nn^\eps_\nodes\bigr)\times B^{6K-2(p_n-s)}\times B^{2(k+L-p+ p^e)}\; \to \;\Cc|_{\De_2},
$$
and the composite  $ \io_{\bP_2}^{-1}\circ \io_{\bP_1} |_{\De_2}$ on the level of domains has the form
\begin{align}\label{eq:chcenter} 
&\quad \io_{\bP_2}^{-1}\circ \io_{\bP_1}|_{\De_2}: 
(\ba_1,\bb_1,\vec \om_1, \vec \zeta_1)\mapsto (\ba_2,\bb_2,\vec \om_2, \vec \zeta_2).
\end{align}
Notice here that the lengths of the individual tuples $\ba_i, i=1,2$,   $\bb_i, i=1,2$, and so on are different, although the sum of these lengths  are the same on the two sides.
This is well defined  over $\psi^{-1}(F\cap X_{\ge S})$ because the map $\io_{\bP_i,\ba_i,\bb_i}$ for $i=1,2$  takes the points with coordinates
$\bb_i,\vec \om_i, \vec \zeta_i$ to  the same marked points $\bw,\bz$ in the fiber $\Si_{\bP_1,\ba_1,\bb_1} = \Si_{\bP_2,\ba_2,\bb_2}$.
Hence this map is well defined over $\De_2$ for  sufficiently small $\De_2$.

Define
$$
U_{\bP_1}|_{\De_2}:= \bigl\{(\vec e, \ba_{1},\bb_{1}, \vec \om,\vec\zeta, f)_{\bP_1}\in U \ | \ [\Si_{\bP_1,\ba_{1},\bb_{1}},  \bw,\bz]\in \De_2\bigr\},
$$
 and denote by $\io_{\bP_1}|_{\De_2}$ the restriction of $\io_{\bP_1}$ to the domains occuring in $U_{\bP_1}|_{\De_2}$.
Given a formula such as \eqref{eq:chcenter} for
the coordinate change on the parametrization of domains, we can derive a formula analogous to
\eqref{eq:changeU} for the effect on the elements of $U_{\bP_1}|_{\De_2}$ of this change of center, namely
\begin{equation}\label{eq:changeU1} 
U_{\bP_1}|_{\De_2}\ni  (\vec e,\ba_1,\bb_1, \vec \om_1,\vec \zeta_1,f)_{\bP_1}\;\mapsto  \bigl(\vec e,\ba_2,\bb_2, \vec \om_2,\vec \zeta_2,f\circ  \phi_\de\bigr)_{\bP_2}\in U_{\bP_2}, 
\end{equation} 
where $\ba_2,\bb_2, \vec \om_2,\vec \zeta_2,$  are 
as in \eqref{eq:chcenter},
and where 
$
\phi_\de :
\Si_{(\bP_2,\de_{02}); \de}\to \Si_{(\bP_1,\de_{01});\de}
$
is the biholomorphic map that equals  
$ \io_{\bP_1, \ba_1,\bb_1}\circ \bigl( \io_{\bP_2}^{-1}\circ \io_{\bP_1}|_{\De_2}\bigr)^{-1}\circ  \io_{\bP_2, \ba_2,\bb_2}^{-1}$
wherever this is defined.

Note the following:
 \begin{itemlist}\item
 The map in \eqref{eq:changeU1} has the same form as that in  \eqref{eq:changeU} but with different $\phi, \phi_\de$.  Hence a map that changes both the center and normalization also has this form.

 \item
The resulting  chart with domain $U_{\bP_2}$ may not be not minimal in the sense of Definition~\ref{def:chart2}
 since $\Ga$ may now be larger than $\Stab(\tau_2)$.
\item
The composite of two such maps that change the center first from $\de_0$ to $\de_1$ and then from $\de_1$ to $\de_2$ equals the direct coordinate change from $\de_0$ to $\de_2$.
\end{itemlist}

  \MS
 
 \NI {\bf  [c]  Change of slicing manifold:}   Let us return to considering the 
chart $U$ with center $\de_0= [\Si_0,\bw_0,\bz_0]$ as in defined in (VII), and 
suppose that we change the slicing manifold from $Q_1$ to $Q_2$.  Let is first consider the case in which
 $Q_2$ is so close to $Q_1$  that the new set of slicing points $\bw_2$ lies in the same set of discs $(D^\ell)_{\ell}$ as $\bw_1$.  Then there is a natural correspondence between the (ordered) tuples $\bw_1$ and $\bw_2$ so
  that we can use the same normalization $\bP$ for both  $\de_1: = [\bn,\bw_1,\bz]$  and  $\de_2: = [\bn,\bw_2,\bz]$. 
Then if $\de_1$ is sufficiently close to the center $\de_0$  the element $\de_2$ will also lie in $\De$.  Hence  the same obstruction space $E$ 
can be used for both charts, and
the corresponding change of coordinates $U_1\to U_2$ 
  is given by replacing  the map $\phi_{\bP_2,\bP_1}$ in the above formulas by
  the map $\Si_{\bP}\to  \Si_{\bP}$ that fixes the points in $\bn_\bP, \bz_\bP$ (that are in standard positions) and takes the points in $\bw_2\subset \Si_{\bP}$ with labels in $\im(\bP)$ 
  to their standard positions.

However, if the new slicing manifold $Q_2$ is sufficiently different  from $Q_1$, there need be no obvious relation between 
 the tuples $\bw_1$ and $\bw_2$.  For example, suppose that 
 the chart is centered on $[\Si_0,\bz_0,f_0]$ where 
 $\Si_0=S^2$, $\bz$ is the single point $\infty$  and $f_0: S^2\to M$ is a double cover that factors through the map $z\mapsto z^2$.  Then the isotropy group is $\Ga = \Z/2\Z$, and we need to add two points to stabilize the domain.  We might choose 
   $Q_1$ to have two components, one transverse to $\im (f_0)$ at $f_0(1) = f_0(-1)$ and the other transverse at 
  $f_0(2) = f_0(-2)$ so that $\bw_1 = (1,-1,2,-2)$, while 
 $Q_2$ might have a single component that is transverse to $\im (f_0)$ at $f_0(3) = f_0(-3)$, so that $\bw= (3,-3)$. 
   Since the obstruction bundle for $U_1$ might depend on all four entries in $\bw_1$,
  while that for $U_2$ depends only on $\bw_2$ there is no obvious relation between the  obstruction spaces.
  Therefore  there is no direct coordinate change from $U_1$ to $U_2$, and the easiest way to  relate  them is 
  via  transition charts.

\MS
\NI {\bf (IX): Constructing the sum of two charts:}
Suppose that we are given two sets of data $\bigl(\tau_i: = [\Si_i,\bz_i,f_i], Q_i,(D^\ell_i)_{1\le \ell\le L_i}, \bP_i,E_i, \la_i\bigr)_{i=1,2}$ that define charts $\bK_i$  with overlapping footprints $F_i$.
Then we aim to define a transition chart with 
\begin{itemize}\item[-]
footprint $F_{12} = F_1\cap F_2$, 
\item[-] obstruction space $E_{12}: = E_1\times E_2$, 
\item[-] group $\Ga_{12}: = \Ga_1\times \Ga_2$, and 
\item[-] domain $U_{12}$ of dimension $\dim (U_{12}) - \dim(E_{12}) = \dim U_i -\dim E_i = \ind(A)$, 
so that $\dim U_{12} = \dim U_1 + \dim E_2 = \dim U_2 + \dim E_1$.
\end{itemize}
In this paragraph we consider the case when the center of one chart is 
contained in the footprint of the other: say $\tau_2\in F_1$ which implies $\tau_2\in F_{12}$.  Then we set up the 
transition chart
using the coordinates provided by $\tau_2$ and 
$ \bP_2$.  
Notice that  the central fiber $\Si_{2,\bP_2}$ contains two sets of discs,
the standard discs  $(D^\ell_2)_{1\le\ell\le L_2}$  for the chart $U_2$  as well as  the image
 $(D^\ell_{\tau_2,1})_{1\le\ell\le L_1}$ in the fiber over $\tau_2$
 of the standard discs $(D^\ell_1)_{1\le\ell\le L_1}$   for the chart $U_1$.  
 By Remark~\ref{rmk:lifts} the labelling of the latter set of discs $\bigcup_\ell D^\ell_{\tau_2,1}$ is only well defined modulo the action of $\Ga_1$,
 which explains the different conditions on the tuples $\vec \om_1,\vec \om_2$ below.

With $L: = L_1+L_2$, $E_{12}: = E_1\times E_2$, and $p_{2,n}$ equal to the number of nodal points in $\bP_2$,
we set up an equation as in (VI) on tuples of the form
 \begin{align}\label{eq:ww12p}  \Ww_{12,\bP_2}: & = 
\left\{(\vec e_1,\vec e_2,\ba,\bb, \vec \om_1,\vec \om_2,\vec \zeta,f)_{\bP_2} \in \right.\\ \notag
& \qquad \qquad  
 E_{12} \times B^{6K-2p_{2,n}}\times B^{2(k+L)}\times W^{1,p}(\Si_{\bP_2,\ba,\bb}, M) \;  \mbox{ where } \\  \notag
& \qquad  \qquad \begin{array}{lll} &\vec e_i\in E_i,\quad \ba,\bb\in B^{6K-2p_n},& f:\Si_{\bP_2,\ba,\bb}\to M, \; \\
& 
\exists \ga\in \Ga_1 : \om_1^{\ga(\ell)}\in D_{\tau_2,1}^\ell,  \;1\le \ell\le L_1,
& \om_2^\ell\in D_2^\ell, 1\le \ell\le L_2.\end{array} 
\Bigl. \Bigr\}
\end{align} 
 Somewhat hidden in this notation
is the fact that the tuple $\vec \om_1$ contains $L_1$ elements since all the points $\bw_1$ can vary, while the number of varying elements in
the tuples 
$\vec \om_2$ and $\vec \zeta$ is $\#(\bw_2\less \bw_{2,\bP_2})$ and $\#(\bz\less \bz_{\bP_2})$. 

 In this notation, the unsliced domain $\Hat U_{12, \bP_2}$ is a suitable open subset
of the following solution space:
\begin{align}\label{eq:lasumI}
& \ \Hat U_{12, \bP_2} \subset  \Bigl\{ (\vec e_1,\vec e_2,\ba,\bb, \vec \om_1,\vec \om_2, \vec \zeta,f)_{\bP_2} \in \Ww_{12,\bP_2} \ \big| \ \vec e_i\in E_i, \; 
\\ \notag
&\ \qquad\ \pbar \bigl((\vec e_1,\vec e_2,\ba,\bb, \vec \om_1,\vec \om_2, \vec \zeta,f)_{\bP_2}\bigr) =
{\textstyle \sum_{i=1,2}\sum_{\ga\in \Ga_i} \ga^*\bigl(\la_i(e_i^\ga)\bigr)|_{{\rm graph } f}\Bigr\},
}
\end{align}
with
\begin{equation*}
  \ga^*\bigl(\la_i(e_i^\ga)\bigr)|_{(z,f(z))}: = \la_i(e_i^\ga)(\phi_{i,\ga}^{-1}(z), f(z)) \circ \rd_z \phi_{i,\ga}^{-1},\;\ \mbox{ for } \ga \in \Ga_i,
\end{equation*}
where $\phi_{i,\ga}: = \phi_{\ga,\de_i}$ as in  as in \eqref{eq:ga1} and \eqref{eq:phiga0},  for 
$$
\de_i: = [\Si_{\bP_2,\ba,\bb}, \bw_i,\bz_i],\quad i=1,2,
$$
with (as usual) $\bw_i = \io_{\bP_2,\ba,\bb}(\vec \om_i)$,  and $\bz = \io_{\bP_2,\ba,\bb}(\vec \zeta_i)$.
This equation has the same form as \eqref{eq:delbar}.  Therefore because $E_1,E_2$  and hence $E_1\times E_2$ satisfy (*) for all lifts of elements in the footprint $F_{12}\subset F_2$ to $ \Ww_{12,\bP_2}$,  we can choose the open set $ \Hat U_{12, \bP_2} $ so that it is a smooth manifold that contains all such lifts.  We can also choose $\Hat U_{12,\bP_2}$ to be invariant under the action of 
 the group $\Ga_{12}: =\Ga_1\times \Ga_2$.  Here, since we normalize with respect to the chart $\bK_2$, the action of $\ga_1\in \Ga_1$ is simply by permutation:
\begin{equation}\label{eq:perm}
\ga_1^*\bigl((\vec e_1,\vec e_2,\ba,\bb, \vec \om_1,\vec \om_2, \vec \zeta,f)_{\bP_2}\bigr)=  (\ga_1\cdot \vec e_1,\vec e_2,\ba,\bb, \ga_1\cdot\vec \om_1, \vec \om_2, \vec \zeta,f)_{\bP_2}.
\end{equation}
However the elements of $\Ga_2$ act by permutation plus renormalization:
\begin{equation}\label{eq:permn}
\ga_2^*\bigl((\vec e_1,\vec e_2,\ba,\bb,\vec \om_1,\vec \om_2, \vec \zeta,f)_{\bP_2}\bigr)=  \bigl( \vec e_1,\ga_2\cdot\vec e_2,\ba',\bb',\phi_{\ga_2}^{-1}( \vec \om_1),\phi_{\ga_2}^{-1}(\ga_2\cdot\om_2),\phi_{\ga_2}^{-1} \vec\zeta,f\circ \phi_{\ga_2,\de_2}\bigr)_{\bP_2},
\end{equation}
where $\phi_{\ga_2}: = \phi_{\bP_2, \ga_2,\de_2} $ 
as in \eqref{eq:gastab}. 
This difference in action is compatible with the different conditions on $\vec \om_1,\vec \om_2$ in the definition of $\Ww_{12,\bP_2}$.

  We now define $U_{12,\bP_2}$ to be the subset of $\Hat U_{12, \bP_2}$ on which the  slicing conditions are satisfied.
  Thus
  \begin{equation}\label{eq:U12}
 U_{12,\bP_2} \subset  \bigl\{ \bigl(\vec e_1,\vec e_2,\ba,\bb,\vec \om_1,\vec \om_2, \vec\zeta,f\bigr)_{\bP_2}\in \Hat U_{12,\bP_2}\ | \io_{\bP,\ba,\bb}(\vec \om_i)\in f^{-1}(Q_i),\; i=1,2\bigr\}.
\end{equation}
Then
 $U_{12,\bP_2}$ is $\Ga_{12}$-invariant (since the slicing conditions are preserved by this action),
and the zero set of $s_{12}:  (\vec e_1,\vec e_2,\ba,\bb,\vec \om_1,\vec \om_2, \bz,f)_{\bP_2}\mapsto  (\vec e_1,\vec e_2)$ 
is taken by the forgetful map
$\psi: (\vec 0,\vec 0,\ba,\bb,\vec \om_1,\vec \om_2, \bz,f)_{\bP_2}\mapsto [\Si_{\bP_2,\ba,\bb}, \bz,f]$ onto $F_{12}$.  
We claim that
 $\bK_{12}: = \bigl(U_{12;\bP_2}, E_{12},\Ga_{12}, s_{12},\psi_{12}\bigr)$ is the required transition chart.
 This is immediate from the construction, except possibly for the fact that the footprint map $\psi$ induces an injection
 $\qq{s_{12}^{-1}(0)}{\Ga_{12}}\to F_{12}.$
 However this holds because the forgetful map
 $$
 \rho_{2, 12}: U_{12,\bP_2}\cap s_{12}^{-1}(E_1) \to U_2,\quad  
(\vec 0, \vec e_2, \ba,\bb,\vec \om_1,  \vec \om_2, \vec\zeta,f)_{\bP_2}\mapsto (\vec e_2,  \ba,\bb, \vec \om_2,  \vec\zeta,f)_{\bP_2}
$$
 induces an injection into $U_2$ from the quotient of $\TU_{2,12}: = U_{12,\bP_2}\cap s_{12}^{-1}(E_1)$ 
 by a free permutation action of $\Ga_1$
 on $\vec \om_1$, and we have already checked that the footprint map $\psi_2$ induces a homeomorphism
$ \qu{s_{2}^{-1}(0)}{\Ga_{2}}\to F_{2}$.

To complete the construction  we must check that the required coordinate changes $\bK_i\to \bK_{12}$  exist.
The coordinate change $\bK_2\to \bK_{12}$ is induced by the above projection $\rho_{2,12}$.
The coordinate change $\bK_1\to \bK_{12}$ has domain
$\TU_{1,12}: = U_{12,\bP_2}\cap s_{12}^{-1}(E_2)$, and is given by
first
changing the normalization\footnote
{
If $Q_1,Q_2$ are disjoint we can simply apply  (VIII) (a) with  slicing manifold $Q_1\cup Q_2$; 
the general case is similar.
}
 from $\bP_2$ to $\bP_1$, and
then forgetting the components of $\vec \om_2$ to obtain a map
$\rho_{1,12}: \TU_{1,12}\to U_{1,\tau_2,\bP_1}$.  The reader can check that this change of normalization 
reverses the conditions on the tuples $\vec \om_i$. In particular, afterwards $\vec \om_2$ has $L_2$ potentially nonzero components
$(\om_2^\ell)_\ell$
with $\om_2^\ell\in D^{\ga(\ell)}_2$  for some $\ga\in \Ga_2$.  Hence the forgetful map is the quotient by a free action of $\Ga_2$ as required.

\begin{rmk}\label{rmk:coordf}\rm   We constructed this transition chart under a restrictive condition on the footprints.
If this condition is not satisfied we may not be able to find one set of coordinates that 
covers a neighbourhood of the full footprint $F_{12}$.  The difficulty here is that the parametrization maps 
$\io_{\bP_1,\ba_1,\bb_1}$ in \eqref{eq:gluedom1} are not defined near the nodes, so that their image may not contain 
all the points in the relevant inverse images  $f^{-1}(Q_2)$.  Therefore, one might not be able to pull all the points in the tuple $\bw_2$ back to the center point $\tau_1$, and similarly, the points in $\bw_1$ might not all pull back to a center for the second chart.   One could deal with this problem by requiring that if 
$F_{12}\ne \emptyset$,  the corresponding slicing manifolds $Q_1,Q_2$ are not too different, but such conditions are hard to formulate precisely.   Instead (as in \cite{Pard}) we dispense 
with the requirement that the chart have global coordinates.
To prepare for the general definition given in (X) below, we now explain 
the  coordinate free version of the above construction.
 \MS
 
 \NI {\bf The coordinate free transition chart:}\,
We define $U_{12}$ to be the image of $U_{12,\bP_2}$ by the injective map
\begin{equation*}
U_{12,\bP_2}\ni \bigl(\vec e_1,\vec e_2, \ba,\bb,\vec \om_1,\vec \om_2,\vec \zeta, f\bigr)_{\bP_2}\mapsto
\bigl(\ul{\vec e}, [\bn,\bw_1,\bw_2,\bz, f]\bigr)
\end{equation*}
where $\ul{\vec e}: =(\vec e_1,\vec e_2) \in E_{12}$ and $[\bn,\bw_i,\bz]\in \De_i$ is the domain stabilized via $Q_i$ i.e. the stable curve
 $[\Si_{\bP_2,\ba,\bb}, \bw_i,\bz]$.
Thus $U_{12}$ is a subset of the following space
\begin{align}\label{eq:coordf3}
& \left\{\bigl(\ul{\vec e}, [\bn,\bw_1,\bw_2,\bz, f]\bigr)\ \big| \ \begin{array}{ll} \ul{\vec e}\in E_{12},& \de_i: = [\bn,\bw_i,\bz] \in \De_i,\\
f(\bw_i)\in Q_i, &\exists \ga\in \Ga_i, w_i^\ell\in D_{\de_i}^{\ga(\ell)}\\
\pbar f = \la(\ul{\vec e})|_{{\rm graph } f}&
\end{array} \right\}.
\end{align}
By choice of $Q_i$,  
the condition $f(\bw_i)\in Q_i$  implies that there is exactly one element of
$\bw_i$ in each disc $D_{\de_i}^\ell$.  The labels of these discs are well defined modulo the action of $\Ga_i$,
and  the condition
 $\exists \ga\in \Ga_i, w_i^\ell\in D_{\de_i}^{\ga(\ell)}$   implies that the tuple $\bw_i$ has one of its admissible 
 labellings as in Remark~\ref{rmk:lifts}. Hence in this symmetric formulation both groups $(\Ga_i)_{ i=1,2}$ act by permuting the elements in $\vec e_i, \bn,\bw_i$  as in \eqref{eq:coordf1}.  Further, the  transition chart depends only on the footprint $F_{12}$ and the choice of
 $(Q_i, E_i,\la_i)$,  the  center $\tau_i$,  normalization $\bP_i$  and  discs $ (D^\ell_i)_{1\le\ell \le L}$ 
 being
irrelevant
 except insofar as they help guide the construction.
$\hfill\er$  \end{rmk}

\MS

\NI {\bf (X): Completion of the construction:}  
Suppose given a  collection $(\bK_i)_{i\in I}$ of basic charts whose  footprints $(F_i)_{1\le i\le N}$
cover $X$.   
We aim to construct an  atlas in the sense of Definition~\ref{def:Ku2} in which 
the charts are indexed by $I\in \Ii_\Kk$ and have
$E_I: = \prod_{i\in I} E_i$, $\Ga_I: = \prod_{i\in I} \Ga_i$.
The easiest way to do this is in the coordinate free language introduced above in \eqref{eq:coordf3}.
As before,  we denote the elements of the obstruction space $E_I$ by underlined tuples:  
$\ul{\vec e}: = (\vec e_i)_{i\in I}$.  Similarly $\ul{\bw}: = (\bw_i)_{i\in I}$ are the tuples  of added marked points.
We define $U_I$ to be a $\Ga_I$-invariant open subset of
the following space:
\begin{align}\label{eq:coordf4}
U_I\subset & \left\{\bigl(\ul{\vec e}, [\bn,\ul{\bw},\bz, f]\bigr)\ \big| \ \begin{array}{ll} \ul{\vec e}\in E_{I},& \de_i: = [\bn,\bw_i,\bz] \in \De_i, \forall i\in I,\\
f(\bw_i)\in Q_i, &\exists \ga\in \Ga_i, w_i^\ell\in D_{\de_i}^{\ga(\ell)}\\
\pbar f = \la(\ul{\vec e})|_{{\rm graph } f}&
\end{array} \right\},
\end{align}
chosen so that the footprint is $F_I$.  Since we take
$$
s_I\bigl(\ul{\vec e}, [\bn,\ul{\bw},\bz, f]\bigr) = \ul{\vec e},\quad \psi_I\bigl(\ul{\vec 0}, [\bn,\ul{\bw},\bz, f]\bigr) = 
 [\bn,\bz, f],
 $$
 and $\Ga_I$ acts by permutation, this condition can always  be satisfied.
We claim  that if $U_I$ is a sufficiently small neighbourhood of  $\psi_I^{-1}(F_I)$ then  it is a smooth manifold.  
For this it suffices to check that each point $u$ of $\psi_I^{-1}(F_I)$  has such a neighbourhood, which one does by
choosing a normalization $\bP_i$ at $u$ for some $i\in I$, and then
writing the definition of $U_I$ in the corresponding local coordinates as in the explicit construction in (IX).
Details are left to the reader.

In this coordinate free language, the atlas coordinate changes $\bK_I\to \bK_J$ are given by first choosing appropriate domains $\TU_{IJ}\subset U_J$ and then simply forgetting the components  $(\bw_i)_{i\in (J\less I)}$.
To see these forgetful maps have the required properties, one should argue in coordinates as 
in the discussion after \eqref{eq:U12}.

\MS

\NI {\bf Orientations:}
Finally we must check that the resulting weak atlas is oriented.  To this end,
recall that we chose the $E_i, \la_i$ to be complex linear, and the slicing manifolds  $Q_i$ to be orientable. 
According to Definition~\ref{def:orient}, we must construct a nonvanishing section of $\det(\s_\Kk)$, or, equivalently, of the bundle $\det(\Kk)$. For this we need compatible $\Ga_I$ invariant sections of the  local bundles $\lm \rT U_I\otimes (\lm E_i)^*$ over $U_I$.  
Since each $E_I$ has a natural orientation,  it suffices to construct compatible orientations of the domains $U_I$.  
To this end, we  now fix an orientation for each $Q_i$.    It then suffices to orient the unsliced domains $\Hat{U}_{I,\bP_I}$, which as in \eqref{eq:ww12p} and  \eqref{eq:lasumI}  are subsets of $E_I\times B^N\times W^{1,p}(\Si, M)$
cut out by an inhomogeneous  $\pbar$ equation.  Here the ball $B^N\subset \C^N$  
parametrizes the marked points and gluing parameters, both of which have natural orientations.
Further, the  solution spaces in $ W^{1,p}(\Si, M)$ have  natural orientations,
 since the linearized Cauchy--Riemann operator is the sum of a complex linear first order operator  with a compact perturbation and hence 
has a canonical perturbation through Fredholm operators to   a complex linear operator.  (For more details of this last step,  see the proof of \cite[Theorem~3.1.5]{JHOL} and also \cite[Appendix~A]{JHOL}.)  Since the resulting orientation of $U_I\times E_I$ is preserved by the group actions and coordinate changes, we obtain an orientation of $\Kk$.

\begin{rmk}\label{rmk:stratX}\rm
The smoothness of the charts, group actions  and coordinate changes
 depends on the gluing theorem used.  At the minimum (i.e. with the gluing theorem in \cite{JHOL}) we get a weakly SS atlas.
 With more analytic input, we can get a $\Cc^1$-atlas or even a smooth atlas.  
 However the sets $U_I$ still have an underlying stratification (by the number of nodes in the domains $\Si_\de$ of its elements) that is respected by
  all maps and coordinate changes.    Hence, as explained in \S\ref{ss:SS} the resulting zero set $|(\s+\nu)^{-1}(0)|$ has a natural stratification that is sometimes useful.
$\hfill\er$   \end{rmk}
 
\MS

\NI {\bf (XI): Constructing cobordisms:}    To prove that the VFC is independent of choices we need to build cobordisms between any two atlases  on $X$.  We defined Kuranishi cobordisms over $X\times [0,1]$ in Definition~\ref{def:CKS}.
These restricted cobordisms are also called {\bf concordances}.  
 The following notion is also useful.

\begin{defn}\label{def:commen}
Two atlases $\Kk, \Kk'$ on $X$ are said to be {\bf directly commensurate}
if they are subatlases of a common atlas $\Kk''$ on $X$.  They are commensurate if there is a sequence of atlases
$\Kk =:\Kk_{1},\dots,\Kk_{\ell}: = \Kk'$ such that any consecutive pair $\Kk_i, \Kk_{i+1}$ are directly commensurate. 
\end{defn}

One useful result is that any two commensurate atlases are 
concordant; see
\cite[\S6.2]{MW2} and Remark~\ref{rmk:saddatlas}
 below.  Note that 
because we can construct the sum of any number of charts provided only that their footprints have nonempty intersection, any two atlases constructed on $X$ by the  method  described above
with basic charts $(\bK_i)_{1\le i\le N_1}$ and $(\bK_i)_{N_1+1\le i\le N_2}$ are subatlases of a common atlas
with basic charts $(\bK_i)_{1\le i\le N_2}$. Thus  they are {\bf directly commensurate}  and hence 
concordant.

\begin{rmk}\label{rmk:Jindep}\rm
A similar argument shows that the 
cobordism class of the 
VFC is independent of the choice of almost complex structure $J$ in the following sense.
Suppose that $J_0, J_1$ are two $\om$-tame almost complex structures on  the symplectic manifold
$(M,\om)$, join them by a path $(J_t)_{t\in [0,1]}$ of $\om$-tame almost complex structures (where $t\mapsto J_t$ is constant for $t$ near $0,1$), and define $X^{01}: = \bigcup_{t\in [0,1]} \oMm_{0,k}(M,A,J_t)$.
In the same way that we build a cobordism atlas over $X\times [0,1]$, we can build a cobordism atlas $\Kk^{01}$ 
over $X^{01}$ since this has collared boundary.  Moreover, we can arrange that its restrictions $\Kk^\al: = \Kk^{01}|_{\al}$  at  the end points $\al=0,1$
equal any given GW atlases $\Kk^\al$ for $X_{\al}: =  \oMm_{0,k}(M,A,J_\al)$, and then prove that
the two elements $\bigl([X_\al]^{vir}_{\Kk^\al}\bigr)_{ \al=0,1}$ have the same image in $\check H_ d(X^{01};\Q)$.
It follows that all GW invariants calculated using $[X]^{vir}_{\Kk}$ are independent of the choice of $J$.
This argument is not formally written anywhere; however its  details are very similar to those in \cite[\S8.2]{MW2} which deal with 
the case of concordances.
\end{rmk}

\NI {\bf (XII): Proof of Theorem~A:}  The above construction explains the proof of Theorem~A.  
We set up the relevant equation in (VI), but the proof that it has the 
required  properties assumes a gluing theorem that we did not even state precisely.  
The paper
 \cite{Castell1} will complete the proof by providing  the  analytic details of a  $\Cc^1$-gluing theorem, thus
establishing 
 a $\Cc^1$ version  of Theorem~A.

\subsection{Comments on the construction}\label{ss:var}

We first describe some variants of the construction, and then make some general comments about Gromov--Witten atlases.

 There are two common variants of $X$:  we can consider the subset of $X$ formed by  elements $[\Si,\bz,f]$ 
where we constrain either the image of 
 the evaluation map  $\ev_k f: = \bigl(f(z_1),\dots, f(z_k)\bigr)\in M^k$  or  the topological type of the domain.  In both cases, it 
 should be easy to modify the construction.   Here we indicate very briefly what needs to be done. For more details see \cite{Castell2}.
\MS

\NI {\bf Adding homological constraints from $M$.}

Let $Z_c = Z_1\times \cdots\times Z_k\subset M^k$ be a closed product submanifold representing a homology class $c\in H_{\dim c}(M^k)$ and consider
\begin{equation}\label{eq:XZc}
X_{Z_c}: = \oMm_{0,k}(M,J,A; Z_c): = \bigl\{[\Si,\bz,f]\in \oMm_{0,k}(M,J,A)\ \big| \ ev_k(f)\in Z_c\bigr\}.
\end{equation}
Then if $d: = \ind(A)$ is the formal dimension $2n + 2c_1(A) + 2k-6$ of $\oMm_{0,k}(M,J,A)$, its subset $X_{Z_c}$ has formal dimension $d + \dim c-2n = d-\codim c$.    We can form a chart  for $X_{Z_c}$ 
near $\tau_0: = [\Si_0,\bz_0,f_0]\in X_{Z_c}$ by modifying the requirement that $E$ satisfy condition (*) as follows.

If $z_i\in (S^2)_{\al(i)}$,  consider the subspace
$$
\Dd_{c}: = \bigl\{(\xi_\al)_{\al\in T}\in \Dd_0 \ \big| \  \xi_{\al(i)}(z_i)\in  \rT_{f_0(z_i)} Z_i\  \forall i\bigr\},
$$
where $\Dd_0$ is as in \eqref{eq:Ldelbar}.
Replace condition (*)  (stated just after  \eqref{eq:Ldelbar}) by the condition
\begin{itemize}\item[$(*_c)$]
the elements in the image of
$\la: E\to \C^{\infty}\bigl(\Hom_J^{0,1}(\Cc|_W\times M)\bigr)$ restrict on ${\rm graph }f_0$ to a subspace of 
$\prod_\al L^p(\Hom^{0,1}_J((S^2)_{\al\in T}, f_{0,\al}^*(\rT M))$ that covers the cokernel of 
$\rd_{f_0}(\pbar)|_{\Dd_c}$.
\end{itemize}
Now consider the unsliced domain $\Hat U$ defined as in \eqref{eq:Hat U}.
Condition $(*_c)$ on $E$ implies that the  linearization
\begin{align}\label{eq:cond*c}\notag
 \rd_{f_0}( \ev_{node}\times\ev_k\times  \pbar):  & \prod_{\al\in T} W^{1,p}((S^2)_\al,f_{0,\al}^*(\rT M))\\
&\qquad
\longrightarrow
 (\rT M)^{2K +k} \times \prod_{\al\in T} L^p(\Hom^{0,1}_J((S^2)_{\al}, f_{0,\al}^*(\rT M))
\end{align}
is transverse to the product of the appropriate $2K$-dimensional diagonal with $Z_c$.    Hence, there is an open neighbourhood of
$(\vec 0,\vec w_0, \bz_0, f_0)$ in 
$$
\Hat U_c: =\bigl \{(\vec e,\vec w, \bz, f)\in \Hat U \ | \  \ev_k(f)\in Z_c\bigr\}
$$ 
that is a manifold of  dimension $\dim (\Hat U)  - \codim (c)$.
The rest of the construction goes through as before, giving  a $0$-dimensional atlas $\Kk_c$ on $X_{Z_c}$, whose virtual class $[X_{Z_c}]^{vir}_\Kk$ is a rational number.   Note that if we use a weak gluing theorem as in Remark~\ref{rmk:stratX} then this process gives a weakly SS atlas as described in \S\ref{ss:SS}, which has a fundamental class by Proposition~\ref{prop:VFCstrat}.

 If $c= c_1\times\cdots\times  c_k\in H_*(M^k)$,
then this number is just 
the Gromov--Witten invariant $\langle c_1,\dots,c_k\rangle_{0,k,A}\in \Q$.
If one has an appropriate gluing theorem, one can also form $[X]^{vir}_\Kk$, where 
$X: = \oMm_{0,k}(M,J,A)$.  Because   $[X]^{vir}_\Kk\in H_*(X,\Q)$ it pushes forward 
by the evaluation map $\ev: X\to  M^k$  to a homology class and one can also define this invariant using
the intersection product in $M^k$:
$$
\langle c_1,\dots,c_k\rangle_{0,k,A}: = \ev_*([X]^{vir}_\Kk)\cdot Z_c.
$$
It 
is proved in \cite{Castell2} that these two definitions agree.

\MS

\NI  {\bf Restricting the domain of the stable maps.}

The easiest way to restrict the domain of a stable map is to specify a minimum number of nodes.  For example, 
consider the space $X_{\ge s}$ of elements in $\oMm_{0,k}(M,J,A)$ whose (stabilized) domain has at least $s$ nodes, where $s\ge 1$.
In this case, the above construction builds an atlas 
that has all the  required properties except that its domains may no longer 
be smooth manifolds.
Rather they are stratified spaces with local models of the form 
$$
\R^k\times (\C^{\un})_s = \{(x; a_1,\dots,a_n): \#\{i: a_i= 0\}\ge s\}.
$$
(See  Example~\ref{ex:RC}.)   
This is a  stratified space with smooth strata of even codimension that we label by the number of nodes.\footnote
{
We forget the finer stratification $\Tt^n$ on $\R^k\times \C^{\un}$ since this does not extend in any natural way to $X$.}
If we require that the domains $U_I$ of the Kuranishi charts 
are locally of this form, and that all group actions and coordinate changes respect this stratification, 
we can define an atlas  on $X_{\ge s}$ as before.
Further, if we  assume that our gluing theorem provides charts that are at least  $\Cc^1$-smooth,  
we can construct   $\Cc^1$-smooth  perturbations  $\nu$,  so that
the stratawise  transversality condition  considered  in \S\ref{ss:SS} is open.
The zero set of $\s+\nu$  will no longer be a (branched) 
manifold, but rather a  (branched) stratified space with strata of codimension at least $2$.  Such a space (if oriented) still carries a fundamental class.  Hence all the arguments go through as before, and as in  Theorem B one obtains a well defined class $[X_{\ge s}]^{vir}_{\Kk}$ in
$\check {H}_*(X_{\ge s})$.

Another way to calculate an invariant for $X_{\ge s}$ is to build a 
Kuranishi atlas $\Kk$ for the whole of $X$ in the standard way, together with
a reduction $\Vv$ and 
transverse perturbation $\nu:\bB_{\Kk}|_{\Vv}\to \bE_{\Kk}|_{\Vv}$ that is constructed so as to satisfy 
 the  transversality condition in  Definition~\ref{def:SS1}.
Then consider the
part $|\bZ|_{\nu, s}$ of the zero set $\bigl(\s+\nu^{-1}(0)\bigr)|_{{\ge s}}$
consisting of elements  in strata at level at least $s$, 
i.e. the domains of the maps have at least $s$ nodes.
Similarly, $|\Kk|$ is  stratified with strata $|\Kk|_t$ consisting of elements whose domains have $t$ nodes, and by construction
$|\bZ|_{\nu, s}$  is the transverse intersection of $|\bZ|_{\nu}$ with $|\Kk|_{\ge s}$.
Hence 
 $|\bZ|_{\nu, s}$ represents a homology class in $H_*(|\Kk|_{\ge s})$, and 
by taking a suitable inverse limit 
it should follow  that one obtains
%
a well defined
 class in $\check {H}_*(X_{\ge s})$, that agrees with the one 
constructed earlier.

Details of (a generalization of) this argument 
appear in \cite{Castell2}.

\MS

\NI{\bf Gromov--Witten atlases.}

In the construction of  \S\ref{ss:GW} each basic chart $\bK_i$ depends on the choice of the following data
$\bigl(\tau_i: = [\Si_i,\bz_i,f_i], Q_i,(D^\ell_i)_{1\le \ell\le L_i}, \bP_i,E_i, \la_i\bigr)$, while
 the information recorded in the atlas is the tuple $(U_i,E_i, \Ga_i, s_i, \psi_i)$.  
The center $\tau_i$ and normalization $\bP_i$ are used to define coordinates over the domain $U_i$ of the chart, 
and, although needed whenever one wants to argue analytically, are not  essential to the description of the chart 
since one can define a coordinate free version of a chart.
All the other information is needed when constructing compatible transition charts, but appears in the atlas only indirectly:
 the maps $\la_i$ appear via the section $s_i$ while  the slicing manifolds  $Q_i$ and associated disc structures appear via 
the set $\bw_i$ of added marked points and hence are seen in the group actions.

\begin{defn}\label{def:GWA} 
Let $X\subset \oMm_{0,k}(M,A,J)$ be a compact subset.
We will say that an atlas $\Kk$ on $X$ is a {\bf  Gromov--Witten atlas} 
if it has been constructed either
 by the
procedure described in \S\ref{ss:GW}  or by one of the above variants.
\end{defn}

Because the GW construction allows one to construct a sum of any family of charts with intersecting footprints,
any two GW atlases on a given set $X$ are directly commensurate and hence concordant (i.e. cobordant over $X\times [0,1]$); see  Definition~\ref{def:commen}.

However, as we will see in  \S\ref{ss:nontriv} and~\S\ref{ss:hybrid}  it is sometimes  useful
to perform abstract constructions on atlases.   Hence it is useful to develop a notion of 
maps between atlases in order to compare different constructions. 

\begin{defn} \label{def:Kumap}   Let $\Kk, \Kk'$ be smooth weak atlases on $X, X'$ respectively, where $X$ is a compact subset of $X'$.
We say that
 a tuple $\f:= \bigl(\ell, (f_I,\Hat f_I, f^\Ga_I)_{I\in \Ii_\Kk}\bigr) $ 
 is a {\bf map from $\Kk$ to $ \Kk'$} if the following conditions hold:
\begin{itemlist}\item there is an injection $\ell: \{1,\dots,N\}\to \{1,\dots, N'\}$ such that $F_i\subset F'_{\ell(i)}$ for all $1\le i\le N$; the associated
injection 
$\Ii_\Kk\to \Ii_{\Kk'}$ is also denoted $\ell$;
\item
for each $I\in \Ii_\Kk$ the triple $(f_I, \Hat f_I, f^\Ga_I)$ consists of   
\begin{itemize}\item  a linear map   $\Hat f_I: E_I\to E'_{\ell(I)}$;
\item an injective group homomorphism  $f^\Ga_I: \Ga_I\to \Ga'_{\ell(I)}$;
\item a $f^\Ga_I$-equivariant smooth map $f_I: U_I\to U'_{\ell(I)}$ such that 
\begin{itemize}\item[-]   $f^\Ga_I$ maps the stabilizer subgroup $\Stab_x\subset \Ga_I$ onto the
stabilizer $\Stab'_{f_I(x)}\subset \Ga'_{\ell(I)}$ for each $x\in U_I$;
\item[-]  $f_I$ is compatible 
with the section and footprint maps, i.e.
$s'_{\ell(I)}\circ f_I = \Hat f_I\circ s_I$ and $\psi'_{\ell(I)}\circ f_I = \psi_I$ on $s_I^{-1}(0)$;  
\end{itemize}
\end{itemize}
\item  $(f_I, \Hat f_I, f^\Ga_I)$ are compatible with coordinate changes, i.e. for all $I\subset J$, 
\begin{itemize}\item[-]  $f_J(\TU_{IJ}) \subset \TU_{{\ell(I)}{\ell(J)}}$; \vspace{.05in}
\item[-]  $\Hat f_J\circ \Hat\phi_{IJ} = \Hat\phi'_{\ell(I)\ell(J)}\circ \Hat f_I: E_I\to E'_{\ell(J)}$;
  \vspace{.05in}
 \item[-]  $(\rho^\Ga_{\ell(I)\ell(J)})' \circ f^\Ga_J =  f^\Ga_I \circ \rho^\Ga_{JI}: \Ga_J\to \Ga'_{\ell(I)}$;  and  \vspace{.05in}
  \item[-]  $(\rho_{\ell(I)\ell(J)})' \circ f_J =  f_I \circ \rho_{JI}: \TU_{IJ}\to  U'_{\ell(I)\ell(J)}$.
\end{itemize}
\end{itemlist}
Further, we say that $\f$ is an {\bf embedding} if each map $f_I$ is an injection onto a smooth submanifold of $U'_{\ell(I)}$ and each $\Hat f_I$ is an isomorphism, and is an {\bf open embedding} if in addition $f_I(U_I)$ is an open subset of $U'_{\ell(I)}$.  
\end{defn}

For example, each shrinking $\Kk_{sh}$ of $\Kk$ gives rise to an open embedding $\Kk_{sh}\to \Kk$ with the maps $\ell, \Hat f_I, f^\Ga_I$ all equal to the identity.

If   $\Kk, \Kk'$ are atlases (rather than weak atlases), each  map $\f$ gives rise to  functors $\f:\bB_\Kk\to \bB_{\Kk'}$  and $\f_{\bE}:\bE_\Kk\to \bE_{\Kk'}$ 
that commute with the structural functors $\s, \pr$ as well as being compatible with the homeomorphisms $\io_\Kk:X\to |\s_\Kk^{-1}(0)| $ and $\io_{\Kk'}:X'\to |(\s'_{\Kk'})^{-1}(0)| $.

The definition above is very general  in order that it apply in a variety of circumstances.  
However to get a useful notion one should add enough conditions 
so that there is an appropriate relation between $[X]^{vir}_\Kk$ and  $[X']^{vir}_{\Kk'}$.
For example, we have the following easy result.

\begin{lemma}  If 
 $X=X'$ and $\f:\Kk\to \Kk'$ is an open embedding, then $[X]^{vir}_{\Kk} = [X]^{vir}_{\Kk'}$.
 \end{lemma}
 \begin{proof} 
Any reduction $\Vv=(V_I)_{I\in \Ii_\Kk}$  of $\Kk$ pushes forward to a
 reduction $\f(\Vv)$ of $\Kk'$ for which 
 $V'_{\ell(I)} = f_I(V_I)$ for $I\in \Ii_\Kk$, while
 $F_J'=\emptyset=V_J'$ for all $J\not\in \ell(\Ii_{\Kk})$.  
Because the maps $\Hat f_I$  are linear isomorphisms,
any (admissible transverse, precompact)
 perturbation section  $\nu$ of $\Kk|_\Vv$ pushes forward to a similar  perturbation section of $\Kk'$. Thus these two atlases define the same virtual class on $X$.
 \end{proof}
 
We also note the following.  
\begin{itemlist}
\item  Since we do not require the group homomorphisms $f^\Ga_I$  to be surjective,
the image $f_I(U_I)$ need not be the full inverse image $(\pi_I')^{-1}(\ul{f_I(U_I)})$.
If $f_I$ is an embedding then $(f_I,f_I^\Ga): (U_I,\Ga_I)\to (U_I',\Ga_I')$ is a group embedding in the sense of Definition~\ref{def:GWA}.
\item  Consider the inclusion $X_{Z_c}\hookrightarrow X$ in \eqref{eq:XZc} above.  Then one can construct compatible GW atlases  
$\Kk_c$ on $X_{Z_c}$ and $\Kk$ on $X$ with indexing set $\Ii_{\Kk_c}\subset \Ii_\Kk$
so that there is an embedding $\f: \Kk_c\to \Kk$ for which $\ell$ is the inclusion and 
the maps $\Hat f_I, f^\Ga_I$ are   the identity; see Castellano~\cite{Castell2}.
\item   There are analogous definitions if $\Kk, \Kk'$ are (filtered) topological atlases or if they are weakly SS.  
\item In \S\ref{ss:hybrid} we will consider maps between atlases with more general indexing sets for which
$\ell:\Ii_\Kk\to \Ii_{\Kk'}$ and the $\Hat f_I$ are not injective; see  Remark~\ref{rmk:fK}.
 \end{itemlist}

%
%
%
%

\section{Examples}\label{sec:ex}

In this section, we discuss some examples.
We begin with the case of orbifolds, showing that each orbifold has a Kuranishi atlas with trivial obstruction spaces.
We also show that each such atlas has a canonical completion to a groupoid that represents the orbifold.
Hence the atlas is a stripped down model for the groupoid that captures all essential information.
The zero set construction in Proposition~\ref{prop:zero} can then be interpreted as a simpler version of the resolution construction given in \cite{Mbr}.

 We next show how to use atlases to compute Gromov--Witten invariants in some easy cases, for example if the moduli space is an orbifold with cokernels of constant rank so that the obstruction spaces form a bundle of constant rank.
Finally, we revisit an argument in \cite{Mcq} about the Seidel representation for general symplectic manifolds  
The \lq\lq proof" given there assumed the existence of a construction for the VFC
with  slightly different properties from the construction using Kuranishi atlases, and does not work with the new definitions.  
However, it is not hard to give a precise proof using the current definitions.

\subsection{Orbifolds}\label{ss:orb}

The aim of this section is to sketch the proof of Proposition C stated in \S\ref{s:intro}, i.e. to show that 
every compact orbifold $Y$ is the realization of  a Kuranishi atlas with trivial obstruction spaces. 
We will take a naive approach to orbifolds, since that suffices for our current purposes.  Thus, as in 
Moerdijk~\cite{Moe} and~\cite{Mbr},  we consider them as equivalence classes of groupoids rather than 
as stacks as in Lerman~\cite{Ler}.

Recall that an
 {\bf  ep (\'etale proper) groupoid} $\bG$ is  a topological category with smooth
spaces of objects $\Obj_\bG$ and morphisms $\Mor_\bG$, such that
\begin{itemize}\item  all structural maps (i.e. source $s$, target $t$, identity, composition and inverse)
 are {\bf  \'etale} (i.e. local diffeomorphisms);
and \item
 the map $s\times t: \Mor_\bG\to \Obj_\bG\times \Obj_\bG$
given by taking a morphism to its source and target is {\bf proper}.  
\end{itemize}
The realization $|\bG|$ of $\bG$ is the quotient of the space of objects by the equivalence relation 
given by the morphisms: thus $x\sim y \ \Leftrightarrow\ \Mor_\bG(x,y)\ne \emptyset$.
When (as here) the domains are locally compact, the
 properness condition implies that $|\bG|$ is Hausdorff. 


\begin{defn}\label{def:orbstr}
A $d$-dimensional  {\bf orbifold structure} on a 
 paracompact Hausdorff 
 space $Y$ is a pair $(\bG, f)$ consisting of a $d$-dimensional ep (\'etale proper) groupoid $\bG$ together with a 
map $f:\Obj_\bG\to Y$ that factors through a homeomorphism $|f|:|\bG|\to Y$.  Two orbifold structures $(\bG, f)$ and
$(\bG', f')$ are {\bf Morita equivalent} if they have a common refinement, 
i.e. if there is a third structure $(\bG'', f'')$ on $Y$ and functors $F:\bG''\to\bG, F':\bG''\to \bG'$ such that 
$f'' = f\circ F = f'\circ F'$.
  
  An {\bf orbifold} is a second countable
paracompact
Hausdorff space $Y$ equipped 
with an equivalence class of orbifold structures.  We say that $Y$ is oriented if for each representing groupoid $\bG$
 the spaces  $\Obj_\bG$  and  $\Mor_\bG$ have  orientations that are preserved by  all structure maps and by the functors $F:\bG\to \bG'$ considered above.
\end{defn}

\begin{defn}\label{def:orbat}  We say that a compact Hausdorff space $Y$ has a {\bf strict orbifold atlas} $\Kk$ if $Y $ has an open covering $Y= \bigcup_{i=1,\dots,N} F_i$ such that
 the following conditions hold with
$$
\Ii_Y: = \bigl\{\emptyset \ne I\subset  \{1,\dots,N\}:F_I: ={\textstyle  \bigcap_{i\in I} }F_i\ne \emptyset\bigr\}.
$$
\begin{itemlist} 
\item[\rm (i)]
For each $I\in \Ii_Y$  
there is a   manifold $W_I$ on which the finite group $\Ga_I: = \prod_{i\in I} \Ga_i$ acts and 
a map $\psi_I: W_I\to F_I$ that induces a homeomorphism $\upsi_I: \qq{W_I}{\Ga_I}\to F_I$;
\item[\rm (ii)] for all (nonempty) subsets $I\subset J$  
there is a group covering 
$$
(\rho_{IJ},\rho_{IJ}^\Ga): (W_J,\Ga_J)\to \bigl(W_{IJ}: = (\psi_I)^{-1}(F_J), \Ga_I\bigr)
$$
in the sense of Definition~\ref{def:cover}, where $\rho_{IJ}^\Ga: \Ga_J\to \Ga_I$ is the projection  and 
$\rho_{IJ}$ is the composite of the quotient of $W_J$ by a free action of 
$\ker \rho_{IJ}^\Ga = \Ga_{J\less I}$  with a $\Ga^I$-equivariant homeomorphism $\qu{W_J}{\Ga_{J\less I}}\to W_{IJ}\subset U_I$;
\item[\rm (iii)]    $\psi_I\circ \rho_{IJ} = \psi_J$, and $\rho_{IJ}\circ \rho_{JK} = \rho_{IK}$ for all $I\subset J\subset K$.
\end{itemlist}
Thus the charts of this atlas $\Kk$ are the tuples $\bigl(\bK_I: = (W_I,\Ga_I,\psi_I)\bigr)_{I\in \Ii_Y}$ with footprints 
$(F_I)_{I\in \Ii_Y}$, and the coordinate changes are induced by the covering maps $\rho_{IJ}$.
\end{defn}

\begin{rmk}\rm  (i)  
Note that in an atlas with trivial obstruction spaces the section maps are identically zero  so that the footprint maps
$\psi_I:U_I\to F_I$ are defined on the whole domain and
induce  homeomorphisms $\upsi_I: \uU_I\to X$.  In particular,   the lifted domains
$\TU_{IJ}\subset U_J$ of the coordinate changes equal the whole of $U_J$, which explains why we assume above that
 the covering map $\rho_{IJ}$ is defined on the whole of $W_J$.
Thus the above definition is precisely that of a Kuranishi atlas with trivial obstruction spaces.
\MS

\NI (ii)
One could also allow a more general choice of groups $\Ga_I$ as in Remark~\ref{rmk:tamea}. However, the proof of Proposition~\ref{prop:orb} shows that the product groups considered above are natural to the problem. \hfill$\er$
\end{rmk}

Such an atlas satisfies the strong cocycle condition, and is oriented if the domains are equipped with an orientation
that is preserved by the group actions and other structure maps.
The corresponding category $\bB_\Kk$ has realization $Y$, 
since in this case the intermediate category $\ubB_\Kk$ (see   Lemma~\ref{le:und}) is the category defined by the covering $(F_i)_{i=1,\dots,N}$ of $Y$, i.e. its objects 
are $\bigsqcup_{I\in I} F_I$  and there is a morphism $(I,x)\to (J,y)$ iff $I\subset J$ and $\psi_I(x) =\psi_J(y)\in F_J$.
Although $\bB_\Kk$ is not a groupoid since the nonidentity maps are not invertible, 
we now show that this category has a canonical {\bf groupoid completion} $\bG_\Kk$. (This justifies our language since it means that any compact Hausdorff space $Y$ with a strict orbifold atlas is in fact an orbifold.)

\begin{defn}\label{def:groupc}  Let $\bM$ be an \'etale proper  category with objects $\bigsqcup_{I\in \Ii} W_I$ 
and realization $Y : = \Obj_\bM/\!\sim$ such that 
\begin{itemize}\item
for each $I\in \Ii$ the full subcategory of $\bM$ with objects $W_I$ 
can be identified with the group quotient $(W_I,\Ga_I)$ for some group $\Ga_I$;
\item for each $I\in \Ii$ the realization map $\pi_\Mm:\Obj_\bM \to Y$ induces a homeomorphism
  $\qu{W_I}{\Ga_I}\to F_I\subset Y$, where $F_I$ is an open subset of $Y$.
\end{itemize}
 Then we say that an ep groupoid $\bG$ is a {\bf groupoid completion} of $\bM$ if
 there is an injective functor $\io:\bM\to \bG$ that induces a bijection on objects and 
 a homeomorphism on the realizations $Y= |\bM|\to |\bG|$. 
\end{defn}

Thus 
for each component $W_I$ of $\Obj_\bM$ the groupoid completion (if it exists)
has the same morphisms from $W_I$ to $W_I$ but (unless  $\bM$ is already a groupoid) will have more morphisms between the different components of $\Obj_\bM$ that are obtained by adding inverses and composites.

\begin{prop}\label{prop:groupcomp}  Let $\Kk = \bigl(W_I,\Ga_I, \rho_{IJ}\bigr)_{I\subset J, I,J\in \Ii_\Kk}$ be a strict orbifold atlas as above. Then the category $\bB_\Kk$ 
has a canonical completion to an ep groupoid $\bG_\Kk$ with the same objects and realization as $\bB_\Kk$ and morphisms
$$
\Mor_{\bG_\Kk} = \bigsqcup_{I,J\in \Ii_\Kk: I\cup J\in \Ii_\Kk} \Mor_{\bG_\Kk}(W_I, W_J), \quad \Mor_{\bG_\Kk}(W_I, W_J): = W_{I\cup J}\times \Ga_{I\cap J},
$$
where $\Ga_\emptyset: = \id$, and with the following structural maps.
\begin{itemize}\item[\rm (i)]  The source and target maps $s\times t:  W_{I\cup J}\times \Ga_{I\cap J}\to W_I\times W_J$ are
$$
(s\times t) \bigl(I,J,z,\ga \bigr) \;=\;  \Bigl( \bigl(I, \ga^{-1}\rho_{I(I\cup J)}(z)\bigr) \,,\, \bigl( J, \rho_{J(I\cup J)}(z) \bigr) \Bigr).
$$
\item[\rm (ii)] Composition is given by
\begin{align*}
&\ m:  \Mor_{\bG_\Kk}(W_I, W_J) \, _t\times_s \Mor_{\bG_\Kk}(W_J, W_K)\to \Mor_{\bG_\Kk}(W_I, W_K), \\
&\ \bigl((I,J,z,\ga), (J,K,w,\de)\bigr) \mapsto (I,K, v',\al \,\de_{IJK}\ga_{IJK})\in W_{I\cup K}\times \Ga_{I\cap K},
\end{align*}
where 
\begin{align*}
&\ v': = \rho_{I\cup K,I\cup J\cup K}(v)\in W_{I\cup K},\\
&\ \ga=(\ga_{IJK},\ga_{IJ\less K})\in \Ga_{I\cap J\cap K}\times \Ga_{(I\cap J)\less K} = \Ga_{I\cap J},\\ 
&\ \de=(\de_{IJK},\de_{JK\less I})\in \Ga_{I\cap J\cap K}\times \Ga_{(J\cap K)\less I}= \Ga_{J\cap K}, 
\end{align*}
and
$(v,\al)\in W_{I\cup J\cup K}\times \Ga_{(I\cap K)\less J}$ is the unique pair  such that 
$$
\rho_{I\cup J, I\cup J\cup K}(v) = \ga_{IJ\less K}^{-1}\al\,\de z, \;\; \rho_{J\cup K, I\cup J\cup K}(v) =  \ga_{IJ\less K}^{-1} w.
$$
\item[\rm (iii)] The inverse is given by
$$
\io: \Mor_{\bG_\Kk}(W_I, W_J) \to \Mor_{\bG_\Kk}(W_J, W_I), \quad  \bigl(I,J,z,\ga \bigr)\mapsto (J,I,\ga^{-1} z, \ga^{-1}).
$$
\end{itemize}
\end{prop}
\begin{proof}[Sketch of proof]  The above formula  for $\Mor_{\bG_\Kk}(W_I,W_J)$ agrees with the definition of 
 $\Mor_{\bB_\Kk}(W_I,W_J)$ when $I\subset J$.  We must extend this definition to all pairs 
 $I,J$ with $F_I\cap F_J\ne \emptyset$ (or equivalently $I\cup J\in \Ii_\Kk$) 
 so as to be consistent with the footprint maps and the local group actions.  In particular, we require
 that for all $x\in W_I, y\in W_J$
\begin{itemize}\item  $\Mor_{\bG_\Kk}(x,y)\ne \emptyset$ iff $(I,x)\sim (J,y)$ in $\Obj_{\bB_\Kk}$
 iff $\psi_I(x) = \psi_J(y)\in F_{I\cup J}$.
\end{itemize}
To check this,
note first that $\Mor_{\bG_\Kk}(x,y)\ne \emptyset$ implies that $x,y$ have the same image in $Y$.  
Conversely,
suppose  given 
 $x\in W_I,y\in W_J$ with  $I\not\subset J$ and such that $\psi_I(x) = \psi_J(y)$.
%
Since $\rho_{J(I\cup J)}: W_{I\cup J}\to W_{J(I\cup J)}: = \psi^{-1}(F_{I\cup J})$ is surjective 
and factors out by the free action of 
$\Ga_{I\less J}$
we may choose $z \in W_{I\cup J}$ so that $\rho_{J(I\cup J)}(z) = y$.  Then  
$\rho_{I(I\cup J)}(z)$ lies in the $\Ga_{I}$-orbit of $x$ because $\psi_{I\cup J}(z) = \psi_I(x)$, so that
by replacing $z$ by $\de z$ for some 
$\de\in \Ga_{I\less J}$ we may arrange  that $\rho_{I(I\cup J)}(z)$ lies in the $\Ga_{I\cap J}$-orbit of $x$,
where $\Ga_{I\cap J}: = \id$ if $I\cap J = \emptyset$.
Therefore there is a pair $(z,\ga)\in (W_{I\cup J}, \Ga_{I\cap J})$ with 
$\rho_{I(I\cup J)}(z) = \ga x, \rho_{J(I\cup J)}(z) =y$ and hence with source $x$ and target $y$, as required.
Hence the equivalence relation on  $\Obj_{\bB_\Kk}$ generated by the morphisms in $\bB_\Kk$ agrees with 
that induced by the morphisms in  $\bG_\Kk$, which implies that the induced map $|\bB_\Kk|\to |\bG_\Kk|$ is a homeomorphism.


To complete the proof that $\bG_\Kk$ is a category we must check that  the composition operation defined on
$\Mor_{\bB_\Kk}$ extends to a well defined and associative multiplication operation.
 Notice that each morphism $(I,J,z,\ga)\in W_{I\cup J}\times W_{I\cap J}$ may be written as a composite  $m_1\circ (m_2)^{-1}$: 
 $$
 (I,\ga^{-1}x)\stackrel{m_1}\mapsto (I\cup J, z)\stackrel{(m_2)^{-1}}\mapsto (J,\rho_{(I\cup J)J}(z)),
 $$
 where $m_1 = (I,I\cup J,\ga,z)$ and 
 $m_2 = 
(J, I\cup J, \id,z)$ are morphisms in $\bB_\Kk$.   Therefore to define the composite of  two such morphisms we must 
find an element $v\in W_{I\cup J\cup K}$ together with an appropriate group element
that completes the  following diagram, where for simplicity we omit the subscripts for 
the projections $\rho$  and label the morphisms  by the relevant group elements:
$$
\xymatrix
{
(I,\ga^{-1} x)\quad  \ar@{->}[r]^{\ga\;\;\;\;}  &\qquad  (I\cup J, z)\qquad \ar@{-->}[r]
  & (I\cup J\cup K, v)\\
  & \bigl(J,\rho(z) = \de^{-1}\rho(w)\bigr) \ar@{->}[u] \ar@{->}[r]^{\de} & \bigl(J\cup K,w=\rho(v))\ar@{->}[u]\\
  & &  (K,\rho (w))\ar@{->}[u]
}
$$
It is not hard to check that the formula for $v$ given in (ii) has the required properties.  Details are in \cite{Morb}.
Moreover, it follows its definition that it is 
  a local diffeomorphism, and is compatible with the inverse operation in (iii).
  
  If the group actions are effective, we can now check indirectly that $m$ is associative as follows.
  (The proof in the general case is given in \cite{Morb}.)
For if not, there are  objects $x,y$ and composable morphisms $f,g,h$ such that $f\circ (g\circ h) \ne (f\circ g)\circ h
  \in \Mor(x,y)$.    
  Since $m$ is a local diffeomorphism and the sets $\Mor(x,y)$ are finite, this identity must continue to hold 
  as $x$ and hence $f,g,h,y$ vary.   Therefore we can identify $f\circ (g\circ h) \circ \bigl((f\circ g)\circ h\bigr)^{-1}\in \Mor(x,x)$ with a group element $\ga\in \Ga_I^x$ that must remain fixed as $x$ varies. 
    Therefore $\ga$ must belong to a component of morphisms 
  that are not equal to identity morphisms but yet have  $s=t$. 
Since the existence of such $\ga$ contradicts the effectivity hypothesis,  it follows that $m$ is associative.
\end{proof}

We now sketch the proof that every orbifold  has a strict orbifold atlas.  Recall from Definition~\ref{def:commen} that two atlases $\Kk, \Kk'$ are said to be 
 {\bf directly commensurate}
if they are subatlases of a common atlas $\Kk''$ on $X$, and {\bf commensurate} if there is a finite 
sequence of pairwise directly commensurate atlases starting with $\Kk$ and ending at $\Kk'$.

\begin{prop}\label{prop:orb}  Every compact orbifold $Y$ has a strict orbifold atlas $\Kk$  whose associated groupoid 
$\bG_{\Kk}$ is an orbifold structure on $Y$.   Moreover, there is a bijective correspondence between commensurability classes of such  atlases and Morita equivalence classes of ep groupoids.
\end{prop}
\begin{proof} Let $\bG$ be an ep groupoid with footprint map $f: \Obj_\bG\to Y$.
Our first aim is to construct an atlas $\Kk$ on $Y$ 
together with a functor $\Ff: \bB_\Kk\to \bG$ that covers the identity map on  $Y$ and hence extends to
an equivalence from the groupoid completion
 $\bG_\Kk$  to $\bG$.

 By Moerdijk~\cite{Moe}, each point in $Y$ is the image of a group quotient that embeds into $\bG$.
 Therefore since $Y$ is compact we can find 
 a finite set of basic charts $\bK_i: = 
\bigl(W_i,\Ga_i,\psi_i\bigr)_{1\le i\le N}$ on $Y$ whose footprints 
$(F_i)_{1\le i\le N}$ cover $Y$,
together with embeddings 
$$
\si:{\textstyle \bigsqcup_iW_i\hookrightarrow \Obj_{\bG},\quad  \Tsi:\bigsqcup_i W_i\times \Ga_i\hookrightarrow \Mor_\bG}
$$
  that are compatible in the sense that the following diagrams commute:
$$
\xymatrix
{
W_i\times \Ga_i  \ar@{->}[d]_{s\times t} \ar@{->}[r]^{\Tsi_i} &  \Mor_{\bG}\ar@{->}[d]_{s\times t}    \\
W_i\times W_i \ar@{->}[r]^{\si_i} & \Obj_{\bG}\times \Obj_{\bG},
} \qquad\qquad 
\xymatrix
{
W_i  \ar@{->}[d]_{\psi_i} \ar@{->}[r]^{\si_i} &\Obj_{\bG}  \ar@{->}[d]_{f}   \\
Y \ar@{->}[r]^{\id} & Y,
}
$$
We claim that there is a Kuranishi atlas $\Kk$ with these basic charts whose footprint maps $\psi_I$ extend 
 $f\circ \si: \bigsqcup_i W_i \to Y$. 
 
  To see this, we first consider the sum of two charts.
Given $I: = \{i_0, i_1\}$ with $F_I\ne \emptyset$, order its elements so that $i_0<i_1$ and
consider the set $$
W_I: = W_{\{i_1,i_0\}}: = \Mor_{\bG}(\si(W_{i_0}),\si(W_{i_1})): = (s\times t)^{-1}\bigl(\si(W_{i_0})\times \si(W_{i_1})\bigr)
$$
of morphisms in $\bG$ from $\si(W_{i_0})$ to $\si(W_{i_1})$. 
Then $W_I$ is the inverse image of an open subset of $\Obj_{\bG}\times \Obj_{\bG}$, hence open in 
$ \Mor_{\bG}$, and thus a smooth manifold. 
 Since the points in $f^{-1}(F_I)\cap \si(W_{i_0})$ are identified with points in $f^{-1}(F_I)\cap \si(W_{i_1})$
 by morphisms in $\bG$,
 the restrictions of $s,t$ to $W_I$ have images
 $$
 s(W_I) = f^{-1}(F_I)\cap \si(W_{i_0}),\quad  t(W_I)= f^{-1}(F_I)\cap \si(W_{i_1}).
 $$
 Moreover, for any $x\in s(W_I)$ and $\al\in \Mor_{\bG}(x,y)\in W_I$, we have
 $$
 s^{-1}(x)\cap W_I \cong \Mor_\bG\bigl(t(\al),\si(W_{i_1})\bigr)\cong \Ga_{i_1},
 $$
 where the second isomorphism holds because  by assumption
$f\circ \si_{i_1}$  is the footprint map $\psi_{i_1}: W_{i_1}\mapsto \qu{W_{i_1}}{\Ga_{i_1}}\cong F_{i_1}$.
Rephrasing this  in terms of the action 
of the group $\Ga_I: = \Ga_{i_1}\times \Ga_{i_0}$  on $\al\in W_{I}$ by 
 $$
(\ga_{i_1},\ga_{i_0})\cdot \al = \Tsi(\ga_{i_1})\circ \al\circ \Tsi(\ga_{i_0}^{-1}),
$$ 
one finds that  $\Ga_{i_1}$ acts freely on $W_I$ and that
the  source map $s:  W_I\to  \si(W_{i_0})$
induces a diffeomorphism $\qu{W_I}{\Ga_{i_1}}\to  \si(W_{i_0})\cap f^{-1}(F_I)$.
Similarly, $\Ga_{i_0}$ acts freely, and the target map
$t:  W_I\to  \si(W_{i_1})$
induces a diffeomorphism $\qu{W_I}{\Ga_{i_0}}\to  \si(W_{i_1})\cap f^{-1}(F_I)$.  Since the footprint map
for the chart  $W_{i}$ factors out by the  action of $\Ga_i$, the same is true for this sum chart:
in other words the footprint map
$\psi_I: W_I\to Y,  \al\mapsto f\bigl( s(\al)\bigr) = f\bigl( t(\al)\bigr)$  induces a homeomorphism $\qu{W_I}{\Ga_I}
\stackrel{\cong}\to F_I$. Therefore $W_I$ satisfies all the requirements of a sum 
of two charts.

  To define a transition chart for general $I \in \Ii_Y$,  enumerate its elements as
$i_0<i_1\cdots<i_k$, where $k+1: = |I|\ge 2$ and define
$W_I$ to be the set of composable $k$-tuples of morphisms $(\al_{i_k},\cdots,  \al_{i_1})$, where 
$(s\times t)(\al_{i_\ell}) \in \bigl(\si(W_{i_{\ell-1}}),\si(W_{i_\ell})\bigr)$. 
Then $W_I$
supports an action of $\Ga_I$ given by $$
\ga\cdot (\al_{i_k},\cdots,  \al_{i_1}) = (\al_{i_k},\cdots,  \al_{i_{\ell+1}}\Tsi( \ga)^{-1}, \Tsi(\ga) \al_{i_{\ell}},\cdots,  \al_{i_1}), \quad \ga\in \Ga _{i_\ell}.
$$
Notice that for any $H\subsetneq I$ the subgroup $\Ga_{I\less H}$ acts freely, and the quotient  can be identified with $W_H$ by means of the appropriate partial compositions and forgetful maps. 
Moreover, we define 
%
%
%
 the functor $\Ff: \bB_\Kk\to \bG$  on objects by
$$
W_I\to \Obj_{\bG}, \quad \left\{\begin{array}{ll}  x\mapsto \si(x), &\mbox{ if } I=\{i_0\}, x\in W_{i_0},\\
(\al_{i_k},\cdots,\al_{i_1}) \mapsto t(\al_{i_k})\in \si(W_{i_k})
&\mbox{ if } |I|>1.
\end{array}\right.
$$
Further details of this construction are given in \cite{Morb}.

This sketches why every orbifold has a Kuranishi atlas of the required type. 
 To see that this atlas is unique up to commensurability, note first that 
 because we can start from any covering set of basic charts
any two atlases constructed in this way from the same  groupoid are directly commensurate. 
Further,
 given  groupoid structures  $(\bG,f),(\bG',f')$ on $Y$ with
 common refinement  $F:(\bG'',f'')\to (\bG,f), F':(\bG'',f'')\to (\bG',f')$, 
 all atlases on $Y$ that are constructed from $\bG$ or from $\bG'$ are commensurate to an atlas 
 whose basic charts are
 pushed forward from $\bG''$, and hence they all belong to the same commensurability class.

Conversely, we must show that if $\Kk, \Kk'$  are commensurate,
 the groupoids $\bG_{\Kk}$ and $\bG_{\Kk'}$ are equivalent.
It suffices to consider the case when $\Kk, \Kk'$ 
are  directly commensurate.  But then they are contained in a common atlas $\Kk''$ on $Y$ that defines a groupoid $\bG_{\Kk''}$
that contains both $\bG_{\Kk}$ and $\bG_{\Kk'}$ as subgroupoids with the same realization $Y$.  Thus the inclusions 
 $\bG_{\Kk}\to \bG_{\Kk''}$ and  $\bG_{\Kk'}\to \bG_{\Kk''}$ are equivalences.  This completes the proof.
\end{proof}

\begin{example}\label{ex:gerbe}\rm  (i)  Example~\ref{ex:foot} explains the \lq\lq football", which is a simple example of this construction.  In \cite[Example~3.6]{Morb} we also discuss in detail the
different possible choices of strict orbifold atlas for the noneffective but nontrivial orbifold structure $\bG_{S^2}$ on $S^2$ with group $\Ga: = \Z/2\Z$:
the nontriviality can appear either in the nontriviality of the individual
covering maps $\rho_{IJ}:W_J\to W_I$ (if, for example, $W_I$ is an annulus), or, if
all components of the $W_I$ are contractible, in the relation between the different free actions  of the groups $\Ga_{J\less I}$ on the sets of these components.
\MS

\NI (ii)
If one instead describes an orbifold via a Kuranishi structure as in Remark~\ref{rmk:FOOO1}, then the
 coordinate changes between 
the local charts $(F_p,\Ga_p)$
are related by injections $\phi_{qp}: (F_q,\Ga_q)\to (F_p,\Ga_p)$ that satisfy a compatibility condition 
of the form $\phi_{qp}\circ \phi_{rq} = \ga_{pqr} \phi_{rp}$ for some locally constant map $\ga_{pqr}: F_{pqr}\to \Ga_p$.
In this language, all the local charts for  the nontrivial groupoid $\bG_{S^2}$ in (i) above have the form $(F,\Ga)$, where $\Ga = \Z/2$ acts trivially on $F\subset S^2$, and the twisting is encoded in a suitable choice for the  $\ga_{pqr}$.
For example, one could describe a structure of this type with three basic charts $(F_i,\Ga_i)$ with  $(F_i)_{i=1,2,3}$ 
as above, where the
 coordinate changes $\phi_{ij}: (W_{ij}, \Ga_i)\to (W_j,\Ga_j)$ satisfy the relation
 $\phi_{13} = \ga \phi_{23}\circ \phi_{12}$ where $\ga = \id$ on $F_S\subset F_{123} $ and 
 $\ga \ne \id$ on $F_N\subset F_{123} $.
 \MS
 
 \NI (iii)
 In our treatment, in order to obtain as simple as structure as possible, we consider only atlas coordinate changes.  However Joyce
 defines a Kuranishi space in \cite[\S4]{Joyce} to consist of charts with more elaborate coordinate changes that include both our atlas coordinate changes as well as ones involving group elements  $\ga_{pqr}$. \hfill$\er$
\end{example}

\begin{rmk}\label{rmk:commen}\rm    (i)
The construction in Proposition~\ref{prop:orb} is reminiscent of 
that given in  \cite[\S4]{Mbr} for
the \lq\lq resolution" of an orbifold.  
However, the two constructions have  different aims: here we want to build a simple model for  the orbifold $Y$,
 while there we wanted to construct a {\it resolution}, i.e.
 a corresponding branched manifold with the same fundamental class.
In fact, we can simplify the construction in \cite{Mbr} of such a resolution by first
constructing the atlas $\Kk_Y$ and its corresponding category $\bB_\Kk$ and then removing morphisms from $\bB_\Kk$ 
by the recipe used in Proposition~\ref{prop:zero}  to construct the zero set.

To this end,
 consider a strict  orbifold atlas   $\Kk$  with realization $Y$
and choose a reduction $\Vv$ of the footprint covering.
Since there is no obstruction bundle, the perturbation $\nu$ is identically zero, so that the \lq\lq local zero set"  $Z_I = V_I$ for all $I$.
Thus the wnb category that we will call $(\bV,\La_V)$ constructed in  Proposition~\ref{prop:zero}
has objects $\sqcup_I V_I$ and morphisms and weighting function as described in \S\ref{ss:zero}.  It is a nonsingular subcategory of the groupoid completion
$\bG_\Kk$ of $\bB_\Kk$, such that inclusion induces a map
$$
|\io_\bV|: |\bV|_\Hh\to |\bG_\Kk| = Y, \qquad \  |\io_\bV|_*(\La_V) = \La_Y, \quad\mbox{ where } \La_Y(y) = \tfrac 1{|\Ga^y|}.
$$
Note that this statement holds also in the noneffective case 
(although \cite{Mbr} mostly considered the effective case).\footnote{
In fact, in the orbifold setting it does suffice for these purposes to restrict to the effective case, since, as  is well known,
each connected component  of $Y$ has an effective quotient (see for example 
 \cite[Lemma~2.15]{Mbr}).  
 However, as explained in Example~\ref{ex:zorb} -- and in contrast to the approach in \cite{Mbr} --
  the weighting function should take into account the number of group elements that act trivially.
  }
By \cite[Proposition~4.19]{Mbr},
the fundamental class of  $( |\bV|_\Hh,\La_V)$ is constructed in such a way that $|\io_\bV|$ pushes it forward to the usual fundamental class of $Y$.
\MS

\NI (ii)  The relations between the notions of directly commensurate, commensurate and concordant are not clear even for atlases 
with trivial obstruction bundles.
For example, the
 concordance relation between atlases $\Kk^0,\Kk^1$ on $X$ requires
 there to be an atlas $\Kk^{01}$ over $[0,1]\times X$ with prescribed collar isomorphisms between the 
  restrictions of $\Kk^{01}$ to the collars $[0,\eps)\times X$ and $(1-\eps,1]\times X$ and the 
  product atlases $[0,\eps)\times \Kk^0, (1-\eps,1]\times   \Kk^1$.  However,  there is no requirement 
  on the interior charts (i.e.
those whose footprint does not intersect $ \{0,1\}\times X$) that they are in any way compatible with the product structure,
for example the local
action of the stabilizer group of a point need not decompose as a product.  Hence 
  even when the obstruction bundles 
  are trivial so that the footprint maps $\psi_I$  are defined over the whole of the domains $U_I$,  it is not immediately clear that 
the relation of commensurability for atlases on $X$ with trivial obstruction spaces 
 is the same as the notion of concordance,  though it could well be true.  Since we are using concordance (or more generally  cobordism) simply for convenience, we will not pursue this question further here.
Moreover, because  orbifolds are most properly described in terms of stacks, the \lq\lq correct" equivalence relation 
would no doubt involve the notion of $2$-category as in the work of Joyce.
$\hfill\er$  \end{rmk}

\subsection{Nontrivial obstruction bundles}\label{ss:nontriv}

When calculating Gromov--Witten invariants one often starts with moduli spaces $X$ that have nice geometric structure, though they are not regular.  For example, $X$ might be a manifold (or more generally orbifold) of solutions to the Cauchy--Riemann equations, such that the cokernels form a bundle over $X$.  In this case the VFC should be the Euler class 
of the (orbi)bundle.  We now explain some simple examples of this type, both in the abstract and as applied in the GW setting.
 
There are two possible ways of incorporating nontrivial obstruction bundles into our framework.
We can trivialize the bundle either by adding a complementary bundle or by using local trivializations.
The first method is simpler, but may not adapt well to more complicated situations.   
The second method abstracts the procedure used to construct GW atlases.

\MS

\NI {\bf Method 1:}  We explain this method in the case when the isotropy is trivial. It generalizes to cases when the obstruction bundle is a global quotient.  With more complicated isotropy, one would need more charts and so should use Method 2.

\begin{lemma}  
Suppose that $\pi_X: T\to X$ is a nontrivial $k$-dimensional oriented dimensional bundle over a  compact $d+k$ dimensional oriented manifold $X$.
Then there is an oriented  Kuranishi atlas $\Kk$  on $X$ whose VMC equals the Euler class $\chi(T)\in H_d(X)$.
\end{lemma}
\begin{proof}  Choose 
 an oriented  complementary bundle $\pi_X^\perp:T^\perp\to X$ such that $T\oplus T^{\perp}\cong X\times \R^m$.  Denote by $\Hat\io: T^\perp\to \R^m$ the composite of the inclusion $T^\perp \to X\times \R^m$ with the projection. Then
define an atlas  $\Kk$ with a single chart 
$$
\bK = \bigl(\,U = T^{\perp}, \; E = \R^m,\; s = \Hat\io, \;\psi\,\bigr),
$$
where $\psi$ identifies the zero section of $T^\perp$ with $X$. It has no nontrivial coordinate changes.
Any 
section $\nu^X:X\to T$ 
gives rise to a 
perturbation  section $\io_T\circ \nu^X\circ \pi_X^\perp: U\to \R^m$, where $\io_T: T\to \R^m$ is the inclusion. Then $s+\io_T\circ \nu(u) = 0$  only if $s(u) = 0\in T^\perp$ and $\nu(u) = 0\in E$.  Hence the sections $\nu^X$ and $s+\io_T\circ \nu^X\circ \pi_X^\perp$ have the same zero set; moreover one is transverse to zero if and only if the other is.
\end{proof}

\begin{rmk}\rm  Suppose that $X = \oMm_{0,k}(M,A,J)$ 
 is a manifold consisting of equivalence classes of stable maps with domains of constant topological type.
If the cokernels of the linearized Cauchy--Riemann operator  of \eqref{eq:lindbar} form a bundle over $X$ of constant rank, then one might be able to carry out this construction in the GW setting since, at least locally, one can always find a suitable 
embedding $\la$ of $E^\perp$; see \cite[Example~7.2.4]{JHOL}.  However there may be no a global stabilization for the domains of the
curves.  The next method is more local, and so more adaptable.
$\hfill\er$  \end{rmk}

\NI  {\bf Method 2:}    We begin with the case of trivial isotropy.  
The first step is to  build an oriented (standard) Kuranishi atlas that models the nontrivial bundle $T\to X$.
To this end,
choose a finite open cover $(F_i)_{i=1,\dots, N}$ of $X$ together with trivializations $\tau_i: T_0\times F_i\to T|_{F_i}$, where $T_0$ is the fiber of $T$.  We will define an atlas with indexing set $\Ii_\Kk= \bigl\{I\subset \{1,\dots, N\}: F_I\ne \emptyset\bigr\}$, 
and basic charts
\begin{equation}\label{eq:basic}
\bK_i: = \bigl(F_i, E_i: = T_0, s_i \equiv 0, \psi_i = \id\bigr).
\end{equation}  
The transition charts are:
$
\bK_I: = \bigl(U_I, E_I: = \prod_{i\in I} E_i,  s_I , \psi_I \bigr)$ where
\begin{equation}\label{eq:sum}
U_I = \Bigl\{ (\vec e, x)\in E_I\times F_I\ | \ {\textstyle \sum _{i\in I} }\tau_i(e_i,x)  = 0\Bigr\},\quad s_I(\vec e, x) = \vec e,\quad \psi_I(\vec 0,x) = x\in F_I.
\end{equation}
The coordinate changes have domains $U_{IJ}: = (E_I\times F_J)\cap U_I$ and are induced by the obvious inclusions
\begin{align*}
 & \Hat\phi_{IJ}: E_I\to E_J, \quad (e_i)_{i\in I}\mapsto \bigl((e_i)_{i\in I}, (0_j)_{j\in J\less I}\bigr),\\
& \phi_{IJ}(\vec e, x)  = \bigl(\Hat\phi_{IJ}(\vec e), x\bigr).
\end{align*}
Thus each chart has dimension $\dim X - \dim T_0$.   The cocycle condition is immediate.  
Moreover the index condition holds
because  the inclusion $\Hat\phi_{IJ}: E_I|_0\to E_J|_0$ and $\Hat\phi_{IJ}: E_I\to E_J$ have
isomorphic  cokernels, where  $E|_0 = \{\vec e\in E:\sum \vec e = 0\}$.
   Therefore this set  $\Kk = \bigl(\bK_I, \Hat\Phi_{IJ})_{I\subset J, I,J \in \Ii_\Kk}$ of charts and coordinate changes
defines an additive atlas, which is tame by construction.
Note the commutative diagram
\begin{equation}\label{eq:diagr}
\xymatrix{
\bE_{\Kk}  \ar@{->}[d]_{\pr} \ar@{->}[r]^{\pi_\Kk} &  |\bE_{\Kk}|   \ar@{->}[d]_{\pr} \ar@{->}[r]^{\tau} & T\ar@{->}[d]_{\pi_X}    \\
\bB_\Kk \ar@{->}[r]^{\pi_\Kk} & |\Kk| \ar@{->}[r]^{\pi}& X,
}
\end{equation}
where  $\tau:  |\bE_{\Kk}|\to T$ is induced by  $\tau\bigl((e_i)_{i\in I}\bigr) = \sum_{i\in I} \tau_i(e_i)$ and
$\pi: |\Kk|\to X$  by the projection $(\vec e,x)\mapsto x$.

By Lemma~\ref{le:cov0}, the footprint covering $(F_i)_{i=1,\dots,N}$ has a reduction $(Z_I)_{I\in \Ii_\Kk}$.
This defines a reduction    $(V_I)_{I\in \Ii_\Kk}$ of $\Kk$ with 
 $V_I: = 
\bigl\{ (\vec e, x)\in U_I \ | x\in  Z_I\bigr\}$.
We now show that 
each section $\nu^X:X\to T$ that is transverse to the zero section lifts to 
 an admissible perturbation section of
$\pr: \bE_{\Kk}|_\Vv\to  \bB_{\Kk}|_\Vv$.

\begin{prop} \label{prop:Euler} Let $\pi_X:T\to X$ be a nontrivial bundle over a compact manifold $X$ with atlas 
 $\Kk$ as above.  Let $(Z_I^X\sqsubset F_I)_{I\in \Ii_\Kk}$ be any reduction of the footprint cover and 
 $\Vv$  the associated reduction of $\Kk$ as above.  
Then any  section $\nu^X:X\to T$ of $\pi_X$ lifts to  a functor $\nu: \bB_{\Kk}|_\Vv\to 
\bE_{\Kk}|_\Vv$ with local zero sets $(\s|_{V_I} + \nu_I)^{-1}(0) = \psi_I^{-1}\bigl(F_I\cap (\nu^X)^{-1}(0)\bigr)$.
  Moreover, if $\nu^X$ is transverse to $0$, so is its lift.  Therefore $[X]^{vir}_\Kk = \chi(T)$.
\end{prop}
\begin{proof}  The first step is to choose a smooth partition of unity $(\be_i)_{i = 1,\dots,N}$ 
on $X$ such that
\begin{equation}\label{eq:partu}
x\in Z_J^X\Longrightarrow  \sum_{i\in J} \be_i(x) = 1.
\end{equation}
To do this, fix a metric $d$ on $X$, and define
$$
\rho_i(x) : = d(x, {\textstyle \bigcup_{J:i\notin J} {\ov Z}\,\!^X_J).}
$$
Thus $\rho_i(x)=0$ if  and only if $x\in \bigcup_{J:i\notin J} {\ov Z}\,\!^X_J$.
We claim that for all $x$ there is $i$ such that $\rho_i(x)\ne 0$.  To see this,
note that for each $x$ the set of $J$ such that 
 $x\in {\ov Z}\,\!_J^X$ is totally ordered and so can be written 
as a chain $J_1^x\subsetneq J_2^x\subsetneq\dots \subsetneq J_q^x$ for some $q\ge 1$.
Therefore if $i\in J_1^x$ and $x\in {\ov Z}\,\!^X_J$ we must have $i\in J$.  Hence $\rho_i(x)\ne 0$ for all $i\in J_1^x$. 
%
Thus the function $\sum_{i=1}^N \rho_i$ does not vanish anywhere on $X$, which implies that
 the function
$\be_i(x): = \frac 1{\sum_j \rho_j(x)} \rho_i(x)$ is well defined, and gives a partition of unity on $X$.
Moreover \eqref{eq:partu} holds because 
$$
x\in Z_J^X,\; i\notin J\;\; \Longrightarrow\;\; \be_i(x) = 0.
$$

Next define $V_I: = \bigl\{(\vec e, x)\in U_I \ | \ x\in Z_I^X\bigr\}$. These sets $(V_I)_{I\in \Ii_\Kk}$ form a reduction of $\Kk$.   
Further, given a section $\nu^X: X\to T$ there is an associated functor
$\nu: \bB_{\Kk}|_\Vv\to 
\bE_{\Kk}|_\Vv$ defined by
$$
\nu_I(x):  =  \bigl(\nu^i_I(x)\bigr)_{i\in I} \; = \;  \Bigl(\be_i(x)\tau_i^{-1}\bigl(\nu^X(x)\bigr)\Bigr)_{i\in I} \in E_I.
$$
These perturbations are compatible with the coordinate transformations and have the property that 
for each $x\in Z_I^X$ we have $\sum_i \tau_i(\nu^i_I(x))  = \nu^X(x)$.  On the other hand, by definition of $U_I$ and $s_I$ 
the elements $\vec e = (e_i)_{i\in I}$  in $\im s_I(x)\subset E_I$ have the property that $\sum_i \tau_i(e_i,x) = 0$.
Therefore $(s_I|_{V_I} + \nu_I)(x) = 0$ precisely if  $ \nu^X(x) = 0$.  Further, the fact that $\nu^X$ is transverse to the zero section easily implies that $s_I|_{V_I} + \nu_I$ is also transverse to zero.  
Similarly, one can check that the orientation of $T\to X$ induces an orientation of $\Kk$ and that the induced orientations (defined as in Lemma~\ref{le:locorient1}) on the zero sets agree. 
Since the zero set is compact, this completes the proof.
\end{proof}  

\begin{rmk}\label{rmk:E}\rm 
  If the manifold $X$ is a GW moduli space (with trivial isotropy) and $T$ is isomorphic to the bundle of cokernels of the linearized Cauchy--Riemann operator,
then it is not hard to build a GW atlas with basic charts  $(\bK_i)_{1,\dots,N}$ as above, provided that we choose
suitably small footprints $F_i$ whose elements can be stabilized by a slicing manifold $Q_i$. Note that the section $s_i$ as geometrically defined in \eqref{eq:spsi} is zero since there are no solutions of the equation
$\pbar f = \la(\vec e)|_{{\rm graph} f}$ with the element $\vec e = e_i \in E_i\less \{0\}$. To see that 
 the transition charts provided by the GW construction have the form described above, we must compare the defining equation
 $\sum_{i\in I}\tau_i(e_i,x) = 0$ for the subset $U_I\subset X\times E_I$ with the equation 
 $\pbar f = \la(e_i)|_{{\rm graph} f}$ that defines 
 the transition chart in  \eqref{eq:coordf4}.   Here the linear map $\la$ (initially defined in  \eqref{eq:linE}) plays the role of the local trivialization and is summed over $i\in I$ as in \eqref{eq:lasumI}, though this is not explicit  in the notation.
 Therefore these equations are the same; the only  difference in the construction is that in equation~\eqref{eq:sum} we take $U_I$ to consist of {\it all }solutions to the equation  $\sum_{i\in I}\tau_i(e_i,x) = 0$ while in the GW construction we restrict to a neighbourhood of
 the zero set of $s_I$; see  \eqref{eq:Hat U}.
Hence, in this situation there is an open embedding in the sense of 
Definition~\ref{def:Kumap} of 
the GW atlas  into the  atlas constructed in 
Proposition~\ref{prop:Euler}.
$\hfill\er$  \end{rmk}

\NI {\bf The case with isotropy.} 
Now suppose given an oriented orbibundle $\pi:T\to X$ over a smooth compact orbifold $X$ with fiber $T_0$.  
We first explain the standard way of defining its Euler class, using the strict orbifold atlases for $X,T$ constructed in \S\ref{ss:orb}.
  
Choose a covering of $X$ by group quotients $\bigl((W_i,\Ga_i)\bigr)$  that is compatible with the bundle 
in the sense that
 for each $i$ the orbibundle $T|_{F_i}$ pulls back to a trivial  bundle $(\psi^X_i)^*(T|_{F_i})$ on which $\Gamma_i$ acts
 by a product action.  
By Proposition~\ref{prop:orb} we may extend this family of basic charts to an orbifold atlas, $\Kk_X^{orb}$ on $X$ with  charts $(W_I,\Ga_I,\psi^X_I)$ and
 footprint cover $(F_I)$.  
 Moreover the orbifold $T$  has a corresponding atlas
 $\Kk_T^{orb}$  with  charts $(T|_{W_I},\Ga_I,\psi^T_I)$ and
 footprint cover $(T|_{F_I})\subset T$, where for simplicity  we denote the pullback 
 $(\psi^X_I)^*(T|_{F_I})$ of $T$ to $W_I$ simply by $T|_{W_I}$.
 By Proposition~\ref{prop:groupcomp} the categories 
 $$
 \bB_X^{orb}: = \bB_{\Kk_X^{orb}},\quad  \bB_T^{orb}: = \bB_{\Kk_T^{orb}},
 $$
 corresponding to these orbifold atlases have completions to 
 ep groupoids  that we will denote  $\bG_T^{orb}, \bG_X^{orb}$ and the projection $\pi$ induces a functor
 $\pi:\bG_T^{orb}\to \bG_X^{orb}$ that restricts on the object spaces to the bundle projection $
 \bigsqcup_I T|_{W_I}\to  \bigsqcup_I {W_I}$.
%
%
%
%
%

 One way to define the Euler class of $\pi: T\to X$ is to consider a \lq\lq nonsingular  resolution" of the groupoid $\bG_X^{orb}$, 
 pull the bundle $T\to X$ back to this resolution and then push forward to $X$ the (weighted) zero set of a generic section $\nu^X$ of this bundle. 
 As we explained in Remark~\ref{rmk:commen}~(i),  
we can take the resolution  of  $\bG_X^{orb}$ to be the wnb groupoid $\bV_X^{orb}$ formed as in Proposition~\ref{prop:zero}
from a reduction of $\Kk_X^{orb}$.
Indeed,if  $(G_I^X)$ is
 a reduction  of the footprint cover $(F_I = \psi^X_I(W_I))$ of  $\Kk^{orb}_X$,
the sets   $\bigl(V_I^{orb} =(\psi^X_I)^{-1}(Z_I^X)\bigr)\sqsubset W_I$ form a reduction of
$\Kk_X^{orb}$. The corresponding wnb groupoid  $\bV_X^{orb}$ 
has 
 \begin{itemize}\item objects
 $\Obj_{\bV^{orb}} = \bigsqcup_I V_I^{orb}$, 
 \item   morphisms 
  generated by the projections $\rho_{IJ}^X:\TV^{orb}_{IJ}\to V_I^{orb}$ for $I\subsetneq J$.
  \item weighting function given by assigning the weight $\frac 1{|\Ga_I|}$ to each component 
  $V_I^{orb}$.
  \end{itemize}
The pullback of $T$ to $\bV_X^{orb}$ is the corresponding wnb groupoid with objects 
$ \bigsqcup_I T|_{V_I^{orb}}$.  Therefore 
the section $\nu^X = (\nu^X_I)$ is given by a compatible family of sections 
\begin{equation}\label{eq:nuXE}
\nu^X_I: V_I^{orb}\to T|_{V_I^{orb}},\quad 
\nu^X_J|_{\TV_{IJ}^{orb}} = \nu^X_I\circ \rho^X_{IJ}.
\end{equation}
If $\nu\pitchfork 0$, there is a full subcategory $\bZ_X^{orb}$ of $\bV_X^{orb}$ 
with objects $\nu_I^{-1}(0)\subset V_I^{orb}$.  Again this forms a wnb groupoid, with the induced weighting function.
It follows from  \cite[Proposition~4.19]{Mbr}
\footnote
{
Warning: In \cite{Mbr} the word \lq\lq atlas" has a  different meaning from our current usage; 
see  \cite[Definition~2.11]{Mbr}.} 
  that its pushforward to $X$ represents the Euler class of $T\to X$.

Let us now compute this Euler class by
using a Kuranishi atlas.  
The first step is to build a Kuranishi atlas $\Kk$ such that the bundle $\pr:\bE_\Kk\to \bB_\Kk$ is a model for the original bundle $T\to X$, in the sense that there is an analog of the  commutative diagram~\eqref{eq:diagr}.
To this end,  choose  $\Gamma_i$-equivariant trivializations 
 $$
 \tau_i: T_i\times W_i\stackrel{\cong}\to (\psi^X_i)^*(T|_{F_i}),
 $$
 where we denote the fiber (which is isomorphic to $T_0$) by $T_i$  to emphasize that it supports an action of $\Ga_i$.
We then define a Kuranishi atlas essentially as before, incorporating the groups in the natural way.
 Thus 
\begin{equation}\label{eq:orbib}
{\textstyle \bK_I: = \bigl(U_{I}, \;\; E_I: = \prod_{i\in I} T_i,\;\;   \Ga_I: = \prod_{i\in I} \Ga_i, \; s_I ,\; \psi_I \bigr)}
\end{equation}
 where
\begin{align*}\label{eq:sum}
U_{I} & = \Bigl\{ (\vec e, x)\in E_I\times W_I\ | \ \sum _{i\in I} \tau_i(e_i,x)  = 0\Bigr\},\quad s_I(\vec e, x) = \vec e,\quad \psi_I(\vec 0,x) = \psi^X_I(x)\in F_I.
\end{align*}
The coordinate changes have lifted domains $\TU_{IJ}: = \Hat\phi_{IJ}(E_I)\times W_J\subset U_J$ and are induced by the obvious projections
$$
\rho_{IJ}: 
\bigl((e_i)_{i\in I}, (0_j)_{j\in J\less I}, x\bigr) \mapsto \bigl((e_i)_{i\in I},\rho^X_{IJ}(x)\bigr) \in U_I,
$$
where $\rho_{IJ}^X: W_J\to W_I$ is part of the corresponding coordinate change in the 
orbifold atlas $\Kk^{orb}_X$.

 As before, $\Kk$ is tame
so that the theory developed in \S\ref{s:iso} applies.  In particular, the reduction $(Z_I^X)$ of the footprint cover
 chosen above induces a compatible reduction $\Vv\subset \Kk$ with $V_I: = U_I\cap (E_I\times V^{orb}_I)$. Further,
the proof of Proposition~\ref{prop:Euler}  carries through to show that every section $\bigl(\nu^X_I\bigr)
$ as in \eqref{eq:nuXE}  may be lifted via a partition of unity to a perturbation  $\nu: \bB_{\Kk}|_{\Vv}\to 
\bE_{\Kk}|_{\Vv}$  in the sense of Definition~\ref{def:sect2}.   Hence, as before,  the VFC defined by the Kuranishi atlas $\Kk$ is the Euler class of $T\to X$.  Further details are left to the interested reader.
Thus we conclude:
\begin{itemize}\item {\it  The analog of Proposition~\ref{prop:Euler} holds for general oriented orbibundles.}
\end{itemize}

\begin{rmk}\rm  If $X$ is more complicated, for example a union of strata each of which has  fixed dimension and 
cokernel bundle of constant rank, then one should 
first build local atlases that model each stratum separately, and then put them together via the gluing parameters.
   Suppose, for example, that $X$ is a compact $2k$-dimensional manifold that contains a codimension $2$ submanifold $Y$ consisting of curves with one node, and that the cokernels have constant rank $2r$ so that they form bundles $T^Y\to Y$ and $T^X\to X\less Y$.  
  The
 complex line bundle $L_Y$ over $Y$ formed by the gluing parameter at the node is the normal bundle to $Y$ in $X$.
 Let $\Nn(Y)\subset X$ be a neighbourhood of $Y$ that forms a disc bundle $\pi_Y: \Nn(Y)\to Y$.  
 Notice that the restriction of $T^X$ to $\Nn(Y)$ may not simply  be the pullback  $\pi_Y^*(T^Y)$  because when one glues with the family of parameters $a=\eps e^{2\pi i\theta}, \theta\in S^1$, one of the components twists by $2\pi \theta$ relative to the other.
In many situations the bundles
$T^X, T^Y$ have a natural complex structure that is preserved by the $S^1$ action in the fibers of
$\p\Nn(Y)\to Y$, and
$T^Y$ decomposes into a finite sum $\oplus_k T^Y_k$ so that
 $T^X|_{\Nn(Y)\less Y} = \oplus_k \pi_Y^*\bigl(T^Y_k\otimes _\C (L_Y)^{\otimes k}\bigr)$.
 One should then build the atlas over $X$
 in stages, with one atlas over a neighbourhood $\Nn_1(Y)$ of $Y$ with obstruction bundle $T^Y$ and gluing parameters in $L_Y$, another over $X\less 
\Nn_2(Y)$ with obstruction bundle $T^X$, and appropriate transition charts over the deleted neighbourhood
$\Nn_1(Y)\less \Nn_2(Y)$.  The interaction of  the gluing parameters and the obstruction bundles will  be seen in the
structure of these transition charts. Of course, if one really wanted to carry out such computations, rather than using the rather rigid notion of atlases one probably should first
 generalize the previous discussion to show that the fundamental cycle can be calculated using more general atlases,
in which, for example,  the bundles $\pr:U_I\times E_I\to U_I$ are allowed to be nontrivial.
$\hfill\er$  \end{rmk}

\subsection{$S^1$ actions}\label{ss:S1}

Finally we reprove a  result from \cite{Mcq} concerning $2$-point Gromov--Witten invariants in the product manifold 
$ (S^2\times M_1, \om_0\times \om_1)$.

\begin{prop}\label{prop:S1}  Let  $M = (S^2\times M_1, \om_0\times \om_1)$, and let  $A = [S^2\times pt] + B$, where $B\in H_2(M_1)$. Then
\begin{equation}
\langle pt, c\rangle_{0,2,A} = 0,\quad \forall B\ne 0, c\in H_*(M_1).
\end{equation}
\end{prop}

This statement about $2$-point Gromov--Witten invariants immediately implies that the Seidel element corresponding to the trivial loop in $\Ham(M_1,\om_1)$ is the identity.  (see  \cite{Mcq} or \cite[Chapter~12.5]{JHOL} for information on the Seidel representation.)  The key idea of the proof is that the manifold $M$ supports an $S^1$ action that rotates the $S^2$ factor with fixed points $0,\infty$.  If $B\ne 0$ and we choose $J = j\times J_1$ to be a product, then 
the
elements in the top stratum of the moduli space $\oMm_{0,2}(M,A,J)$ are simply graphs of nonconstant $J_1$-holomorphic maps to $M_1$. Therefore the action of  $S^1$ on this stratum
is free.  Since we can place the constraints in the fixed fibers over $0$ and $\infty$, it should be impossible to find isolated regular solutions of the equation.  The difficulty with this argument is  that the $S^1$ action does have fixed points on the 
compactified moduli space $\oMm_{0,2}(M,A,J)$, and it is not clear what effect these might have on the invariant.
On the other hand, because the ambient manifold is a product we can construct the
atlas so that it carries an $S^1$ action that is free on the top stratum.   This is the key reason why the invariant vanishes. 
Also our proof assumes known that GW invariants can be constructed by
considering cut down moduli spaces as in \S\ref{ss:var}.

We begin by discussing the abstract situation.  We will work with weakly SS atlases as in \S\ref{ss:SS}.

\begin{defn}  Suppose that $S^1$ acts on $X$.  Then we say that a weakly 
SS (weak) Kuranishi atlas on $X$ {\bf supports an $S^1$ action} if
the following conditions hold:
\begin{itemize}\item   The action of
$S^1$ on each domain $U_I$ is weakly SS smooth and commutes with the  action of $\Ga_I$;
\item $S^1$ acts  trivially on $E_I$;
\item the maps $s_I$ and $\psi_I$ are $S^1$ equivariant.  
\item the subsets $U_{IJ}\subset U_I, \TU_{IJ}\subset U_J$ are $S^1$-invariant and the covering map $\rho_{IJ}: 
 \TU_{IJ}\to  U_{IJ}$ commutes with the $S^1$-action.
 \end{itemize}
 Further we say that the action has {\bf fixed points of codimension at least $2$} if  both the domains $U_I$ and $X$ have a collection of strata that is  respected by the footprint maps,  contains all fixed points of the action, and of codimension $\ge 2$ in $U$.
  \end{defn}

\begin{lemma}\label{le:S1}  
Let $\Kk$ be a a $0$-dimensional weakly SS and weak
 Kuranishi atlas that supports an $S^1$ action.    Suppose further that 
this action has fixed points of codimension at least $2$.
 Then $[X]^{vir}_\Kk = 0$.
\end{lemma}

\begin{proof}  If $\Kk$ supports an $S^1$ action then  the quotient $\uKk/S^1$ of the intermediate atlas
 by the $S^1$ action is a filtered topological atlas on $X/S^1$ in the sense of \S\ref{ss:tatlas} since its domains  $\uU_I/S^1$   
 are locally compact, separable metric spaces and all structural maps descend to the quotient.
Hence we may tame the quotient, and then use the inverse image of these sets to tame $\Kk$.
 For a similar reason, we may suppose that the 
reduction consists of $S^1$ invariant sets; see  Lemma~\ref{le:cov0}.
It remains to construct $\nu_I$ inductively over $I$ (by the method explained in Proposition~\ref{prop:ext})
so that $s_I|_{V_I} + \nu_I$ has no zeros.  Because $S^1$ acts trivially on the obstruction spaces, 
we may assume that $\nu_I: V_I\to E_I$ factors through $V_I/S^1$, which is the quotient of a $k$-dimensional SS manifold by an action of $S^1$
which is free on the top dimensional strata.  Hence $V_I/S^1$ is a CW complex of dimension $k-1$.  Since $E_I$ has dimension $k$, we can extend any nonzero perturbation that is defined on  a closed subset of $V_I/S^1$ to a  perturbation that is nonzero everywhere.
This completes the proof.
\end{proof}

\begin{proof}[Proof of Proposition~\ref{prop:S1}]
It remains to construct an appropriate Kuranishi atlas.  This requires some care.
To reduce the dimension to $0$ we consider the cut down moduli space
 $$
X_{Z_c} = \bigl\{ [\Si, \bz, f]\in \oMm_{0,2}(A, S^2\times M_1, j\times J_1): 
f(z_0)\in \{0\}\times M, f(z_\infty)\in \{\infty\}\times Z_c\bigr\}
$$
as in \S\ref{ss:var},  where  $Z_c$ is a manifold with  $\dim Z_c + 2c_1(B) = \dim M_1$.
We consider $X_{Z_c}$ to be a stratified space as in Remark~\ref{rmk:stratX}.
It supports an $S^1$ action that is free on the top stratum.  Indeed, each stratum of $X_{Z_c}$ has exactly one component that is a graph over $S^2$, and the stratum contains a fixed point only if this component is the constant map. 
 In order that $S^1$ act on each basic chart with free action on the top stratum, we must choose both the obstruction spaces $E_i$ and
  the slicing conditions to be $S^1$ invariant. 
  This argument  crucially uses the fact the ambient manifold is a product.
  
  We first choose suitable obstruction spaces  for the basic charts.
To this end,  we choose the linear map $\la_i: E_i \to
  \Cc^{\infty}\bigl(\Hom_J^{0,1}(\Cc|_{\De}\times S^2\times  M_1) $
  of \eqref{eq:linE} to take values in the pullback of the bundle
  $\Cc^{\infty}\bigl(\Hom_J^{0,1}(\Cc|_{\De}\times  M_1)$ by the projection $S^2\times M_1\to M_1$.
  Thus each element in the image $\la_i(E_i)$ is fixed by the $S^1$ action.
  Note that these obstruction spaces do suffice for regularity because the section component is regular if and only if its projection to $M_1$ is regular,  
  while all fiber components 
   are spheres so that the trivial horizontal bundle does not contribute to the cokernel; see  \cite[Proposition~6.7.9]{JHOL}.
  Further, by choosing the support of the sections in  $\la_i(E_i)$  to be disjoint from the projections to $M_1$ of the nodes
  of the center point of a chart, we can ensure that the elements $\la(e)|_{{\rm graph} f}$ vanish near the nodes,
  as required by the construction  in \S\ref{ss:GW}~(V).  (Here we use the fact that  if the section component
   is constant then, because of the triviality of the bundle, it  is regular, and so has no need of any perturbations to achieve the transversality condition ($*_c$) of \S\ref{ss:var}.)

The next step is to choose suitable slicing manifolds.  We may  use 
 slicing manifolds $Q_i$ of the form $\Uu\times Q'_i$, where $Q_i'\subset M_1$ has codimension $2$ and $\Uu\subset S^2$ is  open,
  to stabilize all fiberwise components of the domain $\Si$ of $[\Si,f]$, 
 and also  the section component provided that this is not constant.
If the section component is constant and if there is a bubble component in some fiber other than $0,\infty$, then the section component is stable, since we already have marked points at $0, \infty$, and it has at least one other nodal point.  Hence the only case when we need to use a non $S^1$-invariant slicing manifold is when  
 $[\Si,f]$ is a fixed point of the action, consisting of the graph of a constant function together with some bubbles in the fibers over $0,\infty$.    The domains of these graph components can be stabilized by slicing with the 
 fiber $Q_F: = \{1\}\times M_1$.
 Even though $Q_F$ is not itself $S^1$-invariant, we can build an $S^1$ invariant chart with center  $[\Si,\bz,f]$ 
 using this slicing manifold 
 as well as the invariant manifolds $\Uu\times Q'_i$, because
 the 
 induced action of $S^1$ on the normalized and stabilized map  $(\Si_{0,\bP}, \bw, \bz, f)$ involves  renormalization (as in 
 \eqref{eq:actionf})
 and so with appropriate choice of normalization $\bP$ will be trivial.
 Here we must choose $\bP$ to contain the three points $0,1,\infty$ on the constant graph, where  at least one of $0,\infty$  is a node (the other might be a marked point), while the point $w_1$ at $1$  maps to $Q_F$.
Since the center point of this chart is fixed by the $S^1$ action, it is possible to build the chart to be $S^1$ invariant by appropriate choice of the domain $U_i$. 
 Thus all the basic charts can be constructed to support an $S^1$ action that is free on the top stratum.
 It follows that one can choose the domains of the transition charts to be $S^1$-invariant.  As before there are no  fixed points
in the top stratum.    Hence the result follows from Lemma~\ref{le:S1}.
  \end{proof}

\begin{rmk}\rm Pardon uses a similar approach in \cite{Pard} to show that the  Floer complex of 
a small time independent Hamiltonian $H$ equals the Morse complex.   The key idea is to replace a Floer trajectory $u: \R\times S^1\to M$ 
of $H$ that depends on the variable $t\in S^1$ by its graph in $\R\times S^1 \times M$, and then build an appropriate atlas
that is invariant under the  action induced by the rotation of the $S^1$ factor in  $\R\times S^1 \times M$.  Of course, in this situation
one needs to build a coherent family of atlases, with boundary and corners.
A few of the issues that come up  in this context
 are discussed at the end of \S\ref{ss:hybrid}. \hfill$\er$
\end{rmk}

\section {Order structures and products}\label{s:order}

Atlases are defined by a rather rigid process from a set of basic charts. In particular, by \eqref{eq:sum0} our definition requires an atlas to be {\bf additive} in the sense that it satisfies the following condition for the obstruction spaces of the transition charts:
\begin{equation}\label{eq:sum1}
E_I = {\textstyle \prod_{i\in I }}E_i.
\end{equation}
On the one hand, this condition is essential for the taming process; on the other it does not hold for product atlases which are 
relevant to various geometrically interesting situations.
In this section, we 
 define a notion of semi-additive atlas, which
 does apply to products.  
 \footnote
 {
  This turns out to be  fairly close to the notion of a \lq\lq MW weak 
good coordinate system" in Joyce~\cite[Definition~A.18]{Joyce}.  
Since Joyce keeps only infinitesimal information along the zero set of a chart, 
the charts in his atlases  inherit enough compatibility  from the zero sets that they 
 do not need to be tamed in our sense.  However, the taming process is essential 
in our approach, since we use open neighbourhoods of the zero set.}
Since these atlases are not filtered, they cannot be tamed; see  Remark~\ref{rmk:notame}.
However, if one builds them starting from tame or good atlases (as in Definition~\ref{def:good}), then they have nice properties.

Here are our main results, stated somewhat informally.

\begin{itemlist}\item A good semi-additive atlas 
has a well defined VFC; see  Proposition~\ref{prop:semivir}.
\item 
A product of good atlases $\Kk_1\times \Kk_2$ is semi-additive and good; see  Lemma~\ref{le:Kprod}. 
By Proposition~\ref{prop:prod}, the corresponding virtual class
$[X_1\times X_2]_{\Kk_1\times \Kk_2}^{vir}$ equals the product $[X_1]_{\Kk_1}^{vir} \times [X_2]_{ \Kk_2}^{vir}$ in 
$\check H_*(X_1\times X_2)$.  
\item   We show that in Proposition~\ref{prop:hybrid2} that
every semi-additive atlas is concordant to an  atlas that can be tamed.  We then sketch   
how this might be applied  in
Hamiltonian Floer theory, where one needs a coherent system of atlases  for a family of spaces $X$ 
whose boundary is a union of products $X_1\times \dots \times X_k$.  
\end{itemlist}

We begin by with a general discussion of posets (i.e. partially ordered sets) since the category $\bB_\Kk$ determined by an
atlas with trivial isotropy (or, more generally, by a topological atlas) is a poset; see  Remark~\ref{rmk:poset}~(ii).

\subsection{Semi-additive atlases}\label{ss:index}

In our approach, an atlas (of whatever kind) on $X$ always consists of a family of basic charts $(\bK_i)_{i=1,\dots, N}$ with footprints 
$F_i\subset X$,  together with transition charts
$(\bK_I)_{I\in \Ii_\Kk}$ that are indexed by a poset  $\Ii_\Kk\subset \Pp^*\bigl(\{1,\dots,N\}\bigr)$ that is determined by
 the intersection pattern 
of the footprints $F_i\subset X$, where $\Pp^*(\Aa)$ denotes the poset  of nonempty subsets of the set $\Aa$.
Indeed there is a transition chart with footprint $F_I = \cap _{i\in I} F_i$ for every (nonempty) collection 
of charts whose footprints have nonempty intersection.
%
 In some cases (such as the  product of two atlases) {\it different} sets of basic charts have the {\it same} footprint intersection. One might therefore consider permitting an atlas to have just one transition chart 
 for each such intersection. The corresponding indexing set $\Ii$ would  no longer be a subset of
 $\Pp^*\bigl(\{1,\dots,N\}\bigr)$.
  Nevertheless, it  is minimally generated
 in the following sense.

\begin{defn}\label{def:minimal}
The finite poset $\Ii$ 
is said to be  {\bf minimally generated } if it satisfies the following assumptions:
\begin{itemize}\item[(a)] every subset $\{I_1,\dots, I_k\}$ of $\Ii$ with an upper  bound 
has a unique
least upper bound $I_1\vee\cdots\vee I_k\in \Ii$;
\item[(b)]   if $\Mm(\Ii)$ denotes the set of minimal elements in $\Ii$, then each $J\in \Ii$ is the 
least upper bound of the set $m(J) = \{H\in \Mm(\Ii)\ | \ H\le J\}$
of minimal elements it dominates.
\end{itemize}
Further, we say that $\Ii$ is {\bf minimally $(\Aa,\tau)$-generated} if in addition there is an order preserving  injection
$\tau:\Ii\to \Pp^*(\Aa)$ such that
\begin{itemize}\item[(c)] for all $I,J\in \Ii$,   $\tau(I\vee J) = \tau(I)\cup \tau(J)$,
\item[(d)] if $A\in \im(\tau), H\in \Mm(\Ii)$ then
$$ 
 \bigl\{B\in \Pp^*(\Aa): \tau(H)\subset B\subset A\bigr\}\;\subset\;\im(\tau).
$$ 
\end{itemize}

\end{defn}

\begin{lemma}  If $\Ii$ is minimally $(\Aa,\tau)$-generated,  then $\tau(J) = \bigcup_{H\in m(J)} \tau(H)$.
\end{lemma}
\begin{proof}  This follows from conditions (b) and (c).
\end{proof}

Condition (d)  is a completeness condition that asserts the existence of \lq\lq enough" elements in $\Ii$.

\begin{example}\label{ex:prod1}\rm (i) A subset $\Ii$ of $\Pp^*(\{1,\dots,N\})$ (with the induced partial order) 
is minimally generated if
\begin{itemize}\item[-]
it  contains all the one point subsets, 
and 
\item[-]  $A\in \Ii,\;\; \emptyset\ne B\subset A \Longrightarrow B\in \Ii$. 
\end{itemize}
In fact, in this case it is minimally $(\Aa,\tau)$-generated, with $\Aa = \{1,\dots,N\}$ and $\tau = \id$.
The index sets of a Kuranishi atlas have this form.
\MS

\NI (ii) Let $\Aa: =  \Aa_1 \sqcup  \Aa_2$ where $\Aa_t: = \{1,\dots, N_t\}$ for $t= 1,2.$
The set of subsets $\Pp(\Aa)$ of $\Aa$ 
consists of all pairs $(I_1,I_2)\in \Pp(\Aa_1)\times  \Pp(\Aa_2)$, where one or both of $I_1,I_2$ may be empty.
The set $\Ii$ of all pairs $(I_1,I_2)$ with $I_t\ne \emptyset$ for $t=1,2$
 is minimally $(\Aa,\tau)$-generated, where
$\tau: \Ii\to \Pp^*(\Aa)$ is given by $\tau(I_1,I_2) = I_1\cup I_2\subset \Aa$.
Note that the minimal subsets of $\Ii$ are the  pairs of one point subsets $(\{i_1\}, \{i_2\})$.  
\MS

\NI (iii) In our examples of minimally $(\Aa,\tau)$-generated posets, the elements of  $\Aa$ 
have the following interpretation.
Each $\al\in \Aa$ indexes a subset $F_\al\subset X$, and we associate to each $A\in \Pp^*(\Aa)$  the set $F_A: = \bigcap_{\al\in A} F_\al$.  Further, the image of the map $\tau:\Ii\to \Pp^*(\Aa)$ 
is contained  in $\{A\in \Pp^*(\Aa): F_A\ne \emptyset\}$, and typically consists of
all such sets $A$ that contain the image $\tau(H)$ of a minimal element in $\Ii$. 
 For example,  with $\Aa = \Aa_1\sqcup\Aa_2$ as in (ii), suppose 
 that $X = X_1\times X_2$, and that, for $t = 1,2$,   $(F_{t\,i})_{i\in \Ii_t}$ is 
 an open covering 
of $X_t$.
For each $I\subset \{1,\dots, N_t\}$ define 
$$
F_{t\,I}: = {\textstyle\bigcap_{i\in I}} F_{t\,i},\quad F_{t \, \emptyset}: = X_t.
$$
Then  the product covering $\bigl(F_{1 I_1}\times F_{2 I_2}\bigr)$ of $X_1\times X_2$ is indexed by the
pairs $(I_1,I_2)\in \Ii_{1}\times \Ii_{2})$ with the product order
$$
(I_1,I_2)\le (J_1,J_2)\;\Longleftrightarrow I_1\le J_1, \; I_2\le J_2.
$$
The injection $\tau: \Ii_{1}\times \Ii_{2}\to \Pp^*(\Aa)$ 
is given by $(I_1,I_2)\mapsto (I_1\cup I_2)\subset \Pp^*(\Aa)$. It clearly satisfies conditions (c),(d) above.
Note that 
 $\Ii_{1}\times \Ii_{2}$  contains no elements 
of the form $(\emptyset, I_2)$ or $(I_1,\emptyset)$,\footnote
{
With our definitions,  if $(F_{t\ i})_{i\in \Aa_t}$ is the footprint covering of an atlas on $X_t$, then
 the products $(F_{1 I_1}\times F_{2 I_2})_{(I_1, I_2)\in \Ii_{1}\times \Ii_{2}}$
form the  footprint covering  of the  product  atlas, see   Definition~\ref{def:Kprod} below.
Pardon~\cite{Pard} works with a more general definition of atlas
that allows charts with  footprints such as $F_{I\emptyset}$ even though such charts may contain no smooth points.
Hence his treatment of products does not meet the particular difficulties discussed here, though it has some technicalities
 of its own.}
and hence its image by $\tau$ contains none of the minimal elements of $\Pp^*(\Aa)$. 
\MS

\NI (iv)  In the Gromov--Witten context it is appropriate to use semi-additive atlases in situations where different 
collections  of basic charts have the {\it same} footprint intersection. For example, in the case of products the two pairs    
$\bK_{1i_1}\times \bK_{2j_3}, \bK_{1i_2}\times \bK_{2j_4}$ and $\bK_{1i_1}\times \bK_{2j_4}, \bK_{1i_2}\times \bK_{2j_3}$ 
have the same footprint intersection, namely
$$
F_{1(i_1i_2)}\times F_{2(j_3j_4)} := \bigl(F_{1i_1}\cap F_{1i_2} \bigr) \times \bigl(F_{2j_3}\cap F_{2j_4}\bigr).
$$
The product atlas has a {\it single} transition chart $\bK_{1(i_1i_2)}\times \bK_{2(j_3j_4)}$ with  this footprint  rather than many different ones. 
In Remark~\ref{rmk:GWext} and the subsequent discussion we indicate how one might show that in the GW context
the product atlas is commensurate
(in an appropriate sense)  to  the additive atlas with the same basic charts. 
$\hfill\er$
\end{example}

\begin{rmk} \rm Property (ii) in Definition~\ref{def:minimal} implies that the map 
$$
\io_m: \Ii\mapsto \Pp^*(\Mm(\Ii)), \quad J\mapsto m(J),
$$ 
is  injective. Further,  the  greatest lower bound   
$I\wedge J$ exists iff and only if there is $H\in \Mm(\Ii)$ such that $H\le I$ and $H\le J$,
and in this case
$H\le I\wedge J$ iff $H\le I$ and $H\le J$.  Hence
we have
\begin{equation}\label{eq:wedge}
 \io_m(I\wedge J) = \io_m(I)\cap \io_m(J)\quad \forall I,J, \; I\wedge J \in \Ii.
\end{equation}
However,  as in the case of products,  $\io_m$ may not satisfy the condition
$
\io_m(I\vee J) = \io_m(I)\cup \io_m(J).
$
\hfill$\er$
%
%
\end{rmk} 

We now define the notion of semi-additive atlas.  
This involves generalizing the definition of  coordinate change $\Hat\Phi_{IJ}$, since the groups $\Ga_I,\Ga_J$ will in general  no longer be products of groups indexed by the elements of $I,J$.
 However, their main characteristic remains unchanged, namely, they are defined by group covering maps $(\rho_{IJ}, \rho_{IJ}^\Ga): (\TU_{IJ},\Ga_J)\to (U_{IJ},\Ga_I)$ as in Definition~\ref{def:cover}, 
 that are compatible 
in the sense of Remark~\ref{rmk:tamea}.  

\begin{rmk}\rm 
Unless explicitly mentioned otherwise, we will assume throughout this section that all charts are smooth.  With obvious changes, one could deal with SS atlases (as in \S\ref{ss:SS}.)     Also, one could attempt
similar definitions in the topological case.  
However, when it comes to the notion of tameable  atlas (see  \S\ref{ss:hybrid}), one would need to check that 
the corresponding filtration satisfies all the required  properties.  (The potential problem here is that we use the index condition for smooth coordinate changes in order to establish the filtration condition (iv) in Definition~\ref{def:Ku3}.
To deal with this, one might need to rephrase/strengthen condition (iv) in  Definition~\ref{def:good} of good atlas.)\hfill$\er$
\end{rmk}

\begin{defn}\label{def:sadd} Consider a family $\Kk: = (\bK_I, \Hat\Phi_{IJ})_{I\le J, I,J\in \Ii}$ of (smooth) Kuranishi charts  and coordinate changes on $X$ whose charts are 
 indexed by a set $\Ii$ that is minimally $(\Aa,\tau)$-generated,
 where the minimal elements in $\Ii$ correspond to the basic charts, and 
 $\Aa$  is related to the footprint covering as follows:
  \begin{itemlist} 
 \item for each $\al \in \Aa$ there is a subset $F_\al\subset X$ such that
  the chart $\bK_I$ has footprint  $F_I: = F_{\tau(I)}$, where 
 for $A \subset  \Aa$ we define
\begin{equation}\label{eq:foot}
   F_A: = {\textstyle \bigcap}_{\al\in A} F_\al,\quad A\subset \Pp^*(\Aa).
\end{equation}
 \end{itemlist}

We say that $\Kk$ is {\bf semi-additive} if in addition the following conditions hold on the charts and coordinate changes.
 \begin{itemlist}  
 \item
There are  tuples $(E_\al)_{\al\in \Aa}$, $ (\Ga_\al)_{\al\in \Aa}$ where 
$E_\al$ is a finite dimensional vector space and 
 $\Ga_\al$ is a finite group that  acts on  $E_{\al}$.
  For each  subset 
$A\subset \Aa$ we define $E_A: = \prod_{\al\in A} E_\al$, so that $\Ga_A: = \prod_{\al\in A} \Ga_\al$  acts on $E_A $ by the product action.
\item The  map  $\tau: \Ii \to  \Pp^*(\Aa)$ 
satisfies the  following conditions
with respect to  the groups $\Ga_I$ and obstruction spaces $E_I$ of the transition charts $\bK_I$.
 \begin{itemize}
 \item[-]  
   $\Ga_I: = \Ga_{\tau(I)}$ for each $I$, and the surjections $\rho_{IJ}^\Ga: \Ga_J\to \Ga_I$ in the coordinate changes are given by the projections $\prod_{\al\in \tau(J)} \Ga_\al  \to \prod_{\al\in \tau(I)} \Ga_\al$ and hence have kernel $\prod_{\al\in \tau(J)\less \tau(I)} \Ga_\al$; in particular, by definition of coordinate change each group 
$\prod_{\al\in \tau(J)\less \tau(I)} \Ga_\al$ acts freely on the set $\TU_{IJ}$ with quotient 
$\Ga_I$-equivariantly isomorphic to $(U_{IJ}, \Ga_I)$.
 \item[-] 
$E_I=\prod_{\al\in \tau(I)} E_\al$  and $\Hat\phi_{IJ} :E_I\to E_J$ is the natural inclusion for all $I\le  J$.
%
\end{itemize}
\end{itemlist}
We say that $\Kk$ is a  {\bf (weak) semi-additive atlas} if in addition the tangent bundle condition and (weak) cocycle condition hold.   Such an atlas is {\bf good} if in addition the conditions of Definition~\ref{def:good} hold.

\NI
For short we will denote a semi-additive atlas by the tuple $(\bK_I, \Hat\Phi_{IJ})_{\Ii, \Aa,\tau}$.
\end{defn}

\begin{example}\label{ex:add}\rm If $\Kk$ is an atlas in which $\tau$ is induced by
an isomorphism
 $ \Mm(\Ii) \cong \Aa$, 
then the above notion of semi-additivity reduces to the notion of additivity in \cite{MW2},
and is already assumed in the definitions of atlas given here.   
In this case we  say that $\Kk$ is  {\bf standard}.   
$\hfill\er$
\end{example}

Our main example is that of products where   
$\Aa = \Aa_1\sqcup \Aa_2$ as in Example~\ref{ex:prod1}.


\begin{defn}\label{def:Kprod}
Let $(\Kk_t)_{ t=1,2}$  be semi-additive atlases $\bigl((\bK_{I}, \Hat\Phi_{IJ})_{\Ii_t,\Aa_t,\tau_t}\bigr)_{ t=1,2}$.
The product atlas $\Kk_1\times \Kk_2$ 
is an atlas on $X_1\times X_2$ whose basic charts are products $\bK_{i_1,i_2}$ of the basic charts in 
$\Kk_1,\Kk_2$, where for $I_t\in \Ii_t$ we define
$$
\bK_{I_1,I_2}: = \bigl( U_{I_1}\times U_{I_2}, E_{I_1}\times E_{I_2}, \Ga_{I_1}\times \Ga_{I_2}, s_{I_1}\times s_{I_2}, \psi_{I_1}\times \psi_{I_2}\bigr).
$$
The  indexing set  of $\Kk_1\times \Kk_2$  is $\Ii = \Ii_1\times \Ii_2$  with the product order 
$$
(I_1,I_2)\le (J_1,J_2)\Longleftrightarrow I_t\le J_t, \;\;t=1,2.
$$
Further $\Aa: = \Aa_1\sqcup \Aa_2$ and $\tau$ is given by
\begin{align*}
 \tau: \; \Mm(\Ii_1\times \Ii_2) & = \Mm(\Ii_1)\sqcup \Mm(\Ii_2)\to \Pp^*(\Aa) = \Pp^*(\Aa_1\sqcup \Aa_2),\\
 \tau ((i_1,i_2)) & : = \tau_{1}(i_1)\sqcup \tau_{2}(i_2),  \quad \mbox{ where } i_t\in \Mm(\Ii_t).
\end{align*}
Thus its transition charts are the products $\bK_{I_1,I_2}$, and we require that
 the coordinate changes are also given by product maps.
\end{defn}

\begin{lemma}\label{le:Kprod}  The product  $\Kk_1\times \Kk_2$ of two semi-additive (weak) atlases $\Kk_1,\Kk_2$  is a semi-additive (weak) atlas.
Further if $\Kk_1,\Kk_2$ are  good atlases (see Definition~\ref{def:good}), then so is $\Kk_1\times \Kk_2$.
\end{lemma} 
\begin{proof} 
 As  in Example~\ref{ex:prod1}, one readily checks  that the given tuple $\Ii, \tau,\Aa$  and footprints
satisfy the  conditions  in Definition~\ref{def:minimal} as well as
 \eqref{eq:foot}, while the product charts and coordinate changes have the form required by 
 the semi-additivity conditions.  
 Further, if $\Kk_1,\Kk_2$ satisfy the cocycle condition, so does the product $\Kk_1\times \Kk_2$.
 This proves the first statement.
 
 If $\Kk_1, \Kk_2$ are atlases, then the  quotient space $|\Kk_1\times \Kk_2|$ equals the product  $|\Kk_1|\times |\Kk_2| $ 
 as a point set because 
 $$
 (I_1,I_2, (x_1,x_2))\sim (J_1,J_2, (y_1,y_2)) \Longleftrightarrow
 (I_t, x_t)\sim (J_t, y_t), \;\; t=1,2.
 $$
   Therefore the map 
 $$
 \Obj_{\Kk_1\times \Kk_2} = \bigsqcup U_{I_1}\times U_{I_2}\to |\Kk_1|\times |\Kk_2|
 $$ 
 descends to a 
continuous injection $|\Kk_1\times \Kk_2|\to |\Kk_1|\times |\Kk_2|$.  If the $\Kk_t$ are good, the spaces $ |\Kk_t|$ are Hausdorff, which implies  that  $|\Kk_1\times \Kk_2|$ is also.
We must check that  $\pi_{\Kk_1\times \Kk_2}$ is a homeomorphism when restricted to $U_{I_1}\times U_{I_2}$.
%
Because the topology on  $U_{I_1}\times U_{I_2}$ has 
a basis consisting of product sets $S_1\times S_2$, it suffices to check that for each such set there is an open set $\Ww\subset 
|\Kk_1\times \Kk_2|$ that intersects $\pi_{\Kk_1\times \Kk_2} (U_{I_1}\times U_{I_2})$ in 
$\pi_{\Kk_1\times \Kk_2} (S_{1}\times S_{2})$. 
But since  for $t=1,2$ the atlas $\Kk_t$  is good, there is  an open subset $\Ww_t\subset |\Kk_t|$ such that $\Ww_t\cap \pi_{\Kk_t}(U_{I_t}) = S_t$.  Hence we may  take $\Ww$ to be the  pullback  of $\Ww_1\times \Ww_2$ to $|\Kk_1\times \Kk_2|$.
%
%
%
%
%
The other two conditions in the definition of \lq\lq good", namely the existence of a suitable metric, and the closedness of the sets 
$\phi_{IJ}(U_{IJ})$,
 also are preserved under taking products.    This completes the proof.
 \end{proof}

\begin{rmk}\label{rmk:notame}\rm
We will see in the next section that for many purposes 
one can work just as well with semi-additive atlases as with additive atlases.  
The main difference is that in the semi-additive case the taming construction does not work
 because these atlases do not have a 
filtration satisfying condition (iii) in Definition~\ref{def:Ku3}.   This implies that
the taming construction in Proposition~\ref{prop:proper1} does not go through.
(See also Remarks~\ref{rmk:tamef} and \ref{rmk:add1}.) 
 Moreover, even if the atlas were tame, the proof in Proposition~\ref{prop:Khomeo}
  that the atlas is therefore good does not go through 
  because the key identity \eqref{eq:tamer}  in Lemma~\ref{le:Ku2} also fails. 
  However, if we know for other reasons (e.g. from its construction)
   that the atlas is good, then we show in the next section that it defines a VFC as before. \hfill$\er$
\end{rmk}

\subsection{From a good semi-additive atlas to the VFC}\label{ss:stame}

We now show that every cobordism class of good, oriented,  semi-additive atlases $\Kk$ on $X$ defines a unique
homology class $[X]_{\Kk}^{vir}\in H_*(X;\Q)$.   
As an  illustration, we  prove the following result.  Note that the statement makes sense because   the product atlas of  Definition~\ref{def:Kprod} is good by Lemma~\ref{le:Kprod}.

\begin{prop}\label{prop:prod}    If $\Kk_t$ is a good oriented, semi-additive Kuranishi atlas on $X_t$
of dimension $d_t$ for $t=1,2$, then
the fundamental class  $[ X_1\times X_2]^{vir}_{\Kk_1\times \Kk_2} $ of the product atlas 
 $\Kk_1\times \Kk_2$ on  $X = X_1\times X_2$ is the product 
$$
[X_1]^{vir}_{\Kk_1}\times [X_2]^{vir}_{\Kk_2} \in \check H_{d_1+d_2}(X_1\times X_2;\Q).
$$
\end{prop}

We construct the VFC by the method explained in \S\ref{ss:red}.  Thus the first step is to construct a reduction as in Definition~\ref{def:vicin}.
First we prove a general result about coverings. 

\begin{lemma}\label{le:cov0} Let $\Ii$ be a poset that is minimally generated in the sense of
 Definition~\ref{def:minimal}, and
suppose given a finite open cover of a compact metrizable Hausdorff
space $X=\bigcup_{I\in \Ii} P_I$  such that
$$
P_I\cap P_J\subset P_{I\vee J}\quad \forall I,J\in \Ii.
$$
Then
there exists a {\bf cover reduction} $\bigl(Z_I\bigr)_{I\in \Ii}$ where
 $Z_I\subset X$ is a (possibly empty) open subset satisfying
\begin{enumerate}
\item
$Z_I\sqsubset P_I$
for all $I$;
\item
if $\ov{Z_I}\cap \ov{Z_J}\ne \emptyset$ then $I\le J$ or $J\le  I$;
\item
$X\,=\, \bigcup_{I} Z_I$.
\end{enumerate}
\end{lemma}

\begin{proof}
Since $X$ is compact Hausdorff, we may choose precompact open subsets $Q_I'\sqsubset P_I$ 
that still cover $X$.  
We first claim we may enlarge these sets to $Q_I$ with $Q_I'\subset Q_I\sqsubset  P_I$ so that $Q_I\cap Q_J\subset Q_{I\vee J}$ for all $I,J$.  For this, we define $Q_I=Q_I'$ if $I$ is minimal, and then define $Q_I$ by induction over $|I|: = |m(I)|$ 
by setting 
$$
Q_I = Q_I'\cup \bigcup_{S \in \Ss(I)}\;
{\textstyle \bigcap_{j\in S} Q_{I_j}}, 
$$ 
where $\Ss(I)$ is the set of all collections $S: = (I_j)_{j\in S}$ such that 
$
\vee_{j\in S} I_j=I$ and $I_j\ne I.
$
Note that for each such collection $S$, the induction hypothesis implies that 
 $$ {\textstyle \bigcap_{j\in S} Q_{I_j}\;\; \sqsubset\;\; \bigcap_{j\in S} P_{I_j} \subset P_I.}
 $$
Therefore $Q_I\sqsubset P_I$ since it is a finite union of precompact subsets of $P_I$.  

Repeating this procedure $2M$ times where $M = \max_{I\in \Ii} |m(I)|$, we obtain 
 families of nested sets 
\begin{equation}\label{eq:FGI}
Q_I^0\,\sqsubset\, P_I^1 \,\sqsubset\, Q_I^1
\,\sqsubset\, P_I^2 \,\sqsubset\,\ldots \,\sqsubset\,
Q_I^{M}: = P_I,
\end{equation}
such that $P_I^k\cap P_J^k\subset P_{I\vee J}^k,$ and $ Q_I^k\cap Q_J^k\subset Q_{I\vee J}^k$ for all $k$.
For $I\in \Ii$, define
\begin{equation}\label{eq:ZGI}
Z_I \,: =\;P_I^{|I|} \;\less\; {\textstyle \bigcup_{|J|> |I|}} \; \ov{Q^{|I|}_J} ,\qquad |I|: = |m(I)|.
\end{equation}
These sets are open since they are the complement of a finite union of closed sets in the open set $P_I^{|I|}$. The precompact inclusion $Z_I\sqsubset P_I$ in (i) holds since $P_I^{|I|}\sqsubset P_I$.

To prove the covering property (iii),  let $x\in X$ be given. Then we claim that $x \in Z_{I_x}$ for
$$
I_x :=  \underset{I: x \in P^{|I|}_I}
{\textstyle  \bigvee} 
  I  \;\;\;\subset\;\;\; \{1,\ldots, N\} .
$$
To see this, denote by $S_x$ the collection of $I\in \Ii$ such that $x\in P^{|I|}_I$.
Then $I_x=\vee_{I\in S_x} I$, which implies $|I|\le |I_x|$ for all $I\in S_x$.
Hence 
$$
{\textstyle x\in \bigcap_{I\in S_x}P^{|I|}_{I}\subset \bigcap_{I\in S_x}P^{|I_x|}_{I}   \subset P^{|I_x|}_{I_x},}
$$
where the last step holds by assumption on the covering $(P^k_I)_{I\in \Ii}$.
On the other hand, 
if $x\in \ov{Q^{|I_x|}_J}$ for some $ J$ with $|J|> |I_x|$, then $x\in \ov{Q^{|I_x|}_J}\subset P^{|J|}_J$, which contradicts the definition of $I_x$.  Hence $x\in Z_{I_x}$ as claimed.

To prove the intersection property (ii), suppose to the contrary that $x\in \ov{Z_I}\cap\ov{Z_J}$ where $|I|\le |J|$ but $m(I)\less m(J)\ne \emptyset$ so that $I\not\le J$.  Then $I\vee J> J$ and 
we have 
$x\in \ov{Z_I}\cap \ov{Z_J} \subset Q^{|I|}_I\cap Q^{|J|}_J \subset Q^{|J|}_{J\vee I}$,
which is impossible because $ Q^{|J|}_{J\vee I}$ has been removed from $ \ov{Z_J} $.
Thus the sets $Z_I$ form a cover reduction.
\end{proof}

\begin{rmk}\rm  This is a  version of  \cite[Lemma~5.3.1]{MW1}, modified to apply to
a covering that is indexed by a minimally generated poset $\Ii$,
 rather than just 
by the poset $\Pp^*(\Mm(\Ii))$. \hfill$\er$
\end{rmk}

We define the {\bf reduction} of a good semi-positive atlas just as in Definition~\ref{def:vicin}.
Thus  it is  a collection of open sets $V_I\sqsubset U_I$ such that 
\begin{align*} &\
\io_\Kk(X)\;\subset\;{\textstyle \bigcup_{I\in \Ii}}\, \pi_\Kk(V_I\cap s_I^{-1}(0));\;\; \mbox{ and }\\
& \ \pi_\Kk(\ov{V_I})\cap 
\pi_\Kk(\ov{V_J}) \ne \emptyset\;\;\Longrightarrow \;\; I\le J \mbox{ or } J\le I.
\end{align*}
The next result is the semi-additive analog of Proposition~\ref{prop:cov2}.  

\begin{lemma}\label{le:redprod}  Suppose that $\Kk$ is a good semi-additive atlas with footprint cover $(F_I)_{I\in \Ii}$, and let $P_I\subset F_I$ be any open sets that cover $X$ and are such that $P_I\cap P_J\subset  P_{I\vee J}$.  Suppose further that  $W_I\subset U_I$ are $\Ga_I$-invariant sets such that $W_I\cap s_I^{-1}(0) = \psi_I^{-1}(P_I)$.  Then
$\Kk$ has a reduction $\Vv: = (V_I)_{I\in \Ii}$ such that $V_I\subset W_I$ 
for all $I$.   It is unique up to concordance.
\end{lemma}
\begin{proof}  Suppose first that the isotropy is trivial.  Choose a cover reduction $(Z_I\sqsubset P_I)_{I\in \Ii}$ of $X$ as in the previous lemma, and then for each $I$ choose an open set $W_I'\sqsubset W_I$  such that $W_I'\cap s_I^{-1}(0) = \psi_I^{-1}(Z_I)$.
For each $I$, let $\Cc(I) = \{J\in \Ii: I\le J, \mbox{ or } J\le I\}$ and then define
$Y_I: = \bigcup_{J\notin \Cc(I)} \ov {W_I'}\cap \pi_\Kk^{-1}(\pi_\Kk(\ov{W_J'}))$. Then $Y_I$ is closed because $\pi_\Kk: U_J\to |\Kk|$ is  homeomorphism for each $J$ by goodness.  Further, by construction  $$
Y_I\cap s_I^{-1}(0) \subset \bigcup_{J \notin\Cc(I)} \psi_I^{-1}(F_I\cap F_J) =  \emptyset.
$$
 Hence we may choose an open neighbourhood $\Nn(Y_I)$ of $Y_I$ in $U_I\less s_I^{-1}(0)$,
and then set $V_I: = W_I'\less \Nn(Y_I)$.   This has the same zero set of $W_I'$ 
and satisfies the other conditions by construction.
The statement about concordance follows as in \cite{MW1}.

If the isotropy is nontrivial, the argument is the essentially  same except that we construct the corresponding sets $\uW_I\subset\uU_I$ on the intermediate level and then define $W_I: = \pi_I^{-1}(\uW_I)$,
\end{proof}

The notions of   Kuranishi concordance, cobordism
and of orientation bundle extend immediately
to the semi-additive case.   The only difference is that because the atlases that we now consider cannot be tamed, we must assume 
from the outset that all  atlases under consideration are good.  Thus the relevant relation is that of  {\bf good cobordism}.
 We can now proceed to construct representatives for the fundamental class $[X]^{vir}_\Kk$ as in \S\ref{ss:zero}.
Thus, given a reduction $\Vv$ of a  good atlas, we  use the method of Proposition~\ref{prop:ext} to 
  construct a reduced perturbation $\nu: = \bigl(\nu_I:V_I\to E_I\bigr)_{I\in \Ii}$ that satisfies 
  the conditions of Definition~\ref{def:sect2}
  whose zero set
$$
\bZ^\nu: = {\textstyle \bigcup_{I\in \Ii, \ga\in \Ga_I}}  (\s_I|_{V_I} + \nu)^{-1}(0)
$$
can be completed to a wnb groupoid $\Hat\bZ^\nu$ as in Proposition~\ref{prop:zero}.  
This construction is done by induction over $|m(I)|$, which is permissible because $I < J \Longrightarrow |m(I)| < |m(J)|$. 
It follows that  $\Hat\bZ^\nu$ has a fundamental class that is represented in the rational singular homology of a small neighbourhood  in $|\Kk|$ of the zero set $\io_\Kk(X)$.  
Taking a sequence $\nu_k$ of admissible perturbations with norm converging to $0$, one obtains   an element in
the \v Cech homology of $X$.  
It is unique because, as in Proposition~\ref{prop:ext2},
 it is possible to join any two admissible, precompact, transverse  perturbations of $\Kk$ by an admissible, precompact,
 transverse 
 perturbation over the product $[0,1]\times \Kk$.
Therefore, finally we obtain the following result.
  
\begin{prop}\label{prop:semivir}   Every oriented, good, semi-additive atlas $\Kk$ on $X$ of dimension $d$  determines a \v Cech homology class $[X]^{vir}_\Kk\in \check H_d(X;\Q)$  that depends only on the good oriented cobordism class of $\Kk$.
\end{prop}

\begin{proof}[{\bf Proof of Proposition~\ref{prop:prod}}]  
Again, let us first suppose that the isotropy is trivial. Since  each $\Kk_i$ is good, we may choose reductions $\Vv_i$ of $\Kk_i$ and admissible perturbations
 $\nu_i:\bB_{\Kk_i}\big|_{\Vv_i}\to 
\bE_{\Kk_i}\big|_{\Vv_i}$.
For each chart $\bK_{I_i}\times \bK_{I_2}$ of the product atlas, define $P_{I_1,I_2} = V_{I_1}\times V_{I_2}\subset 
U_{I_1}\times U_{I_2}$.  Since $P_{I_1,I_2}\cap P_{J_1,J_2} = P_{I_1\vee J_1,I_2\vee J_2}$, we can apply 
Lemma~\ref{le:redprod} to construct a reduction  $\Vv = (V_{I_1,I_2})_{(I_1,I_2)\in \Ii_{\Kk_1}\times \Ii_{\Kk_2}}$ with $V_{I_1,I_2}\subset P_{I_1,I_2}  : = V_{I_1}\times V_{I_2} $.\footnote
{
Note that the sets $P_{I_1,I_2}$ themselves do {\it not} satisfy the requirements of a reduction: for example if $I_1 \lneq J_1$ and $I_2 \gneq J_2$ then $(V_{I_1}\times V_{I_2}) \cap (V_{J_1}\times V_{J_2}) = 
V_{J_1}\times V_{I_2}$ but $(I_1,I_2)$ is not comparable to $(J_1,J_2)$ in the product order. }

Now define
$$
\nu: \bB_{\Kk_1\times \Kk_2}\big|_{\Vv}\to 
\bE_{\Kk_1\times \Kk_2}\big|_{\Vv},
$$
by setting
$$
 \nu_{I_1,I_2}: V_{I_1,I_2}\to E_{I_1}\times E_{I_2},\quad     (u_1,u_2) = \bigl(\nu_1(u_1), \nu_2(u_2)\bigr).
 $$ 
 It is easy to check that $\nu$ satisfies the necessary conditions. In particular,
 $\s|_\Vv+\nu$ is transverse to $0$ on each $V_{I_1,I_2}$  with zero set equal to the restriction to 
 $V_{I_1,I_2}$ of the product  of the zero sets of
$\nu_{I_1}$ and $\nu_{I_2}$.  It follows that the realization $|\bZ^\nu|$  of the zero set is the product 
$|\bZ^{\nu_1}|\times |\bZ^{\nu_2}|.$    Hence the resulting homology class $[X_1\times X_2]^{vir}_{\Kk_1\times \Kk_2}$ in $\check H_{d_1+d_2}(X_1\times X_2)$ is the product $[X_1]^{vir}_{\Kk_1} \times [X_2]^{vir}_{\Kk_2}$.

If the isotropy is nontrivial, the argument is again the essentially  same. By  Proposition~\ref{prop:zero}, the zero set
can now be constructed as
a wnb groupoid $\Hat\bZ^\nu$. The inclusions $V_{I_1,I_2}\to V_{I_1}\times V_{I_2}$  induce a functor 
from $\Hat\bZ^\nu$ to the product wnb groupoid
 $\Hat\bZ^{\nu_1}\times \Hat\bZ^{\nu_2}$ that takes each local zero set $Z_{I_1,I_2}\subset V_{I_1,I_2}$ diffeomorphically 
 onto the subset $(Z_{I_1}\times Z_{I_2})\cap   V_{I_1,I_2}$ of $V_{I_1}\times V_{I_2}$, preserving its weight
 $\frac 1{|\Ga_{I_1}|\ |\Ga_{I_2}|}$.  
 Because the weighting function on $|\Hat\bZ^\nu|_\Hh$ is induced by the weights of the components $Z_{I_1,I_2}$,
it follows that  this functor is a layered covering in the sense of \cite{Mbr}, and hence
 pushes the fundamental class of $\Hat\bZ^\nu$ forward to that of the product.
\end{proof}

\subsection{From semi-additive to tameable atlases}  \label{ss:hybrid}
In this section we first show that every semi-additive atlas is cobordant to a tameable atlas, and then discuss the relevance of this result to
Hamiltonian Floer theory.

Remark~\ref{rmk:tamea} defines the notion of a (smooth) {\bf tameable} atlas that we will call $\Kk'$ to avoid confusion.  
The transition charts of such an atlas  have the usual indexing set $\Ii_{\Kk'}\subset \Pp^*(\Mm(\Ii_{\Kk'}))$ (where $\Mm(\Ii_{\Kk'})$ is the set of  minimal elements in $\Ii_{\Kk'}$), and the obstruction spaces 
$E_I' = \prod_{i\in I} E_i'$
are products.  However the isotropy groups $(\Ga_I')_{I\in \Ii_{\Kk'}}$ are no longer required to be the products
of the corresponding basic groups.  For example, 
suppose given a tuple $(\Ii, \Aa,\tau)$ with associated spaces and groups  $E_\al,\Ga_\al$ as in Definition \ref{def:sadd},
where $\Ii\subset  \Pp^*( \Mm(\Ii))$.  Then
the transition charts might be indexed by the subsets $K\in \Pp^*( \Mm(\Ii))$ and have obstruction spaces
$E_K': = \prod_{i\in K} \prod_{\al\in \tau(i)} E_\al$ as usual, but the groups $\Ga_K'$ could
have the form $ \prod_{\al\in \tau(K)} \Ga_\al$.
Then $E_K'$ might contain several factors of $E_\al$, while $\Ga_K'$ contains only one, acting diagonally on the $E_\al$ factors in $E_K'$, and trivially on the rest.
(For an explicit example, see Proposition~\ref{prop:hybrid1} below.)

Since the corresponding intermediate atlas with domains $\qu{U_K'}{\Ga_K'}$ is a filtered topological atlas, all taming results apply to these atlases.

Here is the first main result of this section.

\begin{prop} \label{prop:hybrid1} 
Every semi-additive (weak) atlas  
$\Kk = (\bK_I, \Hat\Phi_{IJ})_{\Ii,\Aa,\tau}$ 
has a canonical extension to a tameable (weak) atlas $\Kk'$ with the same basic charts   as $\Kk$.  
\end{prop}
\begin{proof}  
Define 
$$
\Ii_{\Kk'}: = \{K\subset \Pp^*(\Mm(\Ii))\ \big| \ F_K: ={\textstyle \bigcap_{i\in K}} F_{\tau(i)}\ne \emptyset\},
$$
where $\tau:\Mm(\Ii)\to \Aa$ is part of the defining structure for  $\Kk$.
By assumption on $\Ii$, the least upper bound function  
\begin{equation}\label{eq:ell}
\ell: \Pp^*(\Mm(\Ii))\to \Ii,\quad 
K \mapsto \ell(K): = \lub(K) = \vee_{i\in K}\;  i 
\end{equation}
 defines a map $\ell:\Ii_{\Kk'}\to \Ii$ such that
$F_K = F_{\ell(K)}$ for all $K\in \Ii_{\Kk'}$.  
The weak atlas $\Kk'$  has  charts indexed by $K\in \Ii_{\Kk'}$. 

For simplicity, let us first consider the case when all isotropy groups are trivial. 
For $i\in  \Mm(\Ii)$ define $E_i': =E_i =   \prod_{\al\in \tau(i)} E_\al$. 
Then the obstruction space $E_K'$ for the chart $\bK_K'$ is  $E_K'= \prod_{i\in K}E_i'$.
This may also be written as the product
$$
E_K': = \prod_{\al\in \tau(\ell(K))} E_{\al}^{\ m_{\al,K}},
$$
 where the multiplicities $m_{\al,K}\ge 1$ are defined as follows:
$$
m_{\al,K} = \big| \{i\in K\ | \ \al\in \tau(i)\}\big |.
$$
Therefore we can write the elements $e_K'$ of $E_K' $ as tuples $\bigl(\vec e_\al\bigr)_{\al\in\tau(\ell(K))}$ where
$\vec e_\al = (e^k_\al)_{1\le k\le m_{\al,K}}$ is an $ m_{\al,K}$-tuple of vectors in $E_\al$.
With this notation the map  $\Hat\phi_{KL}':  E_K'\to E_L'$ is the obvious inclusion with image equal to
\begin{equation}\label{eq:hphiI}
\Hat\phi_{KL}'(E_K') = \Bigl\{\bigl( \vec e_\al\bigr)_{\al\in\tau(\ell(L))}\ \big | \ e_\al^k= 0,\
m_{\al,K} < k\le m_{\al,L}
 \Bigr\}\subset E_L'.
\end{equation}
Further, there is a projection:
\begin{equation}\label{eq:siI}
\si_K: E_K'\to E_{\ell(K)},  \quad \bigl(\vec e_\al\bigr)_{\al\in\tau(\ell(K))} \mapsto  \bigl(\si_\al(\vec e_\al)\bigr)_{\al\in\tau(\ell(K))},
\end{equation}
where we define $\si_\al(\vec e_\al): = \sum_{k=1}^{m_{\al,K}} e^k_\al \in E_\al$.
This map 
satisfies  the compatibility condition:
$$
\si_L\circ \Hat\phi_{KL}' = \Hat\phi_{KL}\circ \si_K, \quad E_K'\to E_L.
$$
Now define the domains $U_K'$ of the charts of $\Kk'$ and the section $s_K'$ by setting:
\begin{equation}\label{eq:Uadd}
U_K': = \Bigl\{ (e_K',u)\in E_K' \times U_{\ell(K)} \ | \  s_{\ell(K)}(u) = \si_K( {e_K'})\Bigr\}, \quad s_K'({e_K'},u): = {e_K'}.
\end{equation}
Thus, since  $\ell(i) = i$, the basic chart $\bK_i= (U_i',E_i'=E_i,  s_i', \psi_i')$ 
has  domain $U_i'$ 
consisting of all pairs $ (e_i,u)\in E_i \times U_i $ with $s_i(u) = e_i\in  E_i$.
Thus we can identify $U_i'$ with $U_i$, and take the footprint map $\psi_i' : = \psi_i$ so that  $\bK_i'\cong \bK_i$.
  On the other hand, if  $|K|>1$ there is a fibration
$$
f_K: U_K'\to U_{\ell(K)}, \quad ({e_K'},u)\mapsto (\si_K'({e_K'}),u) \mapsto u\in U_{\ell(K)},
$$
where the second map is a diffeomorphism since $\si_K'({e_K'})\in E_{\ell(K)} = s_{\ell(K)}(u)$ is uniquely determined by $u$.  Thus the fiber of $f_K$ at $u\in U_K$ is the affine  space $\{e\in E_K' \ \big| \ \si_K(e) = s_{\ell(K)}(u)\}$.
Further $f_K$ restricts to a diffeomorphism from $(s_K')^{-1}(0) = \{ ({e_K'},u)\in U_K'\, | \  e_K' = 0\}$
to $s_{\ell(K)}^{-1}(0)$.  Therefore the footprint map 
$$
\psi_K': = \psi_{\ell(K)}\circ f_K: (s_K')^{-1}(0)\to F_K
$$ 
can be identified with $\psi_{\ell(K)}: (s_{\ell(K)})^{-1}(0)\to F_{\ell(K)} = F_K$ and so induces a  homeomorphism 
$(s_K')^{-1}(0)\to F_K$.
Therefore the chart  
$$
\bK_K': = \bigl(U_K',   E_K', s_K', \psi_K'\bigr),
$$
has the footprint $F_K = F_{\ell(K)}$. 
If $K\subset L$  define  
\begin{equation}\label{eq:UKL'}
U_{KL}': =\bigl \{(e_K',u)\in U'_K\ | \ \  u\in  U_{\ell(K),\ell(L)}
\bigr\} = f_K^{-1}(U_{\ell(K),\ell(L)}),
\end{equation}
and define the coordinate change $\phi_{KL}': U_{KL}\to U_L$ to be 
the product  $\Hat\phi'_{KL}\times \phi_{KL}$.
This is compatible with the section $s_{K}'$ and footprint maps, and satisfies the weak cocycle condition.
Since $\Kk'$ is tameable by construction, this completes the proof of (i) in the case when the isotropy is trivial.

In the general case, we use the same definition of $E_K'$ and $U_K'$, and 
take $\Ga_K': = \Ga_{\ell(K)} = \prod_{\al\in \ell(K)} \Ga_\al$.
It acts on $E_K'$ in the obvious way, with each factor $\Ga_\al$ acting diagonally on the tuple
$\vec e_\al \in E_\al^{m_\al}$, and trivially on the other components. By taking the product of this action with its action on $u\in U_{\ell(K)}$,
we obtain an action of $\Ga_K' = \Ga_{\ell(K)}$ on 
 the domain $U_K'$ in \eqref{eq:Uadd} that lifts its action on $U_K$.  Hence it is compatible with the footprint maps, 
 the projections $\si_K$ 
 and the fibration $f_K$.  Moreover, because $(s_K')^{-1}(0)$ can again be identified with $s_{\ell(K)}^{-1}(0)$,
 the fact that the footprint map induces a homeomorphism $\qu{(s_K')^{-1}(0)}{\Ga_K}\to F_K = F_{\ell(K)}$ can be deduced from the corresponding result for the chart $\bK_{\ell(K)}$.
 As for coordinate changes, we define $\rho_{KL}^\Ga: \Ga_L'\to \Ga_K'$ to equal $\rho_{\ell(K)\ell(L)}^\Ga$, and define
\begin{equation}\label{eq:TUKL'}
 \TU_{KL} = \bigl\{(\Hat\phi_{JK}(e_K'), u) \in U_L' \ \big| \ e_K'\in E_K', u\in  \TU_{\ell(K)\ell(L)}\bigr\}
\end{equation}
Then $f_L: \TU_{KL}\to \TU_{\ell(K)\ell(L)}$ is surjective, so that $\ker(\rho_{KL}^\Ga)$    acts freely on 
$ \TU_{KL}$ with quotient $U_{KL}'$ as defined in \eqref{eq:UKL'}.
Thus we may take $$
\rho'_{KL} = \id_{E_K'}\times \rho_{\ell(K)\ell(L)}:\TU_{KL}\to U_{KL}.
$$
It is then easy to see that all the compatibility requirements in Remark~\ref{rmk:tamea} are satisfied.

It remains to  check that if $\Kk$ is an atlas (rather than a weak atlas), then so is $\Kk'$.
This follows  
  because the domains $U_{KL}'$ of the coordinate changes 
in $\Kk'$ fiber over the corresponding domains $U_{KL}$ in $\Kk$.   Hence  if  
 $\Kk$ satisfies the  cocycle condition then $\Kk'$ does as well.
Hence there is a well defined category $\bB_{\Kk'}$. Further the fibrations $(f_K)$ fit together to provide a functor  
$\f: \bB_{\Kk'}\to \bB_{\Kk}$
 that intertwines the section and footprint functors.
\end{proof}

\begin{rmk}\label{rmk:fK}  \rm Definition~\ref{def:Kumap} spells out a notion of map between 
(weak) atlases that readily extends to the semi-additive atlases considered here; all we need is a 
suitable map $\ell$ between the indexing sets.  In particular, the
 maps $f_K: U_K'\to U_{\ell(K)}$ defined above extend to a map $\f: \Kk'\to \Kk$  with $\ell:\Ii_{\Kk'}\to \Ii$
 given by \eqref{eq:ell}, 
 $f_K^\Ga: \Ga_K'\to \Ga_{\ell(K)}$ equal to the identity, and $\Hat f_K: E_K'\to E_{\ell(K)}$ equal to the projection $\si_K$ 
 of \eqref{eq:siI}. \hfill$\er$
\end{rmk}

We next show that the atlas $\Kk'$ constructed above is concordant to $\Kk$.

\begin{prop}\label{prop:hybrid2} Let  $\Kk$ and $\Kk'$ be as in Proposition~\ref{prop:hybrid1}.  Then
there is a (weak) cobordism atlas $\Kk^{01}$ on $[0,1]\times X$ from $\p^0 \Kk^{01} =\Kk$ to
$\p^1 \Kk^{01} =\Kk'$ and with  basic charts 
$$
[0,\tfrac 23)\times \bK_i,\;\; (\tfrac13,1]\times \bK_i,\quad i\in \Mm(\Ii).
$$
\end{prop}
\begin{proof}  
We define  the cobordism (weak) atlas to have an indexing set $\Ii^{01}$ that is a hybrid of the given order set $\Ii$ with the
standard order set $\Ii'$ of the atlas $\Kk'$.  Thus 
 \begin{itemlist}
\item $\Mm(\Ii^{01})$ is the union of two copies of $\Mm(\Ii)$ that for clarity we denote $\Mm(\Ii^0): = \Mm(\Ii)$ and $\Mm(\Ii^1): = \Mm(\Ii')$; 
\item we take $\Aa^{01} = \Aa \sqcup \Mm(\Ii^1)$, and 
\item we define 
\begin{align*}
 \tau^{01}: \Mm(\Ii^0)\sqcup \Mm(\Ii^1)&\  \to\  \Pp^*(\Aa^{01})\subset \Pp(\Aa)\times \Pp(\Mm(\Ii^1)),\\
 (I,K) &\ \mapsto\ (\tau(I), m(K))
\end{align*}
\end{itemlist}
The transition charts are of three types: 
 \begin{itemlist}
\item  $[0,\frac 23)\times \bK_{(I,\emptyset)}$ where 
$\bK_{(I,\emptyset)} = \bK_I $ is a transition chart in $\Kk$;
\item  $ (\frac 13,1]\times \bK_{(\emptyset,K)}$ where 
$\bK_{(\emptyset, K)} = \bK_K $ is a transition chart in $\Kk'$;
\item \lq\lq transition" charts $(\frac 13,\frac 23)\times \bK_{(I,K)}$ where $\bK_{(I,K)}$ is a chart on $X$ with footprint $F_{I\vee \ell(K)} $, obstruction space $E_I\times E'_K$, group $\Ga_{I,K} =  \Ga_{I\vee\ell(K)}$, and domain
$$
U_{(I,K)} = \bigl\{(e_I,e_K',u) \ \big| \ u\in U_{I\vee\ell(K)}, \Hat\phi_{I (I\vee\ell(K))}(e_I) + \si_K(e'_K) = s_{I\vee\ell(K)}(u)\bigr\}.
$$
 The section $s_{(I,K)}$ is given by $(e_I,e_K',u)\mapsto (e_I,e_K')$.
  \end{itemlist}
 If the isotropy groups are trivial, then
the coordinate changes for the cobordism charts involve the obvious inclusions on
 the obstruction spaces together with the following maps on the domains:
\begin{align*}
 (\tfrac 13,\tfrac 23) \times  U_{(I,\emptyset)} \to  (\tfrac 13,\tfrac 23) \times U_{(J,K)}:&\  (t,u)\mapsto \bigl(t,\Hat\phi_{IJ}(s_I(u)), 0, \phi_{I(J\vee\ell(K))}(u)\bigr),\\
 (\tfrac 13,\tfrac 23) \times  U_{(\emptyset, K)}  \to  (\tfrac 13,\tfrac 23) \times U_{(J,L)}: &\  (t,0,e_K',u)\mapsto \bigl(t,0,\Hat\phi_{KL}(e_K'), \phi_{\ell(K)(J\vee\ell(L))}(u)\bigr),\\
 (\tfrac 13,\tfrac 23) \times U_{(I, K)} \to  (\tfrac 13,\tfrac 23) \times U_{(J,L)} : &\  (t,e_I,e_K',u)\mapsto\\
 & \qquad  \bigl(t,\Hat\phi_{IJ}(e_I),\Hat\phi'_{KL}(e_K'), \phi_{(I\vee \ell(K))\ (J\vee\ell(L))}(u)\bigr).
\end{align*}
 It is straightforward to check that the cocycle condition holds, so that this is a cobordism atlas.  
 
If the isotropy is nontrivial, then we define the domains $$
 (\tfrac 13,\tfrac 23) \times \TU_{(I,K) (J,L)}\subset  (\tfrac 13,\tfrac 23) \times U_{(J,L)}
$$ 
of the coordinate changes as in \eqref{eq:TUKL'}, with  elements
$$
\Bigl\{\bigl(t,\Hat\phi_{IJ}(e_I),\Hat\phi_{KL}'(e_K'), u\bigr) \ \big| \  e_I\in E_I, e_K'\in E_K', u\in \TU_{(I\vee \ell(K))\,(J\vee\ell(L))}, t\in (\tfrac 13,\tfrac 23) \Bigr\},
$$
where $\TU_{(I\vee \ell(K))\, (J\vee\ell(L))}$ is the corresponding domain in the atlas $\Kk$.
We must check that $\TU_{(I,K) (J,L)}$  supports a free action of the kernel of the group homomorphism $\Ga_{(J,L)}\to \Ga_{(I,K)}$. 
But $\Ga_{(J,L)} = \Ga_{J\cup\ell(L)} $ is the isotropy group for the chart $\bK_{J\vee\ell(L)}$ in $\Kk$.  Hence this kernel acts freely on
$\TU_{(I,K) (J,L)}$ because it does on the corresponding set $\TU_{(I\vee \ell(K))\, (J\vee\ell(L))}$ in $\Kk$.  The rest of the details are straightforward, and
 are left to the interested reader.  
\end{proof}

\begin{cor} If both $\Kk$ and the cobordism $\Kk^{01}$ constructed above are good, oriented  atlases,
then  $[X]^{vir}_\Kk= [X]^{vir}_{\Kk'}$.
\end{cor}
\begin{proof}  This follows by the arguments that establish  Proposition~\ref{prop:semivir}.
\end{proof}

\begin{rmk}\label{rmk:saddatlas}\rm 
%
(i)   The proof of Proposition~\ref{prop:hybrid2} builds a cobordism atlas over $[0,1]\times X$ that has a natural restriction over $ \{\frac 12\}\times X$ to an atlas that contains both $\Kk$ and $\Kk'$.
   Hence  these two atlases are directly commensurate.  Conversely, given 
 two standard atlases that are directly commensurate (or, more generally, commensurate) in the sense of Definition~\ref{def:commen}, one can build a cobordism atlas on $[0,1]\times X$  as above and hence deduce that the two atlases are concordant.\MS
 
 \NI (ii)  Formally, when the isotropy is trivial  the construction of $\Kk'$ from $\Kk$ given above is very similar to the construction in Proposition~\ref{prop:Euler}
 that builds an atlas from a nontrivial bundle over a manifold.  
 Indeed, if the basic charts of $\Kk$ are  GW charts, the corresponding standard GW atlas with these basic charts can be mapped into $\Kk'$ by an open embedding in the sense of Definition~\ref{def:Kumap}; the only real difference between the constructions is that in the present context we include 
 the full fiber of $f_I$ into the domain $U_I$.  
 Just after Remark~\ref{rmk:E} we sketched how to 
 obtain a (standard) atlas in the corresponding orbifold case.  This construction was possible because in this case the isotropy is determined by its action on the  orbifold $E$ so that we could use the results of \S\ref{ss:orb} to build the needed transition charts.  
 We explain briefly below how  in the Gromov--Witten situation it is also possible to promote 
  the tameable atlas $\Kk'$ to a standard atlas.
For general atlases, 
 it is not immediately clear how we might do this.   However,  this does not matter since
 all our constructions work for  tameable atlases  just as well as they do for standard atlases.   
 $\hfill\er$ 
\end{rmk}

\begin{lemma}\label{le:good3}  Suppose that $\Kk$ is a good semi-additive atlas  with trivial isotropy.
Then the atlas $\Kk'$ constructed  in Proposition~\ref{prop:hybrid1} and the cobordism atlas $\Kk^{01}$ 
 constructed  in Proposition~\ref{prop:hybrid2} are also good.
\end{lemma}
\begin{proof}
Define the vector space $E'_\Mm: = \prod_{i\in \Mm(\Ii)} E_i'$, and consider the product category
$E'_\Mm\times \bB_\Kk$ with objects $\bigsqcup_K E_\Mm'\times U_K'$ and morphisms 
$\id_{E_\Mm'}\times \Mor_{\Bb_\Kk'}$.
Then  there is an injective functor $\io_\Mm: {\Kk'}\to E'_\Mm \times \Kk$ defined on objects by
$$
\io_\Mm: \;\; (K,(e'_K,u))) \mapsto \bigl(\io_K^E(e'_K),\ f_K(K,(e'_K,u))\bigr),
$$
that intertwines the 
coordinate change maps $\phi_{KL}'$ of $\Kk'$ with the products $\id \times \phi_{\ell(K)\ell(L)}$.
It follows that $\io_\Mm$ descends to give continuous injections  
$$
|\io_\Mm|:  |\Kk'|\to |E_{\Mm}' \times \bB_{\Kk}| \to  E'_\Mm \times |\Kk|,
$$
 where the last space has the product topology.
Then the composite $$
|\io_\Mm|\circ \pi_{\Kk'}: U_I'\to E'_\Mm \times |\Kk|
$$ is a homeomorphism to its image, which implies that
$\pi_\Kk'$ is.  The Hausdorff and metrizability conditions  in Definition~\ref{def:good}  follow similarly, while $\im \phi_{IJ}'$ is closed in $U_J'$ since it fibers over $\im\phi_{IJ}$ with fibers that are closed in $E'_J$.
The proof that $\Kk^{01}$ is good is similar.
\end{proof}

\begin{rmk}\label{rmk:good3}\rm
 If $\Kk$ has nontrivial isotropy, then the above argument extends to show there is a continuous map
$|\Kk'|\to \qu{E'_\Mm}{\Ga'_\Mm} \times |\Kk|$, where
$\Ga_\Mm' = \prod_{\al\in \Ga}\Ga_\al$ acts diagonally on the factors of $E_\al$ in $E'_\Mm$.  However, in general this map
is no longer injective and so does not quite give what we need.  When $\Kk$ is good, one should nevertheless be able to
prove directly that the atlases $\Kk'$ and $\Kk^{01}$ 
constructed above satisfies all the conditions in Definition~\ref{def:good} 
except perhaps for metrizability.  The problem here is that we assumed so little about 
the metric topology  $(|\Kk|,d)$  that it is not clear how to proceed.  However, the original proof of
 metrizability  in Lemma~\ref{le:metriz} uses the fact that the atlas under consideration is \lq\lq preshrunk".  
 This implies that
 its domains can be coherently compactified,  so that $|\Kk|$ injects into a compact space $|\ov\Kk|$ 
 that can be directly shown to be metrizable.
The good semi-additive atlases that we encounter  in practice 
are constructed from preshrunk tame atlases, for example, by taking products.  Hence this argument should adapt 
to prove the analog of Lemma~\ref{le:good3} in all cases of interest.  
Since we do not fully develop any such applications here, we will leave this point for now. \hfill$\er$
 \end{rmk}

Propositions~\ref{prop:hybrid1} and ~\ref{prop:hybrid2} describe abstract constructions that apply to all semi-additive atlases.  We now explain how one can do better  in the Gromov--Witten context, promoting the atlas $\Kk'$ from a tameable atlas to a standard atlas $\Kk''$, because  in this case  the isotropy groups are controlled   explicitly via the added marked points $\bw$.

\begin{rmk}\label{rmk:GWext}  \rm 
 In \S\ref{ss:GW}  we gave a  recipe for constructing  basic and transition charts using  stabilizations that are determined by the choice of a slicing manifold $Q_i\subset M$ for each basic chart $\bK_i$.
The isotropy group $\Ga_i$ acts on the tuples  $(\vec e_i, [\bn,\bw,\bz,f])$ in $U_i$  by permuting  the corresponding admissibly ordered set $\bw_i$ of marked points and also the components $\vec e_i$ of the obstruction space $E_i: = \prod_{\ga\in \Ga_i} E_i^0$. (This structure is easiest to see when one uses the coordinate free definitions in  \eqref{eq:coordf1} (for a basic chart) and \eqref{eq:coordf3},
\eqref{eq:coordf4} for transition charts.  Note also that
even though the elements in $\bw$ are distinct, the isotropy action may not be free
because of the effect of renormalization.)
The elements of the transition chart $\bK_K$ contain tuples $\vec{\ul{e}}: = (\vec e_i)_{i\in K}, \ul{\bw}: = (\bw_i)_{i\in K}$ of such elements that are again permuted by the components of the groups $\Ga_{K}: = \prod _{i\in K} \Ga_i$.
   The construction of a standard atlas treats  the extra data $E,Q$ for each basic chart as distinct. For example,  if 
we choose the same obstruction space $E_1=E_2$ and slicing manifold $Q$ for 
two different basic charts $\bK_1,\bK_2$ and if the two groups $\Ga_1,\Ga_2$ 
are isomorphic, then the standard construction
adds two copies of the marked points $\bw_1,\bw_2$ (agreeing modulo admissible labellings; cf. Remark~\ref{rmk:lifts})  to the elements in the domain 
of the transition chart $U_{12}$ as well as two copies $E_{12} = E_1\times E_2$ of the obstruction space 
thus giving rise to an  action of $\Ga_1\times \Ga_2$ on the domain.
However,  we could instead simply add one set of marked points and take the group to be $\Ga_1$ with the diagonal action on $E_{12} = E_1\times E_2$, thereby getting a simplified chart whose domain  embeds as an open subset of  the usual transition chart $\bK_{12}$.   Note that this simplification of the structure of the chart does not affect which stable maps 
$[\bn,\bz,f]$ lie in the domain of this chart: the map $f$ satisfies the equation $\pbar f = \la(\ul{\vec e})|_{{\rm graph} f}$  
where the map $\la$ sums over  the action of the relevant group $\Ga_I$ over all components of $E_I$, and one can check that having repeated copies of one of the components $E_i, \Ga_i$ does not increase the image of $\la$.  
 
We explain below how the tameable atlas $\Kk'$ constructed in Proposition~\ref{prop:hybrid1} can similarly be considered as a simplified version of the standard GW atlas with the same basic charts and obstruction spaces.
$\hfill\er$
\end{rmk}

The  variations on the construction described in Remark~\ref{rmk:GWext} above, though possible, are not very useful when the moduli spaces $X$ consist of closed curves, since in this case one can always build  a standard atlas.    However, as we now explain  they are relevant in the context of {\bf Hamiltonian Floer theory}.\footnote
{
The general set-up  is described and analyzed in Pardon~\cite{Pard}, and  we do not attempt a complete treatment here.  However, 
 what we have done above is enough  to deal formally with trajectories that are once broken; to deal with more breaks one 
would need to introduce a (collared) corner structure. 
}

\begin{rmk}\label{rmk:HamFT}\rm
  In this situation we have a family of (noncompact)  moduli spaces $X(p,r)$ consisting of unbroken trajectories from $p$ to $r$ (where $p,r$  are periodic orbits of the Hamiltonian flow) 
whose boundary decomposes into products of the form $X(p,q_1)\times \dots\times X(q_k,r)$ consisting of trajectories that are broken at $k$ interior orbits $q_1,\dots,q_k$. For example if the spaces ${X}(p,q), X(q,r)$ are compact, then $X(p,r)$ will have a boundary component given by the product $X(p,q)\times X(q,r)$. 
To be consistent with our approach to cobordism, we will thicken up $X(p,r)$ near its boundary by adding appropriate collars (in general, cornered).
In particular, we will assume  that this component of the boundary of $X(p,r)$ has 
 a neighborhood that can be  identified with $[0,\eps)\times X(p,q)\times X(q,r)$.
In order to define a chain complex one needs to build a coherent family of atlases, reductions, and sections.    By induction, one can assume that atlases $\Kk_{p,q}$ for $X(p,q)$ and  $\Kk_{q,r}$ for $X(q,r)$ are already constructed, and need  to extend the product
 $\Kk_{p,q}\times  \Kk_{q,r}$   to a good  atlas over 
  $X(p,r)$.   Once this is done, the arguments in Proposition~\ref{prop:prod} would allow us to extend a product  perturbation section
  $\nu_{p,q}\times \nu_{q,r}$ over the interior of the atlas.

Given a basic chart $\bK_1$ at $\tau_1\in X(p,q)$ and $\bK_2$ at $\tau_2\in X(q,r)$, the  elements of the domain $U_1\times U_2$ of the product chart $\bK_1\times \bK_2$ are tuples whose map component $f=(f_1,f_2)$ is a broken trajectory from $p$ to $r$ with first component $f_1$ (going from $p$ to $q$)   stabilized by points in $f_1^{-1}(Q_1)$ and second component 
$f_2$(going from $q$ to $r$)   stabilized by points in $f_2^{-1}(Q_2)$.  By gluing at $q$, we can  build a chart with domain
 $[0,\eps)\times  U_1\times U_2$ that models a neighbourhood of $(\tau_1,\tau_2)$ in $\ov X(p,r)$.
Note that this chart has the essentially the same form  as a basic GW chart that would be constructed for trajectories going straight from $p$ to $r$; the only difference is that  the slicing manifold is the union $Q_1\cup Q_2$ of two submanifolds of the ambient manifold $M$ which we do not assume to be disjoint.  However, because the associated marked points $\bw_t: = f^{-1}(Q_t)$ belong to distinct parts of the domain of the curve it is not hard to extend the definition to allow this variation.  Therefore we  can use the GW construction to build a standard atlas $\Kk_{p,r}$  over
$X(p,r)$ that restricts over 
$[0,\eps)\times  X(p,q)\times X(q,r) $ to  $[0,\eps)\times  \Kk''$, where $\Kk''$ is the standard atlas over $X(p,q)\times X(q,r)$
 with these product basic charts.
On the other hand, the product $[0,\eps)\times  \Kk_{p,q}\times \Kk_{q,r}$
is another (semi-additive) atlas over the neighbourhood  $[0,\eps)\times  X(p,q)\times X(q,r)$, which, by the inductive hypothesis 
and the construction in Proposition~\ref{prop:prod} already has a reduction and a perturbation section.
To be able to extend these structures over $\Kk_{p,r}$ it suffices to check that the product atlas $\Kk_{p,q}\times \Kk_{q,r}$ is concordant to the standard atlas $\Kk''$.  Note that we do need the concordance to be good (in particular metrizable)
in order to be able to apply 
the results about the extension of perturbation sections; see  Proposition~\ref{prop:ext2} and Remark~\ref{rmk:good3}.
 
 If the isotropy is trivial,
 the arguments we gave above prove what we need.
    Indeed the construction   of 
 $\Kk'$ from the product semi-additive $\Kk$ given in Proposition~\ref{prop:hybrid1} 
is mirrored in the GW set up:   
the only difference in the constructions is that we usually do not include the whole fiber of $f_K$ into the chart $U_K'$.
 Thus the category of the GW atlas embeds as an open subcategory of $\bB_{\Kk'}$.  Hence, if we similarly restrict the domains of the transition charts over $(\frac 13,\frac 23)$ we can define the cobordism as in   
 Proposition~\ref{prop:hybrid2}.  Moreover, it is good by Lemma~\ref{le:good3}. 

   In the case of nontrivial isotropy,  the two atlases $\Kk'$ and $\Kk''$  are different because the 
   chart with label $K$ has isotropy group   $\Ga_K': =  \prod_{\al\in \ell(K)} \Ga_\al$ in the first case and
    $\Ga_K'': = \prod_{i\in K}\Ga_i = \prod_{\al\in \ell(K)} \Ga_\al^{m_{\al,K}}$ in the second.
As in Remark~\ref{rmk:GWext} 
   one can obtain $\Kk'$ from $\Kk''$ by
 simplifying the charts in $\Kk''$. 
 (This can be accomplished by either by forgetting all repetitions of the added marked points $\bw_\al$, so that the points in the domain only contain one such set instead of $m_{\al,K}$ of them,  or, equivalently, by restricting to those elements in the domain $U_K''$ of the standard GW chart for which the repetitions of $\bw_\al$ are identical (rather than being identical modulo order).)
However, more relevant  here is the fact  that one can build a concordance from the product to the standard atlas $\Kk''$ much as in  
Proposition~\ref{prop:hybrid2}, using the indexing set $\Ii^{01}$ and function $\tau^{01}$ described there, but 
this time using the set $\tau^{01}\bigl((I,K)\bigr)$ to define the group as well as the obstruction space for
the transition chart $\bK_{(I,K)}$.  Thus
 $\Ga''_{(I,K)} = \Ga_I\times \Ga_K''$ instead of $\Ga_{I\vee \ell(K)}$.
To make this possible  the tuples of added marked points in the elements of the domain  $U''_{(I,K)}$ should also be indexed by 
the elements of   $\tau^{01}\bigl((I,K)\bigr)$.   One would then have to check directly that the concordance  can be constructed  to be good.
  As suggested in Remark~\ref{rmk:good3}, this should be possible if one assumes that the initial atlases $\Kk_{p,q}, \Kk_{q,r}$ are preshrunk.  Further details are left for later.
\hfill$\er$
\end{rmk}

\NI {\bf Acknowledgements}  
I wish to thank Jingchen Niu, Robert Castellano and Yasha Savelyev for making some useful comments on earlier drafts of these notes, and the referee for a careful reading.
I also benefitted by conversations with John Pardon, Dominic Joyce and Eugene Lerman about orbifolds. Finally I wish to thank Katrin 
Wehrheim for contributing so many ideas to this collaboration and patiently helping to clarify many arguments.
Sections \S\ref{s:noiso} and ~\S\ref{s:iso} of these notes are taken very directly from the joint papers with Katrin Wehrheim.  However I am solely responsible for the exposition of these results as well as for the rest of these  notes. 
I wish to thank the Simons Center for Geometry and Physics for its hospitality 
during Spring 2014 while these notes were being prepared.

\end{document}